\documentclass[3p,10pt,nopreprintline]{elsarticle}

\usepackage{amsmath}
\usepackage{amsthm}
\usepackage{amssymb}
\usepackage{bm}
\usepackage{graphicx}
\usepackage{epsfig}
\usepackage[normal]{subfigure}
\usepackage{float}
\usepackage{paralist}
\usepackage{caption2}
\usepackage{hyperref}
\usepackage{aliascnt}
\usepackage[T1]{fontenc}
\usepackage{natbib}
\usepackage{CJK}
\usepackage{color}
\usepackage[table]{xcolor}
\usepackage{colortbl,booktabs}

\newtheoremstyle{theorem}{6pt}{6pt}{\itshape}{}{\bfseries}{.}{.5em}{}
\newtheoremstyle{definition}{6pt}{6pt}{\upshape}{}{\bfseries}{.}{.5em}{}

\theoremstyle{theorem}
\newtheorem{theorem}{Theorem}[section]

\newaliascnt{corollary}{theorem}
\newtheorem{corollary}[corollary]{Corollary}
\aliascntresetthe{corollary}

\newaliascnt{lemma}{theorem}
\newtheorem{lemma}[lemma]{Lemma}
\aliascntresetthe{lemma}

\newaliascnt{sublemma}{theorem}

\aliascntresetthe{sublemma}

\theoremstyle{definition}

\newtheorem{remark}{Remark}[section]

\newaliascnt{proposition}{theorem}
\newtheorem{proposition}[proposition]{Proposition}
\aliascntresetthe{proposition}

\newcommand{\R}{{\mathbb R}}

\newcommand{\dif}{{\mathrm d}}

\newcommand{\bn}{\begin{eqnarray}}
\newcommand{\en}{\end{eqnarray}}
\newcommand{\bnn}{\begin{eqnarray*}}
\newcommand{\enn}{\end{eqnarray*}}
\renewcommand{\div}{ {\rm div }  }

\newcommand{\rL}{L}
\newcommand{\T}{\mathbb{T}}

\newcommand{\al}{\alpha}
\newcommand{\curle}{\mathrm{curl}}
\newcommand{\frD}{\mathfrak{D}}
\newcommand{\M}{\mathcal{M}}
\newcommand{\calE}{\mathcal{E}}
\newcommand{\pt}{\partial_t}

\newcommand{\D}{\nabla}

\newcommand{\vr}{\bar{\rho}}

\usepackage{scalerel,stackengine}
\stackMath
\newcommand\reallywidehat[1]{%
\savestack{\tmpbox}{\stretchto{%
  \scaleto{%
    \scalerel*[\widthof{\ensuremath{#1}}]{\kern-.6pt\bigwedge\kern-.6pt}%
    {\rule[-\textheight/2]{1ex}{\textheight}}
  }{\textheight}%
}{0.5ex}}%
\stackon[1pt]{#1}{\tmpbox}%
}

\allowdisplaybreaks

\numberwithin{equation}{section}

\begin{document}

\begin{frontmatter}

\title{Global well-posedness and asymptotic behavior of large strong \\ solutions to the 3D full compressible Navier-Stokes equations \\ with temperature-dependent coefficients}

\author[label1]{Yachun Li}
\address[label1]{School of Mathematical Sciences, CMA-Shanghai, MOE-LSC, and SHL-MAC,\\ Shanghai Jiao Tong University, Shanghai 200240, P.R.China;}
\ead{ycli@sjtu.edu.cn}

\author[label2]{Peng Lu}
\address[label2]{Department of Mathematics, Zhejiang Sci-Tech University, Hangzhou, Zhejiang 310018, P.R.China;}
\ead{plu25@zstu.edu.cn}

\author[label3]{Zhaoyang Shang\corref{cor2}}
\address[label3]{School of Finance, Shanghai Lixin University of  Accounting and Finance, Shanghai 201209, P.R.China;}
\cortext[cor2]{Corresponding author. }
\ead{shangzhaoyang@sjtu.edu.cn}

\begin{abstract}
It is well known that the global well-posedness of the Navier-Stokes equations with temperature-dependent coefficients is a challenging problem, especially in multi-dimensional space. In this paper, we study the 3D Navier-Stokes equations with temperature-dependent coefficients in the whole space. For any fixed positive far-field density, we establish the global existence of large strong solutions provided that the far-field temperature is sufficiently high, with a threshold depending on the fixed density and the initial data. The solutions satisfy uniform-in-time bounds and remain uniformly positive. Moreover, for these solutions, the optimal decay rates of the solution to its associated equilibrium are established when the initial data belong to $L^{p_0}(\R^3)$ for some $p_0\in[1,2]$.
\end{abstract}


\end{frontmatter}

\section{Introduction}
	
The motion of a compressible viscous, heat-conductive, and Newtonian polytropic fluid is governed by the following full
compressible Navier-Stokes equations:
\begin{equation}\label{FCNS}
\begin{cases}
\pt\rho+\div(\rho u)=0,\\
\pt(\rho u)+\div(\rho u\otimes u)+\D P=\div\T,\\
\pt(\rho\calE)+\div(\rho\calE u+Pu)=\div(u\T)+\div(\tilde{\kappa}\D\theta),
\end{cases}
\end{equation}
where $x=(x_1,x_2,x_3)^\top\in\R^3$ and $t\geq 0$ denote the spatial coordinate and time coordinate, $\rho\geq 0$ the mass density, $u = (u^1, u^2, u^3)^\top$ the fluid velocity, $P$ the pressure of the fluid, $\theta$ the absolute temperature, $\calE=e+\frac{1}{2}|u|^2$ the specific total energy, $e$ the specific internal energy. The equation of state for polytropic fluids satisfies
\begin{equation}\label{Pecv}
P=R\rho\theta, \quad e=c_v\theta,
\end{equation}
where $R$ is the gas constant, and $c_v$ is the specific heat at
constant volume. $\T$ is the viscosity stress tensor given by $\T=2\tilde{\mu}\frD(u)+\tilde{\lambda}\div u\mathbb{I}_3$, where 
$$\frD(u) =\frac{\D u+(\D u)^\top}{2}$$
is the deformation tensor, $\mathbb{I}_3$ is the $3 \times 3$ identity matrix, $\tilde{\mu}$ is the shear viscosity coefficient, and $\tilde{\lambda}+\frac{2}{3}\tilde{\mu}$ is the bulk viscosity coefficient. $\tilde{\kappa}$ denotes the coefficient of thermal conductivity.
In this paper, we will consider the following  temperature-dependent transport coefficients:
\begin{equation}\label{temperature-dependent}
\tilde{\mu}(\theta)=\mu\theta^\al, \quad \tilde{\lambda}(\theta)=\lambda\theta^\al, \quad \tilde{\kappa}=\kappa\theta^\beta,
\end{equation}
where $\mu$, $\lambda$, $\al$ and $\beta$ are constants satisfying $\mu>0, 2\mu+3\lambda\geq 0, \al>0, \beta>0$.

When the viscosity coefficients are constants, there have been a lot of literature on the well-posedness theory for the full compressible Navier-Stokes equations \eqref{FCNS} in multi-dimensional space. The local existence and uniqueness of strong and classical solutions in bounded and unbounded domains were established in \cite{MR0283426, MR0149094,MR0106646} and for the global existence of weak solutions, we refer to \cite{MR4813232,MR4294284,MR1339675,MR1360077, MR1480244,MR2091508}. For the global well-posedness results of strong and smooth solutions in three-dimensional space, in 1979, Matsumura and Nishida \cite{MR555060} first demonstrated the global strong solutions with the initial data close to a non-vacuum equilibrium in Sobolev space, see also \cite{MR0564670, MR713680}. 
In 1996, Jiang \cite{MR1389908} considered the equations in the domain exterior to a ball and proved the global existence of spherically symmetric smooth solutions for large initial data with spherical symmetry. In 1998, Xin \cite{MR1488513} presented a sufficient condition on the blowup of smooth solution in arbitrary space dimensions with initial density of compact support. It is shown that any smooth solution for polytropic fluids in the absence of heat conduction will blow up in finite time as long as the initial density has compact support, see also Jiu, Wang and Xin \cite{MR3360663} for a simplified and unified proof. 
In 2013, Xin and Yan \cite{MR3063918} improved the blowup results of \cite{MR1488513} by removing the crucial assumptions that the initial density has compact support and the smooth solution has finite total energy. In 2017, Wen and Zhu \cite{MR3597161} investigated the Cauchy problem and showed that the strong solution exists globally in time if the initial mass is small for the fixed coefficients of viscosity $\mu$ and heat conductivity $\kappa$.  In the next year, Huang and Li \cite{MR3744381} proved the global existence of classical solution to the Cauchy problem with smooth initial data which are of small energy but possibly large oscillations, see also \cite{MR4492674} given by Lai, Xu and Zhang for the case that the far-field limit is equal to zero. 
In 2020, Li \cite{MR4097330} considered the Cauchy problem with far-field vacuum, and found a scaling invariant quantity
$$N_0=:\|\rho_0\|_{L^\infty}\left(\|\rho_0\|_{L^3}+\|\rho_0\|_{L^\infty}^2\|\sqrt{\rho_0}u_0\|_{L^2}^2\right)\left(\|\nabla u_0\|_{L^2}^2+\|\rho_0\|_{L^\infty}\|\sqrt{\rho_0}\theta_0\|_{L^2}^2\right),$$
with respect to the scaling transform $\left(\rho_{0,\lambda}(x), u_{0,\lambda}(x), \theta_{0,\lambda}(x)\right)=\left(\rho_0(\lambda x), \lambda u_0(\lambda x), \lambda^2\theta_0(\lambda x)\right)$. Under the assumptions that $2\mu>\lambda$ and $N_0$ is sufficiently small, the global well-posedness of strong solution is established. Very recently, Li, Li and L\"u \cite{MR4980411} extended the result of \cite{MR3744381} to the initial-boundary-value problem in an exterior domain, Wen \cite{MR4978813} extended the previous result of \cite{MR4097330} by removing the technical assumption $2\mu>\lambda$. For the cases of global well-posedness in 1D and 2D spaces, we refer to \cite{MR4924422, MR637519, MR651877, MR4346514, MR4039142, MR4491875, MR4740640}, etc., and the references cited therein.

However, the viscosity coefficients are not necessarily constants, they may depend on density or temperature, or both. We first review some well-posedness results when viscosities depend only on the density. In 1985, Kawohl \cite{MR791841} proved the global existence of large classical solution to initial-boundary value problem in one dimensional space when the transport coefficients satisfy 
\begin{equation}\label{temperature-dependent-2}
0<\mu_0\leq\tilde\mu(v)\leq \mu_1,\quad \kappa_0(1+\theta^q)\leq \tilde\kappa(v,\theta)\leq \kappa_1(1+\theta^q), \quad |\tilde\kappa_v(v,\theta)|+|\tilde\kappa_{vv}(v,\theta)|\leq \kappa_1(1+\theta^q),
\end{equation}
where $v$ is specific volume and $q\geq2$.
In 2016, Wang \cite{MR3461630} studied the initial and initial-boundary value problems  for the $p$-th power Newtonian fluid  and established the existence and uniqueness of global smooth non-vacuum solutions when the transport coefficients $\mu$ and $\kappa$ are given as 
\begin{equation}\label{temperature-dependent-4}
\tilde{\mu}(\rho)=\rho^\al, \quad \tilde{\kappa}(\theta)=\theta^\beta,
\end{equation}
with some positive parameters $\alpha$ and $\beta$.
In 2017, Duan, Guo and Zhu \cite{MR3603273} considered the initial-boundary value problem with the stress-free and heat insulated boundary condition, and established the global existence of strong solution  with  the initial density is away from vacuum and the following density-dependent viscosity and temperature-dependent heat conductivity 
\begin{equation}\label{temperature-dependent-3}
\tilde\mu(\rho)=1+\rho^\alpha, \quad\tilde\kappa(\theta)=\theta^\beta,\quad \alpha\geq 0,\quad \beta>0.
\end{equation}
For the three dimensional case, in 2007, Bresch and Desjard\^ins \cite{MR2297248} considered the Cauchy problem and periodic problem, and proved that the weak solutions exist globally in time under the condition
$$\tilde\lambda(\rho)=2(\tilde\mu'(\rho)\rho-\tilde\mu(\rho)),\quad \tilde\kappa(\rho,\theta)=\kappa_0(\rho,\theta)(1+\rho)(1+\theta^\al),\quad \al\geq 2.$$
In 2024, Li, Lu, Shang and Yu \cite{2024arXiv240805138L} studied the Cauchy problem with the following 
density-dependent viscosities
\begin{equation}\label{temperature-dependent-5}
\tilde{\mu}(\rho)=\rho^\al, \quad \tilde{\lambda}(\rho)=\rho^\al,\quad\tilde{\kappa}=constant>0,\quad \alpha>4.
\end{equation}
By using the modified effective viscous flux and using the bootstrap
argument, they established the global existence of large strong solution when the initial density is linearly equivalent to a large constant state. Very recently, Shen, Xu and Zhang \cite{MR4858604} considered the Dirichlet initial-boundary value problem with density-dependent viscosities $\tilde{\mu}(\rho), \tilde{\lambda}(\rho)\in C^1$. The global well-posedness of strong solutions with vacuum is established provided that the initial total energy $\|\sqrt{\rho_0}u_0\|_{L^2}^2+\|\rho_0\theta_0\|_{L^1}$ is suitably small.


In the above, we introduced some well-posedness results on the density-dependent viscosity coefficients. In fact, the system (\ref{FCNS}) can be deduced from the Boltzmann equation through the Chapman-Enskog expansion to the second order, under some proper physical assumptions, the transport coefficients are power functions of temperature of the form (\ref{temperature-dependent}). On the other hand, experimental evidence \cite{MR791841,10.1115/1.3607836} pointed out that the transport coefficients usually depend on  temperature, in particular, they increase as (\ref{temperature-dependent}) when the temperature is high. In this case, it is challenging to prove the global well-posedness of strong or smooth solutions to the nonisentropic compressible Navier–Stokes equations, most known results are for one-dimensional case. In 2014, Liu, Yang, Zhao and Zou \cite{MR3225502} considered the Cauchy problem, and first obtained the existence and uniqueness of a globally smooth nonvacuum solution provided that 
$$(\gamma-1)\|(v_0-1,u_0,s_0-1)\|_{H^3}\leq C,$$
where $\gamma$ is the adiabatic exponent. When $\gamma$ is close to 1, the initial norm $\|(v_0-1,u_0,s_0-1)\|_{H^3}$ can be large. In 2016, Wang and Zhao \cite{MR3564590} considered the case that the transport coefficients are given by
\begin{equation}\label{Intro:2.3}
\tilde\mu=\mu_0h(v)\theta^\alpha, \quad \tilde\kappa=\kappa_0h(v)\theta^\alpha,
\end{equation}
where $\mu_0>0$, $\kappa_0>0$ and $\alpha$ are constants, $h(v)$ is a smooth function of the specific volume $v$ satisfying
\begin{equation}\label{Intro:2.4}
v^{l_1}+v^{-l_2}\leq Ch(v), \quad h'(v)^2v\leq Ch(v)^3, \quad\text{for all} \quad v\in(0, +\infty),
\end{equation}
for some $C>0$ and $l_1\geq 3$, $l_2\geq 3$, they proved the existence and uniqueness of the global-in-time classical solution to the Cauchy problem when the exponent $|\alpha|$ is sufficiently small.
In 2021, Sun, Zhang and Zhao \cite{MR4235250} considered the case that the transport coefficients are given by
\begin{equation}\label{Intro:2.5}
\tilde\mu(\theta)=\mu_0\theta^\alpha, \quad \tilde\kappa(\theta)=\kappa_0\theta^\beta,
\end{equation}
where $\mu_0$ and $\kappa_0$ are positive constants and $\alpha\geq 0$, $\beta \geq 0$. They proved the global-in-time existence of strong solutions to the initial-boundary value problem when $\alpha$ is sufficiently small.
Recently, the Cauchy problem was investigated by Dong and Guo in \cite{MR4921984}.

In three-dimensional space, there are rich literature of global weak solutions with large initial data. In 2012, Feireisl, Mucha, Novotn\'y and Pokorn\'y \cite{MR2917121} proved the existence of at least one time-periodic weak solution to the initial-boundary value problem,
where the transport coefficients are supposed to be continuously differentiable functions of $\theta$ satisfying
\begin{equation}\label{tpc-0}
\underline{\mu}(1+\theta)\leq\tilde{\mu}(\theta),\quad |\tilde{\mu}^{\prime}(\theta)|\leq C, \quad 0\leq\tilde{\lambda}(\theta)\leq \overline{\lambda}(1+\theta),
\end{equation}
\begin{equation}\label{tpc-00}
\underline{\kappa}(1+\theta^3)\leq\tilde{\kappa}(\theta)\leq \overline{\kappa}(1+\theta^3),
\end{equation}
where $\underline{\mu}$, $\overline{\lambda}$, $\underline{\kappa}$ and $\overline{\kappa}$ are positive constants. Feireisl and Novotn\'y \cite{MR2909912} proved the weak–strong uniqueness of initial-boundary value problem with the transport coefficients satisfying
$$\tilde{\mu}(\theta)=\mu_0+\mu_1\theta, \quad \tilde{\lambda}(\theta)=0, \quad \tilde{\kappa}(\theta)=\kappa_0+\kappa_2\theta^2+\kappa_3\theta^3,$$
for some positive constants $\mu_0$, $\mu_1$, $\kappa_0$, $\kappa_2$, $\kappa_3$.
In 2022, Chaudhuri and Feireisl \cite{MR4469007} considered general non-homogeneous Dirichlet boundary conditions for $u$ and $\theta$, and introduced a new concept of weak solution satisfying the entropy and energy inequalities. When the transport coefficients are supposed to be continuously differentiable functions of $\theta$ satisfying 
\begin{equation}\label{tpc}
\underline{\mu}(1+\theta^\al)\leq\tilde{\mu}(\theta)\leq \overline{\mu}(1+\theta^\al),\quad |\tilde{\mu}^{\prime}(\theta)|\leq C, \quad 0\leq\tilde{\lambda}(\theta)\leq \overline{\lambda}(1+\theta^\al),
\end{equation}
\begin{equation}\label{tpc-1}
\underline{\kappa}(1+\theta^\beta)\leq\tilde{\kappa}(\theta)\leq \overline{\kappa}(1+\theta^\beta).
\end{equation}
for $\frac{1}{2}\leq\al\leq 1$ and $\beta=3$, they established the weak–strong uniqueness and the existence of global-in-time weak solutions. In 2024, Feireisl, Lu and Sun \cite{MR4810457} studied the initial-boundary value problem with $\al=1$ and $\beta>6$ in \eqref{tpc}--\eqref{tpc-1}, and proved the global stability of an equilibrium with small perturbations. For the well-posedness results of strong solutions, in 2020, Yu and Zhang \cite{MR4079010} studied the initial-boundary value problem in a smooth bounded domain $\Omega$ with Dirichlet boundary condition
$$(u, \theta)(x,t)=(0,0), \quad \text{on} \quad \partial\Omega\times (0,T),$$
and density–temperature–dependent viscosities
$$\tilde\mu(\rho,\theta), \tilde\lambda(\rho,\theta)\in C^1(\R^2), \quad \tilde{\mu}(\rho, \theta)>0, \quad 2\tilde{\mu}(\rho, \theta)+3\tilde{\lambda}(\rho, \theta)\geq 0,$$
and proved the global existence of strong solution provided that $\|\nabla u_0\|_{L^2}^2+\|\nabla \theta_0\|_{L^2}^2$ is sufficiently small. In 2022, Cao and Li \cite{MR4496704} established the local well-posedness of strong solution to the initial-boundary value problem with the transport coefficients are given in \eqref{temperature-dependent} with $\alpha, \beta\geq 0$, where the initial density is allowed to contain vacuum. In 2024, Hou, Liu, Wang and Xu \cite{MR4803458} studied the vanishing viscosity limit for Navier–Stokes equations with temperature-dependent transport coefficients satisfying \eqref{temperature-dependent} for $\alpha=\beta>0$.

From the above mentioned known results, we can see that there are already some literature concerning the well-posedness of nonisentropic compressible Navier-Stokes equations with variable coefficients. However, compared to the case of constant coefficients and density-dependent viscosities, there are much fewer results with temperature-dependent transport coefficients, which are more in line with physical experiments, especially in multi-dimensional case. In this paper, we are interested in investigating the global well-posedness and asymptotic behavior of strong solution to the system \eqref{FCNS} with temperature-dependent transport coefficients and large initial data in three-dimensional space.  Without loss of generality, we assume the constant $c_v=1$ in \eqref{Pecv}. It follows from \eqref{Pecv}--\eqref{temperature-dependent} that \eqref{FCNS} can be rewritten as
Throughout the following equations and estimates, we take $c_v=1$ in \eqref{Pecv}.
\begin{equation}\label{FCNS:2}
\begin{cases}
\pt\rho+\div(\rho u)=0,\\
\rho(\pt u+(u\cdot\D)u)+\D P=2\mu\div(\theta^\al\frD(u)) +\lambda\D(\theta^\al\div u),\\
\rho(\pt\theta+u\cdot\D\theta)+P\div u=2\mu\theta^\al |\frD(u)|^2+\lambda\theta^\al(\div u)^2+\kappa\div(\theta^\beta\D\theta).
\end{cases}
\end{equation}
When system \eqref{FCNS:2} is supplemented with the following initial data and far field behavior
\begin{equation}\label{data and far}
\begin{aligned}
&(\rho, u, \theta)(x,0)=(\rho_0, u_0, \theta_0)(x) \quad \text{for} \quad x\in\R^3, \\
& (\rho , u, \theta)(t, x) \to \left(\bar{\rho}, 0, \bar{\theta}\right) \quad \text{as} \quad |x|\to\infty \quad \text{for} \quad t \geq 0,
\end{aligned}
\end{equation}
where $\bar{\rho}$ and $\bar{\theta}$ are positive constants, we have our main theorem as follows.
\begin{theorem}\label{th}$($Global existence of strong solutions$)$ 
Let the positive constants $\bar{\rho}, \bar{\theta}>1$, $2\leq\alpha<\beta<\alpha+2$, and $4\alpha-3\beta<4$. Suppose that the initial data $(\rho_0, u_0, \theta_0)$ satisfy
\begin{equation}\label{data:condition}
\left(\rho_0(x)-\bar{\rho}, ~ u_0(x), ~ \theta_0(x)-\bar{\theta}\right)\in H^2,
\end{equation}
\begin{equation}
\|\rho_0-\bar{\rho}\|_{H^2}+\|u_0\|_{H^2}+\|\theta_0-\bar{\theta}\|_{H^2}\leq C_0,
\end{equation}
\begin{equation}\label{compatibility condition}
\frac{3}{4}\bar{\rho}\leq\rho_0(x)\leq \frac{5}{4}\bar{\rho}, \quad \frac{3}{4}\bar{\theta}\leq\theta_0(x)\leq\frac{5}{4}\bar{\theta},
\end{equation}
for some positive constant $C_0$ independent of $\bar{\rho}$ and $\bar{\theta}$. Then there exist positive constants $\rL$ and $\eta$ independent of $\bar{\rho}$ and $\bar{\theta}$, such that if $\bar{\rho}\geq\rL$ and $\bar{\theta}\geq\bar{\rho}^{\eta}$,  the Cauchy problem \eqref{FCNS:2}--\eqref{data and far} has a global strong solution $(\rho,u,\theta)$ satisfying
\begin{equation}\label{sp}
\begin{aligned}
& \big(\rho-\bar{\rho}, u, \theta-\bar{\theta}\big)\in C([0,\infty);H^2), \quad \rho_t\in L^\infty(0,\infty;H^1), \\[6pt]
& ({u}_t,\theta_t)\in L^\infty(0,\infty;L^2)\cap L^2(0,\infty;H^1), 
\end{aligned}
\end{equation}  
and
\begin{equation}\label{bc}
\frac{2}{3}\bar{\rho}\leq\rho(x,t)\leq \frac{4}{3}\bar{\rho}, \quad \frac{2}{3}\bar{\theta}\leq\theta(x,t)\leq\frac{4}{3}\bar{\theta}, \quad \text{for} \quad (x,t)\in\R^3\times(0,\infty).
\end{equation}
\end{theorem}
\begin{remark}
    We would like to mention that Theorem \ref{th} is the first result concerning the global existence of strong solution  for the Cauchy problem to the 3D full Navier-Stokes equation with temperature-dependent coefficients and large initial data. Our global strong solution satisfying (\ref{sp}) is bounded uniformly in time. In particular, the uniform-in-time lower and upper bounds of the density and temperature are also given in (\ref{bc}).
\end{remark}


\begin{remark}
In Theorem \ref{th}, we consider the transport coefficients depending on the temperature (\ref{temperature-dependent}). This is certainly a restriction which is physically motivated, experimental evidence \cite{MR791841,10.1115/1.3607836} pointed out that the coefficients of viscosity and heat
conductivity usually depend on temperature. In particular, the transport coefficients increase such as (\ref{temperature-dependent}) when the temperature is high, which is highly consistent with our assumption that the initial temperature is linearly equivalent to some large constant states.
The restrictions upon the parameters $2\leq\alpha\leq 10$,  $\alpha<\beta<\alpha+2$ are technical requirements. In fact, they are physically reasonable. For example,  when intermolecular potential varies as $r^{-b}$ with $r$ being the molecule distance, see \cite{MR258399}, the parameter range of $\alpha$ ($\alpha=\frac{1}{2}+\frac{2}{b}, b\in[\frac{4}{19},\frac{4}{3}]$), especially for ionized gas ($\alpha=\frac{5}{2}, b=1$), is included in our results. It should be noted that the above  restrictions upon the parameters $\alpha,\beta$ are not sharp, it would be interesting to study the problem \eqref{FCNS:2}--\eqref{data and far} when $0<\alpha<2$ and $\beta\leq\alpha$. This is left for interested readers.
\end{remark}

Moreover, our global strong solution can be a classical one.
\begin{corollary}\label{co}$($Global existence of classical solutions$)$ 
Under the assumptions in Theorem \ref{th}, suppose that $\left(\rho_0(x)-\bar{\rho}, ~ u_0(x), ~ \theta_0(x)-\bar{\theta}\right)\in H^3$, then there exist positive constants $\rL$ and $\eta$ independent of $\bar{\rho}$ and $\bar{\theta}$, such that if $\bar{\rho}\geq\rL$ and $\bar{\theta}\geq\bar{\rho}^{\eta}$,  the Cauchy problem \eqref{FCNS:2}--\eqref{data and far} has a global solution $(\rho,u,\theta)$ satisfying
\begin{equation}\label{sp'}
\begin{aligned}
& \big(\rho-\bar{\rho}, u, \theta-\bar{\theta}\big)\in C([0,\infty);H^3), \quad \rho_t\in L^\infty(0,\infty;H^2), \\[6pt]
& ({u}_t,\theta_t)\in L^\infty(0,\infty;H^1)\cap L^2(0,\infty;H^2), 
\end{aligned}
\end{equation}  
and
\begin{equation}\label{bc'}
\frac{2}{3}\bar{\rho}\leq\rho(x,t)\leq \frac{4}{3}\bar{\rho}, \quad \frac{2}{3}\bar{\theta}\leq\theta(x,t)\leq\frac{4}{3}\bar{\theta}, \quad \text{for} \quad (x,t)\in\R^3\times(0,\infty).
\end{equation}
\end{corollary}

The largeness assumption on the far-field density can be removed if the temperature threshold is allowed to depend on a fixed positive density. More precisely, we have the following strengthened result.
\begin{theorem}\label{th:fixed}$($Global strong solutions for a fixed positive far-field density$)$
Let $\bar{\rho}>0$ be arbitrary and fixed, and suppose that $2\leq\alpha<\beta<\alpha+2$ and $4\alpha-3\beta<4$. For any $C_0>0$, there exists a constant
\begin{equation}\label{fixed:threshold}
C^*=C^*(\bar{\rho},C_0,\mu,\lambda,\kappa,R,\alpha,\beta)\geq1
\end{equation}
such that, if $\bar{\theta}\geq C^*$ and the initial data satisfy \eqref{data:condition}, \eqref{compatibility condition} and
\[
\|\rho_0-\bar{\rho}\|_{H^2}+\|u_0\|_{H^2}+\|\theta_0-\bar{\theta}\|_{H^2}\leq C_0,
\]
then the Cauchy problem \eqref{FCNS:2}--\eqref{data and far} has a unique global strong solution satisfying \eqref{sp} and \eqref{bc}. Moreover,
\begin{equation}\label{fixed:H2-bound}
\sup_{t\geq0}\left(\|\rho-\bar{\rho}\|_{H^2}^2+\|u\|_{H^2}^2+\|\theta-\bar{\theta}\|_{H^2}^2\right)
\leq C(\bar{\rho},\bar{\theta}).
\end{equation}
Here the constant is independent of time. No largeness assumption is imposed on $\bar{\rho}$, and $C^*$ is not required to be uniform as $\bar{\rho}\to0$.
\end{theorem}
\begin{remark}
The proof of Theorem \ref{th:fixed} is given in \S\ref{sec:fixed-density}. The density gradient bounds are derived after the velocity and temperature estimates, using a mass-weighted density estimate and a logarithmic comparison in the time variable $\tau=\bar{\theta}^{1-\alpha}t$. The decay rates in Theorem \ref{th2} also hold for these fixed-density solutions, without any further restriction on the far-field temperature. The classical regularity result in Corollary \ref{co} retains its assumptions from Theorem \ref{th}.
\end{remark}

\begin{remark}
Under the framework of this paper, the requirement on the lower bounds of $\bar{\rho}$ and $\bar{\theta}$ can be removed when the initial data satisfy small energy assumption. As a consequence, we extend the results in \cite{MR3744381, MR4492674} with small initial energy to the temperature-dependent coefficients case.
\end{remark}

\bigskip

Now, we turn to the large time behavior of solutions. We first review some related known results. When the transport coefficients are constants, we refer to \cite{MR1675129, MR555060, MR0564670, MR0785713} and \cite{Gao2021De, MR4389852, MR4493879, MR4613439, MR2562709, Zhang2020Con} for decay rates of solutions with or without smallness restrictions on the initial perturbation, respectively.

However, when the transport coefficients depend on the temperature, there are much fewer results about the large time behavior of global solutions. For the case of weak solutions, in 2007, Feireisl and Petzeltov\'a \cite{MR2350243} considered the equation \eqref{FCNS} in a three dimensional bounded domain $\Omega$ with the transport coefficients satisfying \eqref{tpc-0}--\eqref{tpc-00}. When the system is supplied with an external force $\rho\D F$ for some $F(x)\in W^{1,\infty}(\Omega)$, they proved that the global weak solution $(\rho, \rho u)$ converges to a non-constant equilibrium $(\bar{\rho}(x), 0)$ in $L^1(\Omega)$ as time goes to infinity, and 
$$\lim\limits_{t\to+\infty}\int_\Omega\rho s(\rho, \theta)\dif x=\int\bar{\rho} s(\bar{\rho}, \bar{\theta})\dif x,$$
where $s(\rho, \theta)$ is the entropy, and the equilibrium $(\bar{\rho}(x), \bar{\theta}(x))$ solves the equation $\D P(\bar{\rho}, \bar{\theta})=\bar{\rho}\D F$. In 2013, Bella, Feireisl and Pra\v z\'ak \cite{MR3054636} considered the same problem when the driving force takes the form $t^a\omega(t^b)\rho \bm{F}(x)$ with $\omega\in L^\infty(\R)$, $\bm{F}\in (W^{1,\infty}(\Omega))^3$ and $b>a+2+\max\{a, 0\}$. They proved that the global weak solution $(\rho, \rho u, \theta)$ satisfies
$$\lim\limits_{t\to+\infty}\left(\|\rho-\bar{\rho}\|_{L^{\frac{5}{3}}}+\|\rho u\|_{L^{\frac{5}{4}}}+\|\theta-\bar{\theta}\|_{L^4}\right)=0,$$
for some constant equilibrium $(\bar{\rho}, 0, \bar{\theta})$. In 2024, Chiodaroli and Feireisl \cite{MR4846859} considered the three dimensional Cauchy problem with far-field limit $(0, 0, \bar{\theta})$, where the viscosity coefficients satisfy \eqref{tpc-0}, and the heat conductivity coefficient satisfies
$$\tilde{\kappa}(\theta)=\kappa_0\theta+\kappa_1\theta^3.$$
They proved that any global weak solutions approaches the equilibrium in the sense that
$$\lim\limits_{T\to+\infty}\frac{1}{T}\int_1^T\left(\|\rho\|_{L^{\frac{5}{3}}}^{\frac{5}{3}}+\|u\|_{H^1}^2+\|\theta-\bar{\theta}\|_{H^1}^{2}\right)\dif t=0.$$
The results of large time behavior for global strong or smooth solutions mainly focus on the one-dimensional and spherically or cylindrically symmetric situations. In 2014, Liu, Yang, Zhao and Zou \cite{MR3225502} proved that the global classical solution $(v, u, \theta)$ of one-dimensional Cauchy problem with $\tilde{\mu}(\theta)>0, ~ \tilde{\kappa}(\theta)>0$ satisfies
$$\lim\limits_{t\to+\infty}\|(v-\bar{v}, u, \theta-\bar{\theta})\|_{L^\infty}=0,$$
similar results were obtained by Wang and Zhao \cite{MR3564590} in 2016 with the assumptions \eqref{Intro:2.3}--\eqref{Intro:2.4}.
In 2020, Yu and Zhang \cite{MR4079010} proved that the global strong solution $(\rho, u, \theta)$ of three-dimensional initial boundary value problem with $\tilde{\mu}(\rho, \theta)>0, ~ \tilde{\kappa}(\rho, \theta)>0$ satisfies
\begin{equation*}
 \|\sqrt{\rho}u\|_{L^2}^2+\|\sqrt{\rho}\theta\|_{L^2}^2\leq~ Ce^{-Ct}, \qquad \|\sqrt{\rho}u_t\|_{L^2}^2+\|\sqrt{\rho}\theta_t\|_{L^2}^2+\|\D u\|_{L^2}^2+\|\D\theta\|_{L^2}^2\leq~ C(1+t)^{-2}.
\end{equation*}
In \cite{MR4235250}, Sun, Zhang and Zhao studied the global strong solution of one-dimensional initial boundary value problem with \eqref{Intro:2.5}, and proved that the $H^1$-norm of the solution decays exponentially. Very recently, Dong and Guo \cite{MR4921984} studied the one-dimensional Cauchy problem with \eqref{Intro:2.5}, and proved that the global strong solution $(v, u, \theta)$ satisfies
$$\lim\limits_{t\to+\infty}\left(\|(v-\bar{v}, u, \theta-\bar{\theta})\|_{L^q}+\|(v_x, u_x, \theta_x)\|_{H^1}\right)=0,$$
for any $2<q\leq\infty$. However, to the authors' best knowledge, there are few results concerning the optimal decay rate for the Cauchy problem of \eqref{FCNS:2}--\eqref{data and far} with temperature-dependent coefficients.

Motivated by the above results, in the following, we give our second result concerning the algebraic decay rates of the global solutions obtained in Theorem \ref{th} or Theorem \ref{th:fixed}. In particular, the decay result applies to any fixed positive far-field density when only the far-field temperature is sufficiently high.
\begin{theorem}\label{th2}$($Decay rates of the solution$)$ 
Let $(\rho, u, \theta)$ be the global strong solution of \eqref{FCNS:2}--\eqref{data and far} obtained in Theorem \ref{th} or Theorem \ref{th:fixed}. Suppose that $(\rho_0-\bar{\rho}, u_0, \theta_0-\bar{\theta})\in L^{p_0}(\R^3)$ for some $p_0\in[1,2]$, then we have
\begin{equation}\label{decay-0}
\big\|\big(\rho-\bar{\rho}, u, \theta-\bar{\theta}\big)\big\|_{L^2}\leq C(1+t)^{-\frac{3}{4}(\frac{2}{p_0}-1)},
\end{equation}
\begin{equation}\label{decay-1}
\|\left(\D\rho, \D u, \D\theta\right)\|_{L^2}+\|(\rho_t, u_t, \theta_t)\|_{L^2}\leq C(1+t)^{-\frac{3}{4}(\frac{2}{p_0}-1)-\frac{1}{2}},
\end{equation}
\begin{equation}\label{decay-2}
\|\left(\D^2\rho, \D^2u, \D^2\theta\right)\|_{L^2}\leq C(1+t)^{-\frac{3}{4}(\frac{2}{p_0}-1)-1},
\end{equation}
where the constant $C$ depends on the fixed far-field state, the physical parameters, $\alpha$, $\beta$, $p_0$ and the initial data. Moreover, also at the endpoint $p_0=2$,
\begin{equation}\label{decay:strong-convergence}
\lim_{t\to\infty}\| (\rho-\bar\rho,u,\theta-\bar\theta)(t)\|_{H^2}=0.
\end{equation}
\end{theorem}
\begin{remark}
The constants in the decay estimates are not required to be uniform as $\bar\rho\to0$ or $\bar\theta\to\infty$. When $p_0=2$, \eqref{decay-0} gives a uniform bound, while \eqref{decay:strong-convergence} gives convergence without a prescribed algebraic rate for the zero-order norm. A positive algebraic zero-order rate generally requires additional low-frequency information on the initial data, as is already the case for the linear heat equation.
\end{remark}

\begin{remark}
In the previous works \cite{MR3054636, MR4846859, MR4921984, MR2350243, MR3703560, MR3225502, MR3624545, MR3564590}, the authors established the long-time behavior of solutions to full compressible Navier-Stokes equations with temperature-dependent coefficients. More precisely, they proved that the solution converges to an equilibrium in some $L^p$ and $H^2$ spaces as time goes to infinity. However, the explicit decay rates were not given. In this paper, we establish the decay rates of the first-order and the second-order derivatives of large strong solution. Moreover, if $(\rho_0-\bar{\rho}, u_0, \theta_0-\bar{\theta})\in L^{p_0}$ for some $p_0\in[1,2)$, additionally, we obtain the decay rates of the solution itself. The decay estimates are sharp in the sense that they are the same as those of the corresponding linear problem. As a result, we improve the results in \cite{MR3054636, MR4846859, MR4921984, MR2350243, MR3703560, MR3225502, MR3624545, MR3564590}. As far as we know, this is the first result concerning the decay rates of full compressible Navier-Stokes equations with temperature-dependent coefficients in the whole space.
\end{remark}

We now make some comments on the analysis of this paper.
In this paper, based on \cite{2024arXiv240805138L}, we continue to study the global well-posedenss of strong solutions to 3D nonisentropic Naiver-Stokes equations with temperature-dependent coefficients and large initial data. The main difficulties come from the dependence on temperature of transport coefficients, the strong coupling of density, velocity and temperature, and the strong nonlinearity arising therefrom, as well as the lack of smallness of initial data. First, due to the dependence of the transport coefficients on the temperature, the key points here are to derive the a priori estimates on the temperature related nonlinear terms. It is worth emphasizing that, in order to close the a priori estimates, there must be some restrictions on the upper bounds of the parameters $\alpha$ and $\beta$ due to the high complexity of the estimates, although the larger the $\alpha$ and $\beta$ are, the stronger the dissipation are. This is different from the density-dependent viscosities case \cite{2024arXiv240805138L} in which the parameter $\alpha$ can be arbitrarily large. Second, bootstrap argument plays an important role in our analysis. Due to the strong interaction between temperature, velocity and density, the upper bounds of $(\nabla u, \nabla\theta)$ in $H^1$-norm and $\nabla\rho$ in $L^2\cap L^4$-norm have to be predetermined. In fact, these bounds imply that of all the norms required by the solution space, and can be controlled by some $(\bar{\rho},\bar{\theta})$-dependent constants which depend on each other and are not easy to be determined. Third, based on the uniform-in-time global estimates of solutions, we apply the Fourier analysis and energy method in frequency space, the optimal decay rates of the solution to its associated equilibrium in $H^2$ Sobolev space are established when the initial data belong to $L^{p_0}(\R^3)$ for some $p_0\in[1,2]$. More precisely,
\begin{itemize}
\item Uniform-in-time estimates of velocity and temperature.

In the lower-order estimates, we first give an entropy-type estimate, by which the $L^2$-estimates of the solution are obtained. Here, the estimates are uniform-in-time and depend on the far-field state $(\bar{\rho}, \bar{\theta})$ of density and temperature, which will play a crucial role in our bootstrap arguments and the asymptotic behavior. In the first-order estimates, since the viscosities depend on the temperature, we have to deal with some extra nonlinear terms, for example,
$$\al\int\theta^{\al-1}\theta_t\left(2\mu|\frD(u)|^2+\lambda (\div u)^2\right)\dif x,$$
which needs the $L^2$-norms of $\theta_t$ and $|\nabla u|^2$ to get controlled. The $L^2$-norm of $\theta_t$ can be controlled by the bootstrap assumptions. For the $L^2$-norm of $|\nabla u|^2$, we introduce a new effective viscous flux
$$G=(2\mu+\lambda)\div u-R(\rho\theta^{1-\al}-\bar{\rho}\bar{\theta}^{1-\al}),$$
which is different from the cases of constant coefficients \cite{MR3744381} and density-dependent viscosities \cite{2024arXiv240805138L}. By the $L^p$ elliptic estimate, we succeed in deriving some new uniform estimates on the key norms $\|\D u\|_{L^6}$ and $\|\D G\|_{L^p}$ for $p\in[2,6]$. Noticing that the $H^2$-norm of the temperature is needed in the above estimates, and the $H^2$-estimate of the temperature can not be obtained directly due to the temperature-dependence of thermal conductivity. Because of this, we introduce an additional bootstrap assumption
\begin{equation}\label{2dt}
\|\nabla^2\theta\|_{L^2}\leq 2\bar{\theta},
\end{equation}
which can be verified by using the bootstrap assumption \eqref{assume:1}. In the second-order estimates, the key step is to estimate the upper bound of $\|\sqrt{\rho}\dot u\|_{L^2}$ and $\|\sqrt{\rho}\dot\theta\|_{L^2}$. Following the argument of Hoff \cite{MR1339675} for the constant coefficients case, we find that $\|\sqrt{\rho}\dot u\|_{L^2}$ and $\|\sqrt{\rho}\dot \theta\|_{L^2}$ may be large due to the large initial data and the temperature-dependence of the transport coefficients. Based on the bootstrap assumptions \eqref{assume:1}--\eqref{assume:2} as well as \eqref{2dt}, we can control the large terms by the upper bounds of the initial data and close the a priori estimates by choosing $\bar{\theta}$ sufficiently large.






\item Uniform-in-time estimates of the density.

In the process of the first-order estimates of the density $\|\nabla \rho\|_{L^2\cap L^4}$, we have to estimate the key term $\|\nabla u\|_{L^\infty}$:
\begin{equation}\label{xxx}
\frac{\dif}{\dif t}\|\D\rho\|_{L^r}^{r}+\frac{R}{2\mu+\lambda}\int\rho{\theta}^{1-\al}|\D\rho|^{r}\dif x\leq C\|\D u\|_{L^\infty}\|\D \rho\|_{L^r}^{r}+\cdots, \quad r=2,4.
\end{equation}
It should be noted that in classical references, people usually use the following Beale-Kato-Majda type inequality to control $\|\nabla u\|_{L^\infty}$,
$$\|\nabla u\|_{L^\infty}\leq C\left(1+\|\nabla u\|_{L^2}+(\|\text{div} u\|_{L^\infty}+\|\nabla\times u\|_{L^\infty})\ln(e+\|\nabla^2 u\|_{L^q})\right),\quad 3<q<+\infty.$$
A direct use of this inequality on $[0,T]$ gives constants depending on time. For the parameter-explicit estimates in Theorem \ref{th}, we first combine interpolation with elliptic estimates:
$$\|\nabla u\|_{L^\infty}\leq C\|\nabla u\|_{L^2}^{\frac{1}{7}}\left(\bar\rho\bar{\theta}^{-\al}\|\dot{u}\|_{L^4}+\bar{\theta}^{-7}\|\D\theta\|_{L^4}^7\|\D u\|_{L^2}
+\bar{\theta}^{1-\al}\|\D\rho\|_{L^4}+\bar\rho\bar{\theta}^{-\al}\|\D\theta\|_{L^2}^{\frac{1}{4}}\|\D^2\theta\|_{L^2}^{\frac{3}{4}}\right)^{\frac{6}{7}},$$
where the key term $\|\dot{u}\|_{L^4}$ can be controlled by $\|\sqrt{\rho}\dot u\|_{L^2}$ and $\|\nabla\dot u\|_{L^2}$ thanks to the interpolation. When the upper bounds of $\|\sqrt{\rho}\dot u\|_{L^2}$ and $\|\nabla\dot u\|_{L^2}$ are too large, we cannot use bootstrap argument to close our a priori assumptions directly. On this occasion, we divide our estimates with respect to time into two parts $(0,\sigma(T))$ and $(\sigma(T),T)$. In the time interval $(0,\sigma(T))$, estimates on the first-order and the second-order spatial derivatives of $(u,\theta)$ (see $A_i(\sigma(T))(i=1,\cdots,4)$), with elaborate analysis on powers of $\bar{\theta}$, can be obtained provided that $\bar{\theta}$ is suitably large, and all these estimates can be controlled by their corresponding initial data, respectively. Then, in the time interval $(\sigma(T), T)$, the time-weighted estimates on the first-order and the second-order spatial derivatives of $(u,\theta)$ (see $B_i(T)(i=1,\cdots,4)$) can be established in the same way. Since $B_i(0)=0$ and $A_i(0)$ is large, the upper bounds of $B_i(T)$ are better (smaller) than those of $A_i(\sigma(T))$. With estimates of $A_i(\sigma(T))$ and $B_i(T)$ at hand, we can obtain the estimates on the first-order spatial derivatives of density (see $C_1(T)$ and $C_2(T)$ in Lemma \ref{Lem39}). Different from the density-dependent viscosities case \cite{2024arXiv240805138L}, when the transport coefficients depend on temperature and for the high temperature fluid, the largeness of $\bar\theta$ is not enough to control the possible growth of $\|\nabla \rho\|_{L^2\cap L^4}$ due to the pressure term. Combining the mass conservation equation and the new effective viscous flux, we obtain a dissipative property of density (see the term $\int\rho\theta^{1-\alpha}|\nabla\rho|^r\dif x$ on the left-hand side of \eqref{xxx}) by which we hope to control the growth of $\|\nabla\rho\|_{L^2\cap L^4}$. Since the temperature is high and $\alpha>1$, this density dissipation is weak. The direct interpolation argument therefore requires a largeness condition on $\bar{\rho}$ to control the pressure-related terms. 
At last, together with $L^2$-norm estimate of $\rho-\vr$ and Gagliardo–Nirenberg inequality, the lower and upper bounds of the density can be obtained, which depend only on $\bar\rho$. Moreover, we also establish the second-order time-independent estimates of the density, which will be used in the decay estimates of the solution.

To prove Theorem \ref{th:fixed}, we remove the dependence of the velocity and temperature estimates on the density gradients. This is achieved by estimating $\nabla u$ directly from the constant-coefficient Lam\'e operator and by using the dissipation of the mass-weighted quantity $\int\rho|\rho-\bar{\rho}|^6\dif x$. For the density gradients, we introduce $\tau=\bar{\theta}^{1-\alpha}t$ and retain the exponential decay of the initial density perturbation in the characteristic formula. A logarithmic comparison then yields time-independent gradient bounds: the comparison constant is fixed first, and the temperature is chosen afterwards. Finally, the square integrability of $\|\nabla u\|_{L^\infty}$, together with the density damping, gives the uniform second-order estimates. This argument does not require the far-field density to be large.

\item Optimal decay rates of large strong solution.

Based on the global existence of solutions established above, we consider the optimal decay rates of the solution to its far-field state. The $L^2$-optimal decay rates of the solution and its derivatives can be obtained by using Fourier time-frequency splitting method proposed in \cite{MR775190} (see also \cite{Gao2021De, 2024arXiv240805138L}). To this end, we need to derive some energy inequalities of the type
\begin{equation}\label{EC}
\frac{\dif}{\dif t}\calE(t)+\mathcal{D}(t)\leq C\mathcal{N}(t)\mathcal{D}(t),
\end{equation}
where $\calE(t)$ is some energy functional, $\mathcal{D}(t)$ is the corresponding dissipative energy, and $\mathcal{N}(t)$ consists of some terms that converges to zero as time goes to infinity. First, we take $\calE(t)$ equivalent to $\|(\rho-\bar{\rho}, u, \theta-\bar{\theta})\|_{H^2}^2$ to obtain the $L^2$-optimal decay rate of the solution itself. For the $L^2$-optimal decay rate of the first-order derivatives of the solution, the $L^2$-norm of the solution itself can not be contained in the energy functional, so we take $\calE(t)$ to be equivalent to $\|(\nabla\rho, \nabla u, \nabla\theta)\|_{H^1}^2$. However, for the $L^2$-optimal decay rate of the second-order derivatives of the solution, if we exclude the $L^2$-norms of the lower-order terms from the energy functional, we can not obtain \eqref{EC}-type inequality directly. For example, the dissipative energy contains only the third-order derivatives of the velocity and the temperature, which cannot absorb the lower-order terms of the right-hand side.
To overcome this difficulty, we decompose the solution into low-frequency and high-frequency parts. For the high-frequency part, the lower-order norms can be controlled by the higher-order norms, while for the low-frequency part, the higher-order norms can be controlled by the lower-order norms. So we can close the energy estimates to obtain the following key inequality,
\begin{equation*}
\frac{\dif}{\dif t}\calE(t)+C\left(\|(\D^3\rho)_L\|_{L^2}^2+\|(\D^2\rho)_H\|_{L^2}^2+\|\D^3u\|_{L^2}^2+\|\D^3\theta\|_{L^2}^2\right)\leq 0,
\end{equation*}
where $\calE(t)$ is equivalent to $\|(\D^2\rho, \D^2u, \D^2\theta)\|_{L^2}^2$, and $f_L, f_H$ denote the low-frequency and high-frequency parts of $f$, respectively. As a result, we extend the works in \cite{MR4921984, Gao2020The, Gao2021De, MR4389852, 2024arXiv240805138L, MR4493879} to the nonisentropic case with temperature-dependent coefficients. To the best of authors' knowledge, this is the first result of optimal decay rates of large solution to the temperature-dependent viscosity fluid model. 
\end{itemize}

The rest of this paper is organized as follows. In \S \ref{prelim}, we give some notations and preliminary lemmas which will be used frequently in our proof. In \S \ref{sub:3.1}, we give the uniform-in-time estimates of the velocity and the temperature up to the second order, as well as the estimates of the density up to the first order, by which, lower and upper bounds of the density and temperature are obtained. The uniform-in-time second-order estimates of the density are given in \S \ref{sec3.4}. In \S \ref{sec:fixed-density}, we establish the high-temperature criterion for an arbitrary fixed positive far-field density. Finally, in \S \ref{sec4}, we establish the optimal decay rates of the solution to its far-field state when the initial data belong to $L^{p_0}(\R^3)$ for some $p_0\in[1,2]$.

\section{Preliminaries}\label{prelim}
	
	
For later purpose, we introduce the following notations. For any $r\in[1,\infty]$ and integer $k\geq 0$, we denote
\begin{equation*}
  L^r=L^r(\R^3), \quad W^{k,r}=W^{k,r}(\R^3), \quad H^k=W^{k,2}, \qquad D^{k,r}=\left\{u\in L^1_{loc}(\R^3) : ~\D^ku\in L^r\right\}.
\end{equation*}
For matrices $A$ and $B$, we denote $A:B=\sum\limits_{i,j=1}^nA_{ij}B_{ij}$ and $|A|^2=A:A$, where $A_{ij}$ is the $(i,j)$-element of $A$. For any function $f(x,t)$, we denote
\begin{equation*}
\int f\dif x=\int_{\R^3}f\dif x,
\end{equation*}
and $\dot{f}=f_t+u\cdot\D f$ is the material derivative of $f$. Moreover, $\|\cdot\|_{L^p}$, $\|\cdot\|_{L^\infty}$ and $\|\cdot\|_{H^s}$ stand for the norms of $L^p$, $L^{\infty}$ and $H^s$, respectively. $\mathcal{F}(f)$ or $\widehat{f}(\xi)$ denote the Fourier transformation of a function $f(x)$. Throughout this paper, we use $C$ to denote a generic positive constant that may depend on $\mu$, $\lambda$, $\kappa$, $R$, $C_0$, but is independent of $T$, $\bar{\rho}$, $\bar{\theta}$, $K_i$, and will change in different places.

The following local existence result can be similarly obtained to \cite{MR0564670}.
\begin{proposition}\label{local}
Let $\bar{\rho},\bar{\theta}>0$ be fixed. Suppose that the initial data satisfy \eqref{data:condition} and
\[
 \inf_{x\in\R^3}\rho_0(x)>0,\qquad
 \inf_{x\in\R^3}\theta_0(x)>0.
\]
Then there exist $T_*>0$ and a unique strong solution $(\rho,u,\theta)$ to \eqref{FCNS:2}--\eqref{data and far} on $\R^3\times[0,T_*]$, with the regularity in \eqref{sp} restricted to $[0,T_*]$ and strictly positive lower bounds for $\rho$ and $\theta$. The existence time can be chosen from the initial $H^2$ bound, the positive lower bounds and the fixed parameters. If the initial data also satisfy \eqref{compatibility condition}, then, by decreasing $T_*$ if necessary, the pointwise bounds \eqref{bc} hold on this interval.
\end{proposition}
	
The following well-known Gagliardo-Nirenberg inequality will be used (see \cite{Ladyzhenskaya1968}).
\begin{lemma}
For $p\in[2,6]$, $q\in(1,\infty)$ and $r\in(3,\infty)$, there exists a constant $C>0$ which may depend on $q,r$  such that for $f\in H^1$ and $g\in L^q \cap D^{1,r}$, it holds
\begin{equation}
\|f\|_{L^p}\leq C\|f\|_{L^2}^{\frac{6-p}{2p}}\|\D f\|_{L^2}^{\frac{3p-6}{2p}},
\end{equation}
\begin{equation}
\|g\|_{L^\infty}\leq C\|g\|_{L^q}^{\frac{q(r-3)}{3r+q(r-3)}}\|\D g\|_{L^r}^{\frac{3r}{3r+q(r-3)}}.
\end{equation}
\end{lemma}

From $(\ref{FCNS:2})_2$, we introduce the modified effective viscous flux $G$,
\begin{equation}\label{notations}
G=(2\mu+\lambda)\div u-R(\rho\theta^{1-\al}-\bar{\rho}\bar{\theta}^{1-\al}),
\end{equation}
which satisfies
\begin{equation}\label{elliptic}
-\Delta G=\div H,
\end{equation}
where
\begin{equation}\label{H}
H=\theta^{-\al}\left(-\rho\dot{u}-R\al\rho\D\theta +2\mu\frD(u)\cdot\D(\theta^\al)+\lambda\div u\D(\theta^\al)\right).
\end{equation}

Then we have the following elementary estimates:
\begin{lemma}
Let $(\rho,u,\theta)$ be a strong solution of \eqref{FCNS:2}--\eqref{data and far}. Then there exists a generic positive constant $C$ depending only on $\mu$, $\lambda$, $\kappa$, and $R$ such that, for any $p\in[2, 6]$, it holds
\begin{equation}\label{2.3:2}
\|\D G\|_{L^p}\leq C\|H\|_{L^p}\leq C\bar{\theta}^{-\al}\left(\|\rho\dot{u}\|_{L^p}+\bar{\rho}\|\D\theta\|_{L^p}+\bar{\theta}^{\al-1}\|\D u\cdot\D\theta\|_{L^p}\right),
\end{equation}
\begin{equation}\label{2.3:3}
\|G\|_{L^p}\leq C\|H\|_{L^2}^{\frac{3p-6}{2p}} \left(\|\D u\|_{L^2}+\|\rho\theta^{1-\al}-\bar{\rho}\bar{\theta}^{1-\al}\|_{L^2}\right)^{\frac{6-p}{2p}},
\end{equation}
\begin{equation}\label{2.3:4}
\|\D u\|_{L^p}\leq C\|\D u\|_{L^2}^{\frac{6-p}{2p}}\left( \|H\|_{L^2}+\|\rho\theta^{1-\al}-\bar{\rho}\bar{\theta}^{1-\al}\|_{L^6}\right)^{\frac{3p-6}{2p}}.
\end{equation}
\end{lemma}
\begin{proof}
Applying standard $L^p$-estimate to \eqref{elliptic} and noting the definition of $H$ in \eqref{H}, we obtain \eqref{2.3:2}.

By the interpolation inequalities of $L^p$ spaces, we have
\begin{align}\label{2.3:?}
\|G\|_{L^p}\leq \|\D G\|_{L^2}^{\frac{3p-6}{2p}}\|G\|_{L^2}^{\frac{6-p}{2p}},
\end{align}
then \eqref{2.3:3} is followed by \eqref{2.3:2} and \eqref{2.3:?}.

Now it remains to prove \eqref{2.3:4}. By the interpolation inequalities of $L^p$ spaces, we have
\begin{equation}\label{2.3:11}
\|\D u\|_{L^p}\leq C\|\D u\|_{L^2}^{\frac{6-p}{2p}}\|\D u\|_{L^6}^{\frac{3p-6}{2p}}.
\end{equation}
Noting that $-\Delta u=-\D\div u+\D\times\D\times u$,
we get
\begin{align*}
-\D u=-\D(-\Delta)^{-1}\D\div u+\D(-\Delta)^{-1}\D\times\D\times u.
\end{align*}
Therefore, standard $L^p$-estimate implies
\begin{align}\label{2.3:??}
\|\D u\|_{L^6}\leq C\left(\|\div u\|_{L^6}+\|\D\times u\|_{L^6}\right).
\end{align}

By the definition of $G$ and \eqref{2.3:2}, we obtain
\begin{align}
 \label{2.3:1} \|\div u\|_{L^6}\leq~& C\left(\|G\|_{L^6}+\|\rho\theta^{1-\al}-\bar{\rho}\bar{\theta}^{1-\al}\|_{L^6}\right) \leq~ C\left(\|\D G\|_{L^2}+\|\rho\theta^{1-\al}-\bar{\rho}\bar{\theta}^{1-\al}\|_{L^6}\right)\nonumber\\
 \leq~& C\left(\|H\|_{L^2}+\|\rho\theta^{1-\al}-\bar{\rho}\bar{\theta}^{1-\al}\|_{L^6}\right).
\end{align}

Noting that $\mu\Delta(\D\times u)=-\D\times H$, we have
\begin{equation}\label{2.3:???}
\|\D\times u\|_{L^6}\leq C\|\D(\D\times u)\|_{L^2}\leq C\|H\|_{L^2}.
\end{equation}
Inserting \eqref{2.3:??}--\eqref{2.3:???} into \eqref{2.3:11}, we obtain \eqref{2.3:4}. This completes the proof.
\end{proof}

\section{Global Existence of Large Strong Solution}\label{Section 3}

In this section, we will establish a priori bounds for local-in-time strong solution to \eqref{FCNS:2}--\eqref{data and far} obtained in Proposition \ref{local}. We thus fix a strong solution
$(\rho,u,\theta)$ of \eqref{FCNS:2}--\eqref{data and far} on $\R^3 \times (0,T]$ for some time $T>0$, with initial data
$(\rho_0,u_0,\theta_0)$ satisfying \eqref{data:condition}.
We first retain the explicit dependence of the estimates on both far-field states. In \S\ref{sec:fixed-density}, we use the same energy identities to prove Theorem \ref{th:fixed} for an arbitrary fixed positive $\bar{\rho}$.
For $\sigma(t) =: \min\{1, t\}$, we define $A_i(T)$, $B_i(T)$ and $C_i(T)$ as follows:
\begin{equation}\label{A1A2}
A_1(T)=\bar{\theta}^\al\sup\limits_{t\in[0,T]}\|\D u\|_{L^2}^2+\int_0^T\|\sqrt{\rho}\dot{u}\|_{L^2}^2\dif t, \quad A_2(T)=\bar{\theta}^\beta\sup\limits_{t\in[0,T]}\|\D\theta\|_{L^2}^2+\int_0^T\| \sqrt{\rho}\dot{\theta}\|_{L^2}^2\dif t,
\end{equation}
\begin{equation}\label{A3A4}
A_3(T)=\sup\limits_{t\in[0,T]}\|\sqrt{\rho}\dot{u}\|_{L^2}^2+\bar{\theta}^\al\int_0^T\|\D \dot{u}\|_{L^2}^2\dif t, \quad A_4(T)=\sup\limits_{t\in[0,T]}\| \sqrt{\rho}\dot{\theta}\|_{L^2}^2+\bar{\theta}^\beta\int_0^T\|\D \dot{\theta}\|_{L^2}^2\dif t,
\end{equation}

\begin{equation}\label{B1B2}
B_1(T)=\bar{\theta}^\al\sup\limits_{t\in[0,T]}(\sigma^3\|\D u\|_{L^2}^2)+\int_0^T\sigma^3\|\sqrt{\rho}\dot{u}\|_{L^2}^2\dif t, ~ B_2(T)=\bar{\theta}^\beta\sup\limits_{t\in[0,T]}(\sigma^4\|\D\theta\|_{L^2}^2)+\int_0^T\sigma^4\| \sqrt{\rho}\dot{\theta}\|_{L^2}^2\dif t,
\end{equation}
\begin{equation}\label{B3B4}
B_3(T)=\sup\limits_{t\in[0,T]}(\sigma^5\|\sqrt{\rho}\dot{u}\|_{L^2}^2)+\bar{\theta}^\al\int_0^T\sigma^5\|\D \dot{u}\|_{L^2}^2\dif t, ~ B_4(T)=\sup\limits_{t\in[0,T]}(\sigma^6\| \sqrt{\rho}\dot{\theta}\|_{L^2}^2)+\bar{\theta}^\beta\int_0^T\sigma^6\|\D \dot{\theta}\|_{L^2}^2\dif t,
\end{equation}
\begin{equation}\label{A6A7}
C_1(T)=\sup_{t\in[0,T]}\|\D \rho\|_{L^2}^2+\bar{\rho}\bar\theta^{1-\al}\int_0^T\|\D \rho\|_{L^2}^2\dif t, \quad C_2(T)=\sup_{t\in[0,T]}\|\D \rho\|_{L^4}^2+\bar{\rho}\bar\theta^{1-\al}\int_0^T\|\D \rho\|_{L^4}^2\dif t,
\end{equation}

We then have the following key a priori estimates on $(\rho,u,\theta)$.
\begin{proposition}\label{p4.1}
For given $\bar{\rho},\bar{\theta}>1$, assume that $(\rho_0,u_0,\theta_0)$ satisfies \eqref{data:condition}
and \eqref{compatibility condition}, then there exist positive constants $K_i$, $\widetilde{\rL}$ and $\rL$ depending on $\mu$, $\lambda$, $\kappa$, $R$, $C_0$, such that if $(\rho,u,\theta)$ is a strong solution of \eqref{FCNS:2}--\eqref{data and far} on $\R^3\times(0,T]$ satisfying
\begin{equation}\label{assume:1}
A_1(\sigma(T))\leq 2K_1\bar{\theta}^\al, ~ A_2(\sigma(T))\leq 2K_2\bar{\theta}^\beta, ~ A_3(\sigma(T))\leq 2K_3\bar{\theta}^{2\alpha}, ~ A_4(\sigma(T))\leq 2K_4\bar{\theta}^{2\beta}, 
\end{equation}
\begin{equation}
 \begin{aligned}
 \label{assume:2}
& B_1(T)\leq 2K_5\bar{\rho}\bar{\theta}, \quad B_2(T)\leq 2K_6\bar{\rho}\bar{\theta}^2, \quad B_3(T)\leq 2\bar{\theta}^{\frac{\alpha}{3}+\frac{2}{3}}, \quad B_4(T)\leq 2K_7\bar{\rho}\bar{\theta}^2,\\
 & C_1(T)\leq 2K_8, \quad C_2(T)\leq 2K_9, \qquad \frac{1}{2}\bar{\theta}\leq\theta\leq\frac{3}{2}\bar{\theta}, \quad \frac{1}{2}\bar{\rho}\leq\rho\leq \frac{3}{2}\bar{\rho}, \qquad\text{for} ~ t\in[0,T], ~ x\in\R^3,
\end{aligned}
\end{equation}
then the following estimates hold
\begin{equation}\label{goal:1}
A_1(\sigma(T))\leq K_1\bar{\theta}^\al, ~ A_2(\sigma(T))\leq K_2\bar{\theta}^\beta, ~ A_3(\sigma(T))\leq K_3\bar{\theta}^{2\al}, ~ A_4(\sigma(T))\leq K_4\bar{\theta}^{2\beta},
\end{equation}
\begin{equation}
 \begin{aligned}
 \label{goal:2}
& B_1(T)\leq K_5\bar{\rho}\bar{\theta}, \quad B_2(T)\leq K_6\bar{\rho}\bar{\theta}^2, \quad B_3(T)\leq \bar{\theta}^{\frac{\alpha}{3}+\frac{2}{3}}, \quad B_4(T)\leq K_7\bar{\rho}\bar{\theta}^2,\\
 & C_1(T)\leq K_8, \quad C_2(T)\leq K_9, \qquad \frac{2}{3}\bar{\theta}\leq\theta\leq\frac{4}{3}\bar{\theta}, \quad \frac{2}{3}\bar{\rho}\leq\rho\leq \frac{4}{3}\bar{\rho}, \qquad\text{for} ~ t\in[0,T], ~ x\in\R^3,
\end{aligned}
\end{equation}
provided that $\bar{\rho}\geq\rL$, $\bar{\theta}\geq\widetilde{\rL}$, $2\leq\alpha<\beta<\alpha+2$, $4\alpha-3\beta<4$, and
\begin{equation}\label{rho:theta}
\bar{\rho}\leq C(K)\min\left\{\bar{\theta}^{\frac{1}{3}(\beta-\alpha)}, \bar{\theta}^{2(\alpha-\beta+2)}, \bar{\theta}^{\frac{1}{3}}\right\},
\end{equation}
with some positive constant $C(K)$ depending on $K_i ~ (i=1,\cdots, 9)$.
\end{proposition}

\begin{proposition}\label{Lem:10}
Under the conditions of Proposition \ref{p4.1}, it holds that
\begin{equation}\label{sec3.4:1}
\sup_{t\in[0,T]}\left(\|u\|_{H^2}^2+\|\theta-\bar{\theta}\|_{H^2}^2\right)+\int_0^T\left(\|u_t\|_{H^1}^2+\|\theta_t\|_{H^1}^2\right)\dif t\leq C\bar{\theta}^{\beta},
\end{equation}
\begin{equation}\label{sec3.4:1.2}
\sup_{t\in[0,T]}\left(\|\rho_t\|_{L^2}^2+\|\rho-\bar{\rho}\|_{H^1}^2\right)\leq C\bar{\theta}^{2},
\end{equation}
and
\begin{equation}\label{sec3.4:1.3}
\sup_{t\in[0,T]}\left(\|\D\rho_t\|_{L^2}^2+\|\D^2\rho\|_{L^2}^2\right)\leq C\bar{\rho}^{\frac{2}{7}}\exp\left\{C\bar{\rho}^{\frac{2}{7}}\right\}.
\end{equation}
\end{proposition}

\subsection{Proof of Proposition \ref{p4.1}}\label{sub:3.1}

In this subsection, we use $C(K)$ to denote a generic constant that may depend on $K_i ~ (i=1,\cdots, 9)$. We begin with the zero-order energy estimate of $(\rho,u,\theta)$, followed by the first-order and the second-order estimates of $(u, \theta)$, as well as the first-order estimate of $\rho$.

\subsubsection{Uniform zero-order estimates of the solution on $[0,T]$}

\begin{lemma}\label{Lem3.1}
Under the conditions of Proposition \ref{p4.1}, there exists a positive constant $C$ depending on $\mu$, $\lambda$, $\kappa$, $C_0$, and $R$ such that if $(\rho,u,\theta)$ is a smooth solution of \eqref{FCNS:2}--\eqref{data and far} on $\R^3\times(0,T]$, the following estimate holds for any $T>0$:
\begin{align}\label{Lem3.1:3-2}
&\sup\limits_{t\in[0,T]}\int\left(\rho|u|^2+\bar{\rho}^{-1}\bar{\theta}(\rho-\bar{\rho})^2+\rho\bar{\theta}^{-1}(\theta-\bar{\theta})^2\right)\dif x\nonumber\\
&+\int_0^T\int\left(\bar{\theta}\left(\lambda\theta^{\al-1}(\div u)^2+2\mu\theta^{\al-1}|\frD(u)|^2\right)+\frac{\kappa\bar{\theta}}{\theta^{2-\beta}} |\D\theta|^2\right)\dif x\dif t\leq C\bar{\rho}\bar{\theta}.
\end{align}
\end{lemma}
\begin{proof}
Adding $(\ref{FCNS:2})_2$ multiplied by $u$ to $(\ref{FCNS:2})_3$ multiplied by $\displaystyle{1-\frac{\bar{\theta}}{\theta}}$, we obtain after integrating the resulting equality over $\R^3$ and using $(\ref{FCNS:2})_1$ that
\begin{align}\label{Lem3.1:1}
&\frac{\dif}{\dif t}\int\left(\frac{1}{2}\rho|u|^2+R\bar{\theta}(\rho\ln\rho-\rho-\rho\ln\bar{\rho}+\bar{\rho})+\rho(\theta-\bar{\theta}\ln\theta-\bar{\theta}+\bar{\theta}\ln\bar{\theta})\right)\dif x\nonumber\\
=~&-\int\left(\lambda\theta^\al(\div u)^2+2\mu\theta^\al |\frD(u)|^2\right)\dif x-\int\frac{\kappa\bar{\theta}}{\theta^{2-\beta}} |\D\theta|^2\dif x+\int\left(1-\frac{\bar{\theta}}{\theta}\right)\left(\lambda\theta^\al(\div u)^2+2\mu\theta^\al|\frD(u)|^2\right)\dif x\nonumber\\
=~&-\int\left(\bar{\theta}\left(\lambda\theta^{\al-1}(\div u)^2+2\mu\theta^{\al-1}|\frD(u)|^2\right)+\frac{\kappa\bar{\theta}}{\theta^{2-\beta}} |\D\theta|^2\right)\dif x.
\end{align}

It follows from the identity
\begin{equation*}
\rho\ln\rho-\rho-\rho\ln\bar{\rho}+\bar{\rho}=(\rho-\bar{\rho})^2\int_0^1\frac{1-y}{y(\rho-\bar{\rho})+\bar{\rho}}\dif y
\end{equation*}
that
\begin{equation}\label{Lem3.1:2}
(10\bar{\rho})^{-1}(\rho-\bar{\rho})^2\leq\rho\ln\rho-\rho-\rho\ln\bar{\rho}+\bar{\rho}\leq 2\bar{\rho}^{-1}(\rho-\bar{\rho})^2.
\end{equation}

Similarly, it follows from the identity
\begin{equation*}
\theta-\bar{\theta}\ln\theta-\bar{\theta}+\bar{\theta}\ln\bar{\theta}=(\theta-\bar{\theta})^2\int_0^1\frac{y}{y(\theta-\bar{\theta})+\bar{\theta}}\dif y
\end{equation*}
that
\begin{equation}\label{Lem3.1:3}
(10\bar{\theta})^{-1}(\theta-\bar{\theta})^2\leq\theta-\bar{\theta}\ln\theta-\bar{\theta}+\bar{\theta}\ln\bar{\theta}\leq 2\bar{\theta}^{-1}(\theta-\bar{\theta})^2.
\end{equation}

Integrating \eqref{Lem3.1:1} with respect to $t$ over $(0, T)$ and using \eqref{compatibility condition} yields
\begin{align}\label{Lem3.1:3-1}
&\sup\limits_{t\in[0,T]}\int\left(\frac{1}{2}\rho|u|^2+R\bar{\theta}(\rho\ln\rho-\rho-\rho\ln\bar{\rho}+\bar{\rho})+\rho(\theta-\bar{\theta}\ln\theta-\bar{\theta}+\bar{\theta}\ln\bar{\theta})\right)\dif x\nonumber\\
&+\int_0^T\int\left(\bar{\theta}\left(\lambda\theta^{\al-1}(\div u)^2+2\mu\theta^{\al-1}|\frD(u)|^2\right)+\frac{\kappa\bar{\theta}}{\theta^{2-\beta}} |\D\theta|^2\right)\dif x\dif t\leq C\bar{\rho}\bar{\theta},
\end{align}
which together with \eqref{Lem3.1:2} and \eqref{Lem3.1:3} leads to \eqref{Lem3.1:3-2}.
\end{proof}

The following corollary follows immediately from Lemma \ref{Lem3.1}.
\begin{corollary}
Under the conditions of Proposition \ref{p4.1}, there exists a positive constant $C$ depending on $\mu$, $\lambda$, $\kappa$, $C_0$, and $R$ such that if $(\rho,u,\theta)$ is a smooth solution of \eqref{FCNS:2}--\eqref{data and far} on $\R^3\times(0,T]$, the following estimates hold for all $T>0$:
\begin{equation}\label{lem3.1:g1}
\|\rho-\bar{\rho}\|_{L^2}^2\leq C\bar{\rho}^2, \qquad \|u\|_{L^2}^2\leq C\bar{\theta}, \qquad \|\theta-\bar{\theta}\|_{L^2}^2\leq C\bar{\theta}^2, \quad\forall ~ t\in(0,T],
\end{equation}
\begin{equation}\label{lem3.1:g2}
\int_0^T\|\D u\|_{L^2}^2\dif t\leq C\bar{\rho}\bar{\theta}^{1-\al}, \qquad \int_0^T\|\D\theta\|_{L^2}^2\dif t\leq C\bar{\rho}\bar{\theta}^{2-\beta}.
\end{equation}
\end{corollary}

\subsubsection{Uniform estimates of velocity and temperature on $[0,\sigma(T)]$}\label{sec3.1}

The following lemma will give the first-order estimates of $u$ and $\theta$.
\begin{lemma}\label{Lem33}$($First-order estimates of velocity and temperature$)$ 
Under the conditions of Proposition \ref{p4.1}, there exists positive constants $K_1$, $K_2$ and $\rL_1$ depending on $\mu$, $\lambda$, $\kappa$, $R$, and $C_0$, such that if $(\rho,u,\theta)$ is a smooth solution of \eqref{FCNS:2}--\eqref{data and far} on $\R^3\times(0,T]$, the following estimate holds
\begin{equation}\label{g:3.3:1}
A_1(\sigma(T))=\bar{\theta}^\al\sup\limits_{t\in[0,\sigma(T)]}\|\D u\|_{L^2}^2+\int_0^{\sigma(T)}\|\sqrt{\rho}\dot{u}\|_{L^2}^2\dif t\leq K_1\bar{\theta}^\al,
\end{equation}
and
\begin{equation}\label{g:3.3:2}
A_2(\sigma(T))=\bar{\theta}^\beta\sup\limits_{t\in[0,\sigma(T)]}\|\D \theta\|_{L^2}^2+\int_0^{\sigma(T)}\|\sqrt{\rho}\dot{\theta}\|_{L^2}^2\dif t\leq K_2\bar{\theta}^\beta,
\end{equation}
provided that $\bar{\theta}\geq \rL_1$, $2\leq\alpha<\beta<\alpha+2$, and
\begin{equation}\label{rho-theta}
\bar{\rho}\leq C(K)\min\left\{\bar{\theta}^{\alpha-1}, \bar{\theta}^{\frac{1}{2}(3\alpha-\beta-1)}, \bar{\theta}^{2(\alpha-\beta+2)}, \bar{\theta}^{\frac{1}{2}(\alpha+\beta-3)}, \bar{\theta}^{2(\beta-\alpha)}\right\}.
\end{equation}
\end{lemma}
\begin{proof}
First, multiplying $(\ref{FCNS:2})_2$ by $2u_t$ and integrating the resulting equality over $\R^3$, we obtain after integration by parts and using \eqref{notations} that
\begin{align}\label{3.3:1}
&\frac{\dif}{\dif t}\int\left(2\mu\theta^\al|\frD(u)|^2+\lambda \theta^\al(\div u)^2 \right)\dif x+\int\rho|u_t|^2\dif x\nonumber\\
\leq~&-2\int\D P\cdot u_t\dif x+\int\rho|u\cdot\D u|^2\dif x+\al\int\theta^{\al-1}\theta_t\left(2\mu|\frD(u)|^2+\lambda (\div u)^2\right)\dif x\nonumber\\
=~&2R\frac{\dif}{\dif t}\int(\rho\theta-\bar{\rho}\bar{\theta})\div u\dif x-2\int P_t\div u\dif x+\int\rho|u\cdot\D u|^2\dif x+\al\int\theta^{\al-1}\theta_t\left(2\mu|\frD(u)|^2+\lambda (\div u)^2\right)\dif x\nonumber\\
=~&2R\frac{\dif}{\dif t}\int(\rho\theta-\bar{\rho}\bar{\theta})\div u\dif x-\frac{2}{2\mu+\lambda}\int P_tG\dif x+\int\rho|u\cdot\D u|^2\dif x\nonumber\\
&+\al\int\theta^{\al-1}\theta_t\left(2\mu|\frD(u)|^2+\lambda (\div u)^2\right)\dif x-\frac{2R}{2\mu+\lambda}\int P_t(\rho\theta^{1-\al}-\bar{\rho}\bar{\theta}^{1-\al})\dif x\nonumber\\
=:~&2R\frac{\dif}{\dif t}\int(\rho\theta-\bar{\rho}\bar{\theta})\div u\dif x+\sum\limits_{i=1}^4I_i.
\end{align}

Next we estimate the terms on the right hand side of \eqref{3.3:1}. Noticing that \eqref{FCNS:2} implies
\begin{equation}\label{eqn:P}
P_t=-\div(Pu)+R\rho\dot{\theta},
\end{equation}
after integration by parts and using \eqref{2.3:3}, \eqref{assume:1}, and Lemma \ref{Lem3.1}, we obtain
\begin{align}\label{3.3:4}
I_1=~&-\frac{2}{2\mu+\lambda}\int P_tG\dif x\nonumber\\
=~& -\frac{2}{2\mu+\lambda}\int\left(-\div(Pu)+R\rho\dot{\theta}\right)G \dif x \nonumber\\
\leq~& C\int P|u||\D G|\dif x+C\bar{\rho}^\frac{1}{2}\int|\sqrt{\rho}\dot{\theta}||G|\dif x\nonumber\\
\leq~& C\left(\bar{\rho}^\frac{1}{2}\bar{\theta}\|\sqrt{\rho}u\|_{L^2}\|\D G\|_{L^2}+\bar{\rho}^\frac{1}{2}\|\sqrt{\rho}\dot{\theta}\|_{L^2}\|G\|_{L^2}\right)\nonumber\\
\leq~& C\left(\bar{\rho}\bar{\theta}\|u\|_{L^2}\|\D G\|_{L^2}+\bar{\rho}^\frac{1}{2}\|\sqrt{\rho}\dot{\theta}\|_{L^2}\|\D u\|_{L^2}+\bar{\rho}^\frac{1}{2}\bar{\theta}^{1-\alpha}\|\sqrt{\rho}\dot{\theta}\|_{L^2}\|\rho-\bar{\rho}\|_{L^2}+\bar{\rho}^\frac{3}{2}\bar{\theta}^{-\alpha}\|\sqrt{\rho}\dot{\theta}\|_{L^2}\|\theta-\bar{\theta}\|_{L^2}\right)\nonumber\\
\leq~& C\left(\bar{\rho}\bar{\theta}^{\frac{3}{2}}\|\D G\|_{L^2}+\bar{\rho}^\frac{1}{2}\|\sqrt{\rho}\dot{\theta}\|_{L^2}\|\D u\|_{L^2}+\bar{\rho}^\frac{3}{2}\bar{\theta}^{1-\alpha}\|\sqrt{\rho}\dot{\theta}\|_{L^2}\right).
\end{align}
where we have used the estimates
\begin{align}
 \label{G} \|G\|_{L^2}\leq~&C\left(\|\D u\|_{L^2}+\|\rho\theta^{1-\al}-\bar{\rho}\bar{\theta}^{1-\al}\|_{L^2}\right) \leq~C\left(\|\D u\|_{L^2}+\bar{\theta}^{-\al}\|\rho(\theta-\bar{\theta})\|_{L^2}+\bar{\theta}^{1-\al}\|\rho-\bar{\rho}\|_{L^2}\right)\nonumber\\
 \leq~&C\left(\|\D u\|_{L^2}+\bar{\rho}\bar{\theta}^{1-\al}\right),
\end{align}
\begin{align}\label{3.3:5}
I_2=\int\rho|u\cdot\D u|^2\dif x\leq C\bar{\rho}\|u\|_{L^6}^2\|\D u\|_{L^2}\|\D u\|_{L^6}\leq C\bar{\rho}\|\D u\|_{L^2}^3\|\D u\|_{L^6},
\end{align}
\begin{align}\label{3.3:6}
I_3=~&\al\int\theta^{\al-1}\theta_t\left(2\mu|\frD(u)|^2+\lambda (\div u)^2\right)\dif x\nonumber\\
\leq~&C\bar{\theta}^{\al-1}\int|\theta_t||\D u|^2\dif x\nonumber\\
\leq~&C\bar{\theta}^{\al-1}\left(\bar{\rho}^{-\frac{1}{2}}\|\sqrt{\rho}\dot{\theta}\|_{L^2}\|\D u\|_{L^2}^{\frac{1}{2}}\|\D u\|_{L^6}^{\frac{3}{2}}+\|u\|_{L^\infty}\|\D\theta\|_{L^2}\|\D u\|_{L^2}^{\frac{1}{2}}\|\D u\|_{L^6}^{\frac{3}{2}}\right)\nonumber\\
\leq~&C\bar{\theta}^{\al-1}\left(\bar{\rho}^{-\frac{1}{2}}\|\sqrt{\rho}\dot{\theta}\|_{L^2}\|\D u\|_{L^2}^{\frac{1}{2}}\|\D u\|_{L^6}^{\frac{3}{2}}+\|\D\theta\|_{L^2}\|\D u\|_{L^2}\|\D u\|_{L^6}^2\right).
\end{align}

After integration by parts and using \eqref{eqn:P}, \eqref{assume:2} and Lemma \ref{Lem3.1}, we obtain
\begin{align}
 \label{3.3:7} I_4=~&-\frac{2R}{2\mu+\lambda}\int P_t(\rho\theta^{1-\al}-\bar{\rho}\bar{\theta}^{1-\al})\dif x =~-\frac{2R}{2\mu+\lambda}\int\left(-\div(Pu)+R\rho\dot{\theta}\right)(\rho\theta^{1-\al}-\bar{\rho}\bar{\theta}^{1-\al})\dif x\nonumber\\
 \leq~&C\bar{\rho}^{\frac{1}{2}}\bar{\theta}^{2-\al}\int|\sqrt{\rho}u||\D\rho|\dif x+C\bar{\rho}^{\frac{3}{2}}\bar{\theta}^{1-\al}\int|\sqrt{\rho} u||\D\theta|\dif x+C\bar{\rho}^{\frac{1}{2}}\int|\sqrt{\rho}\dot{\theta}|\left|\rho\theta^{1-\al}-\bar{\rho}\bar{\theta}^{1-\al}\right|\dif x\nonumber\\
 \leq~&C\bar{\rho}^{\frac{1}{2}}\bar{\theta}^{2-\al}\|\sqrt{\rho}u\|_{L^2}\|\D\rho\|_{L^2}+C\bar{\rho}^{\frac{3}{2}}\bar{\theta}^{1-\al}\|\sqrt{\rho}u\|_{L^2}\|\D\theta\|_{L^2}+C\bar{\rho}^{\frac{1}{2}}\|\sqrt{\rho}\dot{\theta}\|_{L^2}\left\|\rho\theta^{1-\al}-\bar{\rho}\bar{\theta}^{1-\al}\right\|_{L^2}\nonumber\\
 \leq~&C\bar{\rho}\bar{\theta}^{\frac{5}{2}-\al}\|\D\rho\|_{L^2}+C\bar{\rho}^2\bar{\theta}^{\frac{3}{2}-\al}\|\D\theta\|_{L^2}+C\bar{\rho}^{\frac{3}{2}}\bar{\theta}^{1-\al}\|\sqrt{\rho}\dot{\theta}\|_{L^2}.
\end{align}

Noting that
\begin{align*}
\sup\limits_{t\in[0,\sigma(T)]}\|\rho\theta-\bar{\rho}\bar{\theta}\|_{L^2}\|\D u\|_{L^2}\leq C\sup\limits_{t\in[0,\sigma(T)]}\left(\bar{\rho}\|(\theta-\bar{\theta})\|_{L^2}+\bar{\theta}\|\rho-\bar{\rho}\|_{L^2}\right)\|\D u\|_{L^2}\leq CK_1^\frac{1}{2}\bar{\rho}\bar{\theta},
\end{align*}
substituting \eqref{3.3:4}--\eqref{3.3:7} into \eqref{3.3:1}, integrating the resulting inequalities  over $(0,\sigma(T))$ and using \eqref{Lem3.1:3-1}, we obtain
\begin{align}
 \label{3.3:8}
&\bar{\theta}^\al\sup\limits_{t\in[0,\sigma(T)]}\|\D u\|_{L^2}^2+\int_0^{\sigma(T)}\|\sqrt{\rho}\dot{u}\|_{L^2}^2\dif t\nonumber\\
 \leq~&C\bar{\theta}^\al\|\D u_0\|_{L^2}^2+C\|\rho\theta-\bar{\rho}\bar{\theta}\|_{L^2}\|\D u\|_{L^2}+C\|\rho_0\theta_0-\bar{\rho}\bar{\theta}\|_{L^2}\|\D u_0\|_{L^2}\nonumber\\
 &+C\int_0^{\sigma(T)}\left(\bar{\rho}\bar{\theta}^{\frac{3}{2}}\|\D G\|_{L^2}+\bar{\rho}^\frac{1}{2}\|\sqrt{\rho}\dot{\theta}\|_{L^2}\|\D u\|_{L^2}+\bar{\rho}^\frac{3}{2}\bar{\theta}^{1-\alpha}\|\sqrt{\rho}\dot{\theta}\|_{L^2}\right)\dif t+C\bar{\rho}\int_0^{\sigma(T)}\|\D u\|_{L^2}^3\|\D u\|_{L^6}\dif t\nonumber\\
 &+C\bar{\theta}^{\al-1}\int_0^{\sigma(T)}\left(\bar{\rho}^{-\frac{1}{2}}\|\sqrt{\rho}\dot{\theta}\|_{L^2}\|\D u\|_{L^2}^{\frac{1}{2}}\|\D u\|_{L^6}^{\frac{3}{2}}+\|\D\theta\|_{L^2}\|\D u\|_{L^2}\|\D u\|_{L^6}^2\right)\dif t\nonumber\\
 &+C\int_0^{\sigma(T)}\left(\bar{\rho}\bar{\theta}^{\frac{5}{2}-\al}\|\D\rho\|_{L^2}+\bar{\rho}^2\bar{\theta}^{\frac{3}{2}-\al}\|\D\theta\|_{L^2}\right)\dif t\nonumber\\
 \leq~&C\bar{\theta}^\al\|\D u_0\|_{L^2}^2+CK_1^\frac{1}{2}\bar{\rho}\bar{\theta}+CK_2^{\frac{1}{2}}\left(\bar{\rho}\bar{\theta}^{\frac{\beta-\alpha+1}{2}}+\bar{\rho}^{\frac{3}{2}}\bar{\theta}^{\frac{\beta}{2}-\alpha+1}\right)+CK_8^{\frac{1}{2}}\bar{\rho}\bar{\theta}^{\frac{5}{2}-\al}+C\bar{\rho}^{\frac{5}{2}}\bar{\theta}^{\frac{5}{2}-\al-\frac{\beta}{2}}\nonumber\\
 &+C\bar{\rho}\bar{\theta}^{\frac{3}{2}}\int_0^{\sigma(T)}\|\D G\|_{L^2}\dif t+CK_1\bar{\rho}\int_0^{\sigma(T)}\|\D u\|_{L^2}\|\D u\|_{L^6}\dif t\nonumber\\
 &+C\bar{\theta}^{\al-1}\int_0^{\sigma(T)}\left(K_1^\frac{1}{4}\bar{\rho}^{-\frac{1}{2}}\|\sqrt{\rho}\dot{\theta}\|_{L^2}\|\D u\|_{L^6}^{\frac{3}{2}}+K_1^\frac{1}{2}K_2^\frac{1}{2}\|\D u\|_{L^6}^2\right)\dif t =:~\frac{1}{2}K_1\bar{\theta}^\al+\sum\limits_{i=1}^3I_{1i},
\end{align}
provided that $\alpha>\max\big\{\frac{5}{4},\frac{1}{3}(\beta+1)\big\}$,
$$\bar{\rho}\leq C(K) \min\left\{\bar{\theta}^{\alpha-1}, \bar{\theta}^{\frac{1}{2}(3\alpha-\beta-1)}, \bar{\theta}^{2\alpha-\frac{5}{2}}\right\},$$
where we have taken $K_1\geq 4C\|\nabla u_0\|_{L^2}^2$ and $\bar{\theta}$ sufficiently large, such that
$$CK_1^\frac{1}{2}\bar{\rho}\bar{\theta}+CK_2^{\frac{1}{2}}\left(\bar{\rho}\bar{\theta}^{\frac{\beta-\alpha+1}{2}}+\bar{\rho}^{\frac{3}{2}}\bar{\theta}^{\frac{\beta}{2}-\alpha+1}\right)+CK_8^{\frac{1}{2}}\bar{\rho}\bar{\theta}^{\frac{5}{2}-\al}+C\bar{\rho}^\frac{5}{2}\bar{\theta}^{\frac{5}{2}-\al-\frac{\beta}{2}}\leq \frac{1}{4}K_1\bar{\theta}^\al,$$

Now it remains to estimate the terms on the right-hand side of (\ref{3.3:8}). $\|\D u\|_{L^6}$ plays an important role in our analysis. It follows from \eqref{2.3:2},  \eqref{2.3:4}, and \eqref{assume:2} that
\begin{align}\label{L6}
\|\D u\|_{L^6}
\leq~& C\left(\|H\|_{L^2}+\|\rho\theta^{1-\al}-\bar{\rho}\bar{\theta}^{1-\al}\|_{L^6}\right)\nonumber\\
\leq~& C\left(\bar{\theta}^{-\al}\|\rho\dot{u}\|_{L^2}+\bar{\rho}\bar{\theta}^{-\al}\|\D\theta\|_{L^2}+\|\rho\theta^{1-\al}-\bar{\rho}\bar{\theta}^{1-\al}\|_{L^6}+\bar{\theta}^{-1}\|\D\theta\|_{L^4}\|\D u\|_{L^4}\right)\nonumber\\
\leq~& C\left(\bar{\theta}^{-\al}\|\rho\dot{u}\|_{L^2}+\bar{\rho}\bar{\theta}^{-\al}\|\D\theta\|_{L^2}+\bar{\theta}^{-\al}\|\rho(\theta-\bar{\theta})\|_{L^6}+\bar{\theta}^{1-\al}\|\rho-\bar{\rho}\|_{L^6}\right)\nonumber\\
&+C\bar{\theta}^{-1}\|\D\theta\|_{L^4}\|\D u\|_{L^2}^{\frac{1}{4}}\|\D u\|_{L^6}^{\frac{3}{4}}\nonumber\\
\leq~& C\left(\bar{\theta}^{-\al}\|\rho\dot{u}\|_{L^2}+\bar{\rho}\bar{\theta}^{-\al}\|\D\theta\|_{L^2}+\bar{\theta}^{1-\al}\|\D\rho\|_{L^2}+\bar{\theta}^{-4}\|\D\theta\|_{L^4}^4\|\D u\|_{L^2}\right)+\frac{1}{2}\|\D u\|_{L^6}\nonumber\\
\leq~& C\left(\bar{\rho}^\frac{1}{2}\bar{\theta}^{-\al}\|\sqrt{\rho}\dot{u}\|_{L^2}+\bar{\rho}\bar{\theta}^{-\al}\|\nabla\theta\|_{L^2}+\bar{\theta}^{1-\al}\|\D\rho\|_{L^2}+\bar{\theta}^{-4}\|\nabla \theta\|_{L^2}\|\D^2 \theta\|_{L^2}^3\|\D u\|_{L^2}\right)+\frac{1}{2}\|\D u\|_{L^6}.
\end{align}
Moreover,  we claim that
\begin{align}\label{2dtheta-1}
\sup\limits_{t\in[0,\sigma(T)]}\|\D^2 \theta\|_{L^2}\leq 2\bar{\theta}.
\end{align}
Hence, from (\ref{L6}) and (\ref{2dtheta-1}), we have
\begin{align}
 \label{estimate:DuL6}
\|\D u\|_{L^6}&\leq C\left(\bar{\rho}^\frac{1}{2}\bar{\theta}^{-\al}\|\sqrt{\rho}\dot{u}\|_{L^2}+\bar{\rho}\bar{\theta}^{-\al}\|\nabla\theta\|_{L^2}+\bar{\theta}^{1-\al}\|\D\rho\|_{L^2}+\bar{\theta}^{-1}\|\nabla \theta\|_{L^2}\|\D u\|_{L^2}\right)\nonumber\\
 &\leq C\left(K_3^\frac{1}{2}\bar{\rho}^\frac{1}{2}+K_2^\frac{1}{2}\bar{\rho}\bar{\theta}^{-\al}+K_8^\frac{1}{2}\bar{\theta}^{1-\al}+K_1^\frac{1}{2}K_2^\frac{1}{2}\bar{\theta}^{-1}\right) \leq CK_3^\frac{1}{2}\bar{\rho}^\frac{1}{2},
\end{align}
for $t\in[0,\sigma(T)]$ provided that $\alpha>1$,
\begin{align}
 \label{estimate:DuL6-1}
\int_0^{\sigma(T)}\|\D u\|_{L^6}^2\dif t\leq~& C\int_0^{\sigma(T)}\left(\bar{\rho}\bar{\theta}^{-2\al}\|\sqrt{\rho}\dot{u}\|_{L^2}^2+\bar{\rho}^2\bar{\theta}^{-2\al}\|\nabla\theta\|_{L^2}^2+\bar{\theta}^{2-2\al}\|\nabla\rho\|_{L^2}^2+\bar{\theta}^{-2}\|\nabla \theta\|_{L^2}^2\|\D u\|_{L^2}^2\right)\dif t\nonumber\\
 \leq~& C\left(K_1\bar{\rho}\bar{\theta}^{-\alpha}+\bar{\rho}^3\bar{\theta}^{-2\alpha+2-\beta}+K_8\bar{\theta}^{2-2\alpha}+K_2\bar{\rho}\bar{\theta}^{-1-\alpha}\right) \leq~ CK_1\bar{\rho}\bar{\theta}^{-\alpha},
\end{align}
provided that $\alpha\geq 2$, $\bar{\rho}\leq C(K)\bar{\theta}^{\frac{1}{2}(\alpha+\beta-2)}$, and
\begin{align}
 \label{estimate:D2theta-0}
\sup\limits_{t\in[0,\sigma(T)]}\|\D^2\theta\|_{L^2}^2 
\leq~& C\bar{\theta}^{-2\beta}\left(\bar{\rho}\|\sqrt{\rho}\dot{\theta}\|_{L^2}^2+\bar{\rho}^2\bar{\theta}^2\|\D u\|_{L^2}^2+\bar{\theta}^{2\al}\|\D u\|_{L^4}^4+\bar{\theta}^{2\beta-2}\|\D \theta\|_{L^4}^4\right)\nonumber\\
 \leq~& C\bar{\theta}^{-2\beta}\left(\bar{\rho}K_4\bar{\theta}^{2\beta}+\bar{\rho}^2\bar{\theta}^2K_1+\bar{\theta}^{2\al}\|\D u\|_{L^2}\|\D u\|_{L^6}^3+\bar{\theta}^{2\beta-2}\|\D \theta\|_{L^2}\|\D^2 \theta\|_{L^2}^3\right)\nonumber\\
 \leq~& C(K)\left(\bar{\rho}+\bar{\rho}^2\bar{\theta}^{2-2\beta}+\bar{\rho}^\frac{3}{2}\bar{\theta}^{2\al-2\beta}+\bar{\theta}\right) \leq~ \bar{\theta}^2,
\end{align}
provided that $\alpha-\beta<1$ and $\bar{\rho}\leq C(K)\bar{\theta}^{\frac{4}{3}(\beta-\alpha+1)}$. Then (\ref{2dtheta-1}) is proved.

Then $I_{1i}$ in (\ref{3.3:8}) can be estimated as follows:
\begin{align}
 \label{I11} I_{11}=~&C\bar{\rho}\bar{\theta}^{\frac{3}{2}}\int_0^{\sigma(T)}\|\D G\|_{L^2}\dif t \leq~C\bar{\rho}\bar{\theta}^{\frac{3}{2}}\int_0^{\sigma(T)}\left(\bar{\rho}^\frac{1}{2}\bar{\theta}^{-\al}\|\sqrt{\rho}\dot{u}\|_{L^2}+\bar{\rho}\bar{\theta}^{-\al}\|\D\theta\|_{L^2}+\bar{\theta}^{-1}\|\D\theta\|_{L^4}\|\D u\|_{L^4}\right)\dif t\nonumber\\
 \leq~&C\bar{\rho}\bar{\theta}^{\frac{3}{2}}\left(\bar{\rho}^\frac{1}{2}K_1^\frac{1}{2}\bar{\theta}^{-\frac{1}{2}\al}+\bar{\rho}^\frac{3}{2}\bar{\theta}^{-\alpha+1-\frac{1}{2}\beta}+\bar{\theta}^{-1}\int_0^{\sigma(T)}\|\D \theta\|_{L^2}^{\frac{1}{4}}\|\D^2\theta\|_{L^2}^{\frac{3}{4}}\|\D u\|_{L^2}^{\frac{1}{4}}\|\D u\|_{L^6}^{\frac{3}{4}}\dif t\right)\nonumber\\
 \leq~&C(K)\bar{\rho}^\frac{3}{2}\bar{\theta}^{\frac{3}{2}-\frac{1}{2}\al}+C\bar{\rho}^\frac{5}{2}\bar{\theta}^{\frac{5}{2}-\alpha-\frac{1}{2}\beta}+CK_2^\frac{1}{8}\bar{\rho}\bar{\theta}^{\frac{5}{4}}\left(\int_0^{\sigma(T)}\|\D u\|_{L^2}^2\dif t\right)^{\frac{1}{8}}\left(\int_0^{\sigma(T)}\|\D u\|_{L^6}^2\dif t\right)^{\frac{3}{8}}\nonumber\\
 \leq~& C(K)\left(\bar{\rho}^\frac{3}{2}\bar{\theta}^{\frac{3}{2}-\frac{1}{2}\al}+\bar{\rho}^\frac{5}{2}\bar{\theta}^{\frac{5}{2}-\alpha-\frac{1}{2}\beta}+\bar{\rho}^\frac{3}{2}\bar{\theta}^{\frac{11}{8}-\frac{1}{2}\alpha}\right).
\end{align}
\begin{align}\label{I12}
I_{12}=~&CK_1\bar{\rho}\int_0^{\sigma(T)}\|\D u\|_{L^2}\|\D u\|_{L^6}\dif t\leq C(K)\bar{\rho}^2\bar{\theta}^{\frac{1}{2}-\al},
\end{align}
\begin{align}\label{I13}
I_{13}=~&C\bar{\theta}^{\al-1}\int_0^{\sigma(T)}\left(K_1^\frac{1}{4}\bar{\rho}^{-\frac{1}{2}}\|\sqrt{\rho}\dot{\theta}\|_{L^2}\|\D u\|_{L^6}^{\frac{3}{2}}+K_1^\frac{1}{2}K_2^\frac{1}{2}\|\D u\|_{L^6}^2\right)\dif t\nonumber\\
\leq~&C\bar{\theta}^{\al-1}\int_0^{\sigma(T)}\left(C(K)\bar{\rho}^{-\frac{1}{4}}\|\sqrt{\rho}\dot{\theta}\|_{L^2}\|\D u\|_{L^6}+K_1^\frac{1}{2}K_2^\frac{1}{2}\|\D u\|_{L^6}^2\right)\dif t\nonumber\\
\leq~&C(K)\bar{\rho}^{-\frac{1}{4}}\bar{\theta}^{\al-1}\left(\int_0^{\sigma(T)}\|\sqrt{\rho}\dot{\theta}\|_{L^2}^2\dif t\right)^{\frac{1}{2}}\left(\int_0^{\sigma(T)}\|\D u\|_{L^6}^2\dif t\right)^{\frac{1}{2}}+C(K)\bar{\rho}\bar{\theta}^{-1}\nonumber\\
\leq~&C(K)\left(\bar{\rho}^{\frac{1}{4}}\bar{\theta}^{\frac{1}{2}(\alpha+\beta)-1}+\bar{\rho}\bar{\theta}^{-1}\right),
\end{align}

Then, substituting \eqref{I11}--\eqref{I13} into \eqref{3.3:8}, we have
\begin{align}\label{3.3:9}
&\bar{\theta}^\al\sup\limits_{t\in[0,\sigma(T)]}\|\D u\|_{L^2}^2+\int_0^{\sigma(T)}\|\sqrt{\rho}\dot{u}\|_{L^2}^2\dif t\nonumber\\
\leq~&\frac{1}{2}K_1\bar{\theta}^\al+C(K)\left(\bar{\rho}^\frac{3}{2}\bar{\theta}^{\frac{3}{2}-\frac{1}{2}\al}+\bar{\rho}^\frac{5}{2}\bar{\theta}^{\frac{5}{2}-\alpha-\frac{1}{2}\beta}+\bar{\rho}^\frac{3}{2}\bar{\theta}^{\frac{11}{8}-\frac{1}{2}\alpha}+\bar{\rho}^2\bar{\theta}^{\frac{1}{2}-\al}+\bar{\rho}^{\frac{1}{4}}\bar{\theta}^{\frac{1}{2}(\alpha+\beta)-1}+\bar{\rho}\bar{\theta}^{-1}\right)\nonumber\\
\leq~& K_1\bar{\theta}^\al,
\end{align}
provided that $\alpha>\beta-2$ and
$$\bar{\rho}\leq C(K)\min\left\{\bar{\theta}^{\alpha-1}, \bar{\theta}^{\frac{1}{5}(4\alpha+\beta-5)}, \bar{\theta}^{2(\alpha-\beta+2)}\right\},$$
where we have taken $\bar{\theta}$ sufficiently large, such that
$$C(K)\left(\bar{\rho}^\frac{3}{2}\bar{\theta}^{\frac{3}{2}-\frac{1}{2}\al}+\bar{\rho}^\frac{5}{2}\bar{\theta}^{\frac{5}{2}-\alpha-\frac{1}{2}\beta}+\bar{\rho}^\frac{3}{2}\bar{\theta}^{\frac{11}{8}-\frac{1}{2}\alpha}+\bar{\rho}^2\bar{\theta}^{\frac{1}{2}-\al}+\bar{\rho}^{\frac{1}{4}}\bar{\theta}^{\frac{1}{2}(\alpha+\beta)-1}+\bar{\rho}\bar{\theta}^{-1}\right)\leq \frac{1}{2}K_1\bar{\theta}^\al.$$
This complete the proof of \eqref{g:3.3:1}.

Next, we will prove \eqref{g:3.3:2}. Multiplying $(\ref{FCNS:2})_3$ by $\theta_t$ and integrating the resulting equality over $\R^3$, we obtain after integration by parts that
\begin{align}
 \label{3.3:10}
&\frac{\kappa}{2}\frac{\dif}{\dif t}\int\theta^{\beta}|\D\theta|^2\dif x+\int\rho|\dot{\theta}|^2\dif x\nonumber\\
 =~&\int\rho\dot{\theta}u\cdot\D\theta\dif x+\lambda\int\theta^\al(\div u)^2\dot{\theta}\dif x+2\mu\int\theta^\al|\frD(u)|^2\dot{\theta}\dif x-R\int\rho\theta\div u\dot{\theta}\dif x+R\int\rho\theta u\cdot\D\theta\div u\dif x\nonumber\\
 &-\lambda\int\theta^\al u\cdot\D\theta(\div u)^2\dif x-2\mu\int\theta^\al u\cdot\D\theta|\frD(u)|^2\dif x+\frac{1}{2}\kappa\beta\int\theta^{\beta-1} \theta_t|\D\theta|^2\dif x =:~\sum\limits_{i=1}^8J_i,
\end{align}
where the terms $J_1$ to $J_8$ can be estimated as follows.
\begin{align}
 \label{3.3:11} J_1+J_4\leq~&C\bar{\rho}^\frac{1}{2}\int|\sqrt{\rho}\dot{\theta}|\left(|u\cdot\D\theta|+\bar{\theta}|\D u|\right)\dif x \leq~\frac{1}{8}\|\sqrt{\rho}\dot{\theta}\|_{L^2}^2+C\bar{\rho}\left(\|u\D\theta\|_{L^2}^2+\bar{\theta}^2\|\D u\|_{L^2}^2\right)\nonumber\\
 \leq~&\frac{1}{8}\|\sqrt{\rho}\dot{\theta}\|_{L^2}^2+C\bar{\rho}\left(\|\D u\|_{L^2}\|\D\theta\|_{L^2}^2\|\D u\|_{L^6}+\bar{\theta}^2\|\D u\|_{L^2}^2\right),
\end{align}
\begin{align}\label{3.3:13}
J_2+J_3\leq C\bar{\theta}^{\al}\int|\dot{\theta}||\D u|^2\dif x\leq C\bar{\rho}^{-\frac{1}{2}}\bar{\theta}^{\al}\|\sqrt{\rho}\dot{\theta}\|_{L^2}\|\D u\|_{L^2}^{\frac{1}{2}}\|\D u\|_{L^6}^{\frac{3}{2}},
\end{align}
\begin{align}\label{3.3:12}
J_5\leq C\bar{\rho}\bar{\theta}\|u\|_{L^\infty}\|\D u\|_{L^2}\|\D\theta\|_{L^2}\leq C\bar{\rho}\bar{\theta}\|\D u\|_{L^2}^\frac{3}{2}\|\D u\|_{L^6}^\frac{1}{2}\|\D\theta\|_{L^2},
\end{align}
\begin{align}\label{3.3:15}
J_6+J_7\leq C\bar{\theta}^\al\int|u\cdot\D\theta||\D u|^2\dif x
\leq~&C\bar{\theta}^\al\|u\|_{L^\infty}\|\D\theta\|_{L^2}\|\D u\|_{L^3}\|\D u\|_{L^6}\nonumber\\
\leq~&C\bar{\theta}^\al\|\D\theta\|_{L^2}\|\D u\|_{L^2}\|\D u\|_{L^6}^2,
\end{align}
and
\begin{align}\label{3.3:14-1}
J_8\leq~& C\bar{\theta}^{\beta-1}\int|\dot{\theta}||\D \theta|^2\dif x+C\bar{\theta}^{\beta-1}\int|u||\D \theta|^3\dif x\nonumber\\
\leq~&\frac{1}{8}\|\sqrt{\rho}\dot{\theta}\|_{L^2}^2+C\left(\bar{\rho}^{-1}\bar{\theta}^{2\beta-2}\|\D \theta\|_{L^2}\|\D^2 \theta\|_{L^2}^3+\bar{\theta}^{\beta-1}\|u\|_{L^\infty}\|\D\theta\|_{L^2}\|\D \theta\|_{L^3}\|\D \theta\|_{L^6}\right)\nonumber\\
\leq~&\frac{1}{8}\|\sqrt{\rho}\dot{\theta}\|_{L^2}^2+C(K)\bar{\rho}^{-1}\bar{\theta}^{2\beta-1}\|\D^2 \theta\|_{L^2}^2+C\bar{\theta}^{\beta-1}\|\nabla u\|_{L^2}^\frac{1}{2}\|\nabla u\|_{L^6}^\frac{1}{2}\|\D\theta\|_{L^2}^\frac{3}{2}\|\D^2 \theta\|_{L^2}^\frac{3}{2}\nonumber\\
\leq~&\frac{1}{8}\|\sqrt{\rho}\dot{\theta}\|_{L^2}^2+C(K)\left(\bar{\rho}^{-1}\bar{\theta}^{2\beta-1}\|\D^2 \theta\|_{L^2}^2+\bar{\theta}^{\beta-1}\|\nabla u\|_{L^6}^\frac{1}{2}\|\D^2 \theta\|_{L^2}^\frac{3}{2}\right).
\end{align}

At last, we substitute \eqref{3.3:11}--\eqref{3.3:14-1} into \eqref{3.3:10} and integrate the resulting inequality over $(0,\sigma(T))$. Noting that 
\begin{align}
 \label{estimate:D2theta:2-1} &\int_0^{\sigma(T)}\|\D^2\theta\|_{L^2}^2\dif t \leq~ C\bar{\theta}^{-2\beta}\int_0^{\sigma(T)}\left(\bar{\rho}\|\sqrt{\rho}\dot{\theta}\|_{L^2}^2+\bar{\rho}^2\bar{\theta}^2\|\D u\|_{L^2}^2+\bar{\theta}^{2\al}\|\D u\|_{L^4}^4+\bar{\theta}^{2\beta-2}\|\D \theta\|_{L^4}^4\right)\dif t\nonumber\\
 \leq~&C\bar{\theta}^{-2\beta}\left(K_2\bar{\rho}\bar{\theta}^\beta+\bar{\rho}^3\bar{\theta}^{3-\alpha}+\bar{\theta}^{2\al}\int_0^{\sigma(T)}\|\D u\|_{L^2}\|\D^2 u\|_{L^2}^3\dif t+\bar{\theta}^{2\beta-2}\int_0^{\sigma(T)}\|\D\theta\|_{L^2}\|\D^2\theta\|_{L^2}^3\dif t\right)\nonumber\\
 \leq~&CK_2\bar{\rho}\bar{\theta}^{-\beta}+C\bar{\rho}^3\bar{\theta}^{3-\alpha-2\beta}+C(K)^\frac{1}{2}K_1^\frac{1}{2}\bar{\rho}^\frac{1}{2}\bar{\theta}^{2\al-2\beta}\int_0^{\sigma(T)}\|\D^2 u\|_{L^2}^2\dif t+CK_2^\frac{1}{2}\bar{\theta}^{-1}\int_0^{\sigma(T)}\|\D^2\theta\|_{L^2}^2\dif t\nonumber\\
 \leq~&C(K)\left(\bar{\rho}\bar{\theta}^{-\beta}+\bar{\rho}^3\bar{\theta}^{3-\alpha-2\beta}+C(K)\bar{\rho}^\frac{3}{2}\bar{\theta}^{\al-2\beta}+C(K)\bar{\theta}^{-1-\beta}\right) \leq~C(K)\bar{\rho}\bar{\theta}^{-\beta},
\end{align}
we obtain
\begin{align}\label{3.3:16}
&\bar{\theta}^\beta\sup\limits_{t\in[0,\sigma(T)]}\|\D \theta\|_{L^2}^2+\int_0^{\sigma(T)}\|\sqrt{\rho}\dot{\theta}\|_{L^2}^2\dif t\nonumber\\
\leq~&C\bar{\theta}^\beta\|\D\theta_0\|_{L^2}^2+C\bar{\rho}\int_0^{\sigma(T)}\left(\|\D u\|_{L^2}\|\D\theta\|_{L^2}^2\|\D u\|_{L^6}+\bar{\theta}^2\|\D u\|_{L^2}^2\right)\dif t\nonumber\\
&+C\bar{\rho}^{-\frac{1}{2}}\bar{\theta}^{\al}\int_0^{\sigma(T)}\|\sqrt{\rho}\dot{\theta}\|_{L^2}\|\D u\|_{L^2}^{\frac{1}{2}}\|\D u\|_{L^6}^{\frac{3}{2}}\dif t+C\bar{\rho}\bar{\theta}\int_0^{\sigma(T)} \|\D u\|_{L^2}^\frac{3}{2}\|\D u\|_{L^6}^\frac{1}{2}\|\D\theta\|_{L^2}\dif t\nonumber\\
&+C\bar{\theta}^\al\int_0^{\sigma(T)} \|\D\theta\|_{L^2}\|\D u\|_{L^2}\|\D u\|_{L^6}^2\dif t+C(K)\bar{\rho}^{-1}\bar{\theta}^{2\beta-1}\int_0^{\sigma(T)} \|\D^2 \theta\|_{L^2}^2\dif t\nonumber\\
&+C(K)\bar{\theta}^{\beta-1}\int_0^{\sigma(T)}\|\nabla u\|_{L^6}^\frac{1}{2}\|\D^2 \theta\|_{L^2}^\frac{3}{2}\dif t\nonumber\\
\leq~&C\bar{\theta}^\beta\|\D\theta_0\|_{L^2}^2+C(K)\left(\bar{\rho}^\frac{3}{2}\bar{\theta}^{2-\beta}+\bar{\rho}^2\bar{\theta}^{3-\al}+\bar{\rho}^\frac{1}{4}\bar{\theta}^{\frac{1}{2}(\al+\beta)}+\bar{\rho}^2\bar{\theta}^{\frac{7}{4}-\al}+\bar{\rho}+\bar{\theta}^{\beta-1}+\bar{\rho}\bar{\theta}^{\frac{1}{4}\beta-\frac{1}{4}\alpha-1}\right)\nonumber\\
\leq~&K_2\bar{\theta}^\beta,
\end{align}
provided that $2C\|\D\theta_0\|_{L^2}^2\leq K_2$, $\beta>\alpha$, and
$$\bar{\rho}\leq C(K)\min\left\{\bar{\theta}^{\frac{1}{2}(\alpha+\beta-3)}, \bar{\theta}^{2(\beta-\alpha)}\right\},$$
where we have used (\ref{assume:1}),  (\ref{2dtheta-1}), and taken $\bar{\theta}$ sufficiently large, such that
$$C(K)\left(\bar{\rho}^\frac{3}{2}\bar{\theta}^{2-\beta}+\bar{\rho}^2\bar{\theta}^{3-\al}+\bar{\rho}^\frac{1}{4}\bar{\theta}^{\frac{1}{2}(\al+\beta)}+\bar{\rho}^2\bar{\theta}^{\frac{7}{4}-\al}+\bar{\rho}+\bar{\theta}^{\beta-1}+\bar{\rho}\bar{\theta}^{\frac{1}{4}\beta-\frac{1}{4}\alpha-1}\right)\leq\frac{1}{2}K_2\bar{\theta}^\beta.$$
Hence, we complete the proof of \eqref{g:3.3:2}.
\end{proof}

Second-order estimates of velocity and temperature on time interval $[0,\sigma(T)]$ are given in the following lemma.
\begin{lemma}\label{lem:3.4}$($Second-order estimates of velocity and temperature$)$ 
Under the conditions of Proposition \ref{p4.1}, there exists positive constants $K_3$, $K_4$ and $\rL_2$ depending on $\mu$, $\lambda$, $\kappa$, $R$ and $C_0$, such that if $(\rho,u,\theta)$ is a smooth solution of \eqref{FCNS:2}--\eqref{data and far} on $\R^3\times(0,T]$, the following estimate holds
\begin{equation}\label{g:3.4:1}
A_3(\sigma(T))=\sup\limits_{t\in[0,\sigma(T)]}\|\sqrt{\rho}\dot{u}\|_{L^2}^2+\bar{\theta}^\al\int_0^{\sigma(T)}\|\D \dot{u}\|_{L^2}^2\dif t\leq K_3\bar{\theta}^{2\alpha},
\end{equation}
\begin{equation}\label{g:3.4:2}
A_4(\sigma(T))=\sup\limits_{t\in[0,\sigma(T)]}\| \sqrt{\rho}\dot{\theta}\|_{L^2}^2+\bar{\theta}^\beta\int_0^{\sigma(T)}\|\D \dot{\theta}\|_{L^2}^2\dif t\leq K_4\bar{\theta}^{2\beta},
\end{equation}
provided that $\bar{\theta}\geq \rL_2$, $2\leq\alpha<\beta<\al+2$ and \eqref{rho-theta} holds.
\end{lemma}
\begin{proof}
Operating $\dot{u}^j(\pt+\mathrm{div}(u\cdot))$ to $(\ref{FCNS:2})_2^j$, summing with respect to $j$, and integrating the resulting equation over $\R^3$, we obtain
\begin{align}
 \label{3.4.1}
\frac{1}{2}\frac{\dif}{\dif t}\int\rho|\dot u|^2\dif x
=~& -\int\dot u^j\left(\partial_jP_t+\div(u\partial_j P)\right)\dif x+2\mu\int\dot u^j\left[\pt(\div(\theta^\al\frD(u)))+\div(u\div(\theta^\al\frD(u)))\right]\dif x\nonumber\\
 &+\lambda\int\dot u^j\left[\pt(\partial_j(\theta^\al\div u))+\div(u\cdot\partial_j(\theta^\al\div u))\right]\dif x =:~\sum_{i=1}^3M_i.
\end{align} 

First, from \cite{MR3604614} and \eqref{eqn:P}, using integration by parts,  we can see that
\begin{align}\label{3.4.3}
M_1=~&-\int\dot u^j\left(\partial_jP_t+\div (u\partial_j P)\right)\dif x\nonumber\\
=~&R\int\rho\dot{\theta}\div\dot{u}\dif x+\int P\left(\partial_k\left(u^k\div\dot{u}\right)-\div u\div\dot{u}-\partial_j\left(\partial_k\dot{u}^ju^k\right)\right)\dif x\nonumber\\
=~&R\int\rho\dot{\theta}\div\dot{u}\dif x-R\int \rho\theta\partial_k\dot{u}^j\partial_j u^k\dif x\nonumber\\
\leq~&\frac{\mu}{8}\left(\frac{\bar{\theta}}{2}\right)^\al\|\D\dot u\|_{L^2}^2+C\bar{\theta}^{-\alpha}\left(\|\rho\dot\theta\|_{L^2}^2+\|\rho\theta\D u\|_{L^2}^2\right)\nonumber\\
\leq~&\frac{\mu}{8}\left(\frac{\bar{\theta}}{2}\right)^\al\|\D\dot u\|_{L^2}^2+C\left(\bar{\rho}\bar{\theta}^{-\alpha}\|\sqrt{\rho}\dot\theta\|_{L^2}^2+\bar{\rho}^2\bar{\theta}^{-\alpha+2}\|\D u\|_{L^2}^2\right).
\end{align}

Next, we have
\begin{align}\label{3.4.4}
M_2=~&\mu\int\dot{u}^j\left(\pt\partial_i\left(\theta^\al(\partial_iu^j+\partial_ju^i)\right)+\partial_k\left(u^k\partial_i\left(\theta^\al(\partial_iu^j+\partial_ju^i)\right)\right)\right)\dif x\nonumber\\
=~&-\mu\int\partial_i\dot{u}^j\pt\left(\theta^\al(\partial_iu^j+\partial_ju^i)\right)\dif x+\mu\int\dot{u}^j\div u\partial_i\left(\theta^\al(\partial_iu^j+\partial_ju^i)\right)\dif x\nonumber\\
&+\mu\int\dot{u}^ju^k\partial_i\partial_k\left(\theta^\al(\partial_iu^j+\partial_ju^i)\right)\dif x\nonumber\\
=~&-\mu\int\partial_i\dot{u}^j\pt\left(\theta^\al(\partial_iu^j+\partial_ju^i)\right)\dif x-\mu\int\partial_i\dot{u}^j\div u\left(\theta^\al(\partial_iu^j+\partial_ju^i)\right)\dif x\nonumber\\
&-\mu\int\theta^\al\dot{u}^j\partial_i(\div u)(\partial_iu^j+\partial_ju^i)\dif x-\mu\int\partial_i\dot{u}^ju^k\partial_k\left(\theta^\al(\partial_iu^j+\partial_ju^i)\right)\dif x\nonumber\\
&-\mu\int\dot{u}^j\partial_i u^k\partial_k\left(\theta^\al(\partial_iu^j+\partial_ju^i)\right)\dif x\nonumber\\
=~&-\mu\int\partial_i\dot{u}^j\pt\left(\theta^\al(\partial_iu^j+\partial_ju^i)\right)\dif x-\mu\int\partial_i\dot{u}^j\div u\left(\theta^\al(\partial_iu^j+\partial_ju^i)\right)\dif x\nonumber\\
&-\mu\int\theta^\al\dot{u}^j\partial_i(\div u)(\partial_iu^j+\partial_ju^i)\dif x-\mu\int\partial_i\dot{u}^ju^k\partial_k\left(\theta^\al(\partial_iu^j+\partial_ju^i)\right)\dif x\nonumber\\
&+\mu\int\partial_k\dot{u}^j\partial_i u^k\left(\theta^\al(\partial_iu^j+\partial_ju^i)\right)\dif x+\mu\int\dot{u}^j\partial_k\partial_i u^k\left(\theta^\al(\partial_iu^j+\partial_ju^i)\right)\dif x\nonumber\\
=~&-\mu\int\partial_i\dot{u}^j\pt\left(\theta^\al(\partial_iu^j+\partial_ju^i)\right)\dif x-\mu\int\partial_i\dot{u}^ju^k\partial_k\left(\theta^\al(\partial_iu^j+\partial_ju^i)\right)\dif x\nonumber\\
&+\mu\int\theta^\al\partial_k\dot{u}^j\partial_iu^k(\partial_iu^j+\partial_ju^i)\dif x-\mu\int\theta^\al\partial_i\dot{u}^j\div u(\partial_i u^j+\partial_j u^i)\dif x\nonumber\\
\leq~& -2\mu\int\theta^\al|\frD(\dot{u})|^2\dif x+C\bar{\theta}^\al\int|\D\dot{u}||\D u|^2\dif x+C\int\theta^{\al-1}|\dot{\theta}||\D\dot{u}||\D u|\dif x\nonumber\\
\leq~&-2\mu\int\theta^\al|\frD(\dot{u})|^2\dif x+\frac{\mu}{8}\left(\frac{\bar{\theta}}{2}\right)^\al\|\D\dot u\|_{L^2}^2+C\bar{\theta}^\al\|\D u\|_{L^4}^4+C\bar{\theta}^{\al-2}\|\D{u}\|_{L^2}\|\D{u}\|_{L^6}\|\D{\dot{\theta}}\|_{L^2}^2,
\end{align}
where we have used the identity
\begin{align*}
&-\mu\int\partial_i\dot{u}^j\left(\pt+u\cdot\D\right)\left(\theta^\al(\partial_iu^j+\partial_ju^i)\right)\dif x\\
=~&-\mu\al\int\theta^{\al-1}\dot{\theta}\partial_i\dot{u}^j(\partial_iu^j+\partial_ju^i)\dif x-\mu\int\theta^\al\partial_i\dot{u}^j(\partial_i\dot{u}^j+\partial_j\dot{u}^i)\dif x\\
&+\mu\int\theta^\al\partial_i\dot{u}^j\left(\partial_iu^k\partial_ku^j+\partial_ju^k\partial_ku^i\right)\dif x.
\end{align*}

Similarly, we get
\begin{align}\label{3.4.5}
M_3
=~&\lambda\int\dot u^j\left(\pt\partial_j(\theta^\al\div u)+\div (u\partial_j(\theta^\al\div u))\right)\dif x\nonumber\\
=~&-\lambda\int\partial_j\dot u^j\pt(\theta^\al\div u)\dif x+\lambda\int\dot u^j\left(\div u\partial_j(\theta^\al\div u)+u\cdot\D\partial_j(\theta^\al\div u)\right)\dif x\nonumber\\
=~&-\lambda\int\div\dot u\pt(\theta^\al\div u)\dif x-\lambda\int\theta^\al\div\dot{u}(\div u)^2\dif x-\lambda\int\theta^\al\div u\dot{u}^j\partial_j(\div u)\dif x\nonumber\\
&-\lambda\int\div\dot{u}u\cdot\nabla(\theta^\al\div u)\dif x-\lambda\int\dot{u}^j\partial_ju^k\partial_k(\theta^\al\div u)\dif x\nonumber\\
=~&-\lambda\int\div\dot u\pt(\theta^\al\div u)\dif x-\lambda\int\div\dot{u}u\cdot\nabla(\theta^\al\div u)\dif x-\lambda\int\theta^\al\div\dot{u}(\div u)^2\dif x\nonumber\\
&-\lambda\int\theta^\al\div u\dot{u}^j\partial_j(\div u)\dif x+\lambda\int\theta^\al\partial_k\dot{u}^j\partial_ju^k\div u\dif x+\lambda\int\theta^\al\div u\dot{u}^j\partial_j(\div u)\dif x\nonumber\\
=~&-\lambda\al\int\theta^{\al-1}\div\dot{u}\div u\dot\theta \dif x-\lambda\int\theta^\al(\div\dot{u})^2\dif x+\lambda\int\theta^\al\div\dot{u}\partial_iu^j\partial_ju^i\dif x\nonumber\\
&-\lambda\int\theta^\al\div\dot{u}(\div u)^2\dif x+\lambda\int\theta^\al\partial_k\dot{u}^j\partial_ju^k\div u\dif x\nonumber\\
\leq~&-\lambda\int\theta^\al(\div\dot{u})^2\dif x+\frac{\mu}{8}\left(\frac{\bar{\theta}}{2}\right)^\al\|\D\dot{u}\|_{L^2}^2+C\bar{\theta}^\al\|\D{u}\|_{L^4}^4+C\bar{\theta}^{\al-2}\|\D{u}\|_{L^2}\|\D{u}\|_{L^6}\|\D{\dot{\theta}}\|_{L^2}^2,
\end{align}
where we have used the fact
\begin{align*}
&-\lambda\int\div\dot{u}\left(\pt(\theta^\al\div u)+u\cdot\D(\theta^\al\div u)\right)\dif x\\
=~&-\lambda\int\div\dot{u}\left(\div u(\pt(\theta^\al)+u\cdot\D(\theta^\al))+\theta^\al(\div u_t+u\cdot\D(\div u))\right)\dif x\\
=~&-\lambda\int\div\dot{u}\left(\al\theta^{\al-1}\div u(\theta_t+u\cdot\D\theta)+\theta^\al(\div\dot{u}+u\cdot\D(\div u)-\div((u\cdot\D)u))\right)\dif x\\
=~&-\lambda\al\int\theta^{\al-1}\div\dot{u}\div u\dot\theta \dif x-\lambda\int\theta^\al(\div\dot{u})^2\dif x+\lambda\int\theta^\al\div\dot{u}\partial_iu^j\partial_ju^i\dif x.
\end{align*}

The physical condition $2\mu+3\lambda\geq0$ gives
\begin{align*}
 &\int\theta^\al\left(2\mu|\frD(\dot{u})|^2+\lambda(\div\dot{u})^2\right)\dif x \geq~2\mu\left(\frac{\bar{\theta}}{2}\right)^\al\int\left(|\frD(\dot{u})|^2-\frac13(\div\dot{u})^2\right)\dif x\\
 =~&\mu\left(\frac{\bar{\theta}}{2}\right)^\al\left(\|\D\dot{u}\|_{L^2}^2+\frac13\|\div\dot{u}\|_{L^2}^2\right)
\geq\mu\left(\frac{\bar{\theta}}{2}\right)^\al\|\D\dot{u}\|_{L^2}^2.
\end{align*}
Here the first inequality uses $|\frD(\dot{u})|^2-\frac13(\div\dot{u})^2\geq0$ before applying the lower bound for $\theta$. Combining the two stress dissipations in \eqref{3.4.4}--\eqref{3.4.5}, and then using \eqref{3.4.1} and \eqref{3.4.3}, we get
\begin{align}\label{3.4.6}
&\frac{\dif}{\dif t}\int\rho|\dot u|^2\dif x+\frac{\mu}{4}\left(\frac{\bar{\theta}}{2}\right)^\al\|\D\dot u\|_{L^2}^2\nonumber\\
\leq~&C\left(\bar{\rho}\bar{\theta}^{-\alpha}\|\sqrt{\rho}\dot\theta\|_{L^2}^2+\bar{\rho}^2\bar{\theta}^{-\alpha+2}\|\D u\|_{L^2}^2+\bar{\theta}^\al\|\D u\|_{L^4}^4+\bar{\theta}^{\al-2}\|\D{u}\|_{L^2}\|\D{u}\|_{L^6}\|\D{\dot{\theta}}\|_{L^2}^2\right)\nonumber\\
\leq~&C\left(\bar{\rho}\bar{\theta}^{-\alpha}\|\sqrt{\rho}\dot\theta\|_{L^2}^2+\bar{\rho}^2\bar{\theta}^{-\alpha+2}\|\D u\|_{L^2}^2+\bar{\theta}^\al\|\D u\|_{L^2}\|\D u\|_{L^6}^3+\bar{\theta}^{\al-2}\|\D{u}\|_{L^2}\|\D{u}\|_{L^6}\|\D{\dot{\theta}}\|_{L^2}^2\right).
\end{align}

Integrating \eqref{3.4.6} over $(0,\sigma(T))$ and using \eqref{assume:1} and \eqref{estimate:DuL6}, we obtain
\begin{align}
 \label{3.4.7}
&\sup\limits_{t\in[0,\sigma(T)]}\|\sqrt{\rho}\dot{u}\|_{L^2}^2+\bar{\theta}^\al\int_0^{\sigma(T)}\|\D \dot{u}\|_{L^2}^2\dif t\nonumber\\
 \leq~&\|\sqrt{\rho}\dot{u}\|_{L^2}^2\Big|_{t=0}+C\bar{\rho}\bar{\theta}^{-\al}\int_0^{\sigma(T)}\|\sqrt{\rho}\dot\theta\|_{L^2}^2\dif t+C\bar{\rho}^2\bar{\theta}^{-\alpha+2}\int_0^{\sigma(T)}\|\D u\|_{L^2}^2\dif t+C\bar{\theta}^\al\int_0^{\sigma(T)}\|\D u\|_{L^2}\|\D u\|_{L^6}^3\dif t\nonumber\\
 &+C\bar{\theta}^{\al-2}\int_0^{\sigma(T)}\|\D{u}\|_{L^2}\|\D{u}\|_{L^6}\|\D{\dot{\theta}}\|_{L^2}^2\dif t \leq~C\bar{\theta}^{2\al}+C(K)\left(\bar{\rho}\bar{\theta}^{\beta-\al}+\bar{\rho}^3\bar{\theta}^{3-2\al}+\bar{\rho}^\frac{3}{2}+\bar{\rho}^{\frac{1}{2}}\bar{\theta}^{\alpha-2+\beta}\right)\nonumber\\
 \leq~& K_3\bar{\theta}^{2\alpha},
\end{align}
provided that $2C\bar{\theta}^{2\al}\leq K_3\bar{\theta}^{2\alpha}$, $\beta>\al>\beta-2$, and
$$\bar{\rho}\leq C(K)\min\left\{\bar{\theta}^{3\alpha-\beta}, \bar{\theta}^{\frac13(4\alpha-3)}, \bar{\theta}^{\frac{4}{3}\alpha}, \bar{\theta}^{2(\alpha-\beta+2)}\right\},$$
where we have taken $\bar{\theta}$ large enough, such that
$$C(K)\left(\bar{\rho}\bar{\theta}^{\beta-\al}+\bar{\rho}^3\bar{\theta}^{3-2\al}+\bar{\rho}^\frac{3}{2}+\bar{\rho}^{\frac{1}{2}}\bar{\theta}^{\alpha-2+\beta}\right)\leq\frac{1}{2}K_3\bar{\theta}^{2\alpha}.$$
Then we complete the proof of \eqref{g:3.4:1}.

Now we give the proof of (\ref{g:3.4:2}). Applying the operator $\pt+\div(u\cdot)$ to $(\ref{FCNS:2})_3$, we get 
\begin{align}\label{3.5:3}
&\rho\pt\dot{\theta}+\rho u\cdot\D\dot{\theta}-\kappa\div\left(\theta^\beta\D\dot{\theta}\right)\nonumber\\
=~&-R\rho\dot{\theta}\div u-R\rho\theta\div\dot{u}+R\rho\theta\partial_iu^j\partial_ju^i+\theta^\al\div u\left(2\mu|\frD(u)|^2+\lambda(\div u)^2\right)\nonumber\\
&+\al\theta^{\al-1}\dot\theta\left(2\mu|\frD(u)|^2+\lambda(\div u)^2\right)+\mu\theta^\al\left(\partial_iu^j+\partial_ju^i\right)\left(\partial_i\dot{u}^j+\partial_j\dot{u}^i\right)\nonumber\\
&-\mu\theta^\al\left(\partial_iu^j+\partial_ju^i\right)\left(\partial_iu^k\partial_ku^j+\partial_ju^k\partial_ku^i\right)+2\lambda\theta^\al\div u\div\dot{u}-2\lambda\theta^\al\div u\partial_iu^j\partial_ju^i\nonumber\\
&+\kappa\beta\div\left(\theta^{\beta-1}\dot{\theta}\D\theta\right)+\kappa\div\left(u\div(\theta^\beta\D\theta)-\theta^\beta\D(u\cdot\D\theta)-\beta\theta^{\beta-1}u\cdot\D\theta\D\theta\right).
\end{align}

Multiplying $\dot{\theta}$ on both sides of (\ref{3.5:3}), we have
\begin{align}\label{3.5:5}
&\frac{1}{2}\frac{\dif}{\dif t}\int\rho|\dot{\theta}|^2\dif x+\kappa\int\theta^\beta|\D\dot{\theta}|^2\dif x\nonumber\\
\leq~&C\int\rho|\dot{\theta}|^2|\D u|\dif x+C\bar{\theta}\int\rho|\dot{\theta}|\left(|\D\dot{u}|+|\D u|^2\right)\dif x+C\bar{\theta}^{\alpha}\int|\dot{\theta}||\D u|^3\dif x+C\bar{\theta}^{\alpha-1}\int|\dot{\theta}|^2|\D u|^2\dif x\nonumber\\
&+C\bar{\theta}^{\alpha}\int|\dot{\theta}||\D u||\D\dot{u}|\dif x+C\bar{\theta}^{\beta-1}\int|\nabla{\theta}||\dot\theta||\D\dot{\theta}|\dif x+C\bar{\theta}^{\beta-1}\int|u||\nabla \theta|^2|\D\dot\theta|\dif x\nonumber\\
&+C\bar{\theta}^{\beta}\int|u||\nabla^2{\theta}||\D\dot\theta|\dif x+C\bar{\theta}^{\beta}\int|\nabla u||\nabla{\theta}||\D\dot\theta|\dif x\nonumber\\
=:~&\sum_{i=1}^9N_i, 
\end{align}
where $N_i$ can be estimated as 
\begin{align}\label{O2}
N_1&=C\int\rho|\dot{\theta}|^2|\D u|\dif x\leq C\bar{\rho}^\frac{1}{2}\|\sqrt{\rho}\dot{\theta}\|_{L^2}\|\D u\|_{L^3}\|\dot\theta\|_{L^6}\leq C\bar{\rho}^\frac{1}{2}\|\sqrt{\rho}\dot{\theta}\|_{L^2}\|\D u\|_{L^2}^\frac{1}{2}\|\D u\|_{L^6}^\frac{1}{2}\|\D\dot\theta\|_{L^2}\nonumber\\
&\leq\delta\bar{\theta}^\beta\|\D\dot\theta\|_{L^2}^2+C(K)\bar{\rho}^{\frac{3}{2}}\bar{\theta}^{-\beta}\|\sqrt{\rho}\dot{\theta}\|_{L^2}^2,
\end{align}
\begin{align}\label{O3}
N_2&=C\bar{\theta}\int\rho|\dot{\theta}|\left(|\D\dot{u}|+|\D u|^2\right)\dif x\leq C\bar{\rho}^\frac{1}{2}\bar{\theta}\|\sqrt{\rho}\dot{\theta}\|_{L^2}\|\D \dot u\|_{L^2}+C\bar{\rho}^\frac{1}{2}\bar{\theta}\|\sqrt{\rho}\dot{\theta}\|_{L^2}\|\D u\|_{L^3}\|\D u\|_{L^6}\nonumber\\
&\leq C\bar{\rho}^\frac{1}{2}\bar{\theta}\|\sqrt{\rho}\dot{\theta}\|_{L^2}\|\D \dot u\|_{L^2}+C(K)\bar{\rho}\bar{\theta}\|\sqrt{\rho}\dot{\theta}\|_{L^2}\|\D u\|_{L^2}^\frac{1}{2}\|\D u\|_{L^6}^\frac{1}{2},
\end{align}
\begin{align}\label{O4}
N_3&=C\bar{\theta}^{\alpha}\int|\dot{\theta}||\D u|^3\dif x\leq C\bar{\theta}^{\alpha}\|\D u\|_{L^2}\|\D u\|_{L^6}^2\|\dot\theta\|_{L^6}\leq C\bar{\theta}^{\alpha}\|\D u\|_{L^2}\|\D u\|_{L^6}^2\|\D\dot\theta\|_{L^2}\nonumber\\
&\leq \delta\bar{\theta}^\beta\|\D\dot\theta\|_{L^2}^2+C(K)\bar{\rho}\bar{\theta}^{2\alpha-\beta}\|\D u\|_{L^6}^2,
\end{align}
\begin{align}\label{O5}
N_4&=C\bar{\theta}^{\alpha-1}\int|\dot{\theta}|^2|\D u|^2\dif x\leq C\bar{\theta}^{\alpha-1}\|\D u\|_{L^2}\|\D u\|_{L^6}\|\dot\theta\|_{L^6}^2\leq C(K)\bar{\rho}^\frac{1}{2}\bar{\theta}^{\alpha-1}\|\nabla\dot\theta\|_{L^2}^2,
\end{align}
\begin{align}\label{O6}
N_5&=C\bar{\theta}^{\alpha}\int|\dot{\theta}||\D u||\D\dot{u}|\dif x\leq C\bar{\theta}^{\alpha}\|\dot{\theta}\|_{L^6}\|\D {u}\|_{L^3}\|\D \dot{u}\|_{L^2}\leq C\bar{\theta}^{\alpha}\|\D\dot{\theta}\|_{L^2}\|\D u\|_{L^2}^\frac{1}{2}\|\D u\|_{L^6}^\frac{1}{2}\|\D \dot{u}\|_{L^2}\nonumber\\
&\leq \delta\bar{\theta}^\beta\|\D\dot\theta\|_{L^2}^2+C(K)\bar{\rho}^\frac{1}{2}\bar{\theta}^{2\alpha-\beta}\|\D \dot{u}\|_{L^2}^2,
\end{align}
\begin{align}\label{O1}
N_6&=C\bar{\theta}^{\beta-1}\int|\D\dot{\theta}||\D\theta||\dot\theta|\dif x\leq C\bar{\theta}^{\beta-1}\|\D \dot\theta\|_{L^2}\|\D \theta\|_{L^3}\|\dot\theta\|_{L^6}\leq C\bar{\theta}^{\beta-1}\|\D \dot\theta\|_{L^2}^2\|\D \theta\|_{L^2}^\frac{1}{2}\|\D^2 \theta\|_{L^2}^\frac{1}{2}\nonumber\\
&\leq CK_2^\frac{1}{2}\bar{\theta}^{\beta-\frac{1}{2}}\|\D \dot\theta\|_{L^2}^2,
\end{align}
\begin{align}\label{O78}
N_7&=C\bar{\theta}^{\beta-1}\int|u||\nabla \theta|^2|\D\dot\theta|\dif x\leq C\bar{\theta}^{\beta-1}\|\D u\|_{L^2}^{\frac{1}{2}}\|\D u\|_{L^6}^{\frac{1}{2}}\|\D\theta\|_{L^2}^{\frac{1}{2}}\|\D^2\theta\|_{L^2}^{\frac{3}{2}}\|\D\dot{\theta}\|_{L^2}\nonumber\\
&\leq \delta\bar{\theta}^\beta\|\D\dot\theta\|_{L^2}^2+ C\bar{\theta}^{\beta-2}\|\D u\|_{L^2}\|\D u\|_{L^6}\|\D\theta\|_{L^2}\|\D^2\theta\|_{L^2}^3\nonumber\\
&\leq \delta\bar{\theta}^\beta\|\D\dot\theta\|_{L^2}^2+ C(K)\bar{\theta}^{\beta+1}\|\D u\|_{L^2}\|\D u\|_{L^6},
\end{align}
\begin{align}
 \label{O910} N_8+N_9&=C\bar{\theta}^{\beta}\int|u||\nabla^2{\theta}||\D\dot\theta|\dif x+C\bar{\theta}^{\beta}\int|\nabla u||\nabla{\theta}||\D\dot\theta|\dif x \leq C\bar{\theta}^{\beta}\|\D u\|_{L^2}^{\frac{1}{2}}\|\D u\|_{L^6}^{\frac{1}{2}}\|\D^2\theta\|_{L^2}\|\D\dot\theta\|_{L^2}\nonumber\\
 &\leq \delta\bar{\theta}^\beta\|\D\dot\theta\|_{L^2}^2+C(K)\bar{\rho}^{\frac{1}{2}}\bar{\theta}^{\beta}\|\D^2 \theta\|_{L^2}^2.
\end{align}

Taking $\delta=\frac{\kappa}{10\cdot 2^\beta}$ in (\ref{O2})--(\ref{O910}) and inserting them into (\ref{3.5:5}), we obtain after integration with respect to $t\in[0,T]$ that
\begin{align}\label{3.5:12}
&\sup\limits_{t\in[0,\sigma(T)]}\| \sqrt{\rho}\dot{\theta}\|_{L^2}^2+\bar{\theta}^\beta\int_0^{\sigma(T)}\|\D \dot{\theta}\|_{L^2}^2\dif t\nonumber\\
\leq~& \| \sqrt{\rho}\dot{\theta}\|_{L^2}^2\Big|_{t=0}+C(K)\int_0^{\sigma(T)}\bigg(\bar{\rho}^{\frac{3}{2}}\bar{\theta}^{-\beta}\|\sqrt{\rho}\dot{\theta}\|_{L^2}^2+\bar{\rho}^\frac{1}{2}\bar{\theta}\|\sqrt{\rho}\dot{\theta}\|_{L^2}\|\D \dot u\|_{L^2}+\bar{\rho}\bar{\theta}\|\sqrt{\rho}\dot{\theta}\|_{L^2}\|\D u\|_{L^2}^\frac{1}{2}\|\D u\|_{L^6}^\frac{1}{2}\nonumber\\
&+\bar{\rho}\bar{\theta}^{2\alpha-\beta}\|\D u\|_{L^6}^2+\bar{\rho}^{\frac{1}{2}}\bar{\theta}^{2\alpha-\beta}\|\D \dot{u}\|_{L^2}^2+\bar{\theta}^{\beta+1}\|\D u\|_{L^2}\|\D u\|_{L^6}+\bar{\rho}^{\frac{1}{2}}\bar{\theta}^{\beta}\|\D^2 \theta\|_{L^2}^2\bigg)\dif t\nonumber\\
\leq~& C\bar{\theta}^{2\beta}+C(K)\left(\bar{\rho}^\frac{3}{2}+\bar{\rho}^\frac{1}{2}\bar{\theta}^{1+\frac{1}{2}(\al+\beta)}+\bar{\rho}^\frac{3}{2}\bar{\theta}^{\frac{3}{2}-\alpha+\frac{1}{2}\beta}+\bar{\rho}^2\bar{\theta}^{\al-\beta}+\bar{\rho}^\frac{1}{2}\bar{\theta}^{3\alpha-\beta}
+\bar{\rho}\bar{\theta}^{\beta-\al+\frac{3}{2}}+\bar{\rho}^\frac{3}{2}\right)\nonumber\\
\leq~& K_4\bar{\theta}^{2\beta},
\end{align}
provided that $2C\bar{\theta}^{2\beta}\leq K_4\bar{\theta}^{2\beta}$, $\alpha<\beta$, and
$$\bar{\rho}\leq C(K)\min\left\{\bar{\theta}^{\frac{4}{3}\beta}, \bar{\theta}^{3\beta-\alpha-2}, \bar{\theta}^{\frac{2}{3}\alpha+\beta-1}, \bar{\theta}^{\frac{1}{2}(3\beta-\alpha)}, \bar{\theta}^{6(\beta-\alpha)}, \bar{\theta}^{\alpha+\beta-\frac{3}{2}}\right\},$$
where we have taken $\bar{\theta}$ large enough, such that
$$C(K)\left(\bar{\rho}^\frac{3}{2}+\bar{\rho}^\frac{1}{2}\bar{\theta}^{1+\frac{1}{2}(\al+\beta)}+\bar{\rho}^\frac{3}{2}\bar{\theta}^{\frac{3}{2}-\alpha+\frac{1}{2}\beta}+\bar{\rho}^2\bar{\theta}^{\al-\beta}+\bar{\rho}^\frac{1}{2}\bar{\theta}^{3\alpha-\beta}
+\bar{\rho}\bar{\theta}^{\beta-\al+\frac{3}{2}}\right)\leq\frac{1}{2}K_4\bar{\theta}^{2\beta}.$$
Hence, the proof of Lemma \ref{lem:3.4} is completed.
\end{proof}

\subsubsection{Time-weighted uniform estimates of velocity and temperature on $[0,T]$}\label{sec3.2}

Next, in Lemmas \ref{Lem:4}--\ref{Lem:5}, we will establish the time-weighted uniform estimates of $(u,\theta)$ on the time interval $[0,T]$. Moreover, by Gagliardo-Nirenberg inequality, we will give the uniform upper and lower bounds of temperature.

\begin{lemma}\label{Lem:4}$($First-order time-weighted estimates of velocity and temperature$)$ 
Under the conditions of Proposition \ref{p4.1}, there exist positive constants $K_5$, $K_6$ and $\rL_3$ depending on $\mu$, $\lambda$, $\kappa$, and $R$ such that if $(\rho,u,\theta)$ is a smooth solution of \eqref{FCNS:2}--\eqref{data and far} on $\R^3\times(0,T]$, it holds
\begin{equation}\label{Lem:4:1-1}
B_1(T)=\bar{\theta}^\al\sup\limits_{t\in[0,T]}\left(\sigma^3\|\D u\|_{L^2}^2\right)+\int_0^T\sigma^3\|\sqrt{\rho}\dot{u}\|_{L^2}^2\dif t\leq K_5\bar{\rho}\bar{\theta},
\end{equation}
\begin{equation}\label{Lem:4:1-2}
B_2(T)=\bar{\theta}^\beta\sup\limits_{t\in[0,T]}\left(\sigma^4\|\D\theta\|_{L^2}^2\right)+\int_0^T\sigma^4\| \sqrt{\rho}\dot{\theta}\|_{L^2}^2\dif t\leq K_6\bar{\rho}\bar{\theta}^2,
\end{equation}
provided that $2\leq\alpha<\beta<\alpha+2$, $\alpha\leq 10$, $\bar{\theta}\geq\rL_3$,
\begin{equation}\label{rho-theta-1}
\bar{\rho}\leq C(K)\bar{\theta}^{\frac{2}{3}},
\end{equation}
and \eqref{rho-theta} holds.
\end{lemma}
\begin{proof}
First, we claim that
\begin{equation}\label{claim2}
\|\D^2\theta\|_{L^2}\leq 2\bar{\theta}, \quad\text{for} ~ t\in[0,T].
\end{equation}

\textbf{Step 1}: Estimates of $\|\D u\|_{L^6}$ and $\|\D^2\theta\|_{L^2}$.

Combining \eqref{L6}, \eqref{assume:2} and \eqref{claim2}, we obtain
\begin{align}\label{estimate:DuL6:4}
\sigma^5\|\D u\|_{L^6}^2\leq~&C\sigma^5\bar{\rho}\bar{\theta}^{-2\al}\|\sqrt{\rho}\dot{u}\|_{L^2}^2+C\sigma^4\bar{\rho}^2\bar{\theta}^{-2\al}\|\D\theta\|_{L^2}^2+C\bar{\theta}^{2-2\alpha}\|\nabla\rho\|_{L^2}^2+C(K)\sigma^5\bar{\theta}^{-2}\|\D\theta\|_{L^2}^2\|\D u\|_{L^2}^2\nonumber\\
\leq~&C(K)\left(\bar{\rho}\bar{\theta}^{\frac{2}{3}-\frac{5}{3}\al}+\bar{\rho}^3\bar{\theta}^{2-2\al-\beta}+\bar{\theta}^{2-2\al}+\bar{\rho}^{\frac{3}{2}}\bar{\theta}^{-\al-\frac{\beta}{2}}\right)\nonumber\\
\leq~&C(K)\left(\bar{\rho}\bar{\theta}^{\frac{2}{3}-\frac{5}{3}\al}+\bar{\theta}^{2-2\al}+\bar{\rho}^{\frac{3}{2}}\bar{\theta}^{-\al-\frac{\beta}{2}}\right),
\end{align}
provided that $\bar{\rho}\leq C(K)\bar{\theta}^{\frac{\beta}{3}}$. By \eqref{L6}, \eqref{assume:1}, \eqref{assume:2} and \eqref{lem3.1:g2}, we obtain
\begin{align}
 \label{estimate:DuL6:5}
\int_0^T\|\D u\|_{L^6}^2\dif t\leq~&C\bar{\rho}\bar{\theta}^{-2\al}\int_0^T\|\sqrt{\rho}\dot{u}\|_{L^2}^2\dif t+C\bar{\rho}^2\bar{\theta}^{-2\al}\int_0^T\|\D\theta\|_{L^2}^2\dif t\nonumber\\
 &+C\bar{\theta}^{2-2\alpha}\int_0^T\|\D\rho\|_{L^2}^2\dif t+C\bar{\theta}^{-2}\int_0^T\|\D u\|_{L^2}^2\|\D\theta\|_{L^2}^2\dif t\nonumber\\
 \leq~&C\bar{\rho}\bar{\theta}^{-2\al}\left(\int_0^{\sigma(T)}\|\sqrt{\rho}\dot{u}\|_{L^2}^2\dif t+\int_0^T\sigma^3\|\sqrt{\rho}\dot{u}\|_{L^2}^2\dif t\right)+C\bar{\rho}^2\bar{\theta}^{-2\al}\int_0^T\|\D\theta\|_{L^2}^2\dif t\nonumber\\
 &+C\bar{\theta}^{2-2\alpha}\int_0^T\|\D\rho\|_{L^2}^2\dif t+C\bar{\theta}^{-2}\sup\limits_{t\in[0,T]}\|\D\theta\|_{L^2}^2\int_0^T\|\D u\|_{L^2}^2\dif t \leq~CK_8\bar{\rho}^{-1}\bar{\theta}^{1-\al}.
\end{align}

Moreover, from \eqref{assume:1}, \eqref{assume:2}, \eqref{estimate:DuL6}, \eqref{estimate:DuL6-1}, \eqref{claim2}, \eqref{estimate:DuL6:4} and \eqref{estimate:DuL6:5}, we have
\begin{align}\label{estimate:D2theta-1}
\sigma^6\|\D^2\theta\|_{L^2}^2\leq~&C\left(\bar{\rho}\bar{\theta}^{-2\beta}\sigma^6\|\sqrt{\rho}\dot{\theta}\|_{L^2}^2+\bar{\rho}^2\bar{\theta}^{2-2\beta}\sigma^3\|\D u\|_{L^2}^2+\bar{\theta}^{2(\al-\beta)}\sigma^6\|\D u\|_{L^2}\|\D u\|_{L^6}^3\right)\nonumber\\
&+C\bar{\theta}^{-2}\|\D\theta\|_{L^2}\|\D^2\theta\|_{L^2}\left(\sigma^6\|\D^2\theta\|_{L^2}^2\right)\nonumber\\
\leq~&C(K)\left(\bar{\rho}^2\bar{\theta}^{2-2\beta}+\bar{\rho}^3\bar{\theta}^{3-\al-2\beta}\right)+C(K)\bar{\theta}^{2(\al-\beta)}\left(\sigma^3\|\D u\|_{L^2}^2\right)^{\frac{1}{2}}\left(\sigma^5\|\D u\|_{L^6}^2\right)^{\frac{9}{10}}\|\D u\|_{L^6}^{\frac65}\nonumber\\
&+C(K)\bar{\theta}^{-1}\sigma^6\|\D^2\theta\|_{L^2}^2\nonumber\\
\leq~&C(K)\left(\bar{\rho}^2\bar{\theta}^{2-2\beta}+\bar{\rho}^{2}\bar{\theta}^{\frac{11}{10}-2\beta}+\bar{\rho}^{\frac{13}{5}}\bar{\theta}^{\frac{1}{2}+\frac{3}{5}\al-\frac{49}{20}\beta}+\bar{\rho}^{\frac{11}{10}}\bar{\theta}^{\frac{23}{10}-\frac{3}{10}\alpha-2\beta}\right)+\frac{\sigma^6}{2}\|\D^2\theta\|_{L^2}^2\nonumber\\
\leq~&C(K)\bar{\rho}^2\bar{\theta}^{2-2\beta},
\end{align}
where we used $\sup_{t\in[0,T]}\|\D u\|_{L^6}^{6/5}\leq C(K)\bar{\rho}^{3/5}$. The last inequality follows from $\bar{\rho}\leq C(K)\bar{\theta}^{1/3}$, $\alpha\geq2$ and $4\alpha-3\beta<4$. Moreover,
\begin{align}\label{estimate:D2theta-2}
\int_0^T\sigma^4\|\D^2\theta\|_{L^2}^2\dif t\leq~&C\int_0^T\left(\bar{\rho}\bar{\theta}^{-2\beta}\sigma^4\|\sqrt{\rho}\dot{\theta}\|_{L^2}^2+\bar{\rho}^2\bar{\theta}^{2-2\beta}\|\D u\|_{L^2}^2+\bar{\theta}^{2(\al-\beta)}\sigma^4\|\D u\|_{L^2}\|\D u\|_{L^6}^3\right)\dif t\nonumber\\
&+C\bar{\theta}^{-2}\sup\limits_{t\in[0,T]}\left(\|\D\theta\|_{L^2}\|\D^2\theta\|_{L^2}\right)\int_0^T\sigma^4\|\D^2\theta\|_{L^2}^2\dif t\nonumber\\
\leq~&C(K)\left(\bar{\rho}^2\bar{\theta}^{2-2\beta}+\bar{\rho}^3\bar{\theta}^{3-\al-2\beta}+\bar{\theta}^{\frac{11}{6}-\frac{\alpha}{3}-2\beta}+\bar{\rho}^{-\frac{1}{2}}\bar{\theta}^{\frac{5}{2}-2\beta}+\bar{\rho}^{\frac{1}{4}}\bar{\theta}^{\frac{3}{2}-\frac{9}{4}\beta}\right) \nonumber\\
&+\frac{1}{2}\int_0^T\sigma^4\|\D^2\theta\|_{L^2}^2\dif t\nonumber\\
\leq~&C(K)\bar{\theta}^{\frac{5}{2}-2\beta}.
\end{align}
Combining \eqref{estimate:D2theta-0} and \eqref{estimate:D2theta-1}, \eqref{claim2} is proved.

\textbf{Step 2}: Estimates of $B_1(T)$.

Multiplying $(\ref{FCNS:2})_2$ by $\sigma^3\dot{u}$ and integrating the resulting equality over $\R^3$ yield
\begin{equation}
 \label{Lem:4:2} \sigma^3\int\rho|\dot{u}|^2\dif x=~-\sigma^3\int\dot{u}\cdot\D P\dif x-2\mu\sigma^3\int\theta^\al\D\dot{u}:\frD(u)\dif x-\lambda\sigma^3\int\theta^\al\div u\div\dot{u}\dif x =:~\sum\limits_{i=1}^3P_i.
\end{equation}

Noting \eqref{eqn:P}, we get after integration by parts that
\begin{align}
 \label{Lem:4:4}
P_1=~&R\sigma^3\int(\rho\theta-\bar{\rho}\bar{\theta})\div u_t\dif x+R\sigma^3\int(\rho\theta-\bar{\rho}\bar{\theta})u\cdot\D\div u\dif x+R\sigma^3\int(\rho\theta-\bar{\rho}\bar{\theta})\partial_ju^i\partial_iu^j\dif x\nonumber\\
 =~&R\frac{\dif}{\dif t}\left(\sigma^3\int(\rho\theta-\bar{\rho}\bar{\theta})\div u\dif x\right)-3R\sigma^2\sigma'\int(\rho\theta-\bar{\rho}\bar{\theta})\div u\dif x-\sigma^3\int P_t\div u \dif x\nonumber\\
 &+R\sigma^3\int(\rho\theta-\bar{\rho}\bar{\theta})u\cdot\D\div u\dif x+R\sigma^3\int(\rho\theta-\bar{\rho}\bar{\theta})\partial_ju^i\partial_iu^j\dif x\nonumber\\
 =~&R\frac{\dif}{\dif t}\left(\sigma^3\int(\rho\theta-\bar{\rho}\bar{\theta})\div u\dif x\right)-3R\sigma^2\sigma'\int(\rho\theta-\bar{\rho}\bar{\theta})\div u\dif x+\sigma^3\int\div(Pu)\div u \dif x\nonumber\\
 &-R\sigma^3\int\rho\dot{\theta}\div u\dif x+R\sigma^3\int(\rho\theta-\bar{\rho}\bar{\theta})u\cdot\D\div u\dif x+R\sigma^3\int(\rho\theta-\bar{\rho}\bar{\theta})\partial_ju^i\partial_iu^j\dif x\nonumber\\
 =~&R\frac{\dif}{\dif t}\left(\sigma^3\int(\rho\theta-\bar{\rho}\bar{\theta})\div u\dif x\right)-3R\sigma^2\sigma'\int(\rho\theta-\bar{\rho}\bar{\theta})\div u\dif x-R\sigma^3\int\rho\dot{\theta}\div u\dif x\nonumber\\
 &+R\sigma^3\int\rho\theta\partial_ju^i\partial_iu^j\dif x \leq~R\frac{\dif}{\dif t}\left(\sigma^3\int(\rho\theta-\bar{\rho}\bar{\theta})\div u\dif x\right)+C\sigma^2\sigma'\|\rho\theta-\bar{\rho}\bar{\theta}\|_{L^2}\|\D u\|_{L^2}\nonumber\\
 &+C\bar{\rho}^{\frac{1}{2}}\sigma^3\|\sqrt{\rho}\dot{\theta}\|_{L^2}\|\D u\|_{L^2}+C\bar{\rho}\bar{\theta}\sigma^3\|\D u\|_{L^2}^2\nonumber\\
 \leq~&R\frac{\dif}{\dif t}\left(\sigma^3\int(\rho\theta-\bar{\rho}\bar{\theta})\div u\dif x\right)+C\bar{\rho}\bar{\theta}\sigma^2\sigma'\|\D u\|_{L^2}+C\bar{\rho}^{\frac{1}{2}}\sigma^3\|\sqrt{\rho}\dot{\theta}\|_{L^2}\|\D u\|_{L^2}+C\bar{\rho}\bar{\theta}\sigma^3\|\D u\|_{L^2}^2,
\end{align}
\begin{align}\label{Lem:4:7}
P_2=~&-2\mu\sigma^3\int\theta^\al\D u_t:\frD(u)\dif x-2\mu\sigma^3\int\theta^\al\D((u\cdot\D) u):\frD(u)\dif x\nonumber\\
=~&-\mu\frac{\dif}{\dif t}\left(\sigma^3\int\theta^\al|\frD(u)|^2\dif x\right)+3\mu\sigma^2\sigma'\int\theta^\al|\frD(u)|^2\dif x+\al\mu\sigma^3\int\theta^{\al-1}\theta_t|\frD(u)|^2\dif x\nonumber\\
&+\mu\sigma^3\int\div(\theta^\al u)|\frD(u)|^2\dif x-\mu\sigma^3\int\theta^\al\partial_ju^k\partial_ku^i\left(\partial_iu^j+\partial_ju^i\right)\dif x\nonumber\\
\leq~&-\mu\frac{\dif}{\dif t}\left(\sigma^3\int\theta^\al|\frD(u)|^2\dif x\right)+C\bar{\theta}^\al\sigma^2\sigma'\|\D u\|_{L^2}^2\nonumber\\
&+C\bar{\theta}^{\al-1}\sigma^3\|\D u\|_{L^2}^{\frac{3}{2}}\|\D u\|_{L^6}^{\frac{1}{2}}\|\D\dot{\theta}\|_{L^2}+C\bar{\theta}^{\al}\sigma^3\|\D u\|_{L^2}^{\frac{3}{2}}\|\D u\|_{L^6}^{\frac{3}{2}},
\end{align}
and
\begin{align}\label{Lem:4:8}
P_3\leq~& -\frac{\lambda}{2}\frac{\dif}{\dif t}\left(\sigma^3\int\theta^\al(\div u)^2\dif x\right)+C\bar{\theta}^\al\sigma^2\sigma'\|\D u\|_{L^2}^2\nonumber\\
&+C\bar{\theta}^{\al-1}\sigma^3\|\D u\|_{L^2}^{\frac{3}{2}}\|\D u\|_{L^6}^{\frac{1}{2}}\|\D\dot{\theta}\|_{L^2}+C\bar{\theta}^{\al}\sigma^3\|\D u\|_{L^2}^{\frac{3}{2}}\|\D u\|_{L^6}^{\frac{3}{2}}.
\end{align}

Substituting \eqref{Lem:4:4}--\eqref{Lem:4:8} into \eqref{Lem:4:2}, we obtain after integration with respect to $t\in[0,T]$ that
\begin{align}
 \label{Lem:4:9} &\bar{\theta}^\al\sup\limits_{t\in[0,T]}\left(\sigma^3\|\D u\|_{L^2}^2\right)+\int_0^T\sigma^3\|\sqrt{\rho}\dot{u}\|_{L^2}^2\dif t \leq~C\sigma^3\int\theta^\al\left(\mu|\frD(u)|^2+\frac{\lambda}{2}(\div u)^2\right)\dif x+\int_0^T\int\sigma^3\rho|\dot{u}|^2\dif x\dif t\nonumber\\
 \leq~&C\bar{\rho}\bar{\theta}\sup\limits_{t\in[0,T]}\left(\sigma^3\|\D u\|_{L^2}\right)+C\bar{\rho}\bar{\theta}\int_0^{\sigma(T)}\|\D u\|_{L^2}\dif t+C\bar{\rho}^{\frac{1}{2}}\int_0^T\sigma^3\|\sqrt{\rho}\dot{\theta}\|_{L^2}\|\D u\|_{L^2}\dif t\nonumber\\
 &+C\bar{\rho}\bar{\theta}\int_0^T\|\D u\|_{L^2}^2\dif t+C\bar{\theta}^\al\int_0^{\sigma(T)}\|\D u\|_{L^2}^2\dif t+C\bar{\theta}^{\al-1}\int_0^T\sigma^3\|\D u\|_{L^2}^{\frac{3}{2}}\|\D u\|_{L^6}^{\frac{1}{2}}\|\D\dot{\theta}\|_{L^2}\dif t\nonumber\\
 &+C\bar{\theta}^\al\int_0^T\sigma^3\|\D u\|_{L^2}^{\frac{3}{2}}\|\D u\|_{L^6}^{\frac{3}{2}}\dif t \leq~C\bar{\rho}\bar{\theta}+C(K)\left(\bar{\rho}^{\frac{3}{2}}\bar{\theta}^{\frac{3}{2}-\frac{\al}{2}}+\bar{\rho}^2\bar{\theta}^{2-\al}+\bar{\rho}^{\frac{3}{4}}\bar{\theta}^{\frac{1}{2}(\al-\beta+1)}\right) \leq~K_5\bar{\rho}\bar{\theta},
\end{align}
provided that $2C\bar{\rho}\bar{\theta}\leq K_5\bar{\rho}\bar{\theta}$, $1<\alpha<\beta+1$, and $\bar{\rho}\leq C(K)\bar{\theta}^{\alpha-1}$, where we have taken $\bar{\theta}$ sufficiently large, such that
$$C(K)\left(\bar{\rho}^{\frac{3}{2}}\bar{\theta}^{\frac{3}{2}-\frac{\al}{2}}+\bar{\rho}^2\bar{\theta}^{2-\al}+\bar{\rho}^{\frac{3}{4}}\bar{\theta}^{\frac{1}{2}(\al-\beta+1)}\right)\leq \frac{1}{2}K_5\bar{\rho}\bar{\theta}.$$
This completes the proof of \eqref{Lem:4:1-1}.

\textbf{Step 3}: Estimates of $B_2(T)$.

Multiplying $(\ref{FCNS:2})_3$ by $\sigma^4\theta_t$ and integrating the resulting equality over $\R^3$ lead to
\begin{align}\label{Lem:5:2}
&\frac{\kappa}{2}\frac{\dif}{\dif t}\left(\sigma^4\int\theta^\beta|\D\theta|^2\dif x\right)+\sigma^4\int\rho|\dot{\theta}|^2\dif x\nonumber\\
=~&2\kappa\sigma^3\sigma'\int\theta^\beta|\D\theta|^2\dif x+\frac{\kappa\beta}{2}\sigma^4\int\theta^{\beta-1}\theta_t|\D\theta|^2\dif x+\sigma^4\int\rho\dot{\theta}u\cdot\D\theta\dif x\nonumber\\
&-R\sigma^4\int\rho\theta\div u\dot{\theta}\dif x+R\sigma^4\int\rho\theta u\cdot\D\theta\div u\dif x+2\mu\sigma^4\int\theta^\al\dot{\theta}|\frD(u)|^2\dif x\nonumber\\
&+\lambda\sigma^4\int\theta^\al\dot{\theta}(\div u)^2\dif x-2\mu\sigma^4\int\theta^\al u\cdot\D\theta|\frD(u)|^2\dif x-\lambda\sigma^4\int\theta^\al u\cdot\D\theta(\div u)^2\dif x.
\end{align}

Similar to \eqref{3.3:11}--\eqref{3.3:15}, we obtain after integration with respect to $t\in[0,T]$ that
\begin{align}
 \label{Lem:5:4} &\bar{\theta}^\beta\sup\limits_{t\in[0,T]}\left(\sigma^4\|\D\theta\|_{L^2}^2\right)+\int_0^T\sigma^4\| \sqrt{\rho}\dot{\theta}\|_{L^2}^2\dif t \leq~C\sigma^4\int\theta^\beta|\D\theta|^2\dif x+\int_0^T\sigma^4\|\sqrt{\rho}\dot{\theta}\|_{L^2}^2\dif t\nonumber\\
 \leq~&C\bar{\theta}^\beta\int_0^T\|\D\theta\|_{L^2}^2\dif t+C\bar{\rho}\int_0^T\sigma^4\|\D u\|_{L^2}\|\D\theta\|_{L^2}^2\|\D u\|_{L^6}\dif t+C\bar{\rho}\bar{\theta}^2\int_0^T\|\D u\|_{L^2}^2\dif t\nonumber\\
 &+C\bar{\theta}^\al\sup\limits_{t\in[0,T]}\left(\sigma^2\|\D u\|_{L^2}^{\frac{1}{2}}\|\D u\|_{L^6}^{\frac{1}{2}}\right)\int_0^T\sigma^2\|\D u\|_{L^6}\|\sqrt{\rho}\dot{\theta}\|_{L^2}\dif t\nonumber\\
 &+C\bar{\rho}\bar{\theta}\int_0^T\sigma^4\|\D u\|_{L^2}^{\frac{3}{2}}\|\D u\|_{L^6}^{\frac{1}{2}}\|\D\theta\|_{L^2}\dif t+C\bar{\theta}^\al\int_0^T\sigma^4\|\D u\|_{L^2}\|\D u\|_{L^6}^2\|\D\theta\|_{L^2}\dif t\nonumber\\
 &+C\bar{\rho}^{-\frac{1}{2}}\bar{\theta}^{\beta-1}\sup\limits_{t\in[0,T]}\left(\|\D\theta\|_{L^2}^{\frac{1}{2}}\|\D^2\theta\|_{L^2}^{\frac{1}{2}}\right)\int_0^T\sigma^4\|\sqrt{\rho}\dot{\theta}\|_{L^2}\|\D^2\theta\|_{L^2}\dif t\nonumber\\
 &+C\bar{\theta}^{\beta-1}\sup\limits_{t\in[0,T]}\left(\|\D u\|_{L^2}^{\frac{1}{2}}\|\D u\|_{L^6}^{\frac{1}{2}}\|\D\theta\|_{L^2}\right)\int_0^T\sigma^4\|\D\theta\|_{L^2}^{\frac{1}{2}}\|\D^2\theta\|_{L^2}^{\frac{3}{2}}\dif t\nonumber\\
 \leq~&C\bar{\rho}\bar{\theta}^2+C(K)\left(\bar{\rho}^2\bar{\theta}^{3-\al-\beta}+\bar{\rho}^2\bar{\theta}^{3-\al}+\bar{\rho}^{\frac{1}{2}}\bar{\theta}^{\frac{7}{4}+\frac{\alpha}{4}-\frac{\beta}{2}}+\bar{\rho}^{\frac{5}{2}}\bar{\theta}^{3-\al-\frac{\beta}{2}}+\bar{\theta}^{\frac{5}{2}-\frac{\al+\beta}{2}}+\bar{\theta}^{\frac{7}{4}}+\bar{\rho}^{\frac{1}{2}}\bar{\theta}^{\frac{11}{8}-\frac{3}{4}\beta}\right)\nonumber\\
 \leq~&K_6\bar{\rho}\bar{\theta}^2,
\end{align}
provided that $2C\bar{\rho}\bar{\theta}^2\leq K_6\bar{\rho}\bar{\theta}^2$, $2\leq\alpha\leq 10$, $\alpha<\beta<\alpha+2$, and
$$\bar{\rho}\leq C(K)\min\left\{\bar{\theta}^{\alpha+\beta-1}, \bar{\theta}^{\alpha-1}, \bar{\theta}^{\frac{1}{3}(2\alpha+\beta-2)}\right\},$$
where we have taken $\bar{\theta}$ sufficiently large, such that
$$C(K)\left(\bar{\rho}^2\bar{\theta}^{3-\al-\beta}+\bar{\rho}^2\bar{\theta}^{3-\al}+\bar{\rho}^{\frac{1}{2}}\bar{\theta}^{\frac{7}{4}+\frac{\alpha}{4}-\frac{\beta}{2}}+\bar{\rho}^{\frac{5}{2}}\bar{\theta}^{3-\al-\frac{\beta}{2}}+\bar{\theta}^{\frac{5}{2}-\frac{\al+\beta}{2}}+\bar{\theta}^{\frac{7}{4}}+\bar{\rho}^{\frac{1}{2}}\bar{\theta}^{\frac{11}{8}-\frac{3}{4}\beta}\right)\leq\frac{1}{2}K_6\bar{\rho}\bar{\theta}^2.$$
Hence we complete the proof of Lemma \ref{Lem:4}.
\end{proof}

The following Lemma is devoted to proving the uniform higher-order time-weighted estimates of velocity and temperature, as well as the uniform estimate of $\|\theta-\bar{\theta}\|_{L^\infty}$.

\begin{lemma}\label{Lem:5}$($Second-order time-weighted estimates of velocity and temperature$)$ 
Under the conditions of Proposition \ref{p4.1}, there exist positive constants $K_7$ and $\rL_4$ depending on $\mu$, $\lambda$, $\kappa$, and $R$ such that if $(\rho,u,\theta)$ is a smooth solution of \eqref{FCNS:2}--\eqref{data and far} on $\R^3\times(0,T]$, it holds that
\begin{equation}\label{Lem:4:1-3}
B_3(T)=\sup\limits_{t\in[0,T]}\left(\sigma^5\|\sqrt{\rho}\dot{u}\|_{L^2}^2\right)+\bar{\theta}^\al\int_0^T\sigma^5\|\D \dot{u}\|_{L^2}^2\dif t\leq\bar{\theta}^{\frac{\alpha}{3}+\frac{2}{3}},
\end{equation}
\begin{equation}\label{Lem:4:1-4}
B_4(T)=\sup\limits_{t\in[0,T]}\left(\sigma^6\| \sqrt{\rho}\dot{\theta}\|_{L^2}^2\right)+\bar{\theta}^\beta\int_0^T\sigma^6\|\D \dot{\theta}\|_{L^2}^2\dif t\leq K_7\bar{\rho}\bar{\theta}^2,
\end{equation}
and
\begin{align}\label{Lem:4:1-5}
\sup\limits_{t\in[0,T]}\|\theta-\bar{\theta}\|_{L^\infty}^2\leq \frac{\bar{\theta}^2}{9},
\end{align}
provided that $2\leq\alpha<\beta<\alpha+2$, $4\alpha-3\beta<4$, $\bar{\theta}\geq\rL_4$,
\begin{equation}\label{rho-theta-2}
\bar{\rho}\leq C(K)\min\left\{\bar{\theta}^{\frac{1}{3}(\alpha-1)}, \bar{\theta}^{\frac{1}{3}(\beta-\alpha)}\right\},
\end{equation}
and \eqref{rho-theta}, \eqref{rho-theta-1} hold.
\end{lemma}

\begin{proof}
\textbf{Step 1}: Estimates of $B_3(T)$.

Operating $\sigma^5\dot{u}^j(\pt+\mathrm{div}(u\cdot))$ to $(\ref{FCNS:2})_2^j$, summing with respect to $j$, and integrating the resulting equation over $\R^3$, we obtain
\begin{align}\label{3.42}
&\frac{1}{2}\frac{\dif}{\dif t}\int\sigma^5\rho|\dot u|^2\dif x-\frac{5}{2}\sigma^4\sigma'\|\sqrt{\rho}\dot u\|_{L^2}^2\nonumber\\
=~& -\int\sigma^5\dot u^j\left(\partial_jP_t+\div (u\partial_j P)\right)\dif x+2\mu\int\sigma^5\dot u^j\left(\pt(\div(\theta^\al\frD(u)))+\div (u\div(\theta^\al\frD(u)))\right)\dif x\nonumber\\
&+\lambda\int\sigma^5\dot u^j\left(\pt(\partial_j(\theta^\al\div u))+\div (u\cdot\partial_j(\theta^\al\div u))\right)\dif x. 
\end{align} 
Similar to \eqref{3.4.3}--\eqref{3.4.5}, we have
\begin{align}\label{e3.308}
&\frac{1}{2}\frac{\dif}{\dif t}\int\sigma^5\rho|\dot u|^2\dif x+\frac{\mu\sigma^5}{4}\left(\frac{\bar{\theta}}{2}\right)^\al\|\D\dot u\|_{L^2}^2\nonumber\\
\leq~&\frac{5}{2}\sigma^4\sigma'\|\sqrt{\rho}\dot u\|_{L^2}^2+C\sigma^5\Big(\bar{\rho}\bar{\theta}^{-\al}\|\sqrt{\rho}\dot\theta\|_{L^2}^2+\bar{\rho}^{-\frac{1}{2}}\bar{\theta}^{\al-2}\|\D u\|_{L^6}^2\|\sqrt{\rho}\dot{\theta}\|_{L^2}\|\D\dot{\theta}\|_{L^2}\nonumber\\
&+\bar{\theta}^\al\|\D u\|_{L^2}\|\D u\|_{L^6}^3+\bar{\rho}^2\bar{\theta}^{2-\al}\|\D u\|_{L^2}^2\Big).
\end{align}
Before integrating \eqref{e3.308}, we retain the time weight in its mixed temperature term. By \eqref{assume:1}, \eqref{assume:2}, \eqref{estimate:DuL6} and \eqref{estimate:DuL6:4},
\begin{align*}
 &\sup_{t\in[0,T]}\left(\sigma^2\|\D u\|_{L^6}^2\right) \leq~\left(\sup_{t\in[0,T]}\|\D u\|_{L^6}^2\right)^{3/5} \left(\sup_{t\in[0,T]}\sigma^5\|\D u\|_{L^6}^2\right)^{2/5}\\
 \leq~&C(K)\left(\bar{\rho}\bar{\theta}^{\frac4{15}-\frac23\alpha}
+\bar{\rho}^{3/5}\bar{\theta}^{\frac45-\frac45\alpha}
+\bar{\rho}^{6/5}\bar{\theta}^{-\frac25\alpha-\frac15\beta}\right),
\end{align*}
and
\begin{equation*}
 \int_0^T\|\sqrt{\rho}\dot{\theta}\|_{L^2}^2\dif t \leq A_2(\sigma(T))+B_2(T)\leq C(K)\bar{\theta}^{\beta}, \qquad \int_0^T\sigma^6\|\D\dot{\theta}\|_{L^2}^2\dif t \leq C(K)\bar{\rho}\bar{\theta}^{2-\beta}.
\end{equation*}
Consequently, the Cauchy--Schwarz inequality gives
\begin{align*}
&C\bar{\rho}^{-1/2}\bar{\theta}^{\alpha-2}
\int_0^T\sigma^5\|\D u\|_{L^6}^2
\|\sqrt{\rho}\dot{\theta}\|_{L^2}\|\D\dot{\theta}\|_{L^2}\dif t\\
\leq~&C\bar{\rho}^{-1/2}\bar{\theta}^{\alpha-2}
\sup_{t\in[0,T]}\left(\sigma^2\|\D u\|_{L^6}^2\right)
\left(\int_0^T\|\sqrt{\rho}\dot{\theta}\|_{L^2}^2\dif t\right)^{1/2}
\left(\int_0^T\sigma^6\|\D\dot{\theta}\|_{L^2}^2\dif t\right)^{1/2}\\
\leq~&C(K)\left(\bar{\rho}\bar{\theta}^{\frac13\alpha-\frac{11}{15}}
+\bar{\rho}^{3/5}\bar{\theta}^{\frac15(\alpha-1)}
+\bar{\rho}^{6/5}\bar{\theta}^{\frac15(3\alpha-\beta)-1}\right).
\end{align*}
Integrating \eqref{e3.308} over $[0,T]$ and using this bound, we obtain
\begin{align}
 \label{e3.3011}
&\sup\limits_{t\in[0,T]}\left(\sigma^5\|\sqrt{\rho}\dot{u}\|_{L^2}^2\right)+\bar{\theta}^\al\int_0^T\sigma^5\|\D \dot{u}\|_{L^2}^2\dif t\nonumber\\
 \leq~&C\int_0^{\sigma(T)}\sigma^4\|\sqrt{\rho}\dot{u}\|_{L^2}^2\dif t+C\bar{\rho}\bar{\theta}^{-\al}\int_0^T\sigma^5\|\sqrt{\rho}\dot\theta\|_{L^2}^2\dif t+C\bar{\rho}^2\bar{\theta}^{2-\al}\int_0^T\|\D u\|_{L^2}^2\dif t\nonumber\\
 &+C(K)\left(\bar{\rho}\bar{\theta}^{\frac13\alpha-\frac{11}{15}} +\bar{\rho}^{3/5}\bar{\theta}^{\frac15(\alpha-1)} +\bar{\rho}^{6/5}\bar{\theta}^{\frac15(3\alpha-\beta)-1}\right) +C\bar{\theta}^\al\sup\limits_{t\in[0,T]}\left(\sigma^5\|\D u\|_{L^2}\|\D u\|_{L^6}\right)\int_0^T\|\D u\|_{L^6}^2\dif t\nonumber\\
 \leq~&C(K)\left(\bar{\rho}\bar{\theta}+\bar{\rho}^2\bar{\theta}^{2-\al}+\bar{\rho}^3\bar{\theta}^{3-2\al}
+\bar{\rho}\bar{\theta}^{\frac13\alpha-\frac{11}{15}}
+\bar{\rho}^{3/5}\bar{\theta}^{\frac15(\alpha-1)}
+\bar{\rho}^{6/5}\bar{\theta}^{\frac15(3\alpha-\beta)-1}\right.\nonumber\\
 &\left.\qquad+\bar{\theta}^{\frac{11}{6}-\frac{4}{3}\alpha}+\bar{\rho}^{-\frac{1}{2}}\bar{\theta}^{\frac{5}{2}-\frac{3}{2}\alpha}+\bar{\rho}^{\frac{1}{4}}\bar{\theta}^{\frac{3}{2}-\alpha-\frac{\beta}{4}}\right) \leq~\bar{\theta}^{\frac{\alpha}{3}+\frac{2}{3}}.
\end{align}
The last inequality holds for sufficiently large $\bar{\theta}$ under $\alpha\geq2$, $4\alpha-3\beta<4$, \eqref{rho:theta} and \eqref{rho-theta-2}. In particular, using $\bar{\rho}\leq C(K)\bar{\theta}^{1/3}$, the three corrected mixed terms are bounded, respectively, by
\[
C(K)\bar{\theta}^{\frac13\alpha-\frac25},\qquad
C(K)\bar{\theta}^{\frac15\alpha},\qquad
C(K)\bar{\theta}^{\frac15(3\alpha-\beta)-\frac35},
\]
whose exponents are strictly smaller than $(\alpha+2)/3$. For the remaining leading term $\bar{\rho}\bar{\theta}$ we use the small constant in \eqref{rho-theta-2}.
This completes the proof of \eqref{Lem:4:1-3}. 

\bigskip

\textbf{Step 2}: Estimates of $B_4(T)$.

Multiplying $\sigma^6\dot{\theta}$ on both sides of \eqref{3.5:3} and integrating the resulting equality over $\R^3$, we get
\begin{align}
 \label{e3.3012}
&\sup\limits_{t\in[0,T]}\left(\sigma^6\| \sqrt{\rho}\dot{\theta}\|_{L^2}^2\right)+\bar{\theta}^\beta\int_0^T\sigma^6\|\D \dot{\theta}\|_{L^2}^2\dif t\nonumber\\
 \leq~&C\int_0^{\sigma(T)}\sigma^5\|\sqrt{\rho}\dot{\theta}\|_{L^2}^2\dif t+C\int_0^T\sigma^6\int\rho|\dot{\theta}|^2|\D u|\dif x\dif t+C\bar{\theta}\int_0^T\sigma^6\int\rho|\dot{\theta}|\left(|\D\dot{u}|+|\D u|^2\right)\dif x\dif t\nonumber\\
 &+C\bar{\theta}^\al\int_0^T\sigma^6\int|\dot{\theta}||\D u|^3\dif x\dif t+C\bar{\theta}^{\al-1}\int_0^T\sigma^6\int|\dot{\theta}|^2|\D u|^2\dif x\dif t+C\bar{\theta}^\al\int_0^T\sigma^6\int|\dot{\theta}||\D\dot{u}||\D u|\dif x\dif t\nonumber\\
 &+C\bar{\theta}^{\beta-1}\int_0^T\sigma^6\int|\dot{\theta}||\D\dot{\theta}||\D\theta|\dif x\dif t+C\bar{\theta}^\beta\int_0^T\sigma^6\int|\D\dot{\theta}|\left(|u||\D^2\theta|+|\D u||\D\theta|\right)\dif x\dif t\nonumber\\
 &+C\bar{\theta}^{\beta-1}\int_0^T\sigma^6\int|\D\dot{\theta}||u||\D\theta|^2\dif x\dif t =:~C\int_0^{\sigma(T)}\sigma^5\|\sqrt{\rho}\dot{\theta}\|_{L^2}^2\dif t+\sum\limits_{i=1}^8Q_i,
\end{align}
where the terms $Q_i$ can be estimated as follows.
\begin{align}\label{e3.3013}
Q_1= ~&C\int_0^T\sigma^6\int\rho|\dot{\theta}|^2|\D u|\dif x\dif t\leq C\bar{\rho}\int_0^T\sigma^6\|\D u\|_{L^2}\|\dot{\theta}\|_{L^2}^{\frac{1}{2}}\|\D\dot{\theta}\|_{L^2}^{\frac{3}{2}}\dif t\nonumber\\
\leq~&\frac{\bar{\theta}^\beta}{8}\int_0^T\sigma^6\|\D\dot{\theta}\|_{L^2}^2\dif t+C\bar{\rho}^3\bar{\theta}^{-3\beta}\sup\limits_{t\in[0,T]}\left(\sigma^6\|\D u\|_{L^2}^4\right)\int_0^T\|\sqrt{\rho}\dot{\theta}\|_{L^2}^2\dif t\nonumber\\
\leq~&\frac{\bar{\theta}^\beta}{8}\int_0^T\sigma^6\|\D\dot{\theta}\|_{L^2}^2\dif t+C(K)\bar{\rho}^5\bar{\theta}^{2-2\alpha-2\beta},
\end{align}
\begin{align}\label{e3.3014}
Q_2=~&C\bar{\theta}\int_0^T\sigma^6\int\rho|\dot{\theta}|\left(|\D\dot{u}|+|\D u|^2\right)\dif x\dif t\nonumber\\
\leq~&C\bar{\rho}^{\frac{1}{2}}\bar{\theta}\int_0^T\sigma^6\|\sqrt{\rho}\dot{\theta}\|_{L^2}\|\D\dot{u}\|_{L^2}\dif t+C\bar{\rho}^{\frac{1}{2}}\bar{\theta}\sup\limits_{t\in[0,T]}\left(\sigma^4\|\D u\|_{L^2}^{\frac{1}{2}}\|\D u\|_{L^6}^{\frac{1}{2}}\right)\int_0^T\sigma^2\|\sqrt{\rho}\dot{\theta}\|_{L^2}\|\D u\|_{L^6}\dif t\nonumber\\
\leq~&C(K)\bar{\rho}\bar{\theta}^{\frac{7}{3}-\frac{\alpha}{3}}+C(K)\bar{\rho}\bar{\theta}^{\frac{11}{4}-\frac{3}{4}\alpha},
\end{align}
\begin{align}\label{e3.3015}
Q_3=~&C\bar{\theta}^\al\int_0^T\sigma^6\int|\dot{\theta}||\D u|^3\dif x\dif t\leq C\bar{\theta}^\al\int_0^T\sigma^6\|\D u\|_{L^2}\|\D u\|_{L^6}^2\|\D\dot{\theta}\|_{L^2}\dif t\nonumber\\
\leq~&\frac{\bar{\theta}^\beta}{8}\int_0^T\sigma^6\|\D\dot{\theta}\|_{L^2}^2\dif t+C\bar{\theta}^{2\al-\beta}\sup\limits_{t\in[0,T]}\left(\sigma^6\|\D u\|_{L^2}^2\|\D u\|_{L^6}^2\right)\int_0^T\|\D u\|_{L^6}^2\dif t\nonumber\\
\leq~&\frac{\bar{\theta}^\beta}{8}\int_0^T\sigma^6\|\D\dot{\theta}\|_{L^2}^2\dif t+C(K)\bar{\rho}\bar{\theta}^{2-\beta},
\end{align}
\begin{align}\label{e3.3016}
Q_4=~&C\bar{\theta}^{\al-1}\int_0^T\sigma^6\int|\dot{\theta}|^2|\D u|^2\dif x\dif t\nonumber\\
\leq~& C\bar{\theta}^{\al-1}\sup\limits_{t\in[0,T]}\left(\|\D u\|_{L^2}\|\D u\|_{L^6}\right)\int_0^T\sigma^6\|\D\dot{\theta}\|_{L^2}^2\dif t\leq\frac{\bar{\theta}^\beta}{8}\int_0^T\sigma^6\|\D\dot{\theta}\|_{L^2}^2\dif t,
\end{align}
\begin{align}\label{e3.3017}
Q_5=~&C\bar{\theta}^\al\int_0^T\sigma^6\int|\dot{\theta}||\D\dot{u}||\D u|\dif x\dif t\leq C\bar{\theta}^\al\int_0^T\sigma^6\|\D u\|_{L^2}^{\frac{1}{2}}\|\D u\|_{L^6}^{\frac{1}{2}}\|\D\dot{u}\|_{L^2}\|\D\dot{\theta}\|_{L^2}\dif t\nonumber\\
\leq~&\frac{\bar{\theta}^\beta}{8}\int_0^T\sigma^6\|\D\dot{\theta}\|_{L^2}^2\dif t+C\bar{\theta}^{2\al-\beta}\sup\limits_{t\in[0,T]}\left(\sigma\|\D u\|_{L^2}\|\D u\|_{L^6}\right)\int_0^T\sigma^5\|\D\dot{u}\|_{L^2}^2\dif t\nonumber\\
\leq~&\frac{\bar{\theta}^\beta}{8}\int_0^T\sigma^6\|\D\dot{\theta}\|_{L^2}^2\dif t+C(K)\bar{\rho}\bar{\theta}^{\alpha+1-\beta},
\end{align}
\begin{align}\label{e3.3018}
Q_6=~&C\bar{\theta}^{\beta-1}\int_0^T\sigma^6\int|\dot{\theta}||\D\dot{\theta}||\D\theta|\dif x\dif t\nonumber\\
\leq~&C\bar{\theta}^{\beta-1}\sup\limits_{t\in[0,T]}\left(\|\D\theta\|_{L^2}^{\frac{1}{2}}\|\D^2\theta\|_{L^2}^{\frac{1}{2}}\right)\int_0^T\sigma^6\|\D\dot{\theta}\|_{L^2}^2\dif t\leq\frac{\bar{\theta}^\beta}{8}\int_0^T\sigma^6\|\D\dot{\theta}\|_{L^2}^2\dif t,
\end{align}
\begin{align}\label{e3.3019}
Q_7=~&C\bar{\theta}^\beta\int_0^T\sigma^6\int|\D\dot{\theta}|\left(|u||\D^2\theta|+|\D u||\D\theta|\right)\dif x\dif t\leq C\bar{\theta}^\beta\int_0^T\sigma^6\|\D u\|_{L^2}^{\frac{1}{2}}\|\D u\|_{L^6}^{\frac{1}{2}}\|\D^2\theta\|_{L^2}\|\D\dot{\theta}\|_{L^2}\dif t\nonumber\\
\leq~&\frac{\bar{\theta}^\beta}{8}\int_0^T\sigma^6\|\D\dot{\theta}\|_{L^2}^2\dif t+C\bar{\theta}^\beta\sup\limits_{t\in[0,T]}\left(\sigma^2\|\D u\|_{L^2}\|\D u\|_{L^6}\right)\int_0^T\sigma^4\|\D^2\theta\|_{L^2}^2\dif t\nonumber\\
\leq~&\frac{\bar{\theta}^\beta}{8}\int_0^T\sigma^6\|\D\dot{\theta}\|_{L^2}^2\dif t+C(K)\bar{\rho}^3\bar{\theta}^{2-\beta},
\end{align}
\begin{align}\label{e3.3020}
Q_8=~&C\bar{\theta}^{\beta-1}\int_0^T\sigma^6\int|\D\dot{\theta}||u||\D\theta|^2\dif x\dif t\leq C\bar{\theta}^{\beta-1}\int_0^T\sigma^6\|\D u\|_{L^2}\|\D^2\theta\|_{L^2}^2\|\D\dot{\theta}\|_{L^2}\dif t\nonumber\\
\leq~&\frac{\bar{\theta}^\beta}{8}\int_0^T\sigma^6\|\D\dot{\theta}\|_{L^2}^2\dif t+C\bar{\theta}^{\beta-2}\sup\limits_{t\in[0,T]}\left(\sigma^6\|\D^2\theta\|_{L^2}^4\right)\int_0^T\|\D u\|_{L^2}^2\dif t\nonumber\\
\leq~&\frac{\bar{\theta}^\beta}{8}\int_0^T\sigma^6\|\D\dot{\theta}\|_{L^2}^2\dif t+C(K)\bar{\rho}^5\bar{\theta}^{3-\alpha-\beta}.
\end{align}

Inserting \eqref{e3.3013}--\eqref{e3.3020} into \eqref{e3.3012}, then taking $\bar{\theta}\geq\rL_4$ and $K_7\geq CK_6$, we obtain
\begin{equation*}
B_4(T)\leq CK_6\bar{\rho}\bar{\theta}^2\leq K_7\bar{\rho}\bar{\theta}^2,
\end{equation*}
provided that
$$\bar{\rho}\leq C(K)\min\left\{\bar{\theta}^{\frac{\beta}{2}}, \bar{\theta}^{\frac{1}{4}(\alpha+\beta-1)}\right\}.$$

\bigskip

\textbf{Step 3}: Estimates of $\|\theta-\bar{\theta}\|_{L^\infty}$.

At last, from \eqref{estimate:D2theta-0} and \eqref{estimate:D2theta-1}, we have
\begin{equation}\label{Lem4:10}
\sup\limits_{t\in[0,T]}\|\D^2\theta\|_{L^2}^2\leq 2\bar{\theta}^2.
\end{equation}
From \eqref{g:3.3:2} and \eqref{Lem:4:1-2}, we have
\begin{equation}\label{Lem4:11}
\sup\limits_{t\in[0,T]}\|\D\theta\|_{L^2}^2\leq C(K).
\end{equation}
Hence we get
\begin{equation}
\|\theta-\bar{\theta}\|_{L^\infty}^2\leq C\|\D\theta\|_{L^2}\|\D^2\theta\|_{L^2}\leq C(K)\bar{\theta}\leq\frac{\bar{\theta}^2}{9},
\end{equation}
 provided that $\bar{\theta}\geq\rL_4$. Thus we complete the proof.
\end{proof}

\subsubsection{Uniform first-order estimates of density on $[0,T]$}\label{sec3.3}

Next, we will establish the first-order uniform estimates of density, which is achieved by combining the estimates on $[0,\sigma(T)]$ and the time-weighted estimate on $[0,T]$. As a result, the uniform lower and upper bounds of density is obtained by Gagliardo-Nirenberg inequality.

\begin{lemma}\label{Lem39}$($First-order estimates of the density$)$ 
Under the conditions of Proposition \ref{p4.1}, there exist positive constants $\widetilde{\rL}$, $\rL_5$, $K_8$ and $K_9$ depending on $\mu$, $\lambda$, $\kappa$, and $R$, such that if $(\rho,u,\theta)$ is a smooth solution of \eqref{FCNS:2}--\eqref{data and far} on $\R^3\times(0,T]$, it holds that
\begin{equation}\label{3.9:1}
C_1(T)=\sup_{t\in[0,T]}\|\D\rho\|_{L^2}^2+\bar{\rho}\bar\theta^{1-\al}\int_0^T\|\D \rho\|_{L^2}^2\dif t\leq K_8,
\end{equation}
\begin{equation}\label{3.9:2}
C_2(T)=\sup_{t\in[0,T]}\|\D\rho\|_{L^4}^2+\bar{\rho}\bar\theta^{1-\al}\int_0^T\|\D \rho\|_{L^4}^2\dif t\leq K_9,
\end{equation}
and
\begin{equation}\label{3.9:3}
\sup_{t\in[0,T]}\|\rho-\bar{\rho}\|_{L^\infty}\leq\frac{\bar{\rho}}{3},
\end{equation}
provided that $\bar{\rho}\geq\rL$, $\bar{\theta}\geq\rL_5$, $\bar{\rho}\leq C(K)\bar{\theta}^{\frac{2}{23}(7\alpha-1)}$, and \eqref{rho-theta}, \eqref{rho-theta-1}, \eqref{rho-theta-2} hold.
\end{lemma}

\begin{proof}
Applying $\nabla$ operator on $(\ref{FCNS:2})_1$ and multiplying $\nabla\rho|\nabla \rho|^{r-2}$ on both sides of the resulting equation, we obtain
\begin{align}
\frac{1}{r}\left(|\D\rho|^r\right)_t&+\frac{1}{r}\div(|\D \rho|^{r}u)+\frac{r-1}{r}|\D \rho|^r\div u+|\D \rho|^{r-2}(\D\rho)^\top\D u\D \rho\nonumber\\
&+\frac{\rho}{2\mu+\lambda}|\D \rho|^{r-2}\D \rho\cdot\left(\D G+R\theta^{1-\al}\D \rho+(1-\alpha)R\rho\theta^{-\al}\D\theta\right)=0.
\end{align}
Integrating the above equation in $\R^3$, we obtain
\begin{align}\label{lr}
&\frac{\dif}{\dif t}\|\D\rho\|_{L^r}^{r}+\frac{R}{2\mu+\lambda}\int\rho{\theta}^{1-\al}|\D\rho|^{r}\dif x\nonumber\\
\leq~& C\|\D u\|_{L^\infty}\|\D \rho\|_{L^r}^{r}+C\bar{\rho}\|\D G\|_{L^r}\|\D\rho\|_{L^r}^{r-1}+C\bar{\rho}^2\bar\theta^{-\al}\|\D \theta\|_{L^r}\|\D\rho\|_{L^r}^{r-1}.
\end{align}

Multiplying \eqref{lr} by $\sigma^m$ with $m\in\mathbb{N}_+$, we have
\begin{align}\label{lr:sigma}
&\frac{\dif}{\dif t}\left(\sigma^m\|\D\rho\|_{L^r}^{r}\right)+\frac{R\sigma^m}{2\mu+\lambda}\int\rho{\theta}^{1-\al}|\D\rho|^{r}\dif x\nonumber\\
\leq~&m\sigma^{m-1}\sigma'\|\D\rho\|_{L^r}^r+ C\sigma^m\|\D u\|_{L^\infty}\|\D \rho\|_{L^r}^{r}+C\bar{\rho}\sigma^m\|\D G\|_{L^r}\|\D\rho\|_{L^r}^{r-1}+C\bar{\rho}^2\sigma^m\bar\theta^{-\al}\|\D \theta\|_{L^r}\|\D\rho\|_{L^r}^{r-1}.
\end{align}

\textbf{Case 1}: When $r=2$, \eqref{lr} reduces to 
\begin{align}
 \label{l2} &\frac{\dif}{\dif t}\|\D\rho\|_{L^2}^2+\frac{c_0}{4}\bar\rho\left(2\bar\theta\right)^{1-\al}\|\D\rho\|_{L^2}^2 \leq~ C\|\D u\|_{L^\infty}\|\D \rho\|_{L^2}^2+C\bar{\rho}\|\D G\|_{L^2}\|\D\rho\|_{L^2}+C\bar{\rho}^2\bar\theta^{-\al}\|\D \theta\|_{L^2}\|\D\rho\|_{L^2}\nonumber\\
 \leq~&\frac{c_0}{8}\bar\rho\left(2\bar\theta\right)^{1-\al}\|\D\rho\|_{L^2}^2+C\|\D u\|_{L^\infty}\|\D\rho\|_{L^2}^2+C\bar{\rho}\|\D G\|_{L^2}\|\D\rho\|_{L^2}+C\bar\rho^3\bar\theta^{-1-\al}\|\D \theta\|_{L^2}^2\nonumber\\
 \leq~&\frac{c_0}{8}\bar\rho\left(2\bar\theta\right)^{1-\al}\|\D\rho\|_{L^2}^2+C\|\D u\|_{L^2}^\frac{1}{7}\|\D^2 u\|_{L^4}^\frac{6}{7}\|\D\rho\|_{L^2}^2+C\bar{\rho}\|\D G\|_{L^2}\|\D\rho\|_{L^2}+C\bar\rho^3\bar\theta^{-1-\al}\|\D \theta\|_{L^2}^2,
\end{align}
for some positive constant $c_0$ depending on $\lambda$, $\mu$ and $R$.

By \eqref{2.3:2}, we have
\begin{align}\label{f2}
\|\D G\|_{L^2}\leq C\|H\|_{L^2}\leq~&C\bar{\rho}^\frac{1}{2}\bar{\theta}^{-\al}\|\sqrt{\rho}\dot{u}\|_{L^2}+C\bar{\rho}\bar{\theta}^{-\al}\|\D\theta\|_{L^2}+C\bar{\theta}^{-1}\|\D\theta\|_{L^2}\|\D u\|_{L^\infty}
\nonumber\\
\leq~&C\left(\bar{\rho}^\frac{1}{2}\bar{\theta}^{-\al}\|\sqrt{\rho}\dot{u}\|_{L^2}+\bar{\rho}\bar{\theta}^{-\al}\|\D\theta\|_{L^2}+\bar{\theta}^{-1}\|\D\theta\|_{L^2}\|\D u\|_{L^2}^{\frac{1}{7}}\|\D^2u\|_{L^4}^{\frac{6}{7}}\right).
\end{align}

Inserting \eqref{f2} into \eqref{l2}, we have
\begin{align}
 \label{l2-1} &\frac{\dif}{\dif t}\|\D\rho\|_{L^2}^2+\frac{c_0}{4}\bar{\rho}\left(2\bar\theta\right)^{1-\al}\|\D\rho\|_{L^2}^2 \leq~C\|\D u\|_{L^2}^\frac{1}{7}\|\D^2 u\|_{L^4}^\frac{6}{7}\|\D\rho\|_{L^2}^2+C\bar\rho^3\bar{\theta}^{-1-\al}\|\D \theta\|_{L^2}^2\nonumber\\
 &+C\bar{\rho}\|\D\rho\|_{L^2}\left(\bar{\rho}^\frac{1}{2}\bar{\theta}^{-\al}\|\sqrt{\rho}\dot{u}\|_{L^2}+\bar{\rho}\bar{\theta}^{-\al}\|\D\theta\|_{L^2}+\bar{\theta}^{-1}\|\D\theta\|_{L^2}\|\D u\|_{L^2}^{\frac{1}{7}}\|\D^2u\|_{L^4}^{\frac{6}{7}}\right)\nonumber\\
 \leq~&\frac{c_0}{8}\bar{\rho}\left(2\bar\theta\right)^{1-\al}\|\D\rho\|_{L^2}^2+C\|\D u\|_{L^2}^\frac{1}{7}\|\D^2 u\|_{L^{4}}^\frac{6}{7}\|\D\rho\|_{L^2}^2+C\bar\rho^3\bar{\theta}^{-1-\al}\|\D\theta\|_{L^{2}}^2\nonumber\\
 &+C\bar{\rho}^2\bar{\theta}^{-1-\al}\|\sqrt{\rho}\dot{u}\|_{L^{2}}^2+C\bar\rho\bar{\theta}^{\alpha-3}\|\D\theta\|_{L^2}^2\|\D u\|_{L^2}^{\frac{2}{7}}\|\D^2u\|_{L^4}^{\frac{12}{7}}.
\end{align}

It remains to estimate $\|\D^2 u\|_{L^4}$  on the right hand side of \eqref{l2-1}. From $(\ref{FCNS:2})_2$, one has
\begin{align}
 \label{D2u4}
\|\D^2u\|_{L^4}\leq~&C\bar{\theta}^{-\al}
\left(\bar\rho\|\dot{u}\|_{L^4}+\bar{\theta}^{\al-1}\|\D\theta\D u\|_{L^4}+\bar{\theta}\|\D\rho\|_{L^4}+\bar\rho\|\D\theta\|_{L^4}\right)
\nonumber\\
 \leq~&C\bar{\theta}^{-\al}
\left(\bar\rho\|\dot{u}\|_{L^4}+\bar{\theta}^{\al-1}\|\D\theta\|_{L^4}\|\D u\|_{L^\infty}+\bar{\theta}\|\D\rho\|_{L^4}+\bar\rho\|\D\theta\|_{L^2}^{\frac{1}{4}}\|\D^2\theta\|_{L^2}^{\frac{3}{4}}\right)
\nonumber\\
 \leq~&C\left(\bar\rho\bar{\theta}^{-\al}\|\dot{u}\|_{L^4}+\bar{\theta}^{-1}\|\D\theta\|_{L^4}\|\D u\|_{L^2}^{\frac{1}{7}}
\|\D^2u\|_{L^4}^{\frac{6}{7}}
+\bar{\theta}^{1-\al}\|\D\rho\|_{L^4}+ \bar\rho\bar{\theta}^{-\al}\|\D\theta\|_{L^2}^{\frac{1}{4}}\|\D^2\theta\|_{L^2}^{\frac{3}{4}}\right)
\nonumber\\
 \leq~&C\left(\bar\rho\bar{\theta}^{-\al}\|\dot{u}\|_{L^4}+\bar{\theta}^{-7}\|\D\theta\|_{L^4}^7\|\D u\|_{L^2} +\bar{\theta}^{1-\al}\|\D\rho\|_{L^4}+\bar\rho\bar{\theta}^{-\al}\|\D\theta\|_{L^2}^{\frac{1}{4}}\|\D^2\theta\|_{L^2}^{\frac{3}{4}}\right) +\frac{1}{4}\|\D^2u\|_{L^4},
\end{align}
which implies
\begin{equation}\label{D2l4}
\|\D^2u\|_{L^4}
\leq C\left(\bar\rho\bar{\theta}^{-\al}\|\dot{u}\|_{L^4}+\bar{\theta}^{-7}\|\D\theta\|_{L^4}^7\|\D u\|_{L^2}
+\bar{\theta}^{1-\al}\|\D\rho\|_{L^4}+\bar\rho\bar{\theta}^{-\al}\|\D\theta\|_{L^2}^{\frac{1}{4}}\|\D^2\theta\|_{L^2}^{\frac{3}{4}}\right).
\end{equation}
By using \eqref{D2l4} 
and integrating with respect to $t$ over $[0,\sigma(T)]$,
we arrive at
\begin{align}
 \label{p0}
\int_0^{\sigma(T)}\|\D^2u\|_{L^4}^2\dif t
\leq~&C\bar\rho^2\bar{\theta}^{-2\al}\int_{0}^{\sigma(T)}\|\dot{u}\|_{L^4}^2\dif t+C\bar{\theta}^{-14}\int_{0}^{\sigma(T)}\|\D\theta \|_{L^4}^{14}
\|\D u\|_{L^2}^{2}\dif t\nonumber\\
 & +C\bar{\theta}^{2-2\al}\int_0^{\sigma(T)}\|\D\rho\|^2_{L^4}\dif t+C\bar\rho^2\bar{\theta}^{-2\al}\int_0^{\sigma(T)}\|\D\theta\|_{L^2}^{\frac{1}{2}}\|\D^2\theta\|_{L^2}^{\frac{3}{2}}\dif t =:~\sum\limits_{i=1}^{4}R_i,
\end{align}
where,  from \eqref{assume:1}, we have
\begin{align}\label{p1}
R_1=~&C\bar\rho^2\bar{\theta}^{-2\al}\int_0^{\sigma(T)}\|\dot{u}\|_{L^4}^2\dif t
\leq C\bar\rho^\frac{7}{4}\bar{\theta}^{-2\al}
\int_{0}^{\sigma(T)}\|\sqrt{\rho}\dot{u}\|_{L^2}^{\frac{1}{2}}\|\D \dot{u}\|_{L^2}^{\frac{3}{2}}
\dif t
\leq~ C(K)\bar\rho^\frac{7}{4}\bar{\theta}^{-\al}.
\end{align}

By \eqref{assume:1} and \eqref{assume:2}, we have
\begin{align}
 \label{p3}
R_2&=C\bar{\theta}^{-14}\int_{0}^{\sigma(T)}\|\D\theta \|_{L^4}^{14}
\|\D u\|_{L^2}^{2}\dif t\leq C\bar{\theta}^{-14}\sup\limits_{t\in[0,\sigma(T)]}\|\D \theta\|_{L^4}^{14}\int_0^{\sigma(T)}\|\D u\|_{L^2}^2\dif t\nonumber\\
 &\leq C\bar\rho\bar{\theta}^{-13-\alpha}\sup\limits_{t\in[0,\sigma(T)]}\left(\|\D\theta\|_{L^2}^{\frac{7}{2}}\|\D^2\theta\|_{L^2}^{\frac{21}{2}}\right) \leq C(K)\bar\rho\bar{\theta}^{-\frac{5}{2}-\al},
\end{align}
\begin{align}\label{p4}
R_3=~&C\bar{\theta}^{2-2\al}\int_0^{\sigma(T)}\|\D\rho\|^2_{L^4}\dif t
\leq CK_9\bar{\theta}^{2-2\al},
\end{align}
and 
\begin{align}\label{p5}
R_4&=C\bar\rho^2\bar{\theta}^{-2\al}\int_0^{\sigma(T)}\|\D\theta\|_{L^2}^{\frac{1}{2}}\|\D^2\theta\|_{L^2}^{\frac{3}{2}}\dif t
\leq C\bar\rho^2\bar{\theta}^{-2\al}\left(\int_0^{\sigma(T)}\|\D\theta\|_{L^2}^2\dif t\right)^\frac{1}{4}\left(\int_0^{\sigma(T)}\|\D^2\theta\|_{L^2}^2\dif t\right)^\frac{3}{4}\nonumber\\
&\leq C(K)\bar\rho^3\bar{\theta}^{\frac{1}{2}-2\al-\beta}.
\end{align}

Now substituting \eqref{p1}--\eqref{p5} into \eqref{p0}, one can see that
\begin{align}\label{D2u24}
\int_0^{\sigma(T)}\|\D^2u\|_{L^4}^2\dif t\leq   C(K)\bar\rho^\frac{7}{4}\bar{\theta}^{-\al},
\end{align}
provided that $\bar{\rho}\leq C(K)\bar{\theta}^{\frac{2}{3}(\alpha+\beta-\frac{1}{2})}$.

Next, we will estimate $\|\D\rho\|_{L^2}$ in two cases.

\textbf{Case 1.1}: For the case  $t\in(0,\sigma(T))$, integrating \eqref{l2-1} over $(0,\sigma(T))$, we have
\begin{align}\label{l2-2}
&\|\D\rho\|_{L^2}^2+\bar{\rho}\bar\theta^{1-\al}\int_0^{\sigma(T)}\|\D\rho\|_{L^2}^2\dif t\nonumber\\
\leq~& C\|\D\rho_0\|_{L^2}^2+C\int_0^{\sigma(T)}\|\D u\|_{L^2}^\frac{1}{7}\|\D^2 u\|_{L^{4}}^\frac{6}{7}\|\D\rho\|_{L^2}^2\dif t+C\bar{\rho}^3\bar{\theta}^{-1-\al}\int_0^{\sigma(T)}\|\D\theta\|_{L^{2}}^2\dif t\nonumber\\
&+C\bar{\rho}^2\bar{\theta}^{-1-\al}\int_0^{\sigma(T)}\|\sqrt{\rho}\dot{u}\|_{L^{2}}^2\dif t+C\bar{\rho}\bar{\theta}^{\alpha-3}\int_0^{\sigma(T)}\|\D\theta\|_{L^2}^2\|\D u\|_{L^2}^{\frac{2}{7}}\|\D^2u\|_{L^4}^{\frac{12}{7}}\dif t
\nonumber\\
\leq~&C\|\D\rho_0\|_{L^2}^2+C\bar{\rho}^4\bar{\theta}^{1-\al-\beta}+CK_1\bar{\rho}^2\bar{\theta}^{-1}+C(K)\bar{\rho}^\frac{37}{14}\bar{\theta}^{-\frac{20}{7}}+C\int_0^{\sigma(T)}\|\D u\|_{L^2}^\frac{1}{7}\|\D^2 u\|_{L^{4}}^\frac{6}{7}\|\D\rho\|_{L^2}^2\dif t,
\end{align}
where we have used \eqref{assume:1}.

By Gronwall's inequality, we get
\begin{align}\label{l2-2-1}
&\sup\limits_{t\in[0,\sigma(T)]}\|\D\rho\|_{L^2}^2+\bar{\rho}\bar\theta^{1-\al}\int_0^{\sigma(T)}\|\D\rho\|_{L^2}^2\dif t \nonumber\\
\leq~&C\left(\|\D\rho_0\|_{L^2}^2+\bar{\rho}^4\bar{\theta}^{1-\al-\beta}+K_1\bar{\rho}^2\bar{\theta}^{-1}+C(K)\bar{\rho}^\frac{37}{14}\bar{\theta}^{-\frac{20}{7}}\right)\nonumber\\
&\cdot\exp\left\{C\left(\int_0^{\sigma(T)}\|\D u\|_{L^2}^2\dif t\right)^\frac{1}{14}\left(\int_0^{\sigma(T)}\|\D^2 u\|_{L^4}^2\dif t\right)^\frac{3}{7}\right\}\nonumber\\
\leq~&C\exp\left\{C(K)\bar{\rho}^\frac{23}{28}\bar\theta^{\frac{1}{14}-\frac{1}{2}\al}\right\}\nonumber\\
\leq~& C,
\end{align}
provided that $\bar{\rho}\leq C(K)\bar{\theta}^{\frac{2}{23}(7\alpha-1)}$, where we have used \eqref{assume:1}, \eqref{D2u24}, and taken $\bar{\theta}$ large enough, such that
$$C(K)\left(\bar{\rho}^4\bar{\theta}^{1-\al-\beta}+\bar{\rho}^2\bar{\theta}^{-1}+\bar{\rho}^\frac{37}{14}\bar{\theta}^{-\frac{20}{7}}\right)\leq 1 \qquad \text{and} \qquad C(K)\bar{\rho}^\frac{23}{28}\bar\theta^{\frac{1}{14}-\frac{1}{2}\al}\leq\ln 2.$$


\textbf{Case 1.2}: For the case  $t\in(\sigma(T),T)$, by \eqref{D2l4} and \eqref{assume:2}, we have
\begin{align}\label{D2u24:sigma}
\int_0^T\sigma^5\|\D^2u\|_{L^4}^2\dif t
\leq~&C\bar{\rho}^\frac{7}{4}\bar{\theta}^{-2\al}
\int_{0}^{T}\sigma^5\|\sqrt{\rho}\dot{u}\|_{L^2}^{\frac{1}{2}}\|\D \dot{u}\|_{L^2}^{\frac{3}{2}}
\dif t+C\bar{\theta}^{-14}\sup\limits_{t\in[0,T]}\|\D\theta\|_{L^4}^{14}\int_0^T\|\D u\|_{L^2}^2\dif t\nonumber\\
&+C\bar{\theta}^{2-2\al}\int_0^{T}\|\D\rho\|^2_{L^4}\dif t+C\bar{\rho}^2\bar{\theta}^{-2\al}\int_0^{T}\sigma^4\|\D\theta\|_{L^2}^{\frac{1}{2}}\|\D^2\theta\|_{L^2}^{\frac{3}{2}}\dif t\nonumber\\
\leq~&C\bar{\rho}^\frac{7}{4}\bar{\theta}^{-2\al}\left(\int_0^T\sigma^3\|\sqrt{\rho}\dot{u}\|_{L^2}^2\dif t\right)^{\frac{1}{4}}\left(\int_0^T\sigma^5\|\D\dot{u}\|_{L^2}^2\dif t\right)^{\frac{3}{4}}\nonumber\\
&+C\bar{\theta}^{-14}\left(\sup\limits_{t\in[0,T]}\|\D \theta\|_{L^4}^{14}\right)\int_0^{T}\|\D u\|_{L^2}^2\dif t\nonumber\\
&+C\bar{\theta}^{2-2\al}\int_0^{T}\|\D\rho\|^2_{L^4}\dif t
+C\bar{\rho}^2\bar{\theta}^{-2\al}\left(\int_0^T\|\D\theta\|_{L^2}^2\dif t\right)^{\frac{1}{4}}\left(\int_0^T\sigma^4\|\D^2\theta\|_{L^2}^2\dif t\right)^{\frac{3}{4}}\nonumber\\
\leq~&C\left(K_5^\frac{1}{4}\bar{\rho}^2\bar{\theta}^{\frac{3}{4}-\frac{7}{4}\al}+K_2^\frac{7}{4}\bar{\rho}\bar{\theta}^{-\frac{5}{2}-\al}+K_9\bar{\rho}^{-1}\bar{\theta}^{1-\al}+C(K)\bar{\rho}^{\frac{9}{4}}\bar{\theta}^{\frac{19}{8}-2\al-\frac{7}{4}\beta}\right)\nonumber\\
\leq~&CK_9\bar{\rho}^{-1}\bar{\theta}^{1-\al}.
\end{align}

When $r=2$, \eqref{lr:sigma} reduces to 
\begin{align}\label{Dr2:1}
&\frac{\dif}{\dif t}\left(\sigma^m\|\D\rho\|_{L^2}^2\right)+\frac{c_0}{4}\sigma^m\bar\rho\left(2\bar{\theta}\right)^{1-\al}\|\D\rho\|_{L^2}^2\nonumber\\
\leq~&m\sigma^{m-1}\sigma'\|\D\rho\|_{L^2}^2+C\sigma^m\|\D u\|_{L^\infty}\|\D \rho\|_{L^2}^2+C\bar{\rho}\sigma^m\|\D G\|_{L^2}\|\D\rho\|_{L^2}+C\bar{\rho}^2\bar{\theta}^{-\al}\sigma^m\|\D\theta\|_{L^2}\|\D\rho\|_{L^2}\nonumber\\
\leq~&m\sigma^{m-1}\sigma'\|\D\rho\|_{L^2}^2+C\sigma^m\|\D u\|_{L^2}^{\frac{1}{7}}\|\D^2u\|_{L^4}^{\frac{6}{7}}\|\D\rho\|_{L^2}^2\nonumber\\
&+C\bar{\rho}\sigma^m\left(\bar\rho^\frac{1}{2}\bar{\theta}^{-\al}\|\sqrt{\rho}\dot{u}\|_{L^2}+\bar\rho\bar{\theta}^{-\al}\|\D\theta\|_{L^2}+\bar{\theta}^{-1}\|\D\theta\|_{L^2}\|\D u\|_{L^\infty}\right)\|\D\rho\|_{L^2}+C\bar{\rho}^2\bar{\theta}^{-\al}\sigma^m\|\D\theta\|_{L^2}\|\D\rho\|_{L^2}\nonumber\\
\leq~&m\sigma^{m-1}\sigma'\|\D\rho\|_{L^2}^2+C\sigma^m\|\D u\|_{L^2}^{\frac{1}{7}}\|\D^2u\|_{L^4}^{\frac{6}{7}}\|\D\rho\|_{L^2}^2+C\bar\rho^\frac{3}{2}\bar{\theta}^{-\al}\sigma^m\|\sqrt{\rho}\dot{u}\|_{L^2}\|\D\rho\|_{L^2}\nonumber\\
&+C\bar\rho^2\bar{\theta}^{-\al}\sigma^m\|\D\theta\|_{L^2}\|\D\rho\|_{L^2}+C\bar{\rho}\bar{\theta}^{-1}\sigma^m\|\D\theta\|_{L^2}\|\D u\|_{L^2}^{\frac{1}{7}}\|\D^2u\|_{L^4}^{\frac{6}{7}}\|\D\rho\|_{L^2},
\end{align}
where we have used \eqref{f2}.

Taking $m=3$ and integrating \eqref{Dr2:1} over $[0,T]$, by \eqref{l2-2-1}, \eqref{assume:2}, and \eqref{D2u24:sigma}, we obtain
\begin{align}
 \label{Dr2:2}
&\sigma^3\|\D\rho\|_{L^2}^2+\bar\rho\bar{\theta}^{1-\al}\int_0^T\sigma^3\|\D\rho\|_{L^2}^2\dif t\nonumber\\
 \leq~&C\sup\limits_{t\in[0,\sigma(T)]}\|\D\rho\|_{L^2}^2+C\sup\limits_{t\in[0,T]}\|\D\rho\|_{L^2}\left(\int_0^T\|\D u\|_{L^2}^2\dif t\right)^{\frac{1}{14}}\left(\int_0^T\sigma^5\|\D^2u\|_{L^4}^2\dif t\right)^{\frac{3}{7}}\left(\int_0^T\|\D\rho\|_{L^2}^2\dif t\right)^{\frac{1}{2}}\nonumber\\
 &+C\bar\rho^\frac{3}{2}\bar{\theta}^{-\al}\left(\int_0^T\sigma^3\|\sqrt{\rho}\dot{u}\|_{L^2}^2\dif t\right)^{\frac{1}{2}}\left(\int_0^T\|\D\rho\|_{L^2}^2\dif t\right)^{\frac{1}{2}}+C\bar\rho^2\bar{\theta}^{-\al}\left(\int_0^T\|\D\theta\|_{L^2}^2\dif t\right)^{\frac{1}{2}}\left(\int_0^T\|\D\rho\|_{L^2}^2\dif t\right)^{\frac{1}{2}}\nonumber\\
 &+C\bar\rho\sup\limits_{t\in[0,T]}\|\D\rho\|_{L^2}\left(\int_0^T\|\D u\|_{L^2}^2\dif t\right)^{\frac{1}{14}}\left(\int_0^T\sigma^5\|\D^2u\|_{L^4}^2\dif t\right)^{\frac{3}{7}}\left(\int_0^T\|\D\theta\|_{L^2}^2\dif t\right)^{\frac{1}{2}}\nonumber\\
 \leq~&\frac{1}{2}K_8+C(K)\left(\bar{\rho}^{-\frac{6}{7}}+\bar{\rho}^{\frac{3}{2}}\bar{\theta}^{-\frac{1}{2}\al}+\bar{\rho}^2\bar{\theta}^{\frac{1}{2}-\frac{1}{2}\al-\frac{1}{2}\beta}+\bar{\rho}^\frac{8}{7}\bar{\theta}^{\frac{3}{2}-\frac{1}{2}\al-\frac{1}{2}\beta}\right) \leq~\frac{3}{4}K_8,
\end{align}
provided that $K_8\geq 2C\sup\limits_{t\in[0,\sigma(T)]}\|\D\rho\|_{L^2}^2$, and
$$\bar{\rho}\leq C(K)\min\left\{\bar{\theta}^{\frac{\alpha}{3}}, \bar{\theta}^{\frac{1}{4}(\alpha+\beta-1)}, \bar{\theta}^{\frac{7}{16}(\alpha+\beta-3)}\right\},$$
where we have taken $\bar{\rho}$ and $\bar{\theta}$ sufficiently large, such that
$$C(K)\left(\bar{\rho}^{-\frac{6}{7}}+\bar{\rho}^{\frac{3}{2}}\bar{\theta}^{-\frac{1}{2}\al}+\bar{\rho}^2\bar{\theta}^{\frac{1}{2}-\frac{1}{2}\al-\frac{1}{2}\beta}+\bar{\rho}^\frac{8}{7}\bar{\theta}^{\frac{3}{2}-\frac{1}{2}\al-\frac{1}{2}\beta}\right)\leq \frac{1}{8}K_8.$$
Combining (\ref{l2-2-1}) and (\ref{Dr2:1}), we complete the proof of \eqref{3.9:1}.

\textbf{Case 2}: When $r=4$, \eqref{lr} reduces to 
\begin{align}\label{l4}
&\frac{\dif}{\dif t}\|\D\rho\|_{L^4}^{4}+\frac{R}{2\mu+\lambda}\int\rho{\theta}^{1-\al}|\D\rho|^{4}\dif x\nonumber\\
\leq~& C\|\D u\|_{L^\infty}\|\D \rho\|_{L^4}^{4}+C\bar{\rho}\|\D G\|_{L^4}\|\D\rho\|_{L^4}^{3}+C\bar{\rho}^2\bar\theta^{-\al}\|\D \theta\|_{L^4}\|\D\rho\|_{L^4}^{3}.
\end{align}


By \eqref{2.3:2}, we have
\begin{align}\label{DG4}
\|\D G\|_{L^4}
\leq~& C\left(\bar\rho\bar{\theta}^{-\al}\|\dot{u}\|_{L^4}+\bar\rho\bar{\theta}^{-\al}\|\D\theta\|_{L^4}+\bar{\theta}^{-1}\|\D\theta\|_{L^4}\|\D u\|_{L^2}^{\frac{1}{7}}\|\D^2u\|_{L^4}^{\frac{6}{7}}\right)\nonumber\\
\leq~& C\left(\bar\rho^\frac{7}{8}\bar{\theta}^{-\al}\|\sqrt{\rho}\dot{u}\|_{L^2}^{\frac{1}{4}}\|\D \dot{u}\|_{L^2}^\frac{3}{4}+\bar\rho\bar{\theta}^{-\al}\|\D\theta\|_{L^4}+\bar{\theta}^{-1}\|\D\theta\|_{L^4}\|\D u\|_{L^2}^{\frac{1}{7}}\|\D^2u\|_{L^4}^{\frac{6}{7}}\right).
\end{align}

For the square of the $L^4$ norm, substituting \eqref{DG4} into \eqref{l4} gives the following inequality. Regularization covers the case of a vanishing norm:
\begin{align}
 \label{l4:2} &\frac{\dif}{\dif t}\|\D\rho\|_{L^4}^2+\frac{c_0}{4}\bar\rho\left(2\bar\theta\right)^{1-\al}\|\D\rho\|_{L^4}^2 \leq~ C\|\D u\|_{L^\infty}\|\D \rho\|_{L^4}^2+C\bar{\rho}\|\D G\|_{L^4}\|\D\rho\|_{L^4}+C\bar{\rho}^2\bar\theta^{-\al}\|\D \theta\|_{L^4}\|\D\rho\|_{L^4}\nonumber\\
 \leq~&\frac{c_0}{8}\bar\rho\left(2\bar\theta\right)^{1-\al}\|\D\rho\|_{L^4}^2+C\|\D u\|_{L^\infty}\|\D\rho\|_{L^4}^2+C\bar{\rho}\|\D G\|_{L^4}\|\D\rho\|_{L^4}+C\bar\rho^3\bar\theta^{-1-\al}\|\D \theta\|_{L^4}^2\nonumber\\
 \leq~&\frac{c_0}{8}\bar\rho\left(2\bar\theta\right)^{1-\al}\|\D\rho\|_{L^4}^2+C\|\D u\|_{L^2}^\frac{1}{7}\|\D^2 u\|_{L^4}^\frac{6}{7}\|\D\rho\|_{L^4}^2\nonumber\\
 &+C\bar{\rho}\left(\bar\rho^\frac{7}{8}\bar{\theta}^{-\al}\|\sqrt{\rho}\dot{u}\|_{L^2}^{\frac{1}{4}}\|\D \dot{u}\|_{L^2}^\frac{3}{4}+\bar\rho\bar{\theta}^{-\al}\|\D\theta\|_{L^4}+\bar{\theta}^{-1}\|\D\theta\|_{L^4}\|\D u\|_{L^2}^{\frac{1}{7}}\|\D^2u\|_{L^4}^{\frac{6}{7}}\right)\|\D\rho\|_{L^4}\nonumber\\
 &+C\bar\rho^3\bar\theta^{-1-\al}\|\D \theta\|_{L^4}^2.
\end{align}

Then, similar as the estimates of  $\|\D\rho\|_{L^2}$.

\textbf{Case 2.1}: For the case  $t\in(0,\sigma(T))$, integrating \eqref{l4:2} over $(0,\sigma(T))$, we have
\begin{align}\label{l4:3}
&\|\D\rho\|_{L^4}^2+\bar\rho\bar\theta^{1-\al}\int_0^{\sigma(T)}\|\D\rho\|_{L^4}^2\dif t\nonumber\\
\leq~& C\|\D\rho_0\|_{L^4}^2+C\int_0^{\sigma(T)}\|\D u\|_{L^2}^\frac{1}{7}\|\D^2 u\|_{L^4}^\frac{6}{7}\|\D\rho\|_{L^4}^2\dif t\nonumber\\
&+C\bar\rho\int_0^{\sigma(T)}\left(\bar\rho^\frac{7}{8}\bar{\theta}^{-\al}\|\sqrt{\rho}\dot{u}\|_{L^2}^{\frac{1}{4}}\|\D \dot{u}\|_{L^2}^\frac{3}{4}+\bar\rho\bar{\theta}^{-\al}\|\D\theta\|_{L^4}+\bar{\theta}^{-1}\|\D\theta\|_{L^4}\|\D u\|_{L^2}^{\frac{1}{7}}\|\D^2u\|_{L^4}^{\frac{6}{7}}\right)\|\D\rho\|_{L^4}\dif t\nonumber\\
&+C\bar\rho^3\bar\theta^{-1-\al}\int_0^{\sigma(T)}\|\D\theta\|_{L^2}^{\frac{1}{2}}\|\D^2\theta\|_{L^2}^{\frac{3}{2}}\dif t\nonumber\\
\leq~&\frac{1}{2}\bar\rho\bar\theta^{1-\al}\int_0^{\sigma(T)}\|\D\rho\|_{L^4}^2\dif t+C\|\D\rho_0\|_{L^4}^2\nonumber\\
&+C\bar\rho^\frac{15}{8}\bar\theta^{-\al}\sup\limits_{t\in[0,\sigma(T)]}\|\D\rho\|_{L^4}\left(\int_0^{\sigma(T)}\|\sqrt{\rho}\dot{u}\|_{L^2}^2\dif t\right)^{\frac{1}{8}}\left(\int_0^{\sigma(T)}\|\D\dot{u}\|_{L^2}^2\dif t\right)^{\frac{3}{8}}\nonumber\\
&+C\bar\rho^3\bar\theta^{-1-\al}\left(\int_0^{\sigma(T)}\|\D\theta\|_{L^2}^2\dif t\right)^{\frac{1}{4}}\left(\int_0^{\sigma(T)}\|\D^2\theta\|_{L^2}^2\dif t\right)^{\frac{3}{4}}\nonumber\\
&+C\bar\rho\bar{\theta}^{-1}\int_0^{\sigma(T)}\|\D\theta\|_{L^2}^{\frac{1}{4}}\|\D^2\theta\|_{L^2}^{\frac{3}{4}}\|\D u\|_{L^2}^\frac{1}{7}\|\D^2 u\|_{L^4}^\frac{6}{7}\|\D\rho\|_{L^4}\dif t+C\int_0^{\sigma(T)}\|\D u\|_{L^2}^\frac{1}{7}\|\D^2 u\|_{L^4}^\frac{6}{7}\|\D\rho\|_{L^4}^2\dif t\nonumber\\
\leq~&\frac{1}{2}\bar\rho\bar\theta^{1-\al}\int_0^{\sigma(T)}\|\D\rho\|_{L^4}^2\dif t+C\|\D\rho_0\|_{L^4}^2+C(K)\left(\bar\rho^\frac{15}{8}\bar\theta^{-\frac{\al}{2}}+\bar\rho^4\bar\theta^{-\frac{1}{2}-\al-\beta}+\bar\rho^\frac{51}{28}\bar\theta^{-\frac{5}{28}-\frac{\al}{2}}\right)\nonumber\\
&+C\int_0^{\sigma(T)}\|\D u\|_{L^2}^\frac{1}{7}\|\D^2 u\|_{L^4}^\frac{6}{7}\|\D\rho\|_{L^4}^2\dif t,
\end{align}
where we have used \eqref{assume:1}. By Gronwall's inequality, we obtain
\begin{align}\label{l4:4}
&\sup\limits_{t\in[0,\sigma(T)]}\|\D\rho\|_{L^4}^2+\bar{\rho}\bar\theta^{1-\al}\int_0^{\sigma(T)}\|\D\rho\|_{L^4}^2\dif t\nonumber\\
\leq~&C\left(\|\D\rho_0\|_{L^4}^2+C(K)\left(\bar\rho^\frac{15}{8}\bar\theta^{-\frac{\al}{2}}+\bar\rho^4\bar\theta^{-\frac{1}{2}-\al-\beta}+\bar\rho^\frac{51}{28}\bar\theta^{-\frac{5}{28}-\frac{\al}{2}}\right)\right)\cdot\nonumber\\
&\exp\left\{C\left(\int_0^{\sigma(T)}\|\D u\|_{L^2}^2\dif t\right)^\frac{1}{14}\left(\int_0^{\sigma(T)}\|\D^2 u\|_{L^4}^2\dif t\right)^\frac{3}{7}\right\}\nonumber\\
\leq~& C\exp\left\{C(K)\bar{\rho}^\frac{23}{28}\bar\theta^{\frac{1}{14}-\frac{1}{2}\al}\right\}\nonumber\\
\leq~& C,
\end{align}
provided that $\bar{\rho}\leq C(K)\bar{\theta}^{\frac{2}{23}(7\alpha-1)}$, where we have taken $\bar{\theta}$ large enough, such that
$$C(K)\left(\bar\rho^\frac{15}{8}\bar\theta^{-\frac{\al}{2}}+\bar\rho^4\bar\theta^{-\frac{1}{2}-\al-\beta}+\bar\rho^\frac{51}{28}\bar\theta^{-\frac{5}{28}-\frac{\al}{2}}\right)\leq 1 \qquad \text{and} \qquad C(K)\bar{\rho}^\frac{23}{28}\bar\theta^{\frac{1}{14}-\frac{1}{2}\al}\leq\ln 2.$$


\textbf{Case 2.2}: Multiplying the squared-norm form of \eqref{l4} by $\sigma^m$ with $m\in\mathbb{N}_+$, we have
\begin{align}
 \label{l4:sigma}
&\frac{\dif}{\dif t}\left(\sigma^m\|\D\rho\|_{L^4}^2\right)+\frac{c_0}{4}\sigma^m\bar\rho\left(2\bar{\theta}\right)^{1-\al}\|\D\rho\|_{L^4}^2\nonumber\\
 \leq~&m\sigma^{m-1}\sigma'\|\D\rho\|_{L^4}^2+C\sigma^m\|\D u\|_{L^\infty}\|\D\rho\|_{L^4}^2+C\bar{\rho}\sigma^m\|\D G\|_{L^4}\|\D\rho\|_{L^4}+C\bar{\rho}^2\bar{\theta}^{-\al}\sigma^m\|\D\theta\|_{L^4}\|\D\rho\|_{L^4}\nonumber\\
 \leq~&m\sigma^{m-1}\sigma'\|\D\rho\|_{L^4}^2+C\sigma^m\|\D u\|_{L^2}^{\frac{1}{7}}\|\D^2 u\|_{L^4}^{\frac{6}{7}}\|\D\rho\|_{L^4}^2\nonumber\\
 &+C\bar{\rho}\sigma^m\left(\bar\rho^\frac{7}{8}\bar{\theta}^{-\al}\|\sqrt{\rho}\dot{u}\|_{L^2}^{\frac{1}{4}}\|\D \dot{u}\|_{L^2}^\frac{3}{4}+\bar\rho\bar{\theta}^{-\al}\|\D\theta\|_{L^4}+\bar{\theta}^{-1}\|\D\theta\|_{L^4}\|\D u\|_{L^2}^{\frac{1}{7}}\|\D^2u\|_{L^4}^{\frac{6}{7}}\right)\|\D\rho\|_{L^4}\nonumber\\
 &+C\bar{\rho}^2\bar{\theta}^{-\al}\sigma^m\|\D\theta\|_{L^2}^{\frac{1}{4}}\|\D^2\theta\|_{L^2}^{\frac{3}{4}}\|\D\rho\|_{L^4} \leq~m\sigma^{m-1}\sigma'\|\D\rho\|_{L^4}^2+C\sigma^m\|\D u\|_{L^2}^{\frac{1}{7}}\|\D^2 u\|_{L^4}^{\frac{6}{7}}\|\D\rho\|_{L^4}^2\nonumber\\
 &+C\bar{\rho}\sigma^m\left(\bar\rho^\frac{7}{8}\bar{\theta}^{-\al}\|\sqrt{\rho}\dot{u}\|_{L^2}^{\frac{1}{4}}\|\D \dot{u}\|_{L^2}^\frac{3}{4}+\bar\rho\bar{\theta}^{-\al}\|\D\theta\|_{L^2}^{\frac{1}{4}}\|\D^2\theta\|_{L^2}^{\frac{3}{4}}+\bar{\theta}^{-1}\|\D\theta\|_{L^4}\|\D u\|_{L^2}^{\frac{1}{7}}\|\D^2u\|_{L^4}^{\frac{6}{7}}\right)\|\D\rho\|_{L^4},
\end{align}
where we have used \eqref{DG4}.

Taking $m=4$ and integrating \eqref{l4:sigma} over $[0,T]$, by \eqref{l4:4}, \eqref{assume:2}, and \eqref{D2u24:sigma}, we obtain
\begin{align}
 \label{Dr4:2}
&\sigma^4\|\D\rho\|_{L^4}^2+\bar\rho\bar{\theta}^{1-\al}\int_0^T\sigma^4\|\D\rho\|_{L^4}^2\dif t\nonumber\\
 \leq~&C\sup\limits_{t\in[0,\sigma(T)]}\|\D\rho\|_{L^4}^2+C\sup\limits_{t\in[0,T]}\|\D\rho\|_{L^4}\left(\int_0^T\|\D u\|_{L^2}^2\dif t\right)^{\frac{1}{14}}\left(\int_0^T\sigma^5\|\D^2u\|_{L^4}^2\dif t\right)^{\frac{3}{7}}\left(\int_0^T\|\D\rho\|_{L^4}^2\dif t\right)^{\frac{1}{2}}\nonumber\\
 &+C\bar\rho^\frac{15}{8}\bar{\theta}^{-\al}\left(\int_0^T\sigma^3\|\sqrt{\rho}\dot{u}\|_{L^2}^2\dif t\right)^{\frac{1}{8}}\left(\int_0^T\sigma^5\|\D\dot{u}\|_{L^2}^2\dif t\right)^{\frac{3}{8}}\left(\int_0^T\|\D\rho\|_{L^4}^2\dif t\right)^{\frac{1}{2}}\nonumber\\
 &+C\bar\rho^2\bar{\theta}^{-\al}\left(\int_0^T\|\D\theta\|_{L^2}^2\dif t\right)^{\frac{1}{8}}\left(\int_0^T\sigma^4\|\D^2\theta\|_{L^2}^2\dif t\right)^{\frac{3}{8}}\left(\int_0^T\|\D\rho\|_{L^4}^2\dif t\right)^{\frac{1}{2}}\nonumber\\
 &+C\bar{\rho}\bar{\theta}^{-1}\sup\limits_{t\in[0,T]}\|\D\rho\|_{L^4}\left(\int_0^T\|\D\theta\|_{L^2}^2\dif t\right)^{\frac{1}{8}}\left(\int_0^T\sigma^4\|\D^2\theta\|_{L^2}^2\dif t\right)^{\frac{3}{8}}\left(\int_0^T\|\D u\|_{L^2}^2\dif t\right)^{\frac{1}{14}} \nonumber\\
 &\times\left(\int_0^T\sigma^5\|\D^2u\|_{L^4}^2\dif t\right)^{\frac{3}{7}} \leq~\frac{1}{2}K_9+C(K)\left(\bar{\rho}^{-\frac{6}{7}}+\bar{\rho}^{\frac{3}{2}}\bar{\theta}^{-\frac{3}{4}\al-\frac{1}{8}}+\bar{\rho}^{\frac{13}{8}}\bar{\theta}^{\frac{3}{4}-\frac{\al}{2}-\frac{7}{8}\beta}+\bar{\rho}\bar{\theta}^{\frac{7}{4}-\frac{\al}{2}-\frac{7}{8}\beta}\right)\nonumber\\
 \leq~&\frac{3}{4}K_9,
\end{align}
provided that $K_9\geq 1+2C\sup\limits_{t\in[0,\sigma(T)]}\|\D\rho\|_{L^4}^2$, and
$$\bar{\rho}\leq C(K)\min\left\{\bar{\theta}^{\frac{1}{12}(6\alpha+1)}, \bar{\theta}^{\frac{1}{13}(4\alpha+7\beta-6)}\right\},$$
where we have taken $\bar{\rho}$ and $\bar{\theta}$ sufficiently large, such that
$$C(K)\left(\bar{\rho}^{-\frac{6}{7}}+\bar{\rho}^{\frac{3}{2}}\bar{\theta}^{-\frac{3}{4}\al-\frac{1}{8}}+\bar{\rho}^{\frac{13}{8}}\bar{\theta}^{\frac{3}{4}-\frac{\al}{2}-\frac{7}{8}\beta}+\bar{\rho}\bar{\theta}^{\frac{7}{4}-\frac{\al}{2}-\frac{7}{8}\beta}\right)\leq\frac{1}{8}K_9.$$
Combining (\ref{l4:4}) and (\ref{Dr4:2}), we have proved \eqref{3.9:2}. From \eqref{3.9:1} and \eqref{3.9:2}, we get
\begin{equation}
\|\rho-\bar{\rho}\|_{L^\infty}^2\leq C\|\nabla\rho\|_{L^2}^{\frac{2}{3}}\|\nabla\rho\|_{L^4}^{\frac{4}{3}}\leq C(K)\leq\frac{1}{9}\bar{\rho}^2.
\end{equation}
This completes the proof.
\end{proof}

Finally, taking
$$\widetilde{\rL}\geq\max\{\rL_i\}_{i=1}^5,$$
we complete the proof of Proposition \ref{p4.1} from Lemma \ref{Lem33}--\ref{Lem39}.

\subsection{Proof of Proposition \ref{Lem:10}}\label{sec3.4}

In this subsection, we will give the uniform estimate of the solution in $H^2$ space.
\begin{proof}
First, the following estimates come from Proposition \ref{p4.1} and Lemma \ref{Lem3.1} directly:
\begin{align}\label{342}
\|\rho-\bar{\rho}\|_{H^1}^2+\|u\|_{H^1}^2+\|\theta-\bar{\theta}\|_{H^2}^2\leq C\bar{\theta}^2.
\end{align}

From $(\ref{FCNS:2})_2$, one has
\begin{align}
 \label{346-0}
\|\D^2u\|_{L^2}^2\leq~&C\bar{\theta}^{-2\al}\left(\bar{\rho}\|\sqrt{\rho}\dot{u}\|_{L^2}^2+\bar{\theta}^2\|\D\rho\|_{L^2}^2+\bar{\rho}^2\|\D\theta\|_{L^2}^2+\bar{\theta}^{2\al-2}\|\D\theta\|_{L^4}^2\|\D u\|_{L^4}^2\right)
\nonumber\\
 \leq~&C\bar{\theta}^{-2\al}\left(\bar{\rho}\bar{\theta}^{2\al}+\bar{\theta}^2+\bar{\rho}^2+\bar{\theta}^{2\al-2}\|\D\theta\|_{L^2}^\frac{1}{2}\|\D^2\theta\|_{L^2}^\frac{3}{2}\|\D u\|_{L^2}^\frac{1}{2}\|\D^2 u\|_{L^2}^\frac{3}{2}\right)\nonumber\\
 \leq~&C\bar{\rho}+C\bar{\theta}^{2-2\al}+C\bar{\rho}^2\bar{\theta}^{-2\al}+C\bar{\theta}^{-2}\|\D^2\theta\|_{L^2}^\frac{3}{2}\|\D^2 u\|_{L^2}^\frac{3}{2} \leq~\frac{1}{2}\|\D^2 u\|_{L^2}^2+C\bar{\rho}+C\bar{\theta}^{-2}\nonumber\\
 \leq~&\frac{1}{2}\|\D^2 u\|_{L^2}^2+C\bar{\rho},
\end{align}
which implies $$\|\D^2u\|_{L^2}^2\leq C\bar{\rho}.$$

It remains to estimate
$$\|\rho_t\|_{H^1}^2+\|\D^2\rho\|_{L^2}^2+\int_0^T\left(\|u_t\|_{H^1}^2+\|\theta_t\|_{H^1}^2\right)\dif t.$$
We divide the proof into several steps.

\textbf{Step 1}: Estimate of $u_t$ and $\theta_t$ in $L^2\left(0,T;H^1\right)$.

From \eqref{goal:1} and \eqref{goal:2}, one has
\begin{align}
\int_0^T\left(\|u_t+u\cdot\D u\|_{L^2}^2+\|\D u_t+\D(u\cdot\D u)\|_{L^2}^2\right)\dif t\leq C\bar{\theta}^{\al}.
\end{align}

Therefore, by \eqref{goal:1}, \eqref{estimate:DuL6}, and \eqref{goal:2}, we have
\begin{align}\label{347}
\int_0^T\|u_t\|_{H^1}^2\dif t\leq~& C\bar{\theta}^{\al}+C\int_0^T\left(\|u\cdot\D u\|_{L^2}^2+\|\D u\|_{L^4}^4+\|u\D^2u\|_{L^2}^2\right)\dif t\nonumber\\
\leq~&C\bar{\theta}^{\al}+C\int_0^T\left(\|\D u\|_{L^2}^3\|\D u\|_{L^6}+\|\D u\|_{L^2}\|\D u\|_{L^6}^3+\|\D u\|_{L^2}\|\D u\|_{L^6}\|\D^2u\|_{L^2}^2\right)\dif t\nonumber\\
\leq~&C\bar{\theta}^{\al}+C\bar{\rho}^{\frac{1}{2}}\int_0^T\left(\|\D u\|_{L^2}^2+\|\D u\|_{L^6}^2+\|\D^2u\|_{L^2}^2\right)\dif t.
\end{align}

From \eqref{346-0}, one has
\begin{align}
 \label{348}
\int_0^T\|\D^2u\|_{L^2}^2\dif t\leq~&C\bar{\theta}^{-2\al}\int_0^T\left(\bar{\rho}\|\sqrt{\rho}\dot{u}\|_{L^2}^2+\bar{\rho}^2\|\D\theta\|_{L^2}^2+\bar{\theta}^2\|\D\rho\|_{L^2}^2\right)\dif t
\nonumber\\
 &+C\bar{\theta}^{-2}\int_0^T\|\D\theta\|_{L^2}^\frac{1}{2}\|\D^2\theta\|_{L^2}^\frac{3}{2}\|\D u\|_{L^2}^\frac{1}{2}\|\D u\|_{L^6}^\frac{3}{2}\dif t \leq~C\bar{\theta}^{1-\al},
\end{align}
where we have used \eqref{goal:1}, \eqref{goal:2}, \eqref{estimate:D2theta-2}, \eqref{lem3.1:g2}, and \eqref{estimate:DuL6:5}.

Inserting \eqref{lem3.1:g2}, \eqref{estimate:DuL6:5} and \eqref{348} into \eqref{347}, we obtain
\begin{align}\label{349}
\int_0^T\|u_t\|_{H^1}^2\dif t\leq C\bar{\theta}^{\al}.
\end{align}

From \eqref{goal:1} and \eqref{goal:2}, we have
\begin{align}
\int_0^T\left(\|\theta_t+u\cdot\D \theta\|^2+\|\D \theta_t+\D(u\cdot\D \theta)\|^2\right)\dif t\leq C\bar{\theta}^{\beta}.
\end{align}
Therefore, by \eqref{goal:1}, \eqref{estimate:DuL6}, and \eqref{goal:2}, we have
\begin{align}
 \label{3410} &\int_0^T\|\theta_t\|_{H^1}^2\dif t \leq~ C\bar{\theta}^{\beta}+C\int_0^T\left(\|u\cdot\D\theta\|_{L^2}^2+\|\D u\cdot\D\theta\|_{L^2}^2+\|u\D^2\theta\|_{L^2}^2\right)\dif t\nonumber\\
 \leq~&C\bar{\theta}^{\beta}+C\int_0^T\left(\|\D u\|_{L^2}^2\|\D\theta\|_{L^2}\|\D^2\theta\|_{L^2}+\|\D u\|_{L^6}^2\|\D\theta\|_{L^2}\|\D^2\theta\|_{L^2}+\|\D u\|_{L^2}\|\D u\|_{L^6}\|\D^2\theta\|_{L^2}^2\right)\dif t\nonumber\\
 \leq~&C\bar{\theta}^{\beta}+C\bar{\rho}\int_0^T\left(\|\D\theta\|_{L^2}^2+\|\D^2\theta\|_{L^2}^2\right)\dif t.
\end{align}

Inserting \eqref{lem3.1:g2}, \eqref{estimate:D2theta:2-1} and \eqref{estimate:D2theta-2} into \eqref{3410}, we obtain
\begin{align}\label{3411}
\int_0^T\|\theta_t\|_{H^1}^2\dif t\leq C\bar{\theta}^{\beta}.
\end{align}

\bigskip

\textbf{Step 2}: Estimate of $\|\D^2\rho\|_{L^2}^2$ and $\|\rho_t\|_{H^1}^2$.

By \eqref{FCNS:2}$_1$, we have
\begin{align}\label{sec3.4:8}
\|\rho_t\|_{L^2}\leq
C\|{u}\|_{L^\infty}\|\D\rho\|_{L^2}+C\bar{\rho}\|\D u\|_{L^2}\leq C\bar{\rho},
\end{align}
and
\begin{align}\label{sec3.4:9}
\nabla\rho_t+(u\cdot\D)\D\rho+\nabla u\nabla\rho+\nabla\rho\text{div}u+\rho\nabla\text{div}u=0.
\end{align}

Then \eqref{sec3.4:9} implies
\begin{align}\label{sec3.4:10}
\|\nabla\rho_t\|_{L^2}^2
\leq
C\|{u}\|_{L^\infty}^2\|\nabla^2\rho\|_{L^2}^2+C\|\nabla{u}\|_{L^2}^{\frac{1}{2}}\|\nabla{u}\|_{L^6}^{\frac{3}{2}}\|\nabla\rho\|_{L^4}^2+C\bar{\rho}^2\|\nabla^2{u}\|_{L^2}^2.
\end{align}

By direct computation, we derive the equation of $\D^2\rho$ as follows:
\begin{align*}
&(\D^2 \rho)_t+2(\D u\cdot\D)\D\rho+(u\cdot\nabla)\D^2\rho+\D \rho\cdot\D^2 u+\D^2 \rho\text{div}u+2\D\rho\otimes\D\text{div}u+\rho\D^2\text{div}u=0.
\end{align*}  
Multiplying the above equation by $\D^2 \rho$ and integrating the resulting equation over $\R^3$, we obtain
\begin{align}\label{sec3.4:2}
&\frac{1}{2}\frac{\dif}{\dif t}\|\D^2\rho\|_{L^2}^{2}+c_0\bar{\rho}\bar{\theta}^{1-\al}\|\D^2\rho\|_{L^2}^{2}\nonumber\\
\leq~&C\|\D u\|_{L^\infty}\|\D^2\rho\|_{L^2}^{2}+ C\|\D^2u\|_{L^4}\|\D \rho\|_{L^4}\|\D^2 \rho\|_{L^2}+C\bar{\rho}\|\D^2 G\|_{L^2}\|\D^2\rho\|_{L^2}+C\bar{\rho}^2\bar{\theta}^{-\al}\|\D^2\theta\|_{L^2}\|\D^2\rho\|_{L^2}\nonumber\\
&+C\bar{\rho}\bar{\theta}^{-\al}\|\D^2\theta\|_{L^2}\|\D \rho\|_{L^2}^{\frac{1}{3}}\|\D \rho\|_{L^4}^{\frac{2}{3}}\|\D^2 \rho\|_{L^2}+C\bar{\rho}^2\bar{\theta}^{-(\al+1)}\|\D^2\rho\|_{L^2}\|\D\theta\|_{L^2}^{\frac{1}{2}}\|\D^2\theta\|_{L^2}^{\frac{3}{2}}\nonumber\\
\leq~& \frac{c_0}{2}\bar{\rho}\bar{\theta}^{1-\al}\|\D^2\rho\|_{L^2}^{2}+C\bar{\rho}^{-1}\bar{\theta}^{\al-1}\|\D u\|_{L^2}^\frac{2}{7}\|\D^2 u\|_{L^4}^\frac{12}{7}\|\D^2\rho\|_{L^2}^2+ C\bar{\rho}^{-1}\bar{\theta}^{\al-1}\|\D^2 u\|_{L^4}^2\nonumber\\
&+C\bar{\rho}\bar{\theta}^{\al-1}\|\D H\|_{L^2}^2+C\bar{\rho}^3\bar{\theta}^{-(\al+1)}\|\D^2\theta\|_{L^2}^2\nonumber\\
\leq~& \frac{c_0}{2}\bar{\rho}\bar{\theta}^{1-\al}\|\D^2\rho\|_{L^2}^{2}+C\bar{\rho}^{-1}\bar{\theta}^{\al-1}\|\D u\|_{L^2}^\frac{2}{7}\|\D^2 u\|_{L^4}^\frac{12}{7}\|\D^2\rho\|_{L^2}^2+ C\bar{\rho}^{-1}\bar{\theta}^{\al-1}\|\D^2 u\|_{L^4}^2\nonumber\\
&+C\bar{\rho}\bar{\theta}^{\al-\frac{3}{2}}\|\D^2 u\|_{L^4}^2+C\bar{\rho}^3\bar{\theta}^{-(\al+1)}\|\D\dot{u}\|_{L^2}^2\nonumber\\
&+C\bar{\rho}^3\bar{\theta}^{-(\al+1)}\|\D^2\theta\|_{L^2}^2+C\bar{\rho}\bar{\theta}^{\al-1}\|\D u\|_{L^6}^2+C\bar{\rho}^{\frac{3}{4}}\bar{\theta}^{-(\al+1)}\|\sqrt{\rho}\dot{u}\|_{L^2}^{\frac{1}{2}}\|\D\dot{u}\|_{L^2}^{\frac{3}{2}}\nonumber\\
&+C\bar{\rho}\bar{\theta}^{-(\al+1)}\|\D\theta\|_{L^2}^{\frac{1}{2}}\|\D^2\theta\|_{L^2}^{\frac{3}{2}}+C\bar{\rho}\bar{\theta}^{\al-1}\|\D u\|_{L^2}^{\frac{2}{7}}\|\D^2 u\|_{L^4}^{\frac{12}{7}},
\end{align}
where we have used 
\begin{align}\label{sec3.4:3}
&\|\D^2 G\|_{L^2}\leq C\|\nabla H\|_{L^2}\nonumber\\
\leq~& C\bar{\rho}\bar{\theta}^{-(\al+1)}\|\D\theta\|_{L^2}^{\frac{1}{2}}\|\D^2\theta\|_{L^2}^{\frac{1}{2}}\|\D\dot{u}\|_{L^2}
+C\bar{\rho}\bar{\theta}^{-(\al+1)}\|\D\theta\|_{L^2}^{\frac{1}{2}}\|\D^2\theta\|_{L^2}^{\frac{3}{2}}+C\bar{\theta}^{-2}\|\D u\|_{L^6}\|\D^2\theta\|_{L^2}^2\nonumber\\
&+C\bar{\rho}^{-\frac{1}{8}}\bar{\theta}^{-\al}\|\sqrt{\rho}\dot{u}\|_{L^2}^{\frac{1}{4}}\|\D\dot{u}\|_{L^2}^{\frac{3}{4}}\|\D\rho\|_{L^4}+C\bar{\rho}\bar{\theta}^{-\al}\|\D\dot{u}\|_{L^2}+C\bar{\theta}^{-\al}\|\D\rho\|_{L^4}\|\D\theta\|_{L^2}^{\frac{1}{4}}\|\D^2\theta\|_{L^2}^{\frac{3}{4}}\nonumber\\
&+C\bar{\rho}\bar{\theta}^{-\al}\|\D^2\theta\|_{L^2}+C\bar{\theta}^{-1}\|\D\theta\|_{L^2}^{\frac{1}{4}}\|\D^2\theta\|_{L^2}^{\frac{3}{4}}\|\D^2u\|_{L^4}+C\bar{\theta}^{-1}\|\D u\|_{L^2}^{\frac{1}{7}}\|\D^2 u\|_{L^4}^{\frac{6}{7}}\|\D^2\theta\|_{L^2}\nonumber\\
\leq~&C\bar{\rho}\bar{\theta}^{-\al}\|\D\dot{u}\|_{L^2}+C\bar{\rho}\bar{\theta}^{-\al}\|\D^2\theta\|_{L^2}+C\|\D u\|_{L^6}+C\bar{\rho}^{-\frac{1}{8}}\bar{\theta}^{-\al}\|\sqrt{\rho}\dot{u}\|_{L^2}^{\frac{1}{4}}\|\D\dot{u}\|_{L^2}^{\frac{3}{4}}\nonumber\\
&+C\bar{\theta}^{-\al}\|\D\theta\|_{L^2}^{\frac{1}{4}}\|\D^2\theta\|_{L^2}^{\frac{3}{4}}+C\bar{\theta}^{-\frac{1}{4}}\|\D^2u\|_{L^4}+C\|\D u\|_{L^2}^{\frac{1}{7}}\|\D^2 u\|_{L^4}^{\frac{6}{7}}.
\end{align}

Then, integrating \eqref{sec3.4:2} over $(0,T)$, we have
\begin{align}\label{sec3.4:6}
&\|\D^2\rho\|_{L^2}^2+\bar{\rho}\bar{\theta}^{1-\al}\int_0^T\|\D^2\rho\|_{L^2}^2\dif t\nonumber\\
\leq~&\|\D^2\rho_0\|_{L^2}^2+C\bar{\rho}\bar{\theta}^{\al-1}\int_0^{T}\|\D u\|_{L^2}^\frac{2}{7}\|\D^2 u\|_{L^4}^\frac{12}{7}\|\D^2\rho\|_{L^2}^2\dif t+ C\bar{\rho}^{-1}\bar{\theta}^{\al-1}\int_0^T\|\D^2 u\|_{L^4}^2\dif t\nonumber\\
&+C\bar{\rho}\bar{\theta}^{\al-\frac{3}{2}}\int_0^T\|\D^2 u\|_{L^4}^2\dif t\nonumber\\
&+C\bar{\rho}^3\bar{\theta}^{-(\al+1)}\int_0^T\|\D\dot{u}\|_{L^2}^2\dif t+C\bar{\rho}^3\bar{\theta}^{-(\al+1)}\int_0^T\|\D^2\theta\|_{L^2}^2\dif t+C\bar{\rho}\bar{\theta}^{\al-1}\int_0^T\|\D u\|_{L^6}^2\dif t\nonumber\\
&+C\bar{\rho}^{\frac{3}{4}}\bar{\theta}^{-(\al+1)}\int_0^T\|\sqrt{\rho}\dot{u}\|_{L^2}^{\frac{1}{2}}\|\D\dot{u}\|_{L^2}^{\frac{3}{2}}\dif t+C\bar{\rho}\bar{\theta}^{-(\al+1)}\int_0^T\|\D\theta\|_{L^2}^{\frac{1}{2}}\|\D^2\theta\|_{L^2}^{\frac{3}{2}}\dif t\nonumber\\
&+C\bar{\rho}\bar{\theta}^{\al-1}\int_0^T\|\D u\|_{L^2}^{\frac{2}{7}}\|\D^2 u\|_{L^4}^{\frac{12}{7}}\dif t\nonumber\\
\leq~&C\bar{\rho}^{\frac{2}{7}}+C\bar{\rho}\bar{\theta}^{\al-1}\int_0^{T}\|\D u\|_{L^2}^\frac{2}{7}\|\D^2 u\|_{L^4}^\frac{12}{7}\|\D^2\rho\|_{L^2}^2\dif t,
\end{align}
Here \eqref{D2u24} and \eqref{D2u24:sigma} give
\[
\int_0^T\|\D^2u\|_{L^4}^2\dif t
\leq C\left(\bar{\rho}^{\frac{7}{4}}\bar{\theta}^{-\al}+\bar{\rho}^{-1}\bar{\theta}^{1-\al}\right).
\]
Consequently, by H\"older's inequality and \eqref{lem3.1:g2},
\begin{align*}
\bar{\rho}\bar{\theta}^{\al-1}\int_0^T\|\D u\|_{L^2}^{\frac{2}{7}}\|\D^2u\|_{L^4}^{\frac{12}{7}}\dif t
&\leq C\left(\bar{\rho}^{\frac{37}{14}}\bar{\theta}^{-\frac{6}{7}}+\bar{\rho}^{\frac{2}{7}}\right)
\leq C\bar{\rho}^{\frac{2}{7}},\\
\bar{\rho}\bar{\theta}^{\al-\frac{3}{2}}\int_0^T\|\D^2u\|_{L^4}^2\dif t
&\leq C\left(\bar{\rho}^{\frac{11}{4}}\bar{\theta}^{-\frac{3}{2}}+\bar{\theta}^{-\frac{1}{2}}\right)\leq C,
\end{align*}
where the last inequalities follow from \eqref{rho:theta} and $\bar{\rho},\bar{\theta}>1$. Therefore, by Gronwall's inequality, we have
\begin{align}\label{sec3.4:7}
\sup\limits_{t\in[0,T]}\|\D^2\rho\|_{L^2}^2+\bar{\rho}\bar{\theta}^{1-\al}\int_0^T\|\D^2\rho\|_{L^2}^2\dif t\leq~& C\bar{\rho}^{\frac{2}{7}}\exp\left\{C\bar{\rho}\bar{\theta}^{\al-1}\int_0^{T}\|\D u\|_{L^2}^\frac{2}{7}\|\D^2 u\|_{L^4}^\frac{12}{7}\dif t\right\}\nonumber\\
\leq~& C\bar{\rho}^{\frac{2}{7}}\exp\left\{C\bar{\rho}^{\frac{2}{7}}\right\},
\end{align}
where we have used (\ref{g:3.4:1}), (\ref{D2u24:sigma}), (\ref{D2u24}), (\ref{lem3.1:g2}), and (\ref{estimate:D2theta-2}).

Combining \eqref{342}, \eqref{349}, \eqref{3411}, \eqref{sec3.4:7} and \eqref{sec3.4:10}, we complete the proof of Proposition \ref{Lem:10}. To obtain Theorem \ref{th}, choose
\[
\eta>\max\left\{3,\frac{3}{\beta-\alpha},\frac{1}{2(\alpha-\beta+2)}\right\}.
\]
Then $\bar{\theta}\geq\bar{\rho}^{\eta}$ and a sufficiently large $\rL$ imply \eqref{rho:theta} and all the lower temperature thresholds. The strict inequalities in the choice of $\eta$ allow the fixed constants in these relative conditions to be absorbed by increasing $\rL$. The estimates in Propositions \ref{p4.1} and \ref{Lem:10}, together with Proposition \ref{local}, then give global continuation and the conclusion of Theorem \ref{th}.
\end{proof}

\subsection{Proof of Theorem \ref{th:fixed}}\label{sec:fixed-density}

We now fix an arbitrary $\bar{\rho}>0$ and prove the high-temperature criterion in Theorem \ref{th:fixed}. The quantities $A_i,B_i,C_i$ and the effective viscous flux $G$ keep their original definitions. We first establish the velocity and temperature estimates without assuming bounds for the density gradients. We then derive the pointwise density bounds and close the gradient estimates in the time variable $\tau=\bar{\theta}^{1-\al}t$. Finally, we recover the uniform second-order estimates and all the time-derivative bounds in \eqref{sp}. Throughout this subsection the thresholds are independent of the time endpoint $T$.

\subsubsection{Low-order estimates for a fixed positive far-field density}
\label{fixed:low-section}

We now give an alternative low-order argument for each fixed $\bar{\rho}>0$.
The quantities $A_i$, $B_i$ and $C_i$ retain their definitions in
\eqref{A1A2}--\eqref{A6A7}. In this argument, $C(\bar{\rho})$ may depend on
$\bar{\rho}$, the initial $H^2$ bound $C_0$, and the physical parameters,
but is independent of $T$, $\bar{\theta}$ and the bootstrap constants.
The notation $C(\bar{\rho},K,M_6)$ additionally allows dependence on the
fixed constants $K_1,\ldots,K_7,M_6$. The auxiliary constant $M_6$ below
is separate from $K_8,K_9$ in the preceding large-density argument.

\begin{lemma}\label{fixed:low-lemma}
Assume that
\[
 2\leq\al<\beta<\al+2,\qquad 4\al-3\beta<4,
\]
and fix $\bar{\rho}>0$. Suppose that the initial perturbations have an
$H^2$ bound $C_0$ independent of $\bar{\theta}$, with
$3\bar{\rho}/4\leq\rho_0\leq5\bar{\rho}/4$.
There exist constants $\varepsilon>0$, $M_6,K_1,\ldots,K_7$, independent
of $\bar{\theta}$ and $T$, and a subsequently chosen temperature threshold
such that the following improvement holds. Assume on $[0,T]$ that
\begin{equation}\label{fixed:low-pointwise}
 \frac12\bar{\rho}\leq\rho\leq\frac32\bar{\rho},\qquad
 \|\theta-\bar{\theta}\|_{L^\infty}\leq\varepsilon\bar{\theta},
\end{equation}
and that
\begin{align}\label{fixed:low-bootstrap}
&A_1(\sigma(T))\leq2K_1\bar{\theta}^{\al},\quad
 A_2(\sigma(T))\leq2K_2\bar{\theta}^{\beta},\quad
 A_3(\sigma(T))\leq2K_3\bar{\theta}^{2\al},\quad
 A_4(\sigma(T))\leq2K_4\bar{\theta}^{2\beta},\nonumber\\
&B_1(T)\leq2K_5\bar{\rho}\bar{\theta},\quad
 B_2(T)\leq2K_6\bar{\rho}\bar{\theta}^2,\quad
 B_3(T)\leq2\bar{\theta}^{(\al+2)/3},\quad
 B_4(T)\leq2K_7\bar{\rho}\bar{\theta}^2.
\end{align}
Assume also that
\begin{equation}\label{fixed:low-mass-bootstrap}
 \sup_{0\leq t\leq T}\left(\int\rho|\rho-\bar{\rho}|^6\dif x\right)^{1/3}
 +\bar{\theta}^{1-\al}\int_0^T
       \left(\int\rho|\rho-\bar{\rho}|^6\dif x\right)^{1/3}\dif t
 \leq2M_6.
\end{equation}
Then all the factors $2$ on the right-hand sides of
\eqref{fixed:low-bootstrap} and \eqref{fixed:low-mass-bootstrap} can be
replaced by $1$, and
\[
 \sup_{0\leq t\leq T}\|\theta-\bar{\theta}\|_{L^\infty}
 \leq\frac12\varepsilon\bar{\theta}.
\]
No bounds for $C_1(T)$ or $C_2(T)$ are assumed.
\end{lemma}

\begin{proof}
The relative entropy identity obtained from \eqref{FCNS:2} is
\begin{align*}
&\frac{\dif}{\dif t}\int\left[\frac12\rho|u|^2
 +R\bar{\theta}\left(\rho\log\frac{\rho}{\bar{\rho}}-\rho+\bar{\rho}\right)
 +\rho\left(\theta-\bar{\theta}-\bar{\theta}\log\frac{\theta}{\bar{\theta}}\right)
 \right]\dif x\\
&\quad+\int\left[\bar{\theta}\theta^{\al-1}
       \left(2\mu|\frD(u)|^2+\lambda(\div u)^2\right)
       +\kappa\bar{\theta}\theta^{\beta-2}|\D\theta|^2\right]\dif x=0.
\end{align*}
Consequently, under \eqref{fixed:low-pointwise},
\begin{equation}
 \label{fixed:low-entropy} \sup_{0\leq t\leq T}\int\left[\rho|u|^2 +\bar{\rho}^{-1}\bar{\theta}|\rho-\bar{\rho}|^2 +\rho\bar{\theta}^{-1}|\theta-\bar{\theta}|^2\right]\dif x +\int_0^T\left(\bar{\theta}^{\al}\|\D u\|_{L^2}^2 +\bar{\theta}^{\beta-1}\|\D\theta\|_{L^2}^2\right)\dif t \leq C(\bar{\rho})\bar{\theta}.
\end{equation}
Here the stress is coercive under $2\mu+3\lambda\geq0$, as in the
combined stress estimate preceding \eqref{3.4.6}. In particular,
\begin{equation}\label{fixed:low-density-lp}
 \sup_{0\leq t\leq T}\|\rho-\bar{\rho}\|_{L^p}
 \leq C(\bar{\rho},p),\qquad 2\leq p<\infty.
\end{equation}

\textbf{Step 1. Elliptic estimates without density derivatives.}
Write the momentum equation as
\begin{equation*}
 \bar{\theta}^{\al}\left[\mu\Delta u+(\mu+\lambda)\D\div u\right] ={}\rho\dot u+R\D(\rho\theta-\bar{\rho}\bar{\theta}) -\div\left[(\theta^{\al}-\bar{\theta}^{\al}) \left(2\mu\frD(u)+\lambda\div u\,I\right)\right].
\end{equation*}
For $p=6,12$ and $q=3p/(3+p)$, the first-order potential estimate and
boundedness of the zero-order multipliers yield
\begin{align*}
 \|\D u\|_{L^p}\leq C(\bar{\rho},p)\big[&
 \bar{\theta}^{-\al}\|\dot u\|_{L^q}
 +\bar{\theta}^{1-\al}\|\rho-\bar{\rho}\|_{L^p}
 +\bar{\theta}^{-\al}\|\theta-\bar{\theta}\|_{L^p}\big]
 +C(p,\al)\varepsilon\|\D u\|_{L^p}.
\end{align*}
Choose $\varepsilon$ small enough to absorb the last term. Then
\begin{equation}\label{fixed:low-lame6}
 \|\D u\|_{L^6}^2\leq C(\bar{\rho})\left[
 \bar{\theta}^{-2\al}\left(\|\sqrt{\rho}\dot u\|_{L^2}^2
                               +\|\D\theta\|_{L^2}^2\right)
 +\bar{\theta}^{2-2\al}\|\rho-\bar{\rho}\|_{L^6}^2\right],
\end{equation}
and
\begin{equation}
 \label{fixed:low-lame12} \|\D u\|_{L^{12}}\leq C(\bar{\rho})\big[ \bar{\theta}^{-\al}\|\sqrt{\rho}\dot u\|_{L^2}^{3/4} \|\D\dot u\|_{L^2}^{1/4} +\bar{\theta}^{1-\al}\|\rho-\bar{\rho}\|_{L^{12}} +\bar{\theta}^{-\al}\|\D\theta\|_{L^2}^{3/4} \|\D^2\theta\|_{L^2}^{1/4}\big].
\end{equation}
Combining the assumed $A_i$ bounds on $[0,\sigma(T)]$ and the $B_i$
bounds on $[1,T]$ with \eqref{fixed:low-lame6}, we first have
\begin{equation}\label{fixed:low-uniform-first}
 \sup_{0\leq t\leq T}\left(
 \|\D u\|_{L^2}+\|\D u\|_{L^6}+\|\D\theta\|_{L^2}
 +\bar{\theta}^{-\al}\|\sqrt{\rho}\dot u\|_{L^2}
 +\bar{\theta}^{-\beta}\|\sqrt{\rho}\dot\theta\|_{L^2}\right)
 \leq C(\bar{\rho},K).
\end{equation}
Dividing the heat equation by $\kappa\theta^{\beta}$ gives
\begin{align}\label{fixed:low-heat-elliptic}
 \|\D^2\theta\|_{L^2}\leq C(\bar{\rho})\big[&
 \bar{\theta}^{-\beta}\|\sqrt{\rho}\dot\theta\|_{L^2}
 +\bar{\theta}^{1-\beta}\|\D u\|_{L^2}
 +\bar{\theta}^{\al-\beta}\|\D u\|_{L^2}^{1/2}
                                  \|\D u\|_{L^6}^{3/2}\big]
 +C\bar{\theta}^{-1}\|\D\theta\|_{L^4}^2.
\end{align}
The integration-by-parts inequality
$\|\D\theta\|_{L^4}^2\leq
 C\|\theta-\bar{\theta}\|_{L^\infty}\|\D^2\theta\|_{L^2}$
allows us to absorb its last term by decreasing the already fixed
$\varepsilon$. Thus
\begin{equation}\label{fixed:low-heat-uniform}
 \sup_{0\leq t\leq T}\left(\|\D^2\theta\|_{L^2}
                   +\|\theta-\bar{\theta}\|_{L^\infty}\right)
 \leq C(\bar{\rho},K),
\end{equation}
where the second bound uses
$\|\theta-\bar{\theta}\|_{L^\infty}^2
 \leq C\|\D\theta\|_{L^2}\|\D^2\theta\|_{L^2}$.

\textbf{Step 2. Mass-weighted $L^6$ damping.}
The continuity equation and the original effective flux $G$ give
\begin{equation}
 \label{fixed:low-mass-equation} (\partial_t+u\cdot\D)(\rho-\bar{\rho}) +\frac{R}{2\mu+\lambda}\rho\theta^{1-\al}(\rho-\bar{\rho}) =-\frac{\rho}{2\mu+\lambda} \left[G+R\bar{\rho}(\theta^{1-\al}-\bar{\theta}^{1-\al})\right].
\end{equation}
Multiplication by $\rho|\rho-\bar{\rho}|^4(\rho-\bar{\rho})$ and the
continuity equation cancel the transport contribution exactly. Hence
\begin{align*}
 &\frac16\frac{\dif}{\dif t}\int\rho|\rho-\bar{\rho}|^6\dif x
 +\frac{R}{2\mu+\lambda}\int\rho^2\theta^{1-\al}
                                      |\rho-\bar{\rho}|^6\dif x\\
 &\quad=-\frac{1}{2\mu+\lambda}\int\rho^2|\rho-\bar{\rho}|^4
  (\rho-\bar{\rho})
       \left[G+R\bar{\rho}(\theta^{1-\al}-\bar{\theta}^{1-\al})\right]\dif x.
\end{align*}
It follows by Young's inequality that
\begin{equation}
 \label{fixed:low-mass-damping} \frac{\dif}{\dif t}\left(\int\rho|\rho-\bar{\rho}|^6\dif x\right)^{1/3} +c(\bar{\rho})\bar{\theta}^{1-\al} \left(\int\rho|\rho-\bar{\rho}|^6\dif x\right)^{1/3} \leq C(\bar{\rho})\bar{\theta}^{\al-1} \left[\|G\|_{L^6}^2 +\bar{\theta}^{-2\al}\|\D\theta\|_{L^2}^2\right].
\end{equation}
A positive regularization of the integral justifies this calculation
when it vanishes. The original equation $-\Delta G=\div H$, with the
same $H$ as before, contains no density derivative. By
\eqref{fixed:low-heat-uniform},
\begin{equation}\label{fixed:low-flux6}
 \|G\|_{L^6}^2\leq C(\bar{\rho},K)\left[
 \bar{\theta}^{-2\al}\left(\|\sqrt{\rho}\dot u\|_{L^2}^2
                                      +\|\D\theta\|_{L^2}^2\right)
 +\bar{\theta}^{-2}\|\D u\|_{L^6}^2\right].
\end{equation}
After substituting \eqref{fixed:low-lame6}, the density feedback in
\eqref{fixed:low-mass-damping} is bounded by
$C(\bar{\rho},K)\bar{\theta}^{-2}$ times its damping and can be absorbed.
Moreover, the bootstrap bounds and \eqref{fixed:low-entropy} give
\[
 \int_0^T\|\sqrt{\rho}\dot u\|_{L^2}^2\dif t
 \leq C(\bar{\rho},K)(\bar{\theta}^{\al}+\bar{\theta}),\qquad
 \int_0^T\|\D\theta\|_{L^2}^2\dif t
 \leq C(\bar{\rho})\bar{\theta}^{2-\beta}.
\]
Thus the left-hand side of \eqref{fixed:low-mass-bootstrap} is at most
\begin{equation}\label{fixed:low-mass-improvement}
 C(\bar{\rho})+C(\bar{\rho},K)
 \left[\bar{\theta}^{-1}+\bar{\theta}^{-\al}
                         +\bar{\theta}^{1-\al-\beta}\right].
\end{equation}
Its first constant depends only on the initial density quantity.
In particular, it is independent of $M_6$ and of the $K_i$.
Integrating \eqref{fixed:low-lame6} now gives
\begin{align}
 \int_0^{\sigma(T)}\|\D u\|_{L^6}^2\dif t
 &\leq C(\bar{\rho},K)\bar{\theta}^{-\al},
                      \label{fixed:low-velocity-short}\\
 \int_0^T\|\D u\|_{L^6}^2\dif t
 &\leq C(\bar{\rho},K,M_6)\bar{\theta}^{1-\al},
                      \label{fixed:low-velocity-integral}\\
 \sup_{0\leq t\leq T}\sigma^5\|\D u\|_{L^6}^2
 &\leq C(\bar{\rho},K)
 \left[\bar{\theta}^{(2-5\al)/3}+\bar{\theta}^{2-2\al}\right].
                      \label{fixed:low-velocity-weighted}
\end{align}
For the first inequality the density term is bounded by
$C\bar{\theta}^{2-2\al}\leq C\bar{\theta}^{-\al}$ since $\al\geq2$.

\textbf{Step 3. Replacements in the first-order calculations.}
The $I_4$ calculation in \eqref{3.3:7} is replaced by the exact identity
\begin{align*}
 \int\rho\theta u\cdot\D(\rho\theta^{1-\al})\dif x
 ={}&-\frac12\int\rho^2\theta^{2-\al}\div u\dif x
 -\frac{\al}{2}\int\rho^2\theta^{1-\al}u\cdot\D\theta\dif x.
\end{align*}
Subtract $\bar{\rho}^2\bar{\theta}^{2-\al}$ in the first integral.
Using \eqref{fixed:low-entropy} gives
\[
 |I_4|\leq C(\bar{\rho})\left[
 \bar{\theta}^{2-\al}\|\D u\|_{L^2}
 +\bar{\theta}^{3/2-\al}\|\D\theta\|_{L^2}
 +\bar{\theta}^{1-\al}\|\sqrt{\rho}\dot\theta\|_{L^2}\right].
\]
The three initial-layer integrals are bounded, respectively, by
\[
 C(\bar{\rho},K)\bar{\theta}^{(5-3\al)/2},\qquad
 C(\bar{\rho},K)\bar{\theta}^{5/2-\al-\beta/2},\qquad
 C(\bar{\rho},K)\bar{\theta}^{1-\al+\beta/2}.
\]
Each is strictly below the target scale $\bar{\theta}^{\al}$.
The second replacement concerns \eqref{estimate:D2theta:2-1}: retain
$\|\D u\|_{L^4}^4$ rather than introduce $\D^2u$. Indeed,
\[
 \int_0^{\sigma(T)}\|\D u\|_{L^4}^4\dif t
 \leq\sup_{[0,\sigma(T)]}\left(\|\D u\|_{L^2}\|\D u\|_{L^6}\right)
      \int_0^{\sigma(T)}\|\D u\|_{L^6}^2\dif t
 \leq C(\bar{\rho},K)\bar{\theta}^{-\al}.
\]
Squaring and integrating \eqref{fixed:low-heat-elliptic} therefore gives
\begin{equation}\label{fixed:low-heat-short}
 \int_0^{\sigma(T)}\|\D^2\theta\|_{L^2}^2\dif t
 \leq C(\bar{\rho},K)\bar{\theta}^{-\beta}.
\end{equation}
For the full weighted interval, use
\begin{equation*}
 \int_0^T\sigma^4\|\D u\|_{L^2}\|\D u\|_{L^6}^3\dif t \leq\left(\sup\sigma^3\|\D u\|_{L^2}^2\right)^{1/2} \left(\sup\sigma^5\|\D u\|_{L^6}^2\right)^{1/2} \int_0^T\|\D u\|_{L^6}^2\dif t.
\end{equation*}
The two resulting nonlinear powers in the heat elliptic estimate are
$11/6-\al/3-2\beta$ and $5/2-\al/2-2\beta$;
both are at most $2-2\beta$. Hence
\begin{equation}
 \label{fixed:low-heat-integrals} \int_0^T\sigma^4\|\D^2\theta\|_{L^2}^2\dif t \leq C(\bar{\rho},K,M_6)\bar{\theta}^{2-2\beta}, \qquad \int_0^T\|\D^2\theta\|_{L^2}^2\dif t \leq C(\bar{\rho},K,M_6)\bar{\theta}^{-\beta}.
\end{equation}
For the second inequality split at $t=1$ and use $\beta>2$.

\textbf{Step 4. Energy budgets.}
The remaining estimates follow from the energy identities already
established above, using \eqref{fixed:low-lame6} and
\eqref{fixed:low-velocity-short}--\eqref{fixed:low-heat-integrals} wherever
a density-gradient estimate previously entered through $\|\D u\|_{L^6}$.
We give the budgets, including the material-derivative couplings.
From \eqref{3.4.6}, the four integrated remainder terms in $A_3$ are
bounded by
\[
 C(\bar{\rho},K)\left[
 \bar{\theta}^{\beta-\al}+\bar{\theta}^{3-2\al}
 +1+\bar{\theta}^{\al+\beta-2}\right].
\]
The last term is the coupling
$\bar{\theta}^{\al-2}\int\|\D u\|_{L^2}\|\D u\|_{L^6}
 \|\D\dot\theta\|_{L^2}^2\dif t$.
After division by $\bar{\theta}^{2\al}$ it has the strictly negative
power $-(\al-\beta+2)$.
For $A_4$, use the nine terms in \eqref{3.5:5}. The terms with
$\bar{\theta}^{\al-1}|\dot\theta|^2|\D u|^2$ and
$\bar{\theta}^{\beta-1}|\D\dot\theta||\D\theta||\dot\theta|$
are absorbed with relative coefficients
$C(\bar{\rho},K)\bar{\theta}^{\al-\beta-1}$ and
$C(\bar{\rho},K)\bar{\theta}^{-1}$. The other integrated remainders are
bounded by
\[
 C(\bar{\rho},K)\left[
 1+\bar{\theta}^{1+(\al+\beta)/2}
 +\bar{\theta}^{1+(\beta-\al)/2}
 +\bar{\theta}^{\al-\beta}
 +\bar{\theta}^{3\al-\beta}+\bar{\theta}^{-2}\right].
\]
In particular, the $\D\dot u$ coupling contributes
$C(\bar{\rho},K)\bar{\theta}^{3\al-\beta}$,
which divided by $\bar{\theta}^{2\beta}$ is
$C(\bar{\rho},K)\bar{\theta}^{-3(\beta-\al)}$.
The initial material derivatives satisfy
\[
 \|\sqrt{\rho_0}\dot u(0)\|_{L^2}^2\leq C(\bar{\rho})\bar{\theta}^{2\al},
 \qquad
 \|\sqrt{\rho_0}\dot\theta(0)\|_{L^2}^2
              \leq C(\bar{\rho})\bar{\theta}^{2\beta},
\]
directly from \eqref{FCNS:2} and the initial $H^2$ bound.

For $B_1$, \eqref{Lem:4:9} and the new integral estimates give
\begin{equation}
 \label{fixed:low-B1-budget} B_1(T)\leq{}C(\bar{\rho})\bar{\theta} +C(\bar{\rho},K,M_6)\big[ \bar{\theta}^{(3-\al)/2}+\bar{\theta}^{2-\al} +\bar{\theta}^{(\al-\beta+1)/2} +\bar{\theta}^{17/12-2\al/3} +\bar{\theta}^{7/4-3\al/4}\big].
\end{equation}
The same-order term here is solely
$\bar{\theta}^{\al}\int_0^{\sigma(T)}\|\D u\|_{L^2}^2\dif t$
from \eqref{fixed:low-entropy}.
For completeness, the coefficient-derivative term is bounded by
\begin{align*}
 &\bar{\theta}^{\al-1}\int_0^T\sigma^3
       \|\D u\|_{L^2}^{3/2}\|\D u\|_{L^6}^{1/2}
       \|\D\dot\theta\|_{L^2}\dif t\\
 &\quad\leq\bar{\theta}^{\al-1} \left(\sup\|\D u\|_{L^2}^2 \int_0^T\|\D u\|_{L^2}\|\D u\|_{L^6}\dif t\right)^{1/2} \left(\int_0^T\sigma^6\|\D\dot\theta\|_{L^2}^2\dif t\right)^{1/2} \leq C(\bar{\rho},K,M_6)\bar{\theta}^{(\al-\beta+1)/2}.
\end{align*}
For the cubic term retain the weighted $L^6$ bound:
\begin{align*}
 &\sup\sigma^3\|\D u\|_{L^2}^{1/2}\|\D u\|_{L^6}^{1/2} \leq \left(\sup\sigma^3\|\D u\|_{L^2}^2\right)^{1/4} \left(\sup\sigma^5\|\D u\|_{L^6}^2\right)^{1/4},\\
 &\int_0^T\|\D u\|_{L^2}\|\D u\|_{L^6}\dif t
 \leq C(\bar{\rho},K,M_6)\bar{\theta}^{1-\al}.
\end{align*}
These give the last two powers in \eqref{fixed:low-B1-budget}.

For $B_2$, start from \eqref{Lem:5:2}--\eqref{Lem:5:4} before their
final estimates. The heating term is bounded directly by
\begin{align*}
 &\bar{\theta}^{\al}\int_0^T\sigma^4 \|\sqrt{\rho}\dot\theta\|_{L^2} \|\D u\|_{L^2}^{1/2}\|\D u\|_{L^6}^{3/2}\dif t \leq\bar{\theta}^{\al} \left(\int_0^T\sigma^4\|\sqrt{\rho}\dot\theta\|_{L^2}^2\dif t\right)^{1/2} \left(\int_0^T\sigma^4\|\D u\|_{L^2}\|\D u\|_{L^6}^3\dif t\right)^{1/2}\\
 &\quad\leq C(\bar{\rho},K,M_6)
       \left[\bar{\theta}^{23/12-\al/6}
                         +\bar{\theta}^{9/4-\al/4}\right].
\end{align*}
The coefficient-derivative term involving $\dot\theta$ is at most
$C\bar{\theta}^{\beta-1}\int\sigma^4
 \|\sqrt{\rho}\dot\theta\|_{L^2}\|\D^2\theta\|_{L^2}\dif t
 \leq C(\bar{\rho},K,M_6)\bar{\theta}$,
using the uniform $\|\D\theta\|_{L^2},\|\D^2\theta\|_{L^2}$ bounds and
\eqref{fixed:low-heat-integrals}.
For its convection part use
\begin{align*}
 &\bar{\theta}^{\beta-1}\int_0^T\sigma^4
   \|\D u\|_{L^2}^{1/2}\|\D u\|_{L^6}^{1/2}
   \|\D\theta\|_{L^2}^{3/2}\|\D^2\theta\|_{L^2}^{3/2}\dif t\\
 &\quad\leq C(\bar{\rho},K)\bar{\theta}^{\beta-1}
 \left(\int_0^T\|\D\theta\|_{L^2}^2\dif t\right)^{1/2}
 \left(\int_0^T\sigma^8\|\D^2\theta\|_{L^2}^2\dif t\right)^{1/2}
 \leq C(\bar{\rho},K,M_6)\bar{\theta}^{1-\beta/2}.
\end{align*}
The remaining terms in \eqref{Lem:5:4} are controlled by
\eqref{fixed:low-entropy}, the uniform first-order bounds and
\eqref{fixed:low-velocity-integral}. Consequently,
\begin{align}\label{fixed:low-B2-budget}
 B_2(T)\leq{}& C(\bar{\rho})\bar{\theta}^2
 +C(\bar{\rho},K,M_6)\big[
 \bar{\theta}^{1-\al}+\bar{\theta}^{3-\al}
 +\bar{\theta}^{23/12-\al/6}+\bar{\theta}^{9/4-\al/4}\nonumber\\
 &\hspace{33mm}+\bar{\theta}^{(5-\al-\beta)/2}
                  +\bar{\theta}+\bar{\theta}^{1-\beta/2}\big].
\end{align}
The same-order term is the cutoff contribution
$\bar{\theta}^{\beta}\int_0^{\sigma(T)}\|\D\theta\|_{L^2}^2\dif t$,
controlled by \eqref{fixed:low-entropy}. Every other displayed power is
at most $7/4$ when $\al\geq2$ and $\beta>\al$.

The $B_3$ identity \eqref{e3.308} gives, explicitly,
\begin{align}\label{fixed:low-B3-budget}
 \frac{B_3(T)}{\bar{\theta}^{(\al+2)/3}}
 \leq C(\bar{\rho},K,M_6)\big[&
 \bar{\theta}^{-(\al-1)/3}
 +\bar{\theta}^{-4(\al-1)/3}
 +\bar{\theta}^{-7(\al-1)/3}
 +\bar{\theta}^{-7/5}\nonumber\\
 &+\bar{\theta}^{-(2\al+13)/15}
 +\bar{\theta}^{-(10\al-7)/6}
 +\bar{\theta}^{-11(\al-1)/6}\big].
\end{align}
Here the mixed thermal term uses
\[
 \sup\sigma^2\|\D u\|_{L^6}^2
 \leq\left(\sup\|\D u\|_{L^6}^2\right)^{3/5}
     \left(\sup\sigma^5\|\D u\|_{L^6}^2\right)^{2/5},
\]
followed by Cauchy--Schwarz with $\int\|\sqrt{\rho}\dot\theta\|_{L^2}^2$
and $\int\sigma^6\|\D\dot\theta\|_{L^2}^2$.
Thus no initial-layer time weight is discarded.

For $B_4$, use $Q_1,\ldots,Q_8$ in \eqref{e3.3012}.
The cutoff is bounded by $C B_2(T)$ and hence by
$C(\bar{\rho})K_6\bar{\theta}^2$. After absorbing the small multiples
of the thermal dissipation, the remainders in $Q_1,Q_2,Q_3,Q_5,Q_7,Q_8$
are bounded, respectively, by
\begin{align*}
 C(\bar{\rho},K,M_6)\big[&
 \bar{\theta}^{2-2\al-2\beta},\quad
 \bar{\theta}^{(7-\al)/3}+\bar{\theta}^{11/4-3\al/4},\quad
 \bar{\theta}^{2-\beta},\\
 &\bar{\theta}^{\al-\beta+4/5}
            +\bar{\theta}^{14\al/15-\beta+16/15},\quad
 \bar{\theta}^{2-\beta},\quad
 \bar{\theta}^{\beta-\al-1}\big].
\end{align*}
The terms $Q_4,Q_6$ are absorbed with relative coefficients
$C(\bar{\rho},K)\bar{\theta}^{\al-\beta-1}$ and
$C(\bar{\rho},K)\bar{\theta}^{-1}$.
In particular, the potentially delicate $Q_5$ estimate follows from
\begin{align*}
 &\sup\sigma\|\D u\|_{L^2}\|\D u\|_{L^6}
 \leq C(\bar{\rho},K)
       \left[\bar{\theta}^{(2-5\al)/15}
                         +\bar{\theta}^{(2-2\al)/5}\right],\\
 &Q_5\leq\frac{\bar{\theta}^{\beta}}8
            \int_0^T\sigma^6\|\D\dot\theta\|_{L^2}^2\dif t
 +C\bar{\theta}^{2\al-\beta}
    \left(\sup\sigma\|\D u\|_{L^2}\|\D u\|_{L^6}\right)
               \int_0^T\sigma^5\|\D\dot u\|_{L^2}^2\dif t.
\end{align*}
For $Q_8$, the stronger uniform bound
\eqref{fixed:low-heat-uniform} already gives the stated remainder
$C(\bar{\rho},K)\bar{\theta}^{\beta-\al-1}$.
Every non-cutoff remainder divided by $\bar{\theta}^2$ therefore has
a negative power at least $1/3$.

Combining these calculations with the first-order identities
\eqref{3.3:8} and \eqref{3.3:16}, using the two replacements in Step~3,
yields the following complete normalized budget. Each entry consists of
its displayed same-order term plus
$C(\bar{\rho},K,M_6)\bar{\theta}^{-\delta_i}$:
\begin{center}
\renewcommand{\arraystretch}{1.2}
\begin{tabular}{c|c|c}
Quantity & Same-order term & Admissible $\delta_i$\\\hline
$A_1(\sigma(T))/\bar{\theta}^{\al}$
 & $C(\bar{\rho})$ & $\min\{1/2,(\al-\beta+2)/2\}$\\
$A_2(\sigma(T))/\bar{\theta}^{\beta}$
 & $C(\bar{\rho})$ & $\min\{1,(\beta-\al)/2\}$\\
$A_3(\sigma(T))/\bar{\theta}^{2\al}$
 & $C(\bar{\rho})$ & $\min\{2,\al-\beta+2\}$\\
$A_4(\sigma(T))/\bar{\theta}^{2\beta}$
 & $C(\bar{\rho})$ & $\min\{1,3(\beta-\al)\}$\\
$B_1(T)/(\bar{\rho}\bar{\theta})$
 & $C(\bar{\rho})$ & $1/2$\\
$B_2(T)/(\bar{\rho}\bar{\theta}^2)$
 & $C(\bar{\rho})$ & $1/4$\\
$B_3(T)/\bar{\theta}^{(\al+2)/3}$
 & $0$ & $1/3$\\
$B_4(T)/(\bar{\rho}\bar{\theta}^2)$
 & $C(\bar{\rho})K_6$ & $1/3$
\end{tabular}
\end{center}
For $A_1$ the least favorable mixed term is
\[
 \bar{\theta}^{\al-1}\int\|\sqrt{\rho}\dot\theta\|_{L^2}
 \|\D u\|_{L^2}^{1/2}\|\D u\|_{L^6}^{3/2}\dif t,
\]
which is at most $C(\bar{\rho},K)\bar{\theta}^{(\al+\beta)/2-1}$.
For $A_2$ the corresponding heating term has power $(\al+\beta)/2$.
These give precisely the two parameter-dependent exponents in the
first two rows. All same-order $A_i$ terms come from their initial
values; those for $B_1,B_2$ come from the zeroth-order energy.

\textbf{Step 5. Choice of constants.}
First choose $\varepsilon$ for the Lam\'e and heat absorptions.
Next choose $M_6$ larger than four times the initial same-order constant
in \eqref{fixed:low-mass-improvement}. Choose $K_1,\ldots,K_6$
larger than four times their respective same-order constants in the
table, and finally choose $K_7\geq4C(\bar{\rho})K_6$.
All these constants are independent of $\bar{\theta}$ and $T$.
In particular, no same-order mutual domination of $K_3$ and $K_4$
is required. Set
\[
 \delta=\min\left\{\frac14,\frac{\beta-\al}{2},
                              \frac{\al-\beta+2}{2}\right\}>0.
\]
After fixing the constants, increase $\bar{\theta}$ so that
$C(\bar{\rho},K,M_6)\bar{\theta}^{-\delta}$ is smaller than
$\frac14\min\{K_1,\ldots,K_7,M_6,1\}$, all preceding absorptions hold,
and $C(\bar{\rho},K)\leq\varepsilon\bar{\theta}/2$ in
\eqref{fixed:low-heat-uniform}. The table and
\eqref{fixed:low-mass-improvement} give the claimed improvements.
\end{proof}

\subsubsection{The density interval and uniform first-order estimates at fixed far-field density}
\label{fixed:density-section}

Throughout this subsection, $\bar{\rho}>0$ is fixed. The constants
$C(\bar{\rho})$ may depend on this fixed density, the initial-data bound,
the physical parameters, and the lower-order constants already fixed in
Lemma \ref{fixed:low-lemma}, but are independent of $\bar{\theta}$, $T$,
and any density-gradient bound. We retain the definitions of $G$ and $H$
in \eqref{notations} and \eqref{H}. The parameter $\varepsilon$ in the
lower-order estimates is chosen beforehand small enough for the finite
set of Lam\'e estimates used below.

We first record two auxiliary estimates.

\begin{lemma}[A logarithmic estimate with an adjustable remainder]
\label{fixed:density-log}
If $\D v\in L^{12}(\R^3)$ and $\D^2v\in L^4(\R^3)$, then for every
$\varepsilon_0>0$,
\begin{equation}\label{fixed:density-log-estimate}
\begin{aligned}
 \|\D v\|_{L^\infty}
 \leq{}& C\|\D v\|_{L^{12}}
 +C\big(\|\div v\|_{L^\infty}
              +\|\curle v\|_{L^\infty}\big)
       \log\left(\mathrm e+\frac{\|\D^2v\|_{L^4}}{\varepsilon_0}\right)
 +\varepsilon_0.
\end{aligned}
\end{equation}
\end{lemma}
\begin{proof}
Use an inhomogeneous Littlewood--Paley decomposition. Bernstein's
inequality gives
$\|S_0\D v\|_{L^\infty}\leq C\|\D v\|_{L^{12}}$.
For $0\leq j\leq N$, the localized Hodge operators have convolution
kernels with uniformly bounded $L^1$ norms. Consequently,
\[
 \|\Delta_j\D v\|_{L^\infty}
 \leq C\big(\|\div v\|_{L^\infty}
                    +\|\curle v\|_{L^\infty}\big).
\]
For the remaining frequencies, Bernstein's inequality yields
\[
 \sum_{j>N}\|\Delta_j\D v\|_{L^\infty}
 \leq C\sum_{j>N}2^{-j/4}\|\Delta_j\D^2v\|_{L^4}
 \leq C2^{-N/4}\|\D^2v\|_{L^4}.
\]
Choose $N\geq0$ so that the last expression is at most
$\varepsilon_0$. Since
$N+1\leq C\log(\mathrm e+\|\D^2v\|_{L^4}/\varepsilon_0)$,
summing the three parts proves the assertion.
\end{proof}

\begin{lemma}[Uniform comparison with an integrable transient and a small source]
\label{fixed:density-comparison}
Let $S>0$, and let $Y\geq0$ be absolutely continuous on $[0,S]$.
Suppose that $a,F\geq0$, $\eta\geq0$,
$\|a\|_{L^1(0,S)}+\|a\|_{L^{4/3}(0,S)}\leq C_a$, and
\begin{equation}\label{fixed:density-abstract}
 Y'+cY\leq C(a+\eta+F)\log(\mathrm e+Y+F)Y+C(a+\eta+F)
 \quad\hbox{a.e. on }(0,S),
\end{equation}
where $c>0$, $C$, $C_a$, and an upper bound for $Y(0)$ are fixed.
There is a constant $K$, chosen only from these fixed quantities, such
that sufficiently small $\eta$ and $\|F\|_{L^2(0,S)}$ imply
$\sup_{[0,S]}Y<K$. Neither $K$ nor the required smallness depends on $S$.
\end{lemma}
\begin{proof}
Assume first, by continuity, that $Y\leq K$. We have
\[
 \log(\mathrm e+Y+F)\leq\log(\mathrm e+Y)+CF^{1/2}.
\]
Require $C\eta\log(\mathrm e+K)\leq c/8$. The other nontransient
coefficients are bounded by constant multiples of
$F\log(\mathrm e+K)$, $\eta F^{1/2}$, and $F^{3/2}$.
For any prescribed small $\delta>0$, each of these is at most
$\delta+C_KF^2$; here we also require $\eta\leq1$.
The terms $\delta Y$ are absorbed by the damping, while
$F^2Y\leq KF^2$. The transient mixed term satisfies
\[
 \int_0^S aF^{1/2}Y\,\dif\tau
 \leq K\|a\|_{L^{4/3}(0,S)}\|F\|_{L^2(0,S)}^{1/2}.
\]
Using also $CF\leq C+CF^2$, we obtain a nonnegative function $Q_K$
such that
\begin{equation}
 \label{fixed:density-comparison-reduced} Y'+\frac c2Y \leq Ca(1+Y)\log(\mathrm e+Y)+C_0+Q_K, \qquad \int_0^S Q_K\,\dif\tau \leq C_K\big(\|F\|_{L^2(0,S)}^{1/2} +\|F\|_{L^2(0,S)}^2\big),
\end{equation}
where $C_0$ is independent of $K$ and $S$.

Choose
\[
 M\geq\max\left\{\mathrm e+Y(0),\,
                     \mathrm e+\frac{4C_0}{c},\,2\mathrm e\right\},
 \qquad Z=\max\{\mathrm e+Y,M\}.
\]
On the set $\mathrm e+Y\geq M$, the damping absorbs $C_0$.
Taking the upper derivative at the joining points gives
\[
 (\log Z)'\leq Ca\log Z+\frac{Q_K}{M},
 \qquad
 \log Z(\tau)
 \leq\left(\log M+\frac1M\int_0^S Q_K\,\dif s\right)
                          \exp(C\|a\|_{L^1(0,S)}).
\]
We may therefore first choose
\[
 K\geq4+\exp\{2(\log M+1)\exp(CC_a)\},
\]
and only afterwards choose $\eta$ and $\|F\|_{L^2(0,S)}$ small enough
for the preceding absorptions and $\int_0^S Q_K\,\dif\tau\leq M$.
It follows that $\mathrm e+Y\leq\sqrt K$, a strict improvement of
$Y\leq K$. Continuity completes the proof. In particular, the choice
of $K$ precedes the small-source choice.
\end{proof}

\begin{lemma}[Density estimates for every fixed positive $\bar{\rho}$]
\label{fixed:density-prop}
Assume the initial bounds of the fixed-density theorem and the
lower-order estimates in Lemma \ref{fixed:low-lemma} on $[0,T]$.
For sufficiently large $\bar{\theta}$, with a threshold independent of
$T$, one has
\begin{equation}\label{fixed:density-conclusions}
 \frac23\bar{\rho}\leq\rho\leq\frac43\bar{\rho},\qquad
 C_1(T)+C_2(T)\leq C(\bar{\rho}),
\end{equation}
where $C_1,C_2$ retain their definitions in \eqref{A6A7}, and
\begin{equation}\label{fixed:density-velocity-integrals}
 \int_0^T\big(\|\D^2u\|_{L^4}^2+\|\D u\|_{L^\infty}^2\big)\,\dif t
 \leq C(\bar{\rho})\bar{\theta}^{1-\al}.
\end{equation}
\end{lemma}
\begin{proof}
The lower-order inputs used here are precisely
\begin{equation}
 \begin{aligned}
 \label{fixed:density-inputs}
 &\frac12\bar{\rho}\leq\rho\leq\frac32\bar{\rho},\qquad
   \|\theta-\bar{\theta}\|_{L^\infty}\leq\varepsilon\bar{\theta},\qquad
   \sup_{[0,T]}\|\rho-\bar{\rho}\|_{L^2}\leq C(\bar{\rho}),\\
 &\sup_{[0,T]}\big(\|\D u\|_{L^2}+\|\D\theta\|_{L^2} +\|\D^2\theta\|_{L^2}\big) \leq C(\bar{\rho}), \qquad \int_0^T\big(\|\sqrt{\rho}\dot u\|_{L^2}^2 +\|\D\dot u\|_{L^2}^2\big)\,\dif t \leq C(\bar{\rho})\bar{\theta}^{\al},\\
 &\int_0^T\|\D u\|_{L^2}^2\,\dif t
                                    \leq C(\bar{\rho})\bar{\theta}^{1-\al},
 \qquad
   \int_0^T\|\D\theta\|_{L^2}^2\,\dif t
                                    \leq C(\bar{\rho})\bar{\theta}^{2-\beta},\\
 &\int_0^T\|\D^2\theta\|_{L^2}^2\,\dif t
                                    \leq C(\bar{\rho})\bar{\theta}^{-\beta},
 \qquad
   \bar{\theta}^{1-\al}\int_0^T\|\rho-\bar{\rho}\|_{L^6}^2\,\dif t
                                    \leq C(\bar{\rho}).
\end{aligned}
\end{equation}
The last bound follows from the mass-weighted $L^6$ density damping in
Lemma \ref{fixed:low-lemma}; no density-gradient estimate is used here.

\paragraph{Constant-coefficient Lam\'e estimates.}
The momentum equation can be written as
\[
 \bar{\theta}^{\al}\{\mu\Delta u+(\mu+\lambda)\D\div u\}
 =\rho\dot u+R\D(\rho\theta)
  -\div\big[(\theta^{\al}-\bar{\theta}^{\al})
                      (2\mu\frD(u)+\lambda\div u\,I)\big].
\]
For $p=6,12$, apply the first-order potential estimate
$L^{3p/(3+p)}\to L^p$ and the $L^p$ boundedness of the zero-order
operators. Since
$\|\theta^{\al}/\bar{\theta}^{\al}-1\|_{L^\infty}\leq C\varepsilon$,
the last term is absorbed by the beforehand choice of $\varepsilon$.
Thus
\begin{equation}
 \begin{aligned}
 \label{fixed:density-lame}
 \|\D u\|_{L^6}
 \leq C(\bar{\rho})\big[&
 \bar{\theta}^{-\al}\|\sqrt{\rho}\dot u\|_{L^2}
 +\bar{\theta}^{1-\al}\|\rho-\bar{\rho}\|_{L^6}
 +\bar{\theta}^{-\al}\|\D\theta\|_{L^2}\big],\\
 \|\D u\|_{L^{12}} \leq C(\bar{\rho})\big[& \bar{\theta}^{-\al}\|\sqrt{\rho}\dot u\|_{L^2}^{3/4} \|\D\dot u\|_{L^2}^{1/4} +\bar{\theta}^{1-\al}\|\rho-\bar{\rho}\|_{L^{12}} +\bar{\theta}^{-\al}\|\D\theta\|_{L^2}^{3/4} \|\D^2\theta\|_{L^2}^{1/4}\big].
\end{aligned}
\end{equation}
Here
$\|\dot u\|_{L^{12/5}}\leq
C(\bar{\rho})\|\sqrt{\rho}\dot u\|_{L^2}^{3/4}
\|\D\dot u\|_{L^2}^{1/4}$ and
$\|\theta-\bar{\theta}\|_{L^{12}}\leq
C\|\D\theta\|_{L^2}^{3/4}\|\D^2\theta\|_{L^2}^{1/4}$.

\paragraph{The effective flux and the density barrier.}
For each of the three sources in $H$, use
\[
 \|(-\Delta)^{-1}\div f\|_{L^\infty}
 \leq C\|f\|_{L^2}^{1/3}\|f\|_{L^4}^{2/3}.
\]
For the stress source, the product estimates are
\[
 \|\D\theta\D u\|_{L^2}
 \leq C\|\D\theta\|_{L^2}^{3/4}\|\D^2\theta\|_{L^2}^{1/4}
                  \|\D u\|_{L^{12}},\qquad
 \|\D\theta\D u\|_{L^4}
 \leq C\|\D^2\theta\|_{L^2}\|\D u\|_{L^{12}}.
\]
Combining \eqref{elliptic},
$\mu\Delta\curle u=-\curle H$, and \eqref{fixed:density-lame},
we obtain
\begin{equation}\label{fixed:density-flux}
\begin{aligned}
 \|G\|_{L^\infty}+\|\curle u\|_{L^\infty}
& \leq C(\bar{\rho})\big[
 \bar{\theta}^{-\al}\|\sqrt{\rho}\dot u\|_{L^2}^{1/2}
                                   \|\D\dot u\|_{L^2}^{1/2}\\
 &\quad+\bar{\theta}^{-\al}\|\D\theta\|_{L^2}^{1/2}
                                   \|\D^2\theta\|_{L^2}^{1/2}
 +\bar{\theta}^{-\al-1}\|\D\theta\|_{L^2}^{1/4}
       \|\D^2\theta\|_{L^2}^{3/4}
       \|\sqrt{\rho}\dot u\|_{L^2}^{3/4}\|\D\dot u\|_{L^2}^{1/4}\\
 &\quad+\bar{\theta}^{-\al}\|\D\theta\|_{L^2}^{1/4}
       \|\D^2\theta\|_{L^2}^{3/4}\|\rho-\bar{\rho}\|_{L^{12}}
 +\bar{\theta}^{-\al-1}\|\D\theta\|_{L^2}\|\D^2\theta\|_{L^2}
 \big].
\end{aligned}
\end{equation}
Applying the linear potential operators separately to the three
sources avoids any cross terms from interpolation of a sum.

Set
\begin{equation}\label{fixed:density-slow-time}
 \tau=\bar{\theta}^{1-\al}t,\qquad
 S=\bar{\theta}^{1-\al}T.
\end{equation}
All physical-time quantities in expressions involving $\tau$ are
evaluated at $t=\bar{\theta}^{\al-1}\tau$. Define the nonnegative source
\begin{equation}
 \label{fixed:density-F} F(\tau)=\bar{\theta}^{-1}\big( \|\sqrt{\rho}\dot u\|_{L^2}^{1/4}\|\D\dot u\|_{L^2}^{3/4} +\|\sqrt{\rho}\dot u\|_{L^2}^{1/2}\|\D\dot u\|_{L^2}^{1/2} +\|\sqrt{\rho}\dot u\|_{L^2}^{3/4}\|\D\dot u\|_{L^2}^{1/4}\big).
\end{equation}
H\"older's inequality and \eqref{fixed:density-inputs} give
\begin{equation}\label{fixed:density-F-small}
 \int_0^S F(\tau)^2\,\dif\tau
 \leq C(\bar{\rho})\bar{\theta}^{-2}
                       \bar{\theta}^{1-\al}\bar{\theta}^{\al}
 =C(\bar{\rho})\bar{\theta}^{-1}.
\end{equation}
The nonacceleration sources in \eqref{fixed:density-flux} have
small $L^2$ norms in the same time scale. For example,
\[
 \left\|\bar{\theta}^{-1}\|\D\theta\|_{L^2}^{1/2}
                      \|\D^2\theta\|_{L^2}^{1/2}\right\|_{L^2(0,S)}
 \leq C(\bar{\rho})\bar{\theta}^{-(\al+\beta)/2}.
\]
For the pressure part of the stress source the corresponding bound is
$C(\bar{\rho})\bar{\theta}^{-1/4-(\al+\beta)/2}$; this only uses
$\sup_{[0,T]}\|\rho-\bar{\rho}\|_{L^{12}}\leq C(\bar{\rho})$.
The other stress terms are smaller. Moreover,
$\|\theta-\bar{\theta}\|_{L^\infty}^2
\leq C\|\D\theta\|_{L^2}\|\D^2\theta\|_{L^2}$.
Consequently,
\begin{equation}
 \label{fixed:density-flux-L2} \|\bar{\theta}^{\al-1}G\|_{L^2(0,S;L^\infty)} +\|\bar{\theta}^{\al-1}\curle u\|_{L^2(0,S;L^\infty)} +\|\bar{\theta}^{-1}(\theta-\bar{\theta})\|_{L^2(0,S;L^\infty)} \leq C(\bar{\rho})\bar{\theta}^{-1/2}.
\end{equation}

The continuity equation and \eqref{notations} imply the exact identity
\begin{equation}
 \label{fixed:density-rho-equation} (\partial_t+u\cdot\D)(\rho-\bar{\rho}) +\frac{R}{2\mu+\lambda}\rho\theta^{1-\al}(\rho-\bar{\rho}) =-\frac{\rho}{2\mu+\lambda} \big[G+R\bar{\rho}(\theta^{1-\al}-\bar{\theta}^{1-\al})\big].
\end{equation}
In the time variable $\tau$, its damping has a fixed positive lower
bound $c(\bar{\rho})$. Also,
$\bar{\theta}^{\al-1}|\theta^{1-\al}-\bar{\theta}^{1-\al}|
\leq C\bar{\theta}^{-1}|\theta-\bar{\theta}|$.
Apply the integrating factor along particle trajectories. The $L^2$
norm of the exponential kernel, together with
\eqref{fixed:density-flux-L2}, gives the pointwise bound below; its
$L^1$ norm gives the time-integrated bound:
\begin{equation}
 \label{fixed:density-rho-infty} \|\rho(t)-\bar{\rho}\|_{L^\infty} \leq\|\rho_0-\bar{\rho}\|_{L^\infty}e^{-c(\bar{\rho})\tau} +C(\bar{\rho})\bar{\theta}^{-1/2}, \qquad \|\rho-\bar{\rho}\|_{L^2(0,S;L^\infty)}\leq C(\bar{\rho}).
\end{equation}
Choose $\bar{\theta}$ so that the error in the first line is at most
$\bar{\rho}/24$. Since
$\|\rho_0-\bar{\rho}\|_{L^\infty}\leq\bar{\rho}/4$, this improves the
coarse interval to
$\|\rho-\bar{\rho}\|_{L^\infty}/\bar{\rho}\leq7/24<1/3$.

We shall use two different interpolations of the density perturbation:
\begin{equation}\label{fixed:density-rho-twelve}
\begin{aligned}
 \|\rho-\bar{\rho}\|_{L^{12}}
 &\leq C(\bar{\rho})
            \big(e^{-c(\bar{\rho})\tau}+\bar{\theta}^{-5/12}\big),\\
 \int_0^S\|\rho-\bar{\rho}\|_{L^{12}}^2\,\dif\tau
 &\leq\|\rho-\bar{\rho}\|_{L^2(0,S;L^6)}
                 \|\rho-\bar{\rho}\|_{L^2(0,S;L^\infty)}
 \leq C(\bar{\rho}).
\end{aligned}
\end{equation}
The first line uses $L^2$ and $L^\infty$ interpolation; the second uses
$L^6$ and $L^\infty$ interpolation and the last bound in
\eqref{fixed:density-inputs}.

\paragraph{Uniform control of the density gradient.}
Let $Y(\tau)=\|\D\rho(t)\|_{L^4}$. The nondivergence-form Lam\'e
estimate, after absorbing the small highest-order coefficient, gives
\begin{equation}
 \label{fixed:density-hessian} \|\D^2u\|_{L^4}\leq C(\bar{\rho})\big[ \bar{\theta}^{-\al}\|\sqrt{\rho}\dot u\|_{L^2}^{1/4} \|\D\dot u\|_{L^2}^{3/4} +\bar{\theta}^{1-\al}Y +\bar{\theta}^{-\al}\|\D\theta\|_{L^4} +\bar{\theta}^{-1}\|\D^2\theta\|_{L^2}\|\D u\|_{L^{12}}\big].
\end{equation}
Here the last term follows from
$\|\D\theta\|_{L^6}\leq C\|\D^2\theta\|_{L^2}$.
Using \eqref{fixed:density-lame}, \eqref{fixed:density-flux},
\eqref{fixed:density-rho-infty}, and \eqref{fixed:density-rho-twelve},
we obtain, with $\eta=C(\bar{\rho})\bar{\theta}^{-5/12}$,
\begin{equation}\label{fixed:density-normalized}
\begin{aligned}
 \bar{\theta}^{\al-1}\|\D^2u\|_{L^4}
 &\leq C(\bar{\rho})(Y+1+F),\\
 \bar{\theta}^{\al-1}
 \big(\|\D u\|_{L^{12}}+\|\div u\|_{L^\infty}
                                +\|\curle u\|_{L^\infty}\big)
 &\leq C(\bar{\rho})(e^{-c(\bar{\rho})\tau}+\eta+F),\\
 \bar{\theta}^{\al-1}\|\D G\|_{L^4}
                      +\bar{\theta}^{-1}\|\D\theta\|_{L^4}
 &\leq C(\bar{\rho})(e^{-c(\bar{\rho})\tau}+\eta+F).
\end{aligned}
\end{equation}
In the last line we used the $L^4$ estimate for $H$ in
\eqref{elliptic}. The pressure term in $\div u$ is controlled directly
by \eqref{notations} and \eqref{fixed:density-rho-infty}.

For completeness, differentiating the continuity equation gives
\begin{equation}
 \label{fixed:density-gradient-equation} (\partial_t+u\cdot\D)\D\rho +\big((\D u)^{\mathsf T}+\div u\,I\big)\D\rho +\frac{R}{2\mu+\lambda}\rho\theta^{1-\al}\D\rho =-\frac{\rho}{2\mu+\lambda}\D G +\frac{R(\al-1)}{2\mu+\lambda}\rho^2\theta^{-\al}\D\theta.
\end{equation}
Taking its $L^4$ norm and changing time yields
\[
 Y'+cY\leq C(\bar{\rho})\bar{\theta}^{\al-1}
                                      \|\D u\|_{L^\infty}Y
       +C(\bar{\rho})\bar{\theta}^{\al-1}\|\D G\|_{L^4}
       +C(\bar{\rho})\bar{\theta}^{-1}\|\D\theta\|_{L^4}.
\]
Apply Lemma \ref{fixed:density-log} to
$v=\bar{\theta}^{\al-1}u$, choosing $\varepsilon_0$ to be a fixed small
fraction of the damping constant. By
\eqref{fixed:density-normalized}, the resulting inequality has the
form \eqref{fixed:density-abstract}, with
$a(\tau)=C(\bar{\rho})e^{-c(\bar{\rho})\tau}$.
Its $L^1$ and $L^{4/3}$ norms are fixed independently of
$\bar{\theta}$ and $T$. Lemma \ref{fixed:density-comparison}, followed
by \eqref{fixed:density-F-small}, therefore gives
\begin{equation}\label{fixed:density-gradient-four}
 \sup_{0\leq t\leq T}\|\D\rho(t)\|_{L^4}\leq C(\bar{\rho}).
\end{equation}
In this application the comparison bound is fixed before the final
choice of $\bar{\theta}$.

\paragraph{The damping integrals and the true source terms.}
To obtain estimates uniform in the time endpoint, we now return to
the source terms before the pointwise replacement by $\eta$.
Equations \eqref{fixed:density-lame} and
\eqref{fixed:density-rho-twelve} give
\begin{equation}\label{fixed:density-velocity-twelve-L2}
 \|\bar{\theta}^{\al-1}\D u\|_{L^2(0,S;L^{12})}
 \leq C(\bar{\rho}).
\end{equation}
Estimating $H$ in $L^2$ and $L^4$ separately yields
\begin{equation}\label{fixed:density-true-source-L2}
\begin{aligned}
 \sum_{p\in\{2,4\}}\big(&
 \|\bar{\theta}^{\al-1}\D G\|_{L^2(0,S;L^p)}
 +\|\bar{\theta}^{-1}\D\theta\|_{L^2(0,S;L^p)}\big)
 \leq C(\bar{\rho})\bar{\theta}^{-1/2}.
\end{aligned}
\end{equation}
Indeed, the normalized $L^4$ stress term is
\[
 \bar{\theta}^{-1}\|\D^2\theta\|_{L^2}
             \big(\bar{\theta}^{\al-1}\|\D u\|_{L^{12}}\big),
\]
whose first factor has supremum at most
$C(\bar{\rho})\bar{\theta}^{-1}$. For the $L^2$ stress term use
$\|\D\theta\D u\|_{L^2}
\leq\|\D\theta\|_{L^3}\|\D u\|_{L^6}$ and
\eqref{fixed:density-lame}; the time integral of its density term is
provided by the mass-weighted $L^6$ damping. The acceleration terms
in \eqref{fixed:density-true-source-L2} are bounded by
\eqref{fixed:density-F-small} and
$\bar{\theta}^{-2}\bar{\theta}^{1-\al}
\int_0^T\|\sqrt{\rho}\dot u\|_{L^2}^2\,\dif t
\leq C(\bar{\rho})\bar{\theta}^{-1}$; the temperature terms are smaller.

Once \eqref{fixed:density-gradient-four} is fixed, the same Young
inequalities used in Lemma \ref{fixed:density-comparison} show that,
for every prescribed small $\delta_0>0$ and sufficiently large
$\bar{\theta}$,
\begin{equation}\label{fixed:density-velocity-small-plus-L1}
 \bar{\theta}^{\al-1}\|\D u\|_{L^\infty}
 \leq\delta_0+q(\tau),\qquad
 q\geq0,\qquad \int_0^S q(\tau)\,\dif\tau\leq C(\bar{\rho}).
\end{equation}
Here one chooses the remainder in Lemma \ref{fixed:density-log}
small relative to $\delta_0$. The terms containing $\eta$ and the
constant parts of the Young inequalities enter $\delta_0$; the terms
$e^{-c(\bar{\rho})\tau}$, $F^2$, and
$e^{-c(\bar{\rho})\tau}F^{1/2}$ enter $q$. Thus no positive constant
is integrated over the full time interval.

For $p=2,4$, multiply \eqref{fixed:density-gradient-equation} by
$p|\D\rho|^{p-2}\D\rho$, integrate, and then pass to the squared
$L^p$ norm. The transport term has coefficient $(p-1)\div u$,
and all deformation terms are bounded by
$C\|\D u\|_{L^\infty}\|\D\rho\|_{L^p}^2$.
Consequently, after changing time,
\begin{equation}\label{fixed:density-square-energy}
\begin{aligned}
 \frac{\dif}{\dif\tau}\|\D\rho\|_{L^p}^2
 +c\|\D\rho\|_{L^p}^2
 &\leq Cq(\tau)\|\D\rho\|_{L^p}^2+f_p(\tau),\\
 f_p(\tau)
 &\leq C(\bar{\rho})\big(
      \bar{\theta}^{2\al-2}\|\D G\|_{L^p}^2
                   +\bar{\theta}^{-2}\|\D\theta\|_{L^p}^2\big),
 \qquad \int_0^S f_p\,\dif\tau\leq C(\bar{\rho}).
\end{aligned}
\end{equation}
We chose $\delta_0$ small enough to absorb its contribution in the
damping and used \eqref{fixed:density-true-source-L2} for the source.
Gronwall's inequality, followed by integration of
\eqref{fixed:density-square-energy}, gives
\begin{equation}\label{fixed:density-gradient-damping}
 \sum_{p\in\{2,4\}}\left(
 \sup_{0\leq t\leq T}\|\D\rho(t)\|_{L^p}^2
 +\bar{\theta}^{1-\al}\int_0^T\|\D\rho\|_{L^p}^2\,\dif t\right)
 \leq C(\bar{\rho}).
\end{equation}
Since $\bar{\rho}$ is fixed, this proves the bound for the original
$C_1(T),C_2(T)$ in \eqref{fixed:density-conclusions}.

Finally, integrate the unshortened estimate
\eqref{fixed:density-hessian}. Its acceleration term contributes at
most $C(\bar{\rho})\bar{\theta}^{-\al}$ after squaring and integrating;
its density-gradient term contributes at most
$C(\bar{\rho})\bar{\theta}^{2-2\al}\bar{\theta}^{\al-1}$ by
\eqref{fixed:density-gradient-damping}; and its stress term is bounded
by \eqref{fixed:density-velocity-twelve-L2} and the uniform bound for
$\|\D^2\theta\|_{L^2}$. The temperature term is smaller. Thus
\begin{equation}\label{fixed:density-hessian-L2}
 \int_0^T\|\D^2u\|_{L^4}^2\,\dif t
 \leq C(\bar{\rho})\bar{\theta}^{1-\al}.
\end{equation}
The Gagliardo--Nirenberg inequality and H\"older's inequality in time
then give
\begin{equation}
 \label{fixed:density-velocity-infty-square} \int_0^T\|\D u\|_{L^\infty}^2\,\dif t \leq C\left(\int_0^T\|\D u\|_{L^2}^2\,\dif t\right)^{1/7} \left(\int_0^T\|\D^2u\|_{L^4}^2\,\dif t\right)^{6/7} \leq C(\bar{\rho})\bar{\theta}^{1-\al}.
\end{equation}
This proves \eqref{fixed:density-velocity-integrals}; in particular,
$\bar{\theta}^{\al-1}\int_0^T\|\D u\|_{L^\infty}^2\,\dif t$
is bounded by a fixed constant.

The order of choices is as follows. First fix $\bar{\rho}$, the
initial-data bound, the physical parameters, the lower-order constants,
and the Lam\'e absorption parameter $\varepsilon$. Next fix the
comparison constant $K$ from the initial gradient and the integrable
exponential transient in Lemma \ref{fixed:density-comparison}.
The constants in \eqref{fixed:density-gradient-damping} and hence the
targets for the original $C_1,C_2$ can then be fixed. Finally increase
$\bar{\theta}$ so that the lower-order improvements, the density
barrier, the absorption of $\eta\log(\mathrm e+K)$, and the small-source
budget for $F$ all hold. All thresholds are independent of $T$, and
no lower bound on the size of the fixed positive $\bar{\rho}$ is needed.
\end{proof}

\subsubsection{Higher-order estimates and continuation}

It remains to recover all the norms in \eqref{sp}. In this subsection,
$\bar{\rho}>0$ and $\bar{\theta}\geq C^*$ are fixed. We use
$C(\bar{\rho},\bar{\theta})$ to denote a generic positive constant which may
depend on these two constants, the physical parameters and $C_0$, but is
independent of $T$.

\begin{lemma}\label{fixed:high-lemma}
Under the assumptions of Theorem \ref{th:fixed}, let
$(\rho,u,\theta)$ be a local strong solution on $[0,T]$ for which the
low-order and density estimates established above hold. Then
\begin{align}\label{fixed:high-conclusion}
&\sup_{t\in[0,T]}\left(
 \|\rho-\bar{\rho}\|_{H^2}^2+\|u\|_{H^2}^2
 +\|\theta-\bar{\theta}\|_{H^2}^2+\|\rho_t\|_{H^1}^2
 +\|u_t\|_{L^2}^2+\|\theta_t\|_{L^2}^2\right)\nonumber\\
&\quad+\int_0^T\left(
 \|\D^2\rho\|_{L^2}^2+\|\D^2u\|_{L^2}^2
 +\|u_t\|_{H^1}^2+\|\theta_t\|_{H^1}^2\right)\dif t
 \leq C(\bar{\rho},\bar{\theta}).
\end{align}
Moreover, $(\rho-\bar{\rho},u,\theta-\bar{\theta})$ is strongly continuous
in $H^2$, including the endpoint of any finite maximal existence interval.
\end{lemma}

\begin{proof}
The preceding estimates give
\begin{align}
&\sup_{t\in[0,T]}\big(
 \|\rho-\bar{\rho}\|_{L^2}+\|u\|_{L^2}
 +\|\theta-\bar{\theta}\|_{L^2}+\|\D u\|_{L^2}
 +\|\D\theta\|_{H^1}+\|\dot{u}\|_{L^2}
 +\|\dot{\theta}\|_{L^2}\nonumber\\
 &\quad+\|\D\rho\|_{L^2}+\|\D\rho\|_{L^4}\big)
 \leq C(\bar{\rho},\bar{\theta}),\label{fixed:high-known-sup}\\
&\int_0^T\big(
 \|\D u\|_{L^2}^2+\|\D\theta\|_{H^1}^2
 +\|\dot{u}\|_{H^1}^2+\|\dot{\theta}\|_{H^1}^2
 +\|\rho-\bar{\rho}\|_{L^6}^2+\|\D\rho\|_{L^2}^2\nonumber\\
 &\quad+\|\D\rho\|_{L^4}^2+\|\D^2u\|_{L^4}^2\big)\dif t
 \leq C(\bar{\rho},\bar{\theta}).\label{fixed:high-known-int}
\end{align}
Here the slow-time estimates have been written in the original time
variable; this only changes their bounds by a fixed positive factor.
All estimates below are obtained on the given existence interval.

\textbf{Step 1}: Estimate of $\|\D^2\rho\|_{L^2}$.

The constant-coefficient Lam\'e perturbation estimate gives
\begin{equation}\label{fixed:high-Du6}
 \|\D u\|_{L^6}\leq C(\bar{\rho},\bar{\theta})
 \left(\|\dot{u}\|_{L^2}+\|\rho-\bar{\rho}\|_{L^6}
                         +\|\D\theta\|_{L^2}\right).
\end{equation}
Thus $\D u$ is bounded in both $L^\infty(0,T;L^6)$ and
$L^2(0,T;L^6)$, independently of $T$. Furthermore, the
Gagliardo--Nirenberg and H\"older inequalities yield
\begin{align}\label{fixed:high-Du-infty}
 \int_0^T\|\D u\|_{L^\infty}^2\dif t
 &\leq C\left(\int_0^T\|\D u\|_{L^2}^2\dif t\right)^{\frac17}
          \left(\int_0^T\|\D^2u\|_{L^4}^2\dif t\right)^{\frac67}
 \leq C(\bar{\rho},\bar{\theta}).
\end{align}

Differentiating $H$ defined in \eqref{H}, and using the positive lower
and upper bounds of the density and temperature, we obtain
\begin{align}\label{fixed:high-DH}
 \|\D H\|_{L^2}\leq C(\bar{\rho},\bar{\theta})\big(&
 \|\D\dot{u}\|_{L^2}
 +\|\D\rho\|_{L^3}\|\dot{u}\|_{L^6}
 +\|\D\theta\|_{L^3}\|\dot{u}\|_{L^6}
 +\|\D\rho\|_{L^3}\|\D\theta\|_{L^6}
 +\|\D^2\theta\|_{L^2}\nonumber\\
 &+\|\D\theta\|_{L^4}^2
 +\|\D^2\theta\|_{L^2}\|\D u\|_{L^\infty}
 +\|\D\theta\|_{L^4}\|\D^2u\|_{L^4}
 +\|\D\theta\|_{L^6}^2\|\D u\|_{L^6}\big).
\end{align}
Each term on the right-hand side belongs to $L^2(0,T)$ with a bound
independent of $T$. Indeed, $\|\dot{u}\|_{L^6}\leq C\|\D\dot{u}\|_{L^2}$,
while $\D\rho$ and $\D\theta$ are bounded uniformly in $L^3$. For the
quadratic temperature term, one has
\begin{align}\label{fixed:high-theta-four}
 \int_0^T\|\D\theta\|_{L^4}^4\dif t
 &\leq C\sup_{t\in[0,T]}\|\D^2\theta\|_{L^2}^2
       \int_0^T\|\D\theta\|_{L^2}\|\D^2\theta\|_{L^2}\dif t
 \leq C(\bar{\rho},\bar{\theta}).
\end{align}
The last three terms in \eqref{fixed:high-DH} are controlled by
\eqref{fixed:high-known-sup}, \eqref{fixed:high-known-int},
\eqref{fixed:high-Du6}, and \eqref{fixed:high-Du-infty}. Consequently,
\begin{equation}\label{fixed:high-DH-int}
 \int_0^T\|\D H\|_{L^2}^2\dif t
 \leq C(\bar{\rho},\bar{\theta}).
\end{equation}

Differentiating the mass equation twice and using the definition of $G$
in the term $\rho\D^2\operatorname{div}u$, as in \eqref{sec3.4:2}, we derive
\begin{align}
 \label{fixed:high-rho-before} &\frac12\frac{\dif}{\dif t}\|\D^2\rho\|_{L^2}^2 +c_*\|\D^2\rho\|_{L^2}^2 \leq C(\bar{\rho},\bar{\theta})\|\D u\|_{L^\infty} \|\D^2\rho\|_{L^2}^2\nonumber\\
 &\quad+C(\bar{\rho},\bar{\theta})\big(
 \|\D^2u\|_{L^4}\|\D\rho\|_{L^4}+\|\D H\|_{L^2}
 +\|\D^2\theta\|_{L^2}+\|\D\rho\|_{L^3}\|\D\theta\|_{L^6}
 +\|\D\theta\|_{L^4}^2\big)\|\D^2\rho\|_{L^2},
\end{align}
where $c_*=c_*(\bar{\rho},\bar{\theta})>0$ is independent of $T$, and
we have used $\|\D^2G\|_{L^2}\leq C\|\D H\|_{L^2}$. In particular,
Young's inequality gives
\[
 C(\bar{\rho},\bar{\theta})\|\D u\|_{L^\infty}
                           \|\D^2\rho\|_{L^2}^2
 \leq\frac{c_*}{4}\|\D^2\rho\|_{L^2}^2
   +C(\bar{\rho},\bar{\theta})\|\D u\|_{L^\infty}^2
                                      \|\D^2\rho\|_{L^2}^2.
\]
Applying Young's inequality to the remaining terms and retaining part
of the density dissipation, we obtain
\begin{align}
 \label{fixed:high-rho-energy} &\frac{\dif}{\dif t}\|\D^2\rho\|_{L^2}^2 +c_*\|\D^2\rho\|_{L^2}^2 \leq C(\bar{\rho},\bar{\theta})\|\D u\|_{L^\infty}^2 \|\D^2\rho\|_{L^2}^2\nonumber\\
 &\quad+C(\bar{\rho},\bar{\theta})\left(
 \|\D^2u\|_{L^4}^2+\|\D H\|_{L^2}^2+\|\D^2\theta\|_{L^2}^2
 +\|\D\rho\|_{L^3}^2+\|\D\theta\|_{L^4}^4\right),
\end{align}
after decreasing $c_*$ if necessary. Notice that interpolation between
$L^2$ and $L^4$ also gives $\D\rho\in L^2(0,T;L^3)$. All the
coefficients and source terms in \eqref{fixed:high-rho-energy} are
integrable with bounds independent of $T$. It follows from Gronwall's
inequality that
\begin{equation}\label{fixed:high-rhoH2}
 \sup_{t\in[0,T]}\|\D^2\rho\|_{L^2}^2
       +\int_0^T\|\D^2\rho\|_{L^2}^2\dif t
 \leq C(\bar{\rho},\bar{\theta}).
\end{equation}

\textbf{Step 2}: Estimates of $\|u\|_{H^2}$ and the time derivatives.

The elliptic estimate for the momentum equation, as in \eqref{346-0},
gives
\begin{equation}\label{fixed:high-u-elliptic}
 \|\D^2u\|_{L^2}\leq C(\bar{\rho},\bar{\theta})\left(
 \|\dot{u}\|_{L^2}+\|\D\rho\|_{L^2}+\|\D\theta\|_{L^2}
                +\|\D\theta\|_{L^4}\|\D u\|_{L^4}\right).
\end{equation}
Here $\|\D u\|_{L^4}$ is bounded in both $L^\infty(0,T)$ and
$L^2(0,T)$ by interpolation between $L^2$ and $L^6$. Hence
\begin{equation}\label{fixed:high-uH2}
 \sup_{t\in[0,T]}\|u\|_{H^2}^2
       +\int_0^T\|\D^2u\|_{L^2}^2\dif t
 \leq C(\bar{\rho},\bar{\theta}).
\end{equation}
Using $\rho_t=-u\cdot\D\rho-\rho\operatorname{div}u$, we obtain
\begin{align}\label{fixed:high-rhot}
 \|\rho_t\|_{H^1}
 &\leq C(\bar{\rho},\bar{\theta})\left(
 \|u\|_{L^\infty}\|\D\rho\|_{H^1}+\|\D u\|_{H^1}
                    +\|\D u\|_{L^4}\|\D\rho\|_{L^4}\right)
 \leq C(\bar{\rho},\bar{\theta}).
\end{align}

Since $u_t=\dot{u}-u\cdot\D u$ and
$\theta_t=\dot{\theta}-u\cdot\D\theta$, their $L^\infty(0,T;L^2)$
bounds follow directly from \eqref{fixed:high-known-sup} and
\eqref{fixed:high-uH2}. To estimate their first-order derivatives, we use
\begin{align}\label{fixed:high-convection}
 \int_0^T\|u\cdot\D u\|_{H^1}^2\dif t
 &\leq C\sup_{t\in[0,T]}\|u\|_{L^\infty}^2
       \int_0^T\left(\|\D u\|_{L^2}^2+\|\D^2u\|_{L^2}^2\right)\dif t
       +C\int_0^T\|\D u\|_{L^4}^4\dif t,\nonumber\\
 \int_0^T\|u\cdot\D\theta\|_{H^1}^2\dif t
 &\leq C\sup_{t\in[0,T]}\|u\|_{L^\infty}^2
       \int_0^T\|\D\theta\|_{H^1}^2\dif t
       +C\sup_{t\in[0,T]}\|\D\theta\|_{L^3}^2
          \int_0^T\|\D u\|_{L^6}^2\dif t.
\end{align}
In addition,
\[
 \int_0^T\|\D u\|_{L^4}^4\dif t
 \leq C\sup_{t\in[0,T]}\|\D u\|_{L^6}^2
       \int_0^T\|\D u\|_{L^2}\|\D u\|_{L^6}\dif t
 \leq C(\bar{\rho},\bar{\theta}).
\]
Combining these estimates with \eqref{fixed:high-known-int}, we obtain
\begin{align}\label{fixed:high-time-derivatives}
 &\sup_{t\in[0,T]}\left(\|u_t\|_{L^2}^2+\|\theta_t\|_{L^2}^2\right)
       +\int_0^T\left(\|u_t\|_{H^1}^2+\|\theta_t\|_{H^1}^2\right)\dif t
 \leq C(\bar{\rho},\bar{\theta}).
\end{align}
Together with \eqref{fixed:high-known-sup}, \eqref{fixed:high-rhoH2},
\eqref{fixed:high-uH2}, and \eqref{fixed:high-rhot}, this proves
\eqref{fixed:high-conclusion}.

\textbf{Step 3}: Strong continuity and continuation.

We first recover the third-order spatial derivatives on finite time
intervals. The equations for $G$ and $\operatorname{curl}u$, the Hodge
estimate, and the definition of $G$ imply
\begin{align}\label{fixed:high-uH3}
 \|\D^3u\|_{L^2}\leq C(\bar{\rho},\bar{\theta})\big(&
 \|\D H\|_{L^2}+\|\D^2\rho\|_{L^2}+\|\D^2\theta\|_{L^2}
 +\|\D\rho\|_{L^3}\|\D\theta\|_{L^6}
 +\|\D\theta\|_{L^4}^2\big).
\end{align}
Thus $\int_0^T\|\D^3u\|_{L^2}^2\dif t\leq
C(\bar{\rho},\bar{\theta})$.

For the temperature, set $W=(\theta^{\beta+1}-\bar{\theta}^{\beta+1})/(\beta+1)$. The heat
equation can be written as
\[
 \kappa\Delta W=\rho\dot{\theta}+R\rho\theta\operatorname{div}u
          -\theta^\al\left(2\mu|\frD(u)|^2
                             +\lambda(\operatorname{div}u)^2\right).
\]
Differentiating this equation gives
\begin{align}\label{fixed:high-temperature-transform}
 \|\D\Delta W\|_{L^2}\leq C(\bar{\rho},\bar{\theta})\big(&
 \|\D\rho\|_{L^3}\|\dot{\theta}\|_{L^6}
 +\|\D\dot{\theta}\|_{L^2}
 +\|\D\rho\|_{L^3}\|\D u\|_{L^6}
 +\|\D\theta\|_{L^3}\|\D u\|_{L^6}
 +\|\D^2u\|_{L^2}\nonumber\\
 &+\|\D u\|_{L^\infty}\|\D^2u\|_{L^2}
 +\|\D\theta\|_{L^6}\|\D u\|_{L^6}^2\big).
\end{align}
All the terms on the right-hand side belong to $L^2(0,T)$ with bounds
independent of $T$. In particular, the sixth term is controlled by
\eqref{fixed:high-Du-infty} and the uniform bound of $\|\D^2u\|_{L^2}$;
the last term uses the uniform bounds of $\|\D\theta\|_{L^6}$ and
$\|\D u\|_{L^6}$, together with the square integrability of
$\|\D u\|_{L^6}$. Therefore $\D^3W\in L^2(0,T;L^2)$. The positive
lower and upper bounds of $\theta$ give
\[
 \|\D^3\theta\|_{L^2}\leq C(\bar{\rho},\bar{\theta})\left(
 \|\D^3W\|_{L^2}+\|\D\theta\|_{L^6}\|\D^2\theta\|_{L^3}
                              +\|\D\theta\|_{L^6}^3\right).
\]
By
$\|\D^2\theta\|_{L^3}\leq
C\|\D^2\theta\|_{L^2}^{\frac12}\|\D^3\theta\|_{L^2}^{\frac12}$
and Young's inequality, the term containing $\D^3\theta$ can be
absorbed. Using \eqref{fixed:high-known-sup} and
\eqref{fixed:high-known-int}, we conclude that
\begin{equation}\label{fixed:high-third-derivatives}
 \int_0^T\left(\|\D^3u\|_{L^2}^2+\|\D^3\theta\|_{L^2}^2\right)\dif t
 \leq C(\bar{\rho},\bar{\theta}).
\end{equation}
Consequently $u,\theta-\bar{\theta}\in L^2(0,T;H^3)$ for every finite
$T$. Together with \eqref{fixed:high-time-derivatives}, the Hilbert
space interpolation result for $H^3\subset H^2\subset H^1$ gives
$u,\theta-\bar{\theta}\in C([0,T];H^2)$. The same argument gives strong
$H^2$ limits at any finite maximal existence time.

For the density, differentiating the mass equation twice gives a linear
transport system for $V=\D^2\rho$ of the form
\[
 V_t+u\cdot\D V=\mathcal A(\D u)V+f,
 \qquad
 \|\mathcal A(\D u)\|_{L^\infty}\leq C\|\D u\|_{L^\infty},
\]
where
\[
 \|f\|_{L^2}\leq C\left(
 \|\D^2u\|_{L^4}\|\D\rho\|_{L^4}
                         +\bar{\rho}\|\D^3u\|_{L^2}\right).
\]
On every finite interval, including a finite maximal existence
interval,
\[
 \D u\in L^1(0,T;L^\infty)\quad\text{and}\quad f\in L^1(0,T;L^2).
\]
Applying Friedrichs mollifiers to this transport system and using its
$L^2$ energy identity, we obtain $\D^2\rho\in C([0,T];L^2)$, including
the endpoint in the maximal-time case. Combining this with
\eqref{fixed:high-rhot} yields $\rho-\bar{\rho}\in C([0,T];H^2)$.

All the estimates above are independent of the time endpoint. If the
maximal existence time were finite, the strong $H^2$ limits just
obtained would retain the positive lower bounds of the density and
temperature and would satisfy the assumptions of the local existence
theory. The solution could therefore be continued, which gives the
required contradiction. This completes the proof.
\end{proof}

\begin{proof}[Proof of Theorem \ref{th:fixed}]
Fix $\bar{\rho}>0$, the initial-data bound $C_0$ and the physical parameters. Choose $\varepsilon$, $M_6$ and $K_1,\ldots,K_7$ in the order specified in Lemma \ref{fixed:low-lemma}. The initial material derivatives satisfy
\[
 \|\sqrt{\rho_0}\dot{u}(0)\|_{L^2}^2\leq C(\bar{\rho})\bar{\theta}^{2\al},
 \qquad
 \|\sqrt{\rho_0}\dot{\theta}(0)\|_{L^2}^2\leq C(\bar{\rho})\bar{\theta}^{2\beta},
\]
by the equations, the positive initial density and temperature, and the $H^2$ bound. The factors $2$ in the low-order bootstrap therefore hold strictly on a short existence interval. Since $\|\theta_0-\bar{\theta}\|_{L^\infty}\leq CC_0$, a sufficiently large $\bar{\theta}$ also gives the initial small relative temperature condition. The initial relative density deviation is at most $1/4$.

Next choose the comparison constant in Lemma \ref{fixed:density-comparison} from the initial density gradient and the integrable exponential transient. This constant is chosen before the temperature, and does not depend on $T$. Finally, take $\bar{\theta}$ large enough that the low-order remainders, the density barrier and the logarithmic comparison budgets in Lemma \ref{fixed:density-prop} are all strictly improved. The maximum of these finitely many thresholds defines $C^*$ in \eqref{fixed:threshold}. Its dependence is only on $\bar{\rho},C_0,\mu,\lambda,\kappa,R,\al,\beta$.

By continuity, Lemma \ref{fixed:low-lemma} improves the velocity, temperature and mass-weighted density bounds, and Lemma \ref{fixed:density-prop} improves the density interval and controls its gradients. Thus no finite first exit from these a priori bounds is possible. Lemma \ref{fixed:high-lemma} gives all the uniform $H^2$ and time-derivative estimates, together with strong endpoint continuity. Proposition \ref{local} then excludes a finite maximal existence time. The resulting global solution satisfies \eqref{sp}, \eqref{bc} and \eqref{fixed:H2-bound}. Uniqueness follows from the local strong-solution theory. This completes the proof of Theorem \ref{th:fixed}.
\end{proof}

\section{Asymptotic behavior of the solution}\label{sec4}

The aim of this section is to acquire the convergence rate of the global solution $(\rho,u,\theta)$ obtained in Theorem \ref{th} or Theorem \ref{th:fixed} to the equilibrium $(\bar\rho,0,\bar\theta)$. Throughout this section, $C$ and $C(\bar\rho,\bar\theta)$ denote generic positive constants which may depend on the fixed positive far-field state, the physical parameters, $\alpha$, $\beta$ and the initial data. Constants in the quantitative decay estimates may also depend on $p_0$ and the initial $L^{p_0}$ norms. No uniformity as $\bar\rho\to0$ or $\bar\theta\to\infty$ is asserted.

The estimates in Section \ref{Section 3} give, in either case,
\begin{align}\label{decay:global-bounds}
&\sup_{t\geq0}\|(\rho-\bar\rho,u,\theta-\bar\theta)(t)\|_{H^2}^2\nonumber\\
&\quad+\int_0^\infty\left(\|\D\rho\|_{H^1}^2+\|\D u\|_{H^2}^2+\|\D\theta\|_{H^2}^2+\|(u_t,\theta_t)\|_{H^1}^2\right)\dif t\leq C(\bar\rho,\bar\theta).
\end{align}
For Theorem \ref{th}, the uniform and lower-order bounds follow from Propositions \ref{p4.1} and \ref{Lem:10} and \eqref{sec3.4:7}; the third-order bounds are recovered by the same effective-flux and temperature elliptic estimates as in \eqref{fixed:high-uH3}--\eqref{fixed:high-third-derivatives}. For Theorem \ref{th:fixed}, they follow from \eqref{fixed:H2-bound}, \eqref{fixed:high-known-int}, \eqref{fixed:high-rhoH2}, \eqref{fixed:high-uH2}, \eqref{fixed:high-time-derivatives} and \eqref{fixed:high-third-derivatives}, whose constants are independent of the time endpoint. In both cases, the density and temperature have fixed positive lower and upper bounds. We also have
\begin{equation}\label{decay:auxiliary-bounds}
\int_0^\infty\left(\|\D^2u\|_{L^4}^2+\|\D u\|_{L^\infty}^2\right)\dif t\leq C(\bar\rho,\bar\theta).
\end{equation}
Indeed, the first term is bounded by \eqref{D2u24} and \eqref{D2u24:sigma} in the first case, and by \eqref{fixed:density-hessian-L2} in the second case; the second follows from
$\|\D u\|_{L^\infty}^2\leq C\|\D u\|_{L^2}^{2/7}\|\D^2u\|_{L^4}^{12/7}$ and H\"older's inequality. The reconstructions in \eqref{fixed:high-DH}--\eqref{fixed:high-DH-int}, \eqref{fixed:high-uH3} and \eqref{fixed:high-temperature-transform}--\eqref{fixed:high-third-derivatives} use only these auxiliary bounds, the fixed positive pointwise bounds, the uniform $H^2$ bounds and the established $L^2(0,\infty;H^1)$ bounds of the material derivatives. Thus the same reconstructions give the third-order integrals in \eqref{decay:global-bounds} for Theorem \ref{th}, without any further parameter condition. We first establish the qualitative large-time behavior needed for the energy estimates below.

\begin{lemma}\label{decay:gradient-limit}
Let $(\rho,u,\theta)$ be the global strong solution obtained in Theorem \ref{th} or Theorem \ref{th:fixed}. Then
\begin{equation}\label{decay1}
\lim_{t\to\infty}\left(\|\D\rho\|_{H^1}+\|\D u\|_{H^1}+\|\D\theta\|_{H^1}\right)=0.
\end{equation}
\end{lemma}
\begin{proof}
The Hilbert space energy identity and \eqref{decay:global-bounds} give
\begin{align}\label{1002}
&\int_0^\infty\left|\frac{\dif}{\dif t}\left(\|\D u\|_{H^1}^2+\|\D\theta\|_{H^1}^2\right)\right|\dif t\nonumber\\
&\quad\leq C\int_0^\infty\|(\D u_t,\D\theta_t)\|_{L^2}
 \left(\|(\D u,\D\theta)\|_{L^2}+\|(\D^3u,\D^3\theta)\|_{L^2}\right)\dif t
 \leq C(\bar\rho,\bar\theta).
\end{align}
This identity is justified by spatial regularization, using $\D u,\D\theta\in L^2(0,\infty;H^2)$ and $\D u_t,\D\theta_t\in L^2(0,\infty;L^2)$. Since the squared $H^1$ gradient norms are themselves integrable in time, they tend to zero.

For the density, the derivation of \eqref{fixed:high-rho-energy} uses only the positive pointwise bounds and the uniform and integrated estimates in \eqref{decay:global-bounds}, together with the effective-flux estimates of Section \ref{Section 3}. It therefore applies to both solutions. It has the form
\begin{equation}\label{decay:density-limit}
y'(t)+c_*y(t)\leq f(t)y(t)+g(t),\qquad y(t)=\|\D^2\rho(t)\|_{L^2}^2,
\end{equation}
where $c_*>0$ and $f,g\geq0$ belong to $L^1(0,\infty)$. In particular, $f=C(\bar\rho,\bar\theta)\|\D u\|_{L^\infty}^2$ is integrable by \eqref{decay:auxiliary-bounds}; the source terms are integrable by the estimates used in \eqref{sec3.4:2}--\eqref{sec3.4:7} or \eqref{fixed:high-rho-energy}, respectively. The comparison formula yields
\[
y(t)\leq e^{\|f\|_{L^1(0,\infty)}}\left(e^{-c_*t}y(0)+\int_0^t e^{-c_*(t-s)}g(s)\dif s\right)\longrightarrow0.
\]
Finally,
$\|\D\rho\|_{L^2}^2\leq C\|\rho-\bar\rho\|_{L^2}\|\D^2\rho\|_{L^2}\to0$ by \eqref{decay:global-bounds}. This proves \eqref{decay1} without using any algebraic decay estimate.
\end{proof}

\subsection{Decay rate of $\|(\rho-\bar{\rho}, u, \theta-\bar{\theta})\|_{L^2}$}\label{dec.1}

We first derive the spatial derivative energy estimates. These estimates use only \eqref{decay:global-bounds} and \eqref{decay1}, and will also supply the density dissipation in the zero-order argument.

First, we rewrite \eqref{FCNS:2} as
\begin{equation}\label{FCNS2}
\begin{cases}
\rho_t+\bar{\rho}\div u=-\M_1,\\
u_t+R\bar{\rho}^{-1}\bar{\theta}\D\rho+R\D\theta-\mu\bar{\rho}^{-1}\bar{\theta}^\al\Delta u-(\mu+\lambda)\bar{\rho}^{-1}\bar{\theta}^\al\D\div u=-\M_2,\\
\theta_t+R\bar{\theta}\div u-\kappa\bar{\rho}^{-1}\bar{\theta}^\beta\Delta\theta=-\M_3,
\end{cases}
\end{equation}
where
\begin{align}\label{M123}
\M_1=~&u\cdot\D\rho+(\rho-\bar{\rho})\div u, \nonumber\\
\M_2=~&(u\cdot\D)u+R\left(\rho^{-1}\theta-\bar{\rho}^{-1}\bar{\theta}\right)\D\rho-\mu\left(2\rho^{-1}\div(\theta^\al\frD(u))-\bar{\rho}^{-1}\bar{\theta}^\al(\Delta u+\D\div u)\right)\nonumber\\
&-\lambda\left(\rho^{-1}\D(\theta^\al\div u)-\bar{\rho}^{-1}\bar{\theta}^\al\D\div u\right), \nonumber\\
\M_3=~&u\cdot\D\theta+R(\theta-\bar{\theta})\div u-2\mu\rho^{-1}\theta^\al|\frD(u)|^2-\lambda\rho^{-1}\theta^\al(\div u)^2-\kappa\left(\rho^{-1}\div(\theta^\beta\D\theta)-\bar{\rho}^{-1}\bar{\theta}^\beta\Delta\theta\right).
\end{align}

Then we establish the estimate for the first order spatial derivative of the solution.
\begin{lemma}\label{Lem.1}
The following estimate holds for the global solution $(\rho, u, \theta)$ of \eqref{FCNS2} when $t\geq 1$.
\begin{align}\label{1.0}
&\frac{1}{2}\frac{\dif}{\dif t}\left(R\bar{\rho}^{-2}\bar{\theta}\|\D\rho\|_{L^2}^2+\|\D u\|_{L^2}^2+\bar{\theta}^{-1}\|\D\theta\|_{L^2}^2\right)+\mu\bar{\rho}^{-1}\bar{\theta}^\al\|\D^2 u\|_{L^2}^2+\kappa\bar{\rho}^{-1}\bar{\theta}^{\beta-1}\|\D^2\theta\|_{L^2}^2\nonumber\\
\leq~&C(\bar{\rho},\bar{\theta})\left(\|(\D\rho, \D u, \D\theta)\|_{L^2}^{\frac{1}{2}}+\|\D\rho\|_{L^2}^{\frac{1}{3}}\|\D\rho\|_{L^4}^{\frac{2}{3}}\right)\left(\|\D^2 u\|_{L^2}^2+\|\D^2\rho\|_{L^2}^2+\|\D^2\theta\|_{L^2}^2\right).
\end{align}
\end{lemma}
\begin{proof}
Taking the $L^2$-inner product of \eqref{FCNS2}$_1$ with $-\Delta\rho$, we have
\begin{equation}\label{1.1}
\frac{1}{2}\frac{\dif}{\dif t}\|\D\rho\|_{L^2}^2-\bar{\rho}\int\D^2\rho:\D u\dif x=\int\M_1\Delta\rho\dif x\leq\|\D^2\rho\|_{L^2}\|\M_1\|_{L^2}.
\end{equation}
Direct computation gives
\begin{equation}\label{1.2}
\|\M_1\|_{L^2}\leq\|u\|_{L^3}\|\D\rho\|_{L^6}+\|\rho-\bar{\rho}\|_{L^3}\|\div u\|_{L^6}\leq C\|(\D\rho, \D u)\|_{L^2}^{\frac{1}{2}}\left(\|\D^2 u\|_{L^2}+\|\D^2\rho\|_{L^2}\right).
\end{equation}
Substituting \eqref{1.2} into \eqref{1.1}, we get
\begin{equation}\label{1.3}
\frac{1}{2}\frac{\dif}{\dif t}\|\D\rho\|_{L^2}^2-\bar{\rho}\int\D^2\rho:\D u\dif x\leq C\|(\D\rho, \D u)\|_{L^2}^{\frac{1}{2}}\left(\|\D^2 u\|_{L^2}^2+\|\D^2\rho\|_{L^2}^2\right).
\end{equation}

Taking the $L^2$-inner product of \eqref{FCNS2}$_2$ with $-\Delta u$, we have
\begin{align}\label{1.4}
&\frac{1}{2}\frac{\dif}{\dif t}\|\D u\|_{L^2}^2+\mu\bar{\rho}^{-1}\bar{\theta}^\al\|\D^2u\|_{L^2}^2+(\mu+\lambda)\bar{\rho}^{-1}\bar{\theta}^\al\|\D\div u\|_{L^2}^2+R\bar{\rho}^{-1}\bar{\theta}\int\D^2\rho:\D u\dif x+R\int\D^2\theta:\D u\dif x\nonumber\\
\leq~&\|\D^2 u\|_{L^2}\|\M_2\|_{L^2}.
\end{align}
Direct calculation gives
\begin{align}\label{1.5}
\|\M_2\|_{L^2}\leq~&\|u\|_{L^3}\|\D u\|_{L^6}+C\bar{\rho}^{-1}\|\theta-\bar{\theta}\|_{L^3}\|\D\rho\|_{L^6}+C\bar{\rho}^{-2}\bar{\theta}\|\rho-\bar{\rho}\|_{L^3}\|\D\rho\|_{L^6}\nonumber\\
&+C\bar{\rho}^{-1}\bar{\theta}^{\al-1}\|\theta-\bar{\theta}\|_{L^\infty}\|\D^2u\|_{L^2}+C\bar{\rho}^{-2}\bar{\theta}^\al\|\rho-\bar{\rho}\|_{L^\infty}\|\D^2u\|_{L^2}+C\bar{\rho}^{-1}\bar{\theta}^{\al-1}\|\D\theta\|_{L^3}\|\D u\|_{L^6}\nonumber\\
\leq~&C\|\D u\|_{L^2}^{\frac{1}{2}}\|\D^2u\|_{L^2}+C\bar{\theta}\|(\D\rho, \D\theta)\|_{L^2}^{\frac{1}{2}}\|\D^2\rho\|_{L^2}+C\bar{\theta}^\al\|\D\rho\|_{L^2}^{\frac{1}{3}}\|\D\rho\|_{L^4}^{\frac{2}{3}}\|\D^2u\|_{L^2}\nonumber\\
&+C\bar{\theta}^{\al-1}\|\D\theta\|_{L^2}^{\frac{1}{2}}\|\D^2\theta\|_{L^2}^{\frac{1}{2}}\|\D^2u\|_{L^2}.
\end{align}
Substituting \eqref{1.5} into \eqref{1.4} and using the uniform temperature bound in \eqref{decay:global-bounds}, we get
\begin{align}\label{1.6}
&\frac{1}{2}\frac{\dif}{\dif t}\|\D u\|_{L^2}^2+\mu\bar{\rho}^{-1}\bar{\theta}^\al\|\D^2u\|_{L^2}^2+(\mu+\lambda)\bar{\rho}^{-1}\bar{\theta}^\al\|\D\div u\|_{L^2}^2+R\bar{\rho}^{-1}\bar{\theta}\int\D^2\rho:\D u\dif x+R\int\D^2\theta:\D u\dif x\nonumber\\
\leq~&C(\bar{\rho}, \bar{\theta})\|(\D\rho, \D u, \D\theta)\|_{L^2}^{\frac{1}{2}}\left(\|\D^2\rho\|_{L^2}^2+\|\D^2u\|_{L^2}^2\right)+C(\bar{\rho}, \bar{\theta})\left(\|\D\rho\|_{L^2}^{\frac{1}{3}}\|\D\rho\|_{L^4}^{\frac{2}{3}}+\|\D\theta\|_{L^2}^{\frac{1}{2}}\right)\|\D^2u\|_{L^2}^2.
\end{align}

Taking the $L^2$-inner product of \eqref{FCNS2}$_3$ with $-\Delta\theta$, we have
\begin{equation}\label{1.7}
\frac{1}{2}\frac{\dif}{\dif t}\|\D\theta\|_{L^2}^2+\kappa\bar{\rho}^{-1}\bar{\theta}^\beta\|\D^2\theta\|_{L^2}^2-R\bar{\theta}\int\D^2\theta:\D u\dif x\leq\|\D^2\theta\|_{L^2}\|\M_3\|_{L^2}.
\end{equation}
Direct calculation gives
\begin{align}\label{1.8}
\|\M_3\|_{L^2}\leq~&\|u\|_{L^3}\|\D\theta\|_{L^6}+R\|\theta-\bar{\theta}\|_{L^3}\|\div u\|_{L^6}+C\bar{\rho}^{-1}\bar{\theta}^\al\|\D u\|_{L^3}\|\D u\|_{L^6}\nonumber\\
&+C\bar{\rho}^{-1}\bar{\theta}^{\beta-1}\|\theta-\bar{\theta}\|_{L^\infty}\|\D^2\theta\|_{L^2}+C\bar{\rho}^{-2}\bar{\theta}^\beta\|\rho-\bar{\rho}\|_{L^\infty}\|\D^2\theta\|_{L^2}+C\bar{\rho}^{-1}\bar{\theta}^{\beta-1}\|\D\theta\|_{L^4}^2\nonumber\\
\leq~& C\|\D u\|_{L^2}^{\frac{1}{2}}\|\D^2\theta\|_{L^2}+C\|\D\theta\|_{L^2}^{\frac{1}{2}}\|\D^2u\|_{L^2}+C\bar{\theta}^\al\|\D u\|_{L^2}^{\frac{1}{2}}\|\D u\|_{L^6}^{\frac{1}{2}}\|\D^2u\|_{L^2}\nonumber\\
&+C\bar{\theta}^{\beta-1}\|\D\theta\|_{L^2}^{\frac{1}{2}}\|\D^2\theta\|_{L^2}^{\frac{1}{2}}\|\D^2\theta\|_{L^2}+C\bar{\theta}^\beta\|\D\rho\|_{L^2}^{\frac{1}{3}}\|\D\rho\|_{L^4}^{\frac{2}{3}}\|\D^2\theta\|_{L^2}\nonumber\\
\leq~& C(\bar{\rho},\bar{\theta})\left(\|\D u\|_{L^2}^{\frac{1}{2}}+\|\D\theta\|_{L^2}^{\frac{1}{2}}+\|\D\rho\|_{L^2}^{\frac{1}{3}}\|\D\rho\|_{L^4}^{\frac{2}{3}}\right)\left(\|\D^2\theta\|_{L^2}+\|\D^2u\|_{L^2}\right).
\end{align}
Substituting \eqref{1.8} into \eqref{1.7}, we get
\begin{align}\label{1.9}
&\frac{1}{2}\frac{\dif}{\dif t}\|\D\theta\|_{L^2}^2+\kappa\bar{\rho}^{-1}\bar{\theta}^\beta\|\D^2\theta\|_{L^2}^2-R\bar{\theta}\int\D^2\theta:\D u\dif x\nonumber\\
\leq~&C(\bar{\rho},\bar{\theta})\left(\|\D\rho\|_{L^2}^{\frac{1}{2}}+\|\D u\|_{L^2}^{\frac{1}{2}}+\|\D\theta\|_{L^2}^{\frac{1}{2}}+\|\D\rho\|_{L^2}^{\frac{1}{3}}\|\D\rho\|_{L^4}^{\frac{2}{3}}\right)\left(\|\D^2\theta\|_{L^2}^2+\|\D^2u\|_{L^2}^2\right).
\end{align}
Making the linear combination \eqref{1.3}$\times R\bar{\rho}^{-2}\bar{\theta}+$\eqref{1.6}$+$\eqref{1.9}$\times\bar{\theta}^{-1}$, we complete the proof.
\end{proof}

Now we establish the energy estimate for the second order spatial derivative of the solution.
\begin{lemma}\label{Lem.2}
The following estimate holds for the global solution $(\rho, u, \theta)$ of \eqref{FCNS2} when $t\geq 1$.
\begin{align}\label{2.0}
&\frac{1}{2}\frac{\dif}{\dif t}\left(R\bar{\rho}^{-2}\bar{\theta}\|\D^2\rho\|_{L^2}^2+\|\D^2 u\|_{L^2}^2+\bar{\theta}^{-1}\|\D^2\theta\|_{L^2}^2\right)+\mu\bar{\rho}^{-1}\bar{\theta}^\al\|\D^3 u\|_{L^2}^2+\kappa\bar{\rho}^{-1}\bar{\theta}^{\beta-1}\|\D^3\theta\|_{L^2}^2\nonumber\\
\leq~&C(\bar{\rho}, \bar{\theta})\left(\|\D u\|_{L^2}^{\frac{1}{4}}+\|\D\theta\|_{L^2}^{\frac{1}{2}}+\|\D\rho\|_{L^2}^{\frac{1}{3}}\|\D\rho\|_{L^4}^{\frac{2}{3}}\right)\left(\|\D^2 u\|_{H^1}^2+\|\D^2\rho\|_{L^2}^2+\|\D^2\theta\|_{H^1}^2\right).
\end{align}
\end{lemma}
\begin{proof}
Applying $\D^2$ operator to \eqref{FCNS2}$_1$, and taking the $L^2$-inner product with $\D^2\rho$, we get
\begin{align}
 \label{2.1} &\frac{1}{2}\frac{\dif}{\dif t}\|\D^2\rho\|_{L^2}^2+\bar{\rho}\int\D^2\rho:\D^2\div u\dif x =~-\int\D^2\rho:\D^2(u\cdot\D\rho)\dif x-\int\D^2\rho:\D^2((\rho-\bar{\rho})\div u)\dif x\nonumber\\
 \leq~&-\frac{1}{2}\int u\cdot\D(|\D^2\rho|^2)\dif x+C\|\D u\|_{L^\infty}\|\D^2\rho\|_{L^2}^2+C\|\D^2u\|_{L^6}\|\D\rho\|_{L^3}\|\D^2\rho\|_{L^2}\nonumber\\
 &+C\|\rho-\bar{\rho}\|_{L^\infty}\|\D^3u\|_{L^2}\|\D^2\rho\|_{L^2}\nonumber\\
 \leq~&C\left(\|\rho-\bar{\rho}\|_{L^\infty}+\|\D\rho\|_{L^3}+\|\D u\|_{L^\infty}\right)\|\D^2\rho\|_{L^2}^2+C\left(\|\rho-\bar{\rho}\|_{L^\infty}+\|\D\rho\|_{L^3}\right)\|\D^3u\|_{L^2}^2\nonumber\\
 \leq~&C\|\D\rho\|_{L^2}^{\frac{1}{3}}\|\D\rho\|_{L^4}^{\frac{2}{3}}\left(\|\D^2u\|_{H^1}^2+\|\D^2\rho\|_{L^2}^2\right)+C\|\D u\|_{L^2}^{\frac{1}{4}}\|\D^3u\|_{L^2}^{\frac{3}{4}}\|\D^2\rho\|_{L^2}^2\nonumber\\
 \leq~&C(\bar{\rho}, \bar{\theta})\left(\|\D\rho\|_{L^2}^{\frac{1}{3}}\|\D\rho\|_{L^4}^{\frac{2}{3}}+\|\D u\|_{L^2}^{\frac{1}{4}}\right)\left(\|\D^2u\|_{H^1}^2+\|\D^2\rho\|_{L^2}^2\right),
\end{align}
where we have used the uniform $H^2$ bound in \eqref{decay:global-bounds}. Applying $\D^2$ operator to \eqref{FCNS2}$_2$, and taking the $L^2$-inner product with $\D^2u$, we get
\begin{align}
 \label{2.2}
&\frac{1}{2}\frac{\dif}{\dif t}\|\D^2u\|_{L^2}^2+\mu\bar{\rho}^{-1}\bar{\theta}^\al\|\D^3u\|_{L^2}^2+(\mu+\lambda)\bar{\rho}^{-1}\bar{\theta}^\al\|\D^2\div u\|_{L^2}^2\nonumber\\
 &-R\bar{\rho}^{-1}\bar{\theta}\int\D^2\rho:\D^2\div u\dif x-R\int\D^2\theta:\D^2\div u\dif x \leq~\|\D^3 u\|_{L^2}\|\D\M_2\|_{L^2}.
\end{align}
Direct calculation gives
\begin{align}\label{2.3}
\|\D\M_2\|_{L^2}\leq~&\|\D u\|_{L^3}\|\D u\|_{L^6}+\|u\|_{L^\infty}\|\D^2u\|_{L^2}+C\bar{\rho}^{-1}\|\D\rho\|_{L^3}\|\D\theta\|_{L^6}+C\bar{\rho}^{-1}\|\theta-\bar{\theta}\|_{L^\infty}\|\D^2\rho\|_{L^2}\nonumber\\
&+C\bar{\rho}^{-2}\bar{\theta}\|\D\rho\|_{L^3}\|\D\rho\|_{L^6}+C\bar{\rho}^{-2}\bar{\theta}\|\rho-\bar{\rho}\|_{L^\infty}\|\D^2\rho\|_{L^2}+C\bar{\theta}^\al\|\D\rho\|_{L^3}\|\D^2u\|_{L^6}\nonumber\\
&+C\bar{\theta}^{\al-1}\|\D\theta\|_{L^3}\|\D^2u\|_{L^6}+C\bar{\theta}^{\al-1}\|\theta-\bar{\theta}\|_{L^\infty}\|\D^3u\|_{L^2}+C\bar{\theta}^\al\|\rho-\bar{\rho}\|_{L^\infty}\|\D^3u\|_{L^2}\nonumber\\
&+C\bar{\theta}^{\al-1}\|\D\rho\|_{L^6}\|\D\theta\|_{L^3}\|\D u\|_{L^\infty}+C\bar{\theta}^{\al-2}\|\D\theta\|_{L^6}\|\D\theta\|_{L^3}\|\D u\|_{L^\infty}+C\bar{\theta}^{\al-1}\|\D^2\theta\|_{L^2}\|\D u\|_{L^\infty}\nonumber\\
\leq~&C(\bar{\rho}, \bar{\theta})\left(\|\D u\|_{L^3}+\|u\|_{L^\infty}+\|\D\theta\|_{L^3}+\|\theta-\bar{\theta}\|_{L^\infty}+\|\D\rho\|_{L^3}+\|\rho-\bar{\rho}\|_{L^\infty}\right)\times\nonumber\\
&\left(\|\D^2u\|_{H^1}+\|\D^2\theta\|_{L^2}+\|\D^2\rho\|_{L^2}\right)+C(\bar{\rho}, \bar{\theta})\|\D u\|_{L^2}^{\frac{1}{4}}\|\D^2\theta\|_{L^2}\|\D^3u\|_{L^2}^{\frac{3}{4}}\nonumber\\
\leq~&C(\bar{\rho}, \bar{\theta})\left(\|\D u\|_{L^2}^{\frac{1}{4}}+\|\D\theta\|_{L^2}^{\frac{1}{2}}+\|\D\rho\|_{L^2}^{\frac{1}{3}}\|\D\rho\|_{L^4}^{\frac{2}{3}}\right)\left(\|\D^2u\|_{H^1}+\|\D^2\theta\|_{L^2}+\|\D^2\rho\|_{L^2}\right).
\end{align}
Inserting \eqref{2.3} into \eqref{2.2}, we get
\begin{align}\label{2.4}
&\frac{1}{2}\frac{\dif}{\dif t}\|\D^2u\|_{L^2}^2+\mu\bar{\rho}^{-1}\bar{\theta}^\al\|\D^3u\|_{L^2}^2+(\mu+\lambda)\bar{\rho}^{-1}\bar{\theta}^\al\|\D^2\div u\|_{L^2}^2\nonumber\\
&-R\bar{\rho}^{-1}\bar{\theta}\int\D^2\rho:\D^2\div u\dif x-R\int\D^2\theta:\D^2\div u\dif x\nonumber\\
\leq~&C(\bar{\rho}, \bar{\theta})\left(\|\D u\|_{L^2}^{\frac{1}{4}}+\|\D\theta\|_{L^2}^{\frac{1}{2}}+\|\D\rho\|_{L^2}^{\frac{1}{3}}\|\D\rho\|_{L^4}^{\frac{2}{3}}\right)\left(\|\D^2u\|_{H^1}^2+\|\D^2\theta\|_{L^2}^2+\|\D^2\rho\|_{L^2}^2\right).
\end{align}

Applying $\D^2$ operator to \eqref{FCNS2}$_3$, and taking the $L^2$-inner product with $\D^2\theta$, we get
\begin{equation}\label{2.5}
\frac{1}{2}\frac{\dif}{\dif t}\|\D^2\theta\|_{L^2}^2+\kappa\bar{\rho}^{-1}\bar{\theta}^\beta\|\D^3\theta\|_{L^2}^2+R\bar{\theta}\int\D^2\theta:\D^2\div u\dif x\leq\|\D^3\theta\|_{L^2}\|\D\M_3\|_{L^2}.
\end{equation}
Direct calculation gives
\begin{align}\label{2.6}
\|\D\M_3\|_{L^2}\leq~&\|\D u\|_{L^3}\|\D\theta\|_{L^6}+\|u\|_{L^\infty}\|\D^2\theta\|_{L^2}+C\|\theta-\bar{\theta}\|_{L^\infty}\|\D^2u\|_{L^2}+C\bar{\theta}^\al\|\D u\|_{L^3}\|\D^2u\|_{L^6}\nonumber\\
&+C\bar{\theta}^\al\|\D u\|_{L^\infty}\|\D u\|_{L^6}\|\D\rho\|_{L^3}+C\bar{\theta}^{\al-1}\|\D u\|_{L^\infty}\|\D u\|_{L^6}\|\D\theta\|_{L^3}+C\bar{\theta}^\beta\|\D\rho\|_{L^3}\|\D^2\theta\|_{L^6}\nonumber\\
&+C\bar{\theta}^{\beta-1}\|\D\theta\|_{L^3}\|\D^2\theta\|_{L^6}+C\bar{\theta}^{\beta-1}\|\theta-\bar{\theta}\|_{L^\infty}\|\D^3\theta\|_{L^2}+C\bar{\theta}^\beta\|\rho-\bar{\rho}\|_{L^\infty}\|\D^3\theta\|_{L^2}\nonumber\\
&+C\bar{\theta}^{\beta-1}\|\D\rho\|_{L^3}\|\D\theta\|_{L^6}\|\D\theta\|_{L^\infty}+C\bar{\theta}^{\beta-2}\|\D\theta\|_{L^3}\|\D\theta\|_{L^6}\|\D\theta\|_{L^\infty}\nonumber\\
\leq~&C(\bar{\rho}, \bar{\theta})\left(\|\D u\|_{L^2}^{\frac{1}{2}}+\|\D\theta\|_{L^2}^{\frac{1}{2}}+\|\D\rho\|_{L^2}^{\frac{1}{3}}\|\D\rho\|_{L^4}^{\frac{2}{3}}\right)\left(\|\D^2\theta\|_{H^1}+\|\D^2u\|_{H^1}\right).
\end{align}
Substituting \eqref{2.6} into \eqref{2.5}, we obtain
\begin{align}\label{2.7}
&\frac{1}{2}\frac{\dif}{\dif t}\|\D^2\theta\|_{L^2}^2+\kappa\bar{\rho}^{-1}\bar{\theta}^\beta\|\D^3\theta\|_{L^2}^2+R\bar{\theta}\int\D^2\theta:\D^2\div u\dif x\nonumber\\
\leq~& C(\bar{\rho}, \bar{\theta})\left(\|\D u\|_{L^2}^{\frac{1}{2}}+\|\D\theta\|_{L^2}^{\frac{1}{2}}+\|\D\rho\|_{L^2}^{\frac{1}{3}}\|\D\rho\|_{L^4}^{\frac{2}{3}}\right)\left(\|\D^2u\|_{H^1}^2+\|\D^2\theta\|_{H^1}^2+\|\D^2\rho\|_{L^2}^2\right).
\end{align}
Making the linear combination \eqref{2.1}$\times R\bar{\rho}^{-2}\bar{\theta}+$\eqref{2.4}$+$\eqref{2.7}$\times\bar{\theta}^{-1}$, we complete the proof.
\end{proof}

In order to close the estimates, we need to establish the following dissipation estimate for
$\D^2\rho$.
\begin{lemma}\label{Lem.3}
The following estimate holds for the global solution $(\rho, u, \theta)$ of \eqref{FCNS2} when $t\geq 1$.
\begin{align}\label{3.0}
&\frac{\dif}{\dif t}\int\D u:\D^2\rho\dif x+\frac{R\bar{\theta}}{2\bar{\rho}}\|\D^2\rho\|_{L^2}^2\nonumber\\
\leq~&\overline{C}_2(\bar{\rho}, \bar{\theta})\left(\|\D^2 u\|_{H^1}^2+\|\D^2\theta\|_{L^2}^2\right)+C(\bar{\rho}, \bar{\theta})\left(\|\D u\|_{L^2}^{\frac{1}{2}}+\|\D\theta\|_{L^2}+\|\D\rho\|_{L^2}^{\frac{2}{3}}\right)\|\D^2\rho\|_{L^2}^2.
\end{align}
\end{lemma}
\begin{proof}
Applying $\D$ operator to \eqref{FCNS2}$_2$, and taking the $L^2$-inner product with $\D^2\rho$, we get
\begin{align}\label{3.1}
&\int\D u_t:\D^2\rho\dif x+R\bar{\rho}^{-1}\bar{\theta}\|\D^2\rho\|_{L^2}^2+R\int\D^2\theta:\D^2\rho\dif x-\mu\bar{\rho}^{-1}\bar{\theta}^\al\int\D\Delta u:\D^2\rho\dif x\nonumber\\
&-(\mu+\lambda)\bar{\rho}^{-1}\bar{\theta}^\al\int\D^2\div u:\D^2\rho\dif x=-\int\D\M_2:\D^2\rho\dif x.
\end{align}
Using \eqref{FCNS2}$_1$, we have
\begin{align}\label{3.2}
\int\D u_t:\D^2\rho\dif x=~&\frac{\dif}{\dif t}\int\D u:\D^2\rho\dif x+\int\D\div u\cdot\D\rho_t\dif x\nonumber\\
=~&\frac{\dif}{\dif t}\int\D u:\D^2\rho\dif x-\bar{\rho}\|\D\div u\|_{L^2}^2-\int\D\div u\cdot\D\M_1\dif x.
\end{align}
Hence we get
\begin{align}\label{3.3}
&\frac{\dif}{\dif t}\int\D u:\D^2\rho\dif x+R\bar{\rho}^{-1}\bar{\theta}\|\D^2\rho\|_{L^2}^2\nonumber\\
=~&-R\int\D^2\theta:\D^2\rho\dif x+\mu\bar{\rho}^{-1}\bar{\theta}^\al\int\D\Delta u:\D^2\rho\dif x+(\mu+\lambda)\bar{\rho}^{-1}\bar{\theta}^\al\int\D^2\div u:\D^2\rho\dif x\nonumber\\
&+\bar{\rho}\|\D\div u\|_{L^2}^2+\int\D\div u\cdot\D\M_1\dif x-\int\D\M_2:\D^2\rho\dif x\nonumber\\
\leq~&\frac{R\bar{\theta}}{2\bar{\rho}}\|\D^2\rho\|_{L^2}^2+C(\bar{\rho}, \bar{\theta})\left(\|\D^2u\|_{H^1}^2+\|\D^2\theta\|_{L^2}^2\right)+C(\bar{\rho}, \bar{\theta})\left(\|\M_1\|_{L^2}^2+\|\D\M_2\|_{L^2}^2\right).
\end{align}
Substituting \eqref{1.2} and \eqref{2.3} into \eqref{3.3}, we complete the proof.
\end{proof}

Combining all the estimates obtained in Lemmas \ref{Lem.1}--\ref{Lem.3}, we derive the following energy estimate.
\begin{lemma}\label{Lem.4}
For the global solution $(\rho, u, \theta)$, there exist energy functional $\calE(t)\sim\|(\D\rho, \D u, \D\theta)\|_{H^1}^2$, time $T_1(\bar{\rho}, \bar{\theta})>0$ and a positive constant $K_0(\bar{\rho}, \bar{\theta})$, such that for any $t\geq T_1(\bar{\rho}, \bar{\theta})$, it holds
\begin{equation}\label{4.0}
\frac{\dif}{\dif t}\calE(t)+K_0(\bar{\rho},\bar{\theta})\left(\|\D^2 u\|_{H^1}^2+\|\D^2\theta\|_{H^1}^2+\|\D^2\rho\|_{L^2}^2\right)\leq 0.
\end{equation}
\end{lemma}
\begin{proof}
For fixed $\delta(\bar{\rho}, \bar{\theta})>0$, adding \eqref{1.0}, \eqref{2.0} and $\delta(\bar{\rho}, \bar{\theta})\times$\eqref{3.0} together, we obtain
\begin{align}\label{4.1}
&\frac{1}{2}\frac{\dif}{\dif t}\left(R\bar{\rho}^{-2}\bar{\theta}\|\D\rho\|_{H^1}^2+\|\D u\|_{H^1}^2+\bar{\theta}^{-1}\|\D\theta\|_{H^1}^2+2\delta(\bar{\rho}, \bar{\theta})\int\D u:\D^2\rho\dif x\right)\nonumber\\
&+\mu\bar{\rho}^{-1}\bar{\theta}^\al\|\D^2 u\|_{H^1}^2+\kappa\bar{\rho}^{-1}\bar{\theta}^{\beta-1}\|\D^2\theta\|_{H^1}^2+\frac{R\bar{\theta}}{2\bar{\rho}}\delta(\bar{\rho},\bar{\theta})\|\D^2\rho\|_{L^2}^2\nonumber\\
\leq~&C(\bar{\rho},\bar{\theta})\left(\|\D u\|_{L^2}^{\frac{1}{4}}+\|\D\theta\|_{L^2}^{\frac{1}{2}}+\|\D\rho\|_{L^2}^{\frac{1}{3}}\right)\left(\|\D^2 u\|_{H^1}^2+\|\D^2\rho\|_{L^2}^2+\|\D^2\theta\|_{H^1}^2\right)\nonumber\\
&+\overline{C}_2(\bar{\rho}, \bar{\theta})\delta(\bar{\rho}, \bar{\theta})\left(\|\D^2 u\|_{H^1}^2+\|\D^2\theta\|_{L^2}^2\right).
\end{align}
Taking $\delta(\bar{\rho}, \bar{\theta})$ sufficiently small, such that
\begin{equation*}
\overline{C}_2(\bar{\rho}, \bar{\theta})\delta(\bar{\rho}, \bar{\theta})\leq\frac{1}{2}\min\left\{\mu\bar{\rho}^{-1}\bar{\theta}^\al, ~ \kappa\bar{\rho}^{-1}\bar{\theta}^{\beta-1}\right\},
\end{equation*}
and
\begin{equation*}
2\delta(\bar{\rho},\bar{\theta})\left|\int\D u:\D^2\rho\dif x\right|\leq \frac{1}{2}\left(R\bar{\rho}^{-2}\bar{\theta}\|\D\rho\|_{H^1}^2+\|\D u\|_{H^1}^2\right),
\end{equation*}
we obtain the following inequality
\begin{align}\label{4.2}
&\frac{\dif}{\dif t}\calE(t)+2K_0(\bar{\rho},\bar{\theta})\left(\|\D^2 u\|_{H^1}^2+\|\D^2\theta\|_{H^1}^2+\|\D^2\rho\|_{L^2}^2\right)\\
\leq~&C(\bar{\rho},\bar{\theta})\left(\|\D u\|_{L^2}^{\frac{1}{4}}+\|\D\theta\|_{L^2}^{\frac{1}{2}}+\|\D\rho\|_{L^2}^{\frac{1}{3}}\right)\left(\|\D^2 u\|_{H^1}^2+\|\D^2\rho\|_{L^2}^2+\|\D^2\theta\|_{H^1}^2\right),\nonumber
\end{align}
where
\begin{equation}\label{4.25}
\calE(t)=:\frac{1}{2}\left(R\bar{\rho}^{-2}\bar{\theta}\|\D\rho\|_{H^1}^2+\|\D u\|_{H^1}^2+\bar{\theta}^{-1}\|\D\theta\|_{H^1}^2\right)+\delta(\bar{\rho},\bar{\theta})\int\D u:\D^2\rho\dif x\sim\|(\D\rho, \D u, \D\theta)\|_{H^1}^2.
\end{equation}

Noticing \eqref{decay1}, we can see that there exists a constant $T_1(\bar{\rho},\bar{\theta})>0$, such that for any $t\geq T_1(\bar{\rho},\bar{\theta})$,
\begin{equation}\label{4.3}
C(\bar{\rho},\bar{\theta})\left(\|\D u\|_{L^2}^{\frac{1}{4}}+\|\D\theta\|_{L^2}^{\frac{1}{2}}+\|\D\rho\|_{L^2}^{\frac{1}{3}}\right)\leq K_0(\bar{\rho},\bar{\theta}).
\end{equation}
Combining \eqref{4.2} and \eqref{4.3}, we complete the proof of Lemma \ref{Lem.4}.
\end{proof}

We now combine the derivative energy with the relative entropy. With $\calE(t)$ defined in \eqref{4.25}, set
\begin{align}\label{Xt}
X(t)=:~&\int\left(\frac12\rho|u|^2+R\bar\theta\left(\rho\ln\frac{\rho}{\bar\rho}-\rho+\bar\rho\right)
 +\rho\left(\theta-\bar\theta-\bar\theta\ln\frac{\theta}{\bar\theta}\right)\right)\dif x\nonumber\\
&+\calE(t)+\delta_1\int u\cdot\D\rho\dif x,
\end{align}
where $\delta_1>0$ will be chosen below, independently of the coefficient $\delta(\bar\rho,\bar\theta)$ in \eqref{4.25}.

\begin{lemma}\label{Lem4.1:6-0}
Let $(\rho,u,\theta)$ be the global strong solution obtained in Theorem \ref{th} or Theorem \ref{th:fixed}. For a sufficiently small $\delta_1=\delta_1(\bar\rho,\bar\theta)>0$, there exist positive constants $c_1(\bar\rho,\bar\theta)$ and $\overline C_1(\bar\rho,\bar\theta)$ such that, for $t\geq T_1(\bar\rho,\bar\theta)$,
\begin{equation}\label{X:sim}
\overline C_1^{-1}\|(\rho-\bar\rho,u,\theta-\bar\theta)\|_{H^2}^2
 \leq X(t)\leq\overline C_1\|(\rho-\bar\rho,u,\theta-\bar\theta)\|_{H^2}^2,
\end{equation}
and
\begin{equation}\label{dissi:X}
\frac{\dif}{\dif t}X(t)+c_1(\bar\rho,\bar\theta)
 \left(\|\D\rho\|_{H^1}^2+\|\D u\|_{H^2}^2+\|\D\theta\|_{H^2}^2\right)\leq0.
\end{equation}
\end{lemma}
\begin{proof}
The relative entropy identity \eqref{Lem3.1:1} and the positive pointwise bounds imply
\begin{align}\label{Lem4.1:1}
&\frac{\dif}{\dif t}\int\left(\frac12\rho|u|^2+R\bar\theta\left(\rho\ln\frac{\rho}{\bar\rho}-\rho+\bar\rho\right)
 +\rho\left(\theta-\bar\theta-\bar\theta\ln\frac{\theta}{\bar\theta}\right)\right)\dif x\nonumber\\
&\quad+c_2(\bar\rho,\bar\theta)\left(\|\D u\|_{L^2}^2+\|\D\theta\|_{L^2}^2\right)\leq0.
\end{align}
The entropy integral is equivalent to $\|(\rho-\bar\rho,u,\theta-\bar\theta)\|_{L^2}^2$ for each fixed positive far-field state. The original nonlinear equations give the exact identity
\begin{align}\label{decay:cross-identity}
&\frac{\dif}{\dif t}\int u\cdot\D\rho\dif x+R\int\frac{\theta}{\rho}|\D\rho|^2\dif x\nonumber\\
&=-\int\left((u\cdot\D)u+R\D\theta-\rho^{-1}\div\T\right)\cdot\D\rho\dif x
 +\int\rho(\div u)^2\dif x+\int(u\cdot\D\rho)\div u\dif x.
\end{align}
By \eqref{decay:global-bounds} and the Sobolev inequality,
\[
\|\rho^{-1}\div\T\|_{L^2}
 \leq C(\bar\rho,\bar\theta)\left(\|\D^2u\|_{L^2}+\|\D\theta\|_{L^3}\|\D u\|_{L^6}\right)
 \leq C(\bar\rho,\bar\theta)\|\D^2u\|_{L^2}.
\]
Since $\theta/\rho$ has a positive lower bound, Young's inequality in \eqref{decay:cross-identity} yields
\begin{equation}\label{decay:cross-bound}
\frac{\dif}{\dif t}\int u\cdot\D\rho\dif x+c_2(\bar\rho,\bar\theta)\|\D\rho\|_{L^2}^2
 \leq C(\bar\rho,\bar\theta)\left(\|\D u\|_{H^1}^2+\|\D\theta\|_{L^2}^2\right),
\end{equation}
after decreasing $c_2$ if necessary. Choose $\delta_1>0$ sufficiently small. The bound
$|\int u\cdot\D\rho\dif x|\leq\|u\|_{L^2}\|\D\rho\|_{L^2}$ and \eqref{4.25} give \eqref{X:sim}. Adding \eqref{Lem4.1:1}, \eqref{4.0} and $\delta_1\times$\eqref{decay:cross-bound}, and absorbing the right-hand side by these dissipations, gives \eqref{dissi:X}. All the coefficient choices depend only on fixed-state bounds. In particular, this argument imposes no new restriction on $\bar\rho$ or $\bar\theta$.
\end{proof}

The next lemma is concerned with the low-frequency part of the solution.
\begin{lemma}\label{Lem:4.2}
Let $(\rho,u,\theta)$ be either global strong solution in Theorem \ref{th2}, with the stated $L^{p_0}$ assumption on the initial data. Then
\begin{align}
 \label{fl2-0} &\int_{S(t)}\left(R\bar\theta|\widehat{\rho-\bar\rho}|^2+|\widehat{\rho u}|^2+\bar\theta^{-1}|\reallywidehat{\rho(\theta-\bar\theta)}|^2\right)\dif\xi \leq C(\bar\rho,\bar\theta)(1+t)^{-2\beta(p_0)}+C(\bar\rho,\bar\theta)(1+t)^{-\frac32}\nonumber\\
 &\quad+C(\bar\rho,\bar\theta)(1+t)^{-\frac52}
\left[\int_0^t\left(\|\rho-\bar\rho\|_{L^2}^2+\|u\|_{L^2}^2+\|\theta-\bar\theta\|_{L^2}^2\right)\dif s\right]^2,
\end{align}
where
\[
S(t)=\left\{\xi:|\xi|^2\leq\frac{C_*(\bar\rho,\bar\theta)}{1+t}\right\},
\qquad \beta(p_0)=\frac34\left(\frac2{p_0}-1\right),
\]
and $C_*(\bar\rho,\bar\theta)>0$ will be chosen below. The constant $C(\bar\rho,\bar\theta)$ may also depend on the initial $L^{p_0}$ norms.
\end{lemma}
\begin{proof}
To retain the cancellation of viscous heating, we use the conservative expressions
$(\rho-\bar\rho,\rho u,\rho(\theta-\bar\theta)+\frac12\rho|u|^2)$ without changing the physical unknowns. The mass equation is
\begin{equation}\label{FT:1}
\partial_t(\rho-\bar\rho)+\div(\rho u)=0.
\end{equation}
The momentum equation can be written exactly as
\begin{align}\label{fl2}
&\partial_t(\rho u)+R\bar\theta\D(\rho-\bar\rho)
 +R\D\left(\rho(\theta-\bar\theta)+\tfrac12\rho|u|^2\right)
 -\mu\bar\rho^{-1}\bar\theta^\alpha\Delta(\rho u)
 -(\mu+\lambda)\bar\rho^{-1}\bar\theta^\alpha\D\div(\rho u)\nonumber\\
&=-\div(\rho u\otimes u)+\tfrac R2\D(\rho|u|^2)
 +2\mu\div\left((\theta^\alpha-\bar\theta^\alpha)\frD(u)\right)
 +\lambda\D\left((\theta^\alpha-\bar\theta^\alpha)\div u\right)\nonumber\\
&\quad-\mu\bar\rho^{-1}\bar\theta^\alpha\Delta((\rho-\bar\rho)u)
 -(\mu+\lambda)\bar\rho^{-1}\bar\theta^\alpha\D\div((\rho-\bar\rho)u).
\end{align}
Subtracting $\bar\theta$ times the mass equation from the total energy equation gives
\begin{align}\label{q1}
&\partial_t\left(\rho(\theta-\bar\theta)+\tfrac12\rho|u|^2\right)
 +R\bar\theta\div(\rho u)
 -\kappa\bar\rho^{-1}\bar\theta^\beta\Delta\left(\rho(\theta-\bar\theta)+\tfrac12\rho|u|^2\right)\nonumber\\
&=-\div\left(\left((1+R)\rho(\theta-\bar\theta)+\tfrac12\rho|u|^2\right)u\right)
 +\div(u\T)\nonumber\\
&\quad+\kappa\Delta\left[\frac{\theta^{\beta+1}-\bar\theta^{\beta+1}}{\beta+1}
 -\bar\theta^\beta(\theta-\bar\theta)
 -\bar\rho^{-1}\bar\theta^\beta\left((\rho-\bar\rho)(\theta-\bar\theta)+\tfrac12\rho|u|^2\right)\right].
\end{align}
For the homogeneous constant-coefficient part of \eqref{FT:1}--\eqref{q1}, the Fourier energy with weights $(R\bar\theta,1,\bar\theta^{-1})$ is nonincreasing. Indeed, the density--momentum and temperature--momentum terms cancel, and the remaining quadratic dissipation is nonnegative since $\mu>0$, $\mu+\lambda>0$ and $\kappa>0$. Thus its evolution operator is uniformly bounded in this weighted norm.

The global bounds already established imply
\begin{align}
 \label{q2} &\|\rho u\otimes u\|_{L^1}+\|\rho|u|^2\|_{L^1} +\|\rho(\theta-\bar\theta)u\|_{L^1}+\|\rho|u|^2u\|_{L^1} +\|(\rho-\bar\rho)u\|_{L^1} +\|(\rho-\bar\rho)(\theta-\bar\theta)\|_{L^1}\nonumber\\
 &\quad+\left\|\frac{\theta^{\beta+1}-\bar\theta^{\beta+1}}{\beta+1} -\bar\theta^\beta(\theta-\bar\theta)\right\|_{L^1} \leq C(\bar\rho,\bar\theta)\left(\|\rho-\bar\rho\|_{L^2}^2+\|u\|_{L^2}^2+\|\theta-\bar\theta\|_{L^2}^2\right),
\end{align}
and
\begin{equation}
 \label{q4} \|(\theta^\alpha-\bar\theta^\alpha)\frD(u)\|_{L^1} +\|(\theta^\alpha-\bar\theta^\alpha)\div u\|_{L^1}+\|u\T\|_{L^1} \leq C(\bar\rho,\bar\theta)\left(\|u\|_{L^2}+\|\theta-\bar\theta\|_{L^2}\right)\|\D u\|_{L^2}.
\end{equation}
Here the uniform $L^\infty$ bound of $u$ follows from its uniform $H^2$ bound, and the temperature Taylor remainder is quadratic because $\theta$ remains in a fixed positive interval for fixed $(\bar\rho,\bar\theta)$.

Applying variation of constants to the Fourier transform of \eqref{FT:1}--\eqref{q1}, using the bounded evolution operator and $\|\widehat f\|_{L^\infty}\leq C\|f\|_{L^1}$, we obtain
\begin{align}\label{q5}
&\left(R\bar\theta|\widehat{\rho-\bar\rho}(t)|^2+|\widehat{\rho u}(t)|^2
 +\bar\theta^{-1}\left|\reallywidehat{\rho(\theta-\bar\theta)+\tfrac12\rho|u|^2}(t)\right|^2\right)^{\frac12}\nonumber\\
&\leq C(\bar\rho,\bar\theta)\left(|\widehat{\rho_0-\bar\rho}|+|\widehat{\rho_0u_0}|
 +\left|\reallywidehat{\rho_0(\theta_0-\bar\theta)+\tfrac12\rho_0|u_0|^2}\right|\right)\nonumber\\
&\quad+C(\bar\rho,\bar\theta)(|\xi|+|\xi|^2)
 \int_0^t\left(\|\rho-\bar\rho\|_{L^2}^2+\|u\|_{L^2}^2+\|\theta-\bar\theta\|_{L^2}^2\right)\dif s\nonumber\\
&\quad+C(\bar\rho,\bar\theta)|\xi|\int_0^t
 \left(\|u\|_{L^2}+\|\theta-\bar\theta\|_{L^2}\right)\|\D u\|_{L^2}\dif s.
\end{align}
By the zero-order bounds and the velocity dissipation,
\begin{equation}\label{q6}
\left[\int_0^t\left(\|u\|_{L^2}+\|\theta-\bar\theta\|_{L^2}\right)\|\D u\|_{L^2}\dif s\right]^2
 \leq C(\bar\rho,\bar\theta)t.
\end{equation}
The initial conservative expressions belong to $L^{p_0}$: in particular
$\|u_0\|_{L^{2p_0}}^2\leq\|u_0\|_{L^\infty}\|u_0\|_{L^{p_0}}$.
The Hausdorff--Young inequality therefore bounds the integral of the squared initial term in \eqref{q5} over $S(t)$ by $C(\bar\rho,\bar\theta)(1+t)^{-2\beta(p_0)}$, also for $p_0=2$ by Plancherel's identity. Moreover,
\begin{equation}\label{q7}
\int_{S(t)}(|\xi|+|\xi|^2)^2\dif\xi\leq C(\bar\rho,\bar\theta)(1+t)^{-\frac52},
\qquad
\int_{S(t)}|\xi|^2\dif\xi\leq C(\bar\rho,\bar\theta)(1+t)^{-\frac52}.
\end{equation}
Squaring \eqref{q5} and integrating over $S(t)$ gives the right-hand side of \eqref{fl2-0}, with the total energy expression in place of $\rho(\theta-\bar\theta)$. Finally,
\begin{equation}\label{q8}
\int_{S(t)}\left|\reallywidehat{\rho|u|^2}\right|^2\dif\xi
 \leq C(\bar\rho,\bar\theta)(1+t)^{-\frac32}\|u\|_{L^2}^4
 \leq C(\bar\rho,\bar\theta)(1+t)^{-\frac32},
\end{equation}
which proves \eqref{fl2-0}. The same argument and
$\bar\rho u=\rho u-(\rho-\bar\rho)u$, $\bar\rho(\theta-\bar\theta)=\rho(\theta-\bar\theta)-(\rho-\bar\rho)(\theta-\bar\theta)$ give
\begin{align}
 \label{fl2-1} &\int_{S(t)}\left(|\widehat{\rho-\bar\rho}|^2+|\widehat u|^2+|\widehat{\theta-\bar\theta}|^2\right)\dif\xi \leq C(\bar\rho,\bar\theta)(1+t)^{-2\beta(p_0)}+C(\bar\rho,\bar\theta)(1+t)^{-\frac32}\nonumber\\
 &\quad+C(\bar\rho,\bar\theta)(1+t)^{-\frac52}
 \left[\int_0^t\left(\|\rho-\bar\rho\|_{L^2}^2+\|u\|_{L^2}^2+\|\theta-\bar\theta\|_{L^2}^2\right)\dif s\right]^2.
\end{align}
\end{proof}

Now we are in a position to prove \eqref{decay-0}. In fact, we will prove the following decay estimate.
\begin{lemma}\label{l2d}
Let $(\rho,u,\theta)$ be either global strong solution in Theorem \ref{th2}, with the stated $L^{p_0}$ assumption on the initial data. Then we have
\begin{equation}\label{decay-0:U}
\|\rho-\bar{\rho}\|_{H^1}+\|u\|_{H^1}+ \|\theta-\bar{\theta}\|_{H^1}+\|\nabla \rho\|_{L^4}+\|(\dot{u}, \dot{\theta})\|_{L^2}\leq C(\bar{\rho}, \bar{\theta})(1+t)^{-\frac{3}{4}(\frac{2}{p_0}-1)}.
\end{equation}
\end{lemma}
\begin{proof}
For $t$ sufficiently large that $C_*(\bar\rho,\bar\theta)/(1+t)\leq1$, Fourier splitting in \eqref{dissi:X} gives, after decreasing $c_1$ if necessary,
\begin{equation}\label{decayeq9}
\frac{\dif}{\dif t}X(t)+\frac{c_1(\bar\rho,\bar\theta)C_*(\bar\rho,\bar\theta)}{1+t}
 \|(\rho-\bar\rho,u,\theta-\bar\theta)\|_{H^2}^2
 \leq\frac{C(\bar\rho,\bar\theta)}{1+t}\int_{S(t)}\left(|\widehat{\rho-\bar\rho}|^2+|\widehat u|^2+|\widehat{\theta-\bar\theta}|^2\right)\dif\xi.
\end{equation}
Indeed, outside $S(t)$ the density multiplier $|\xi|^2+|\xi|^4$ controls a fixed multiple of $C_*(1+t)^{-1}(1+|\xi|^2+|\xi|^4)$; the velocity and temperature multipliers have the same property. The omitted $H^2$ weights are bounded on $S(t)$. Choosing $C_*(\bar\rho,\bar\theta)$ sufficiently large and using \eqref{X:sim}, we obtain
\begin{equation}\label{decay:key}
\frac{\dif}{\dif t}X(t)+\frac2{1+t}X(t)
 \leq\frac{C(\bar\rho,\bar\theta)}{1+t}\int_{S(t)}\left(|\widehat{\rho-\bar\rho}|^2+|\widehat u|^2+|\widehat{\theta-\bar\theta}|^2\right)\dif\xi.
\end{equation}
The low-frequency estimate \eqref{fl2-1} is always taken from initial time zero; no propagation of the initial $L^{p_0}$ norm to the later starting time of \eqref{decay:key} is needed.

\textbf{Step 1}. The uniform bound and \eqref{fl2-1} imply that the integral on the right-hand side is at most
$C(\bar\rho,\bar\theta)((1+t)^{-2\beta(p_0)}+(1+t)^{-1/2})$.
With $r_1=\min\{2\beta(p_0),1/2\}$, multiplication of \eqref{decay:key} by $(1+t)^2$ and integration give
\begin{equation}\label{informdecay}
X(t)\leq C(\bar\rho,\bar\theta)(1+t)^{-r_1}.
\end{equation}
By \eqref{X:sim},
\begin{equation}\label{decayeq10}
\|(\rho-\bar\rho,u,\theta-\bar\theta)\|_{H^2}
 \leq C(\bar\rho,\bar\theta)(1+t)^{-r_1/2}.
\end{equation}
If $2\beta(p_0)\leq1/2$, this already gives the desired power for the $H^2$ norm.

\textbf{Step 2}. If $2\beta(p_0)>1/2$, then $r_1=1/2$, and hence
\[
\left[\int_0^t\|(\rho-\bar\rho,u,\theta-\bar\theta)\|_{L^2}^2\dif s\right]^2
 \leq C(\bar\rho,\bar\theta)(1+t).
\]
Using \eqref{fl2-1} again, the low-frequency integral in \eqref{decay:key} is at most
$C(\bar\rho,\bar\theta)((1+t)^{-2\beta(p_0)}+(1+t)^{-3/2})$.
Since $0\leq2\beta(p_0)\leq3/2$, another integration gives
\begin{equation}\label{decayeq11}
X(t)\leq C(\bar\rho,\bar\theta)(1+t)^{-2\beta(p_0)}.
\end{equation}
Thus, in both cases,
\begin{equation}\label{decayeq10-3}
\|(\rho-\bar\rho,u,\theta-\bar\theta)\|_{H^2}
 \leq C(\bar\rho,\bar\theta)(1+t)^{-\beta(p_0)}.
\end{equation}
The original equations and the uniform $H^2$ bound also imply
\begin{equation}\label{decay:material-bound}
\|\D\rho\|_{L^4}+\|(\dot u,\dot\theta)\|_{L^2}
 \leq C(\bar\rho,\bar\theta)\|(\D\rho,\D u,\D\theta)\|_{H^1}.
\end{equation}
Here the quadratic gradient terms satisfy, for example,
\[
\|\D u\|_{L^4}^2\leq\|\D u\|_{L^3}\|\D u\|_{L^6}\leq C(\bar\rho,\bar\theta)\|\D^2u\|_{L^2},
\]
and similarly for $\D\theta$. This proves \eqref{decay-0:U}, with the initial compact time interval absorbed into the constant.

\textbf{Step 3}. To prove \eqref{decay:strong-convergence} when $p_0=2$, retain the actual initial Fourier integral in \eqref{q5}, instead of its uniform bound. The initial conservative expressions belong to $L^2$, so
\begin{align*}
\int_{S(t)}\Big(&|\widehat{\rho_0-\bar\rho}|^2+|\widehat{\rho_0u_0}|^2
 +\left|\reallywidehat{\rho_0(\theta_0-\bar\theta)+\tfrac12\rho_0|u_0|^2}\right|^2\Big)\dif\xi\longrightarrow0.
\end{align*}
This follows from the absolute continuity of the integral, since $S(t)$ shrinks to the origin. The nonlinear contributions in \eqref{q5}--\eqref{q8} tend to zero as well: the uniform zero-order bound makes them $O((1+t)^{-1/2})$ after integration over $S(t)$. Thus the low-frequency integral in \eqref{decay:key} tends to zero. Multiplying that inequality by $(1+t)^2$, integrating, and splitting the resulting integral at a fixed large time shows $X(t)\to0$. By \eqref{X:sim}, this proves \eqref{decay:strong-convergence}. For $p_0<2$, it already follows from \eqref{decayeq10-3}.
\end{proof}

\subsection{Decay rate of $\|(\D\rho, \D u, \D\theta)\|_{L^2}$}\label{dec.2}

At last, we will use the same manner as \cite{Gao2021De} to prove \eqref{decay-1}. More precisely, we shall prove the following decay estimate.
\begin{lemma}\label{dec:2}
Let $(\rho,u,\theta)$ be either global strong solution in Theorem \ref{th2}, with the stated $L^{p_0}$ assumption on the initial data. Then we have
\begin{equation}\label{5.0}
\|\rho_t\|_{L^2}+\|u_t\|_{L^2}+\|\theta_t\|_{L^2}+\|\D\rho\|_{H^1}+\|\D u\|_{H^1} + \|\D\theta\|_{H^1}\leq C(\bar{\rho},\bar{\theta})(1+t)^{-\frac 34\left(\frac 2{p_0}-1\right)-\frac{1}{2}}.
\end{equation}
\end{lemma}
\begin{proof}
First, we rewrite \eqref{4.0} as follows,
\begin{equation}\label{5.1}
\frac{\dif}{\dif t}\calE(t)+\frac{K_0(\bar{\rho},\bar{\theta})}{2}\left(\|\D^2 u\|_{H^1}^2+\|\D^2\theta\|_{H^1}^2+\|\D^2\rho\|_{L^2}^2\right)+\frac{K_0(\bar{\rho},\bar{\theta})}{2}\|\D^2\rho\|_{L^2}^2\leq 0.
\end{equation}

Next, we denote the time-dependent domain
\begin{equation*}
\tilde{S}(t)=:\left\{\xi\in\R^3: ~ |\xi|^2\leq\frac{c_*(\bar{\rho},\bar{\theta})}{1+t}\right\},
\end{equation*}
with a constant $c_*(\bar{\rho},\bar{\theta})$ to be determined later, and split the whole space $\R^3$ into two time-dependent regions $\tilde{S}(t)$ and $\R^3\setminus\tilde{S}(t)$. Now we can estimate the second term in \eqref{5.1} as follows.
\begin{align}
 \label{5.2} &\|\D^2 u\|_{H^1}^2+\|\D^2\theta\|_{H^1}^2+\|\D^2\rho\|_{L^2}^2 =~\int_{\tilde{S}(t)}\left[(|\xi|^6+|\xi|^4)\left(|\widehat{u}|^2+|\widehat{\theta-\bar{\theta}}|^2\right)+|\xi|^4|\widehat{\rho-\bar{\rho}}|^2\right]\dif\xi\nonumber\\
 &+\int_{\R^3\setminus\tilde{S}(t)}\left[(|\xi|^6+|\xi|^4)\left(|\widehat{u}|^2+|\widehat{\theta-\bar{\theta}}|^2\right)+|\xi|^4|\widehat{\rho-\bar{\rho}}|^2\right]\dif\xi\nonumber\\
 \geq~&\frac{c_*(\bar{\rho},\bar{\theta})}{1+t}\int_{\R^3\setminus\tilde{S}(t)}\left[(|\xi|^4+|\xi|^2)\left(|\widehat{u}|^2+|\widehat{\theta-\bar{\theta}}|^2\right)+|\xi|^2|\widehat{\rho-\bar{\rho}}|^2\right]\dif\xi\nonumber\\
 =~&\frac{c_*(\bar{\rho},\bar{\theta})}{1+t}\int_{\R^3}\left[(|\xi|^4+|\xi|^2)\left(|\widehat{u}|^2+|\widehat{\theta-\bar{\theta}}|^2\right)+|\xi|^2|\widehat{\rho-\bar{\rho}}|^2\right]\dif\xi\nonumber\\
 &-\frac{c_*(\bar{\rho},\bar{\theta})}{1+t}\int_{\tilde{S}(t)}\left[(|\xi|^4+|\xi|^2)\left(|\widehat{u}|^2+|\widehat{\theta-\bar{\theta}}|^2\right)+|\xi|^2|\widehat{\rho-\bar{\rho}}|^2\right]\dif\xi \nonumber\\
 \geq~&\frac{c_*(\bar{\rho},\bar{\theta})}{1+t}\left(\|\D u\|_{H^1}^2+\|\D\theta\|_{H^1}^2+\|\D\rho\|_{L^2}^2\right) -\frac{c_*^2(\bar{\rho},\bar{\theta})}{(1+t)^2}\int_{\tilde{S}(t)}\left[(|\xi|^2+1)\left(|\widehat{u}|^2+|\widehat{\theta-\bar{\theta}}|^2\right)+|\widehat{\rho-\bar{\rho}}|^2\right]\dif\xi\nonumber\\
 \geq~&\frac{c_*(\bar{\rho},\bar{\theta})}{1+t}\left(\|\D u\|_{H^1}^2+\|\D\theta\|_{H^1}^2+\|\D\rho\|_{L^2}^2\right)-\frac{c_*^2(\bar{\rho},\bar{\theta})}{(1+t)^2}\left(\|u\|_{H^1}^2+\|\theta-\bar{\theta}\|_{H^1}^2+\|\rho-\bar{\rho}\|_{L^2}^2\right).
\end{align}
Substituting \eqref{5.2} into \eqref{5.1}, we obtain
\begin{align}\label{5.3}
&\frac{\dif}{\dif t}\calE(t)+\frac{K_0(\bar{\rho},\bar{\theta})}{2}\|\D^2\rho\|_{L^2}^2+\frac{K_0(\bar{\rho},\bar{\theta})c_*(\bar{\rho},\bar{\theta})}{2(1+t)}\left(\|\D u\|_{H^1}^2+\|\D\theta\|_{H^1}^2+\|\D\rho\|_{L^2}^2\right)\nonumber\\
\leq~&\frac{K_0(\bar{\rho},\bar{\theta})c_*^2(\bar{\rho},\bar{\theta})}{2(1+t)^2}\left(\|u\|_{H^1}^2+\|\theta-\bar{\theta}\|_{H^1}^2+\|\rho-\bar{\rho}\|_{L^2}^2\right).
\end{align}
Fix the existing constant $C(\bar{\rho},\bar{\theta})\geq1$ large enough that
\begin{equation*}
\calE(t)\leq C(\bar{\rho},\bar{\theta})\left(\|\D u\|_{H^1}^2+\|\D\theta\|_{H^1}^2+\|\D\rho\|_{L^2}^2+\|\D^2\rho\|_{L^2}^2\right).
\end{equation*}
Denote $T_2(\bar{\rho},\bar{\theta})=c_*(\bar{\rho},\bar{\theta})-1$. Then for any $t\geq\max\{1, T_1(\bar{\rho},\bar{\theta}), T_2(\bar{\rho},\bar{\theta})\}$, from \eqref{decay-0:U} we have
\begin{align}\label{5.4}
\frac{\dif}{\dif t}\calE(t)+\frac{K_0(\bar{\rho},\bar{\theta})c_*(\bar{\rho},\bar{\theta})}{2C(\bar{\rho},\bar{\theta})(1+t)}\calE(t)\leq~&\frac{C(\bar{\rho},\bar{\theta})}{(1+t)^2}\left(\|u\|_{H^1}^2+\|\theta-\bar{\theta}\|_{H^1}^2+\|\rho-\bar{\rho}\|_{L^2}^2\right)\nonumber\\
\leq~& C(\bar{\rho},\bar{\theta})(1+t)^{-\frac{3}{2}\left(\frac{2}{p_0}-1\right)-2}.
\end{align}
Taking $c_*(\bar{\rho},\bar{\theta})=\displaystyle{\frac{6C(\bar{\rho},\bar{\theta})}{K_0(\bar{\rho},\bar{\theta})}}$, we get
\begin{equation}\label{5.5}
\frac{\dif}{\dif t}\calE(t)+\frac{3\calE(t)}{1+t}\leq C(\bar{\rho},\bar{\theta})(1+t)^{-\frac{3}{2}\left(\frac{2}{p_0}-1\right)-2},
\end{equation}
from which we can easily obtain
\begin{equation}\label{decay-0.5}
\|\D\rho\|_{H^1}+\|\D u\|_{H^1} + \|\D\theta\|_{H^1}\leq C(\bar{\rho},\bar{\theta})(1+t)^{-\frac 34\left(\frac 2{p_0}-1\right)-\frac{1}{2}}.
\end{equation}

From \eqref{FCNS2}, we have
\begin{equation}\label{5.6}
\|\rho_t\|_{L^2}+\|u_t\|_{L^2}+\|\theta_t\|_{L^2}\leq C(\bar{\rho}, \bar{\theta})\left(\|\D\rho\|_{L^2}+\|\D u\|_{H^1}+\|\D\theta\|_{H^1}+\|\M_1\|_{L^2}+\|\M_2\|_{L^2}+\|\M_3\|_{L^2}\right).
\end{equation}
Therefore, \eqref{5.0} follows from \eqref{5.6}, \eqref{decay-0.5}, \eqref{1.2}, \eqref{1.5} and \eqref{1.8}.
\end{proof}

\subsection{Decay rate of $\|(\D^2\rho, \D^2u, \D^2\theta)\|_{L^2}$}

To estimate the second-order derivatives of the solution, we choose a radial function $\phi(\xi)\in C_c^\infty(\R^3)$, such that
\begin{equation}\label{cut-off}
0\leq\phi(\xi)\leq1,\qquad \phi(\xi)\equiv 0, \quad \text{for} \quad |\xi|\geq 2, \quad \text{and} \quad \phi(\xi)\equiv 1, \quad \text{for} \quad |\xi|\leq 1. 
\end{equation}
We define the low-frequency part and the high-frequency part of a given function $f(x)$ as follows,
\begin{equation}\label{HL}
f_L=:\mathcal{F}^{-1}\left(\phi(\xi)\widehat{f}(\xi)\right), 
 \quad \text{and} \quad f_H=:\mathcal{F}^{-1}\left((1-\phi(\xi))\widehat{f}(\xi)\right)=f-f_L.
\end{equation}

Next, we shall establish the $L^2$-decay rate of $(\D^2\rho, \D^2u, \D^2\theta)$ when $(\rho_0-\bar{\rho}, u_0, \theta_0-\bar{\theta})\in L^{p_0}$ for some $p_0\in[1,2]$. The following lemma is devoted to establishing energy estimates of the second-order derivatives of the solution.
\begin{lemma}\label{lem.10}
For the global solution $(\rho, u, \theta)$, we have
\begin{align}
 \label{10.0} &\frac{\dif}{\dif t}\left(R\bar{\rho}^{-2}\bar{\theta}\|(\D^2\rho)_L\|_{L^2}^2+\|(\D^2 u)_L\|_{L^2}^2+\bar{\theta}^{-1}\|(\D^2\theta)_L\|_{L^2}^2\right) +\mu\bar{\rho}^{-1}\bar{\theta}^\al\|(\D^3 u)_L\|_{L^2}^2+\kappa\bar{\rho}^{-1}\bar{\theta}^{\beta-1}\|(\D^3\theta)_L\|_{L^2}^2\nonumber\\
 \leq~&C(\bar{\rho}, \bar{\theta})\varepsilon_0^{-1}(\bar{\rho},\bar{\theta})\|(\D\rho, \D u, \D\theta)\|_{H^1}\left(\|(\D^2\rho)_H\|_{L^2}^2+\|(\D^3\rho)_L\|_{L^2}^2+\|\D^3u\|_{L^2}^2+\|\D^3\theta\|_{L^2}^2\right)\nonumber\\
 &+\varepsilon_0(\bar{\rho},\bar{\theta})\|(\D^3\rho)_L\|_{L^2}^2,
\end{align}
and
\begin{align}\label{10.1}
&\frac{\dif}{\dif t}\left(R\bar{\rho}^{-2}\bar{\theta}\|(\D^2\rho)_H\|_{L^2}^2+\|(\D^2 u)_H\|_{L^2}^2+\bar{\theta}^{-1}\|(\D^2\theta)_H\|_{L^2}^2\right)\nonumber\\
&+\mu\bar{\rho}^{-1}\bar{\theta}^\al\|(\D^3 u)_H\|_{L^2}^2+\kappa\bar{\rho}^{-1}\bar{\theta}^{\beta-1}\|(\D^3\theta)_H\|_{L^2}^2\nonumber\\
\leq~&C(\bar{\rho}, \bar{\theta})\varepsilon_0^{-1}(\bar{\rho},\bar{\theta})\|(\D\rho, \D u, \D\theta)\|_{H^1}\left(\|(\D^2\rho)_H\|_{L^2}^2+\|(\D^3\rho)_L\|_{L^2}^2+\|\D^3u\|_{L^2}^2+\|\D^3\theta\|_{L^2}^2\right)\nonumber\\
&+\varepsilon_0(\bar{\rho},\bar{\theta})\|(\D^2\rho)_H\|_{L^2}^2,
\end{align}
with some positive constant $\varepsilon_0$ depending on $\bar{\rho}$, $\bar{\theta}$ to be determined later.
\end{lemma}
\begin{proof}
Applying $\D^2$ operator to \eqref{FCNS2}, then taking the low-frequency part, then taking the $L^2$-inner product with $(R\bar{\rho}^{-2}\bar{\theta}\D^2\rho, \D^2u, \bar{\theta}^{-1}\D^2\theta)_L$, we obtain
\begin{align}\label{10.2}
&\frac{1}{2}\frac{\dif}{\dif t}\left(R\bar{\rho}^{-2}\bar{\theta}\|(\D^2\rho)_L\|_{L^2}^2+\|(\D^2 u)_L\|_{L^2}^2+\bar{\theta}^{-1}\|(\D^2\theta)_L\|_{L^2}^2\right)\nonumber\\
&+\mu\bar{\rho}^{-1}\bar{\theta}^\al\|(\D^3u)_L\|_{L^2}^2+(\mu+\lambda)\bar{\rho}^{-1}\bar{\theta}^\al\|(\D^2\div u)_L\|_{L^2}^2+\kappa\bar{\rho}^{-1}\bar{\theta}^{\beta-1}\|(\D^3\theta)_L\|_{L^2}^2\nonumber\\
=~&R\bar{\rho}^{-2}\bar{\theta}\int(\D\M_1)_L\cdot(\D\Delta\rho)_L\dif x+\int(\D\M_2)_L:(\D\Delta u)_L\dif x+\bar{\theta}^{-1}\int(\D\M_3)_L\cdot(\D\Delta \theta)_L\dif x\nonumber\\
\leq~&\varepsilon_0(\bar{\rho},\bar{\theta})\|(\D^3\rho)_L\|_{L^2}^2+\frac{\mu\bar{\theta}^\al}{2\bar{\rho}}\|(\D^3u)_L\|_{L^2}^2+\frac{\kappa\bar{\theta}^\beta}{2\bar{\rho}\bar{\theta}}\|(\D^3\theta)_L\|_{L^2}^2\nonumber\\
&+C(\bar{\rho},\bar{\theta})\left(\varepsilon_0^{-1}(\bar{\rho},\bar{\theta})\|(\nabla\M_1)_L\|_{L^2}^2+\|(\nabla\M_2)_L\|_{L^2}^2+\|(\nabla\M_3)_L\|_{L^2}^2\right).
\end{align}
Recalling the definition of $\M_i ~ (i=1,2,3)$ in \eqref{M123}, by Gagliardo-Nirenberg inequality and Sobolev embedding inequalities, we have
\begin{align}\label{10.3}
\|\nabla\M_1\|_{L^2}^2\leq~&C(\bar{\rho},\bar{\theta})\left(\|\D u\cdot\D\rho\|_{L^2}^2+\|u\D^2\rho\|_{L^2}^2+\|(\rho-\bar{\rho})\D^2u\|_{L^2}^2\right) \nonumber\\
\leq~&C(\bar{\rho},\bar{\theta})\Big(\|\D u\cdot(\D\rho)_H\|_{L^2}^2+\|\D u\cdot(\D\rho)_L\|_{L^2}^2+\|u(\D^2\rho)_H\|_{L^2}^2\nonumber\\
&+\|u(\D^2\rho)_L\|_{L^2}^2+\|(\rho-\bar{\rho})\D^2u\|_{L^2}^2\Big) \nonumber\\
\leq~&C(\bar{\rho},\bar{\theta})\Big(\|\D u\|_{L^3}^2\|(\D\rho)_H\|_{L^6}^2+\|\D u\|_{L^3}^2\|(\D\rho)_L\|_{L^6}^2+\|u\|_{L^\infty}^2\|(\D^2\rho)_H\|_{L^2}^2 \nonumber\\
&+\|u\|_{L^3}^2\|(\D^2\rho)_L\|_{L^6}^2+\|\rho-\bar{\rho}\|_{L^3}^2\|\D^2u\|_{L^6}^2\Big) \nonumber\\
\leq~&C(\bar{\rho},\bar{\theta})\Big(\|\D u\|_{H^1}^2\|(\D^2\rho)_H\|_{L^2}^2+\|u\|_{L^2}\|\D^3u\|_{L^2}\|(\D\rho)_L\|_{L^2}\|(\D^3\rho)_L\|_{L^2} \nonumber\\
&+\|u\|_{L^2}\|\D u\|_{L^2}\|(\D^3\rho)_L\|_{L^2}^2+\|\rho-\bar{\rho}\|_{L^2}\|\D\rho\|_{L^2}\|\D^3u\|_{L^2}^2\Big) \nonumber\\
\leq~&C(\bar{\rho},\bar{\theta})\|(\D\rho, \D u)\|_{H^1}\left(\|(\D^2\rho)_H\|_{L^2}^2+\|(\D^3\rho)_L\|_{L^2}^2+\|\D^3u\|_{L^2}^2\right),
\end{align}
and
\begin{align}\label{10.4}
\|\nabla\M_2\|_{L^2}^2\leq~&\sum\limits_{i=1}^{16}H_i+C(\bar{\rho},\bar{\theta})\Big(\|(\rho-\bar{\rho})\D^2\rho\|_{L^2}^2+\|(\theta-\bar{\theta})\D^2\rho\|_{L^2}^2+\|\D u\cdot\D^2\theta\|_{L^2}^2\Big),\nonumber\\
\sum\limits_{i=1}^{16}H_i=:~&C(\bar{\rho},\bar{\theta})\Big(\|\D u\cdot\D u\|_{L^2}^2+\|u\D^2u\|_{L^2}^2+\|\D\rho\cdot\D\theta\|_{L^2}^2+\|\D\rho\cdot\D\rho\|_{L^2}^2+\|(\rho-\bar{\rho})\D\rho\cdot\D\rho\|_{L^2}^2 \nonumber\\
&+\|(\theta-\bar{\theta})\D\rho\cdot\D\rho\|_{L^2}^2+\|\D\rho\D^2u\|_{L^2}^2+\|\D\theta\D^2u\|_{L^2}^2+\|(\rho-\bar{\rho})\D^3u\|_{L^2}^2+\|(\theta-\bar{\theta})\D^3u\|_{L^2}^2 \nonumber\\
&+\|\D\rho\cdot\D u\cdot\D\theta\|_{L^2}^2+\|\D\theta\cdot\D u\cdot\D\theta\|_{L^2}^2+\|(\rho-\bar{\rho})\D^2u\cdot\D\theta\|_{L^2}^2+\|(\rho-\bar{\rho})\D u\cdot\D^2\theta\|_{L^2}^2 \nonumber\\
&+\|(\theta-\bar{\theta})\D^2u\cdot\D\theta\|_{L^2}^2+\|(\theta-\bar{\theta})\D u\cdot\D^2\theta\|_{L^2}^2\Big),
\end{align}
where $H_1,\ldots,H_{16}$ denote the sixteen displayed terms, including their coefficients. They can be estimated as follows.
\begin{align}
 \label{10.5}
H_1+H_2&\leq C(\bar{\rho},\bar{\theta})\left(\|\D u\|_{L^4}^4+\|u\|_{L^\infty}^2\|\D^2u\|_{L^2}^2\right) \leq~C(\bar{\rho},\bar{\theta})\|\D u\|_{L^2}\|\D^2u\|_{L^2}^3 \nonumber\\
 &\leq~C(\bar{\rho},\bar{\theta})\|u\|_{L^2}\|\D u\|_{L^2}\|\D^3u\|_{L^2}^2 \leq~C(\bar{\rho},\bar{\theta})\|\D u\|_{L^2}\|\D^3u\|_{L^2}^2,
\end{align}
\begin{align}
 \label{10.6} H_3\leq~& C(\bar{\rho},\bar{\theta})\left(\|(\D\rho)_L\cdot\D\theta\|_{L^2}^2+\|(\D\rho)_H\cdot\D\theta\|_{L^2}^2\right) \leq~C(\bar{\rho},\bar{\theta})\left(\|\D\theta\|_{L^3}^2\|(\D\rho)_L\|_{L^6}^2+\|\D\theta\|_{L^3}^2\|(\D\rho)_H\|_{L^6}^2\right)\nonumber\\
 \leq~&C(\bar{\rho},\bar{\theta})\left(\|\D\theta\|_{L^2}\|\D^2\theta\|_{L^2}\|(\D^2\rho)_L\|_{L^2}^2+\|\D\theta\|_{H^1}^2\|(\D^2\rho)_H\|_{L^2}^2\right) \nonumber\\
 \leq~&C(\bar{\rho},\bar{\theta})\left(\|\theta-\bar{\theta}\|_{L^2}^{\frac{1}{3}}\|\D\theta\|_{L^2}\|\D^3\theta\|_{L^2}^{\frac{2}{3}}\|(\rho-\bar{\rho})_L\|_{L^2}^{\frac{2}{3}}\|(\D^3\rho)_L\|_{L^2}^{\frac{4}{3}}+\|\D\theta\|_{H^1}^2\|(\D^2\rho)_H\|_{L^2}^2\right) \nonumber\\
 \leq~&C(\bar{\rho},\bar{\theta})\|\D\theta\|_{H^1}\left(\|(\D^3\rho)_L\|_{L^2}^2+\|(\D^2\rho)_H\|_{L^2}^2+\|\D^3\theta\|_{L^2}^2\right),
\end{align}
\begin{align}\label{10.7}
H_4+H_5+H_6\leq C(\bar{\rho}, \bar{\theta})\|\D\rho\cdot\D\rho\|_{L^2}^2\leq~& C(\bar{\rho},\bar{\theta})\left(\|(\D\rho)_H\cdot\D\rho\|_{L^2}^2+\|(\D\rho)_L\cdot(\D\rho)_L\|_{L^2}^2\right) \nonumber\\
\leq~&C(\bar{\rho},\bar{\theta})\left(\|\D\rho\|_{L^3}^2\|(\D\rho)_H\|_{L^6}^2+\|(\D\rho)_L\|_{L^4}^4\right) \nonumber\\
\leq~&C(\bar{\rho},\bar{\theta})\left(\|\D\rho\|_{H^1}^2\|(\D^2\rho)_H\|_{L^2}^2+\|(\D\rho)_L\|_{L^2}\|(\D^2\rho)_L\|_{L^2}^3\right) \nonumber\\
\leq~&C(\bar{\rho},\bar{\theta})\|\D\rho\|_{H^1}\left(\|(\D^3\rho)_L\|_{L^2}^2+\|(\D^2\rho)_H\|_{L^2}^2\right),
\end{align}
\begin{align}\label{10.8}
H_7+H_8+H_9+H_{10}\leq~&C(\bar{\rho}, \bar{\theta})\left(\|(\D\rho, \D\theta)\|_{L^3}^2\|\D^2u\|_{L^6}^2+\|(\rho-\bar{\rho}, \theta-\bar{\theta})\|_{L^\infty}^2\|\D^3u\|_{L^2}^2\right) \nonumber\\
\leq~&C(\bar{\rho}, \bar{\theta})\|(\D\rho, \D\theta)\|_{H^1}^2\|\D^3u\|_{L^2}^2,
\end{align}
\begin{align}
 \label{10.9} H_{11}+H_{12}\leq~&C(\bar{\rho}, \bar{\theta})\|(\D\rho, \D\theta)\|_{L^6}^2\|\D u\|_{L^6}^2\|\D\theta\|_{L^6}^2 \leq~C(\bar{\rho}, \bar{\theta})\|(\D\rho, \D\theta)\|_{H^1}^2\|\D^2u\|_{L^2}^2\|\D^2\theta\|_{L^2}^2\nonumber\\
 \leq~&C(\bar{\rho}, \bar{\theta})\|(\D\rho, \D\theta)\|_{H^1}^2\|\D u\|_{L^2}\|\D\theta\|_{L^2}\|\D^3u\|_{L^2}\|\D^3\theta\|_{L^2} \nonumber\\
 \leq~&C(\bar{\rho}, \bar{\theta})\|(\D\rho, \D\theta)\|_{H^1}^2\left(\|\D^3u\|_{L^2}^2+\|\D^3\theta\|_{L^2}^2\right),
\end{align}
\begin{align}\label{10.10}
H_{13}+H_{14}+H_{15}+H_{16}\leq~&C(\bar{\rho}, \bar{\theta})\|(\rho-\bar{\rho}, \theta-\bar{\theta})\|_{L^6}^2\left(\|\D^2u\|_{L^6}^2\|\D\theta\|_{L^6}^2+\|\D u\|_{L^6}^2\|\D^2\theta\|_{L^6}^2\right) \nonumber\\
\leq~&C(\bar{\rho}, \bar{\theta})\|(\D\rho, \D\theta)\|_{L^2}^2\left(\|\D^3u\|_{L^2}^2+\|\D^3\theta\|_{L^2}^2\right).
\end{align}

The three additional terms in \eqref{10.4} satisfy
\begin{align*}
&\|(\rho-\bar{\rho})\D^2\rho\|_{L^2}^2+\|(\theta-\bar{\theta})\D^2\rho\|_{L^2}^2\\
\leq~&C(\bar{\rho},\bar{\theta})\|(\D\rho,\D\theta)\|_{H^1}^2\|(\D^2\rho)_H\|_{L^2}^2
+C\Big(\|(\rho-\bar{\rho})_L\|_{L^\infty}^2+\|\theta-\bar{\theta}\|_{L^\infty}^2\Big)\|(\D^2\rho)_L\|_{L^2}^2\\
\leq~&C(\bar{\rho},\bar{\theta})\|(\D\rho,\D\theta)\|_{H^1}^2\|(\D^2\rho)_H\|_{L^2}^2
+C\|\D\rho\|_{L^2}\|(\rho-\bar{\rho})_L\|_{L^2}\|(\D^3\rho)_L\|_{L^2}^2\\
&+C\|\D\theta\|_{L^2}\|\theta-\bar{\theta}\|_{L^2}^{\frac13}\|(\rho-\bar{\rho})_L\|_{L^2}^{\frac23}\|\D^3\theta\|_{L^2}^{\frac23}\|(\D^3\rho)_L\|_{L^2}^{\frac43}\\
\leq~&C(\bar{\rho},\bar{\theta})\|(\D\rho,\D\theta)\|_{H^1}
\left(\|(\D^2\rho)_H\|_{L^2}^2+\|(\D^3\rho)_L\|_{L^2}^2+\|\D^3\theta\|_{L^2}^2\right),\\
&\|\D u\cdot\D^2\theta\|_{L^2}^2
\leq C\|\D u\|_{L^3}^2\|\D^2\theta\|_{L^6}^2
\leq C(\bar{\rho},\bar{\theta})\|\D u\|_{H^1}\|\D^3\theta\|_{L^2}^2.
\end{align*}
Here we used the uniform $H^2$ bounds, the high-frequency estimate
$\|(\rho-\bar{\rho})_H\|_{L^\infty}\leq C\|(\D^2\rho)_H\|_{L^2}$,
and the interpolation inequalities
$\|f\|_{L^\infty}^2\leq C\|\D f\|_{L^2}\|\D^2f\|_{L^2}$ and
$\|\D^2f\|_{L^2}^3\leq C\|f\|_{L^2}\|\D^3f\|_{L^2}^2$.

Similar to \eqref{10.4}--\eqref{10.10}, we have
\begin{align}\label{10.11}
\|\nabla\M_3\|_{L^2}^2\leq~&C(\bar{\rho},\bar{\theta})\Big(\|\D u\cdot\D\theta\|_{L^2}^2+\|u\D^2\theta\|_{L^2}^2+\|(\theta-\bar{\theta})\D^2u\|_{L^2}^2+\|\D\rho\cdot\D u\cdot\D u\|_{L^2}^2 \nonumber\\
&+\|\D u\cdot\D\theta\cdot\D u\|_{L^2}^2+\|\D u\cdot\D^2u\|_{L^2}^2+\|(\rho-\bar{\rho})\D^3\theta\|_{L^2}^2+\|(\theta-\bar{\theta})\D^3\theta\|_{L^2}^2+\|\D\rho\cdot\D^2\theta\|_{L^2}^2 \nonumber\\
&+\|\D\theta\cdot\D^2\theta\|_{L^2}^2+\|\D\rho\cdot\D\theta\cdot\D\theta\|_{L^2}^2+\|\D\theta\cdot\D\theta\cdot\D\theta\|_{L^2}^2\Big) \nonumber\\
\leq~&C(\bar{\rho}, \bar{\theta})\|(\D\rho, \D u, \D\theta)\|_{H^1}\left(\|(\D^2\rho)_H\|_{L^2}^2+\|(\D^3\rho)_L\|_{L^2}^2+\|\D^3u\|_{L^2}^2+\|\D^3\theta\|_{L^2}^2\right).
\end{align}
Inserting \eqref{10.3}--\eqref{10.11} into \eqref{10.2}, and first using $\varepsilon_0/2$ in place of $\varepsilon_0$ in Young's inequality, we have proved \eqref{10.0}.

Applying $\D^2$ operator to \eqref{FCNS2}, then taking the high-frequency part, then taking the $L^2$-inner product with $(R\bar{\rho}^{-2}\bar{\theta}\D^2\rho, \D^2u, \bar{\theta}^{-1}\D^2\theta)_H$, we obtain
\begin{align}\label{10.12}
&\frac{1}{2}\frac{\dif}{\dif t}\left(R\bar{\rho}^{-2}\bar{\theta}\|(\D^2\rho)_H\|_{L^2}^2+\|(\D^2 u)_H\|_{L^2}^2+\bar{\theta}^{-1}\|(\D^2\theta)_H\|_{L^2}^2\right)\nonumber\\
&+\mu\bar{\rho}^{-1}\bar{\theta}^\al\|(\D^3u)_H\|_{L^2}^2+(\mu+\lambda)\bar{\rho}^{-1}\bar{\theta}^\al\|(\D^2\div u)_H\|_{L^2}^2+\kappa\bar{\rho}^{-1}\bar{\theta}^{\beta-1}\|(\D^3\theta)_H\|_{L^2}^2\nonumber\\
=~&-R\bar{\rho}^{-2}\bar{\theta}\int(\D^2\M_1)_H:(\D^2\rho)_H\dif x+\int(\D\M_2)_H:(\D\Delta u)_H\dif x+\bar{\theta}^{-1}\int(\D\M_3)_H\cdot(\D\Delta \theta)_H\dif x\nonumber\\
\leq~&\frac{\mu\bar{\theta}^\al}{2\bar{\rho}}\|(\D^3u)_H\|_{L^2}^2+\frac{\kappa\bar{\theta}^\beta}{2\bar{\rho}\bar{\theta}}\|(\D^3\theta)_H\|_{L^2}^2+C(\bar{\rho},\bar{\theta})\left(\|(\nabla\M_2)_H\|_{L^2}^2+\|(\nabla\M_3)_H\|_{L^2}^2\right) \nonumber\\
&-R\bar{\rho}^{-2}\bar{\theta}\int(\D^2\M_1)_H:(\D^2\rho)_H\dif x,
\end{align}
where we only need to estimate the last term on the right-hand side. Recalling the definition of $\M_1$ in \eqref{M123}, we have
\begin{align}\label{10.12.1}
&-R\bar{\rho}^{-2}\bar{\theta}\int(\D^2\M_1)_H:(\D^2\rho)_H\dif x \nonumber\\
=~&-R\bar{\rho}^{-2}\bar{\theta}\int\D^2(u\cdot\D\rho)_H:(\D^2\rho)_H\dif x-R\bar{\rho}^{-2}\bar{\theta}\int\D^2((\rho-\bar{\rho})\div u)_H:(\D^2\rho)_H\dif x \nonumber\\
=~&-R\bar{\rho}^{-2}\bar{\theta}\int\D^2(u\cdot\D\rho):(\D^2\rho)_H\dif x+R\bar{\rho}^{-2}\bar{\theta}\int\D^2(u\cdot\D\rho)_L:(\D^2\rho)_H\dif x \nonumber\\
&-R\bar{\rho}^{-2}\bar{\theta}\int\D^2((\rho-\bar{\rho})\div u)_H:(\D^2\rho)_H\dif x \nonumber\\
=~&-R\bar{\rho}^{-2}\bar{\theta}\int\D^2(u\cdot(\D\rho)_L):(\D^2\rho)_H\dif x-R\bar{\rho}^{-2}\bar{\theta}\int\D^2(u\cdot(\D\rho)_H):(\D^2\rho)_H\dif x\nonumber\\
&+R\bar{\rho}^{-2}\bar{\theta}\int\D^2(u\cdot\D\rho)_L:(\D^2\rho)_H\dif x-R\bar{\rho}^{-2}\bar{\theta}\int\D^2((\rho-\bar{\rho})\div u)_H:(\D^2\rho)_H\dif x,
\end{align}
where the second term on the right-hand side can be rewritten as follows,
\begin{align}\label{10.12.2}
&-R\bar{\rho}^{-2}\bar{\theta}\int\D^2(u\cdot(\D\rho)_H):(\D^2\rho)_H\dif x\nonumber\\
=~&-R\bar{\rho}^{-2}\bar{\theta}\int\D^2u\cdot(\D\rho)_H\cdot(\D^2\rho)_H\dif x-2R\bar{\rho}^{-2}\bar{\theta}\int\D u\cdot(\D^2\rho)_H:(\D^2\rho)_H\dif x\nonumber\\
&-R\bar{\rho}^{-2}\bar{\theta}\int(u\cdot\D)(\D^2\rho)_H:(\D^2\rho)_H\dif x\nonumber\\
=~&-R\bar{\rho}^{-2}\bar{\theta}\int\D^2u\cdot(\D\rho)_H\cdot(\D^2\rho)_H\dif x-2R\bar{\rho}^{-2}\bar{\theta}\int\D u\cdot(\D^2\rho)_H:(\D^2\rho)_H\dif x\nonumber\\
&+\frac{R\bar{\rho}^{-2}\bar{\theta}}{2}\int\div u|(\D^2\rho)_H|^2\dif x.
\end{align}

Substituting \eqref{10.12.2} into \eqref{10.12.1}, by Gagliardo-Nirenberg inequality and Sobolev embedding inequalities, we have
\begin{align}\label{10.13}
&-R\bar{\rho}^{-2}\bar{\theta}\int(\D^2\M_1)_H:(\D^2\rho)_H\dif x \nonumber\\
=~&-R\bar{\rho}^{-2}\bar{\theta}\int\D^2(u\cdot(\D\rho)_L):(\D^2\rho)_H\dif x-R\bar{\rho}^{-2}\bar{\theta}\int\D^2u\cdot(\D\rho)_H\cdot(\D^2\rho)_H\dif x\nonumber\\
&-2R\bar{\rho}^{-2}\bar{\theta}\int\D u\cdot(\D^2\rho)_H:(\D^2\rho)_H\dif x+\frac{R\bar{\rho}^{-2}\bar{\theta}}{2}\int\div u|(\D^2\rho)_H|^2\dif x\nonumber\\
&+R\bar{\rho}^{-2}\bar{\theta}\int\D^2(u\cdot\D\rho)_L:(\D^2\rho)_H\dif x-R\bar{\rho}^{-2}\bar{\theta}\int\D^2((\rho-\bar{\rho})\div u)_H:(\D^2\rho)_H\dif x \nonumber\\
\leq~&\varepsilon_0(\bar{\rho},\bar{\theta})\|(\D^2\rho)_H\|_{L^2}^2+C(\bar{\rho}, \bar{\theta})\|\D u\|_{L^\infty}\|(\D^2\rho)_H\|_{L^2}^2+C(\bar{\rho}, \bar{\theta})\varepsilon_0^{-1}(\bar{\rho},\bar{\theta})\Big(\|\D^2(u\cdot(\D\rho)_L)\|_{L^2}^2 \nonumber\\
&+\|(\D\rho)_H\cdot\D^2u\|_{L^2}^2+\|\D^2(u\cdot\D\rho)_L\|_{L^2}^2+\|\D^2((\rho-\bar{\rho})\div u))\|_{L^2}^2\Big) \nonumber \\
\leq~&\varepsilon_0(\bar{\rho},\bar{\theta})\|(\D^2\rho)_H\|_{L^2}^2+C(\bar{\rho}, \bar{\theta})\|\D u\|_{L^\infty}\|(\D^2\rho)_H\|_{L^2}^2+C(\bar{\rho}, \bar{\theta})\varepsilon_0^{-1}(\bar{\rho},\bar{\theta})\Big(\|u(\D^3\rho)_L\|_{L^2}^2 \nonumber\\
&+\|\D\rho\cdot\D^2 u\|_{L^2}^2+\|\D u\cdot(\D^2\rho)_L\|_{L^2}^2+\|\D(u\cdot\D\rho)\|_{L^2}^2+\|\D^2\rho\cdot\D u\|_{L^2}^2+\|(\rho-\bar{\rho})\D^3u\|_{L^2}^2\Big) \nonumber \\
\leq~&\varepsilon_0(\bar{\rho},\bar{\theta})\|(\D^2\rho)_H\|_{L^2}^2+C(\bar{\rho}, \bar{\theta})\|\D u\|_{L^\infty}\|(\D^2\rho)_H\|_{L^2}^2+C(\bar{\rho}, \bar{\theta})\varepsilon_0^{-1}(\bar{\rho},\bar{\theta})\Big(\|u\|_{L^\infty}^2\|(\D^3\rho)_L\|_{L^2}^2 \nonumber\\
&+\|\D\rho\|_{L^3}^2\|\D^2u\|_{L^6}^2+\|\D u\|_{L^3}^2\|(\D^2\rho)_L\|_{L^6}^2+\|\D\rho\cdot\D u\|_{L^2}^2+\|u\D^2\rho\|_{L^2}^2 \nonumber\\
&+\|(\D^2\rho)_H\cdot\D u\|_{L^2}^2+\|\rho-\bar{\rho}\|_{L^\infty}^2\|\D^3u\|_{L^2}^2\Big) \nonumber \\
\leq~&\varepsilon_0(\bar{\rho},\bar{\theta})\|(\D^2\rho)_H\|_{L^2}^2+C(\bar{\rho}, \bar{\theta})\varepsilon_0^{-1}(\bar{\rho},\bar{\theta})\|(\D\rho, \D u)\|_{H^1}\left(\|(\D^3\rho)_L\|_{L^2}^2+\|(\D^2\rho)_H\|_{L^2}^2+\|\D^3u\|_{L^2}^2\right) \nonumber\\
&+C(\bar{\rho}, \bar{\theta})\varepsilon_0^{-1}(\bar{\rho},\bar{\theta})\Big(\|\D\rho\cdot\D u\|_{L^2}^2+\|u(\D^2\rho)_L\|_{L^2}^2+\|u(\D^2\rho)_H\|_{L^2}^2+\|(\D^2\rho)_H\|_{L^2}^2\|\D u\|_{L^\infty}^2\Big) \nonumber \\
\leq~&\varepsilon_0(\bar{\rho},\bar{\theta})\|(\D^2\rho)_H\|_{L^2}^2+C(\bar{\rho}, \bar{\theta})\varepsilon_0^{-1}(\bar{\rho},\bar{\theta})\|(\D\rho, \D u)\|_{H^1}\left(\|(\D^3\rho)_L\|_{L^2}^2+\|(\D^2\rho)_H\|_{L^2}^2+\|\D^3u\|_{L^2}^2\right) \nonumber\\
&+C(\bar{\rho}, \bar{\theta})\varepsilon_0^{-1}(\bar{\rho},\bar{\theta})\Big(\|(\D\rho)_L\cdot\D u\|_{L^2}^2+\|(\D\rho)_H\cdot\D u\|_{L^2}^2+\|u\|_{L^3}^2\|(\D^2\rho)_L\|_{L^6}^2 \nonumber\\
&+\|u\|_{L^\infty}^2\|(\D^2\rho)_H\|_{L^2}^2+\|(\D^2\rho)_H\|_{L^2}^2\|\D^2u\|_{L^2}\|\D^3u\|_{L^2}\Big) \nonumber \\
\leq~&\varepsilon_0(\bar{\rho},\bar{\theta})\|(\D^2\rho)_H\|_{L^2}^2+C(\bar{\rho}, \bar{\theta})\varepsilon_0^{-1}(\bar{\rho},\bar{\theta})\|(\D\rho, \D u)\|_{H^1}\left(\|(\D^3\rho)_L\|_{L^2}^2+\|(\D^2\rho)_H\|_{L^2}^2+\|\D^3u\|_{L^2}^2\right).
\end{align}
In particular, the transport multiplier in \eqref{10.13} is controlled by
\begin{align*}
\|\D u\|_{L^\infty}\|(\D^2\rho)_H\|_{L^2}^2
\leq~&C\|\D u\|_{L^2}^{\frac14}\|\D^3u\|_{L^2}^{\frac34}\|(\D^2\rho)_H\|_{L^2}^2\\
\leq~&C\|\D u\|_{L^2}^{\frac14}\|(\D^2\rho)_H\|_{L^2}^{\frac34}
\left(\|\D^3u\|_{L^2}^2+\|(\D^2\rho)_H\|_{L^2}^2\right)\\
\leq~&C\|(\D\rho,\D u)\|_{H^1}\left(\|\D^3u\|_{L^2}^2+\|(\D^2\rho)_H\|_{L^2}^2\right).
\end{align*}
Combining \eqref{10.12} and \eqref{10.13}, and first using $\varepsilon_0/2$ in place of $\varepsilon_0$ in \eqref{10.13}, we have proved \eqref{10.1}.
\end{proof}

Next, we will establish dissipative estimates of $(\D^3\rho)_L$ and $(\D^2\rho)_H$.
\begin{lemma}\label{lem.11}
For the global solution $(\rho, u, \theta)$, we have
\begin{align}\label{11.0}
&\frac{\dif}{\dif t}\int(\D^2u)_L\cdot(\D^3\rho)_L\dif x+\frac{R\bar{\theta}}{2\bar{\rho}}\|(\D^3\rho)_L\|_{L^2}^2 \nonumber\\
\leq~&C(\bar{\rho}, \bar{\theta})\|(\D\rho, \D u, \D\theta)\|_{H^1}\left(\|(\D^2\rho)_H\|_{L^2}^2+\|(\D^3\rho)_L\|_{L^2}^2+\|\D^3u\|_{L^2}^2+\|\D^3\theta\|_{L^2}^2\right) \nonumber\\
&+\overline{C}_3(\bar{\rho}, \bar{\theta})\left(\|(\D^3u)_L\|_{L^2}^2+\|(\D^3\theta)_L\|_{L^2}^2\right),
\end{align}
and
\begin{align}\label{11.1}
&\frac{\dif}{\dif t}\int(\D u)_H:(\D^2\rho)_H\dif x+\frac{R\bar{\theta}}{2\bar{\rho}}\|(\D^2\rho)_H\|_{L^2}^2 \nonumber\\
\leq~&C(\bar{\rho}, \bar{\theta})\|(\D\rho, \D u, \D\theta)\|_{H^1}\left(\|(\D^2\rho)_H\|_{L^2}^2+\|(\D^3\rho)_L\|_{L^2}^2+\|\D^3u\|_{L^2}^2+\|\D^3\theta\|_{L^2}^2\right) \nonumber\\
&+\overline{C}_3(\bar{\rho}, \bar{\theta})\left(\|(\D^3u)_H\|_{L^2}^2+\|(\D^3\theta)_H\|_{L^2}^2\right),
\end{align}
for some positive constant $\overline{C}_3$ depending on $\bar{\rho}$ and $\bar{\theta}$.
\end{lemma}
\begin{proof}
Applying $\D^2$ operator to \eqref{FCNS2}$_2$, then taking the low-frequency part, then taking the $L^2$-inner product with $(\D^3\rho)_L$, we obtain
\begin{align}\label{11.3}
&\frac{\dif}{\dif t}\int(\D^2u)_L\cdot(\D^3\rho)_L\dif x+R\bar{\rho}^{-1}\bar{\theta}\|(\D^3\rho)_L\|_{L^2}^2\nonumber\\
=~&-R\int(\D^3\theta)_L:(\D^3\rho)_L\dif x+\mu\bar{\rho}^{-1}\bar{\theta}^\al\int(\D^2\Delta u)_L\cdot(\D^3\rho)_L\dif x+(\mu+\lambda)\bar{\rho}^{-1}\bar{\theta}^\al\int(\D^3\div u)_L\cdot(\D^3\rho)_L\dif x\nonumber\\
&+\bar{\rho}\|(\D^2\div u)_L\|_{L^2}^2+\int(\D^2\M_1)_L\cdot(\D\Delta u)_L\dif x-\int(\D^2\M_2)_L\cdot(\D^3\rho)_L\dif x\nonumber\\
\leq~&\frac{R\bar{\theta}}{2\bar{\rho}}\|(\D^3\rho)_L\|_{L^2}^2+C(\bar{\rho}, \bar{\theta})\left(\|(\D^3u)_L\|_{L^2}^2+\|(\D^3\theta)_L\|_{L^2}^2\right)+C(\bar{\rho}, \bar{\theta})\left(\|\D\M_1\|_{L^2}^2+\|\D\M_2\|_{L^2}^2\right).
\end{align}
Substituting \eqref{10.3}--\eqref{10.10} into \eqref{11.3}, we complete the proof of \eqref{11.0}.

Applying $\D$ operator to \eqref{FCNS2}$_2$, then taking the high-frequency part, then taking the $L^2$-inner product with $(\D^2\rho)_H$, we obtain
\begin{align}\label{11.2}
&\frac{\dif}{\dif t}\int(\D u)_H:(\D^2\rho)_H\dif x+R\bar{\rho}^{-1}\bar{\theta}\|(\D^2\rho)_H\|_{L^2}^2\nonumber\\
=~&-R\int(\D^2\theta)_H:(\D^2\rho)_H\dif x+\mu\bar{\rho}^{-1}\bar{\theta}^\al\int(\D\Delta u)_H:(\D^2\rho)_H\dif x+(\mu+\lambda)\bar{\rho}^{-1}\bar{\theta}^\al\int(\D^2\div u)_H:(\D^2\rho)_H\dif x\nonumber\\
&+\bar{\rho}\|(\D\div u)_H\|_{L^2}^2+\int(\D\M_1)_H\cdot(\Delta u)_H\dif x-\int(\D\M_2)_H:(\D^2\rho)_H\dif x\nonumber\\
\leq~&\frac{R\bar{\theta}}{2\bar{\rho}}\|(\D^2\rho)_H\|_{L^2}^2+C(\bar{\rho}, \bar{\theta})\left(\|(\D^2u)_H\|_{H^1}^2+\|(\D^2\theta)_H\|_{L^2}^2\right)+C(\bar{\rho}, \bar{\theta})\left(\|\D\M_1\|_{L^2}^2+\|\D\M_2\|_{L^2}^2\right)\nonumber\\
\leq~&\frac{R\bar{\theta}}{2\bar{\rho}}\|(\D^2\rho)_H\|_{L^2}^2+C(\bar{\rho}, \bar{\theta})\left(\|(\D^3u)_H\|_{L^2}^2+\|(\D^3\theta)_H\|_{L^2}^2\right)+C(\bar{\rho}, \bar{\theta})\left(\|\D\M_1\|_{L^2}^2+\|\D\M_2\|_{L^2}^2\right).
\end{align}
Substituting \eqref{10.3}--\eqref{10.10} into \eqref{11.2}, we complete the proof of \eqref{11.1}.
\end{proof}

Now we are in a position to prove \eqref{decay-2}.
\begin{proof}
For fixed $\delta_0(\bar{\rho}, \bar{\theta})>0$, adding $\delta_0(\bar{\rho}, \bar{\theta})\times$ \eqref{11.0} to \eqref{10.0}, we obtain
\begin{align}\label{12.1}
&\frac{\dif}{\dif t}\left(R\bar{\rho}^{-2}\bar{\theta}\|(\D^2\rho)_L\|_{L^2}^2+\|(\D^2 u)_L\|_{L^2}^2+\bar{\theta}^{-1}\|(\D^2\theta)_L\|_{L^2}^2+\delta_0(\bar{\rho}, \bar{\theta})\int(\D^2u)_L\cdot(\D^3\rho)_L\dif x\right)\nonumber\\
&+\mu\bar{\rho}^{-1}\bar{\theta}^\al\|(\D^3 u)_L\|_{L^2}^2+\kappa\bar{\rho}^{-1}\bar{\theta}^{\beta-1}\|(\D^3\theta)_L\|_{L^2}^2+\delta_0(\bar{\rho}, \bar{\theta})\frac{R\bar{\theta}}{2\bar{\rho}}\|(\D^3\rho)_L\|_{L^2}^2\nonumber\\
\leq~&C(\bar{\rho}, \bar{\theta})\varepsilon_0^{-1}(\bar{\rho},\bar{\theta})\|(\D\rho, \D u, \D\theta)\|_{H^1}\left(\|(\D^2\rho)_H\|_{L^2}^2+\|(\D^3\rho)_L\|_{L^2}^2+\|\D^3u\|_{L^2}^2+\|\D^3\theta\|_{L^2}^2\right)\nonumber\\
&+\varepsilon_0(\bar{\rho},\bar{\theta})\|(\D^3\rho)_L\|_{L^2}^2+\delta_0(\bar{\rho}, \bar{\theta})\overline{C}_3(\bar{\rho}, \bar{\theta})\left(\|(\D^3u)_L\|_{L^2}^2+\|(\D^3\theta)_L\|_{L^2}^2\right).
\end{align}

Taking $\delta_0(\bar{\rho}, \bar{\theta})$ sufficiently small, such that
$$\delta_0(\bar{\rho}, \bar{\theta})\overline{C}_3(\bar{\rho}, \bar{\theta})\leq\frac{1}{2}\min\left\{\mu\bar{\rho}^{-1}\bar{\theta}^\al, \kappa\bar{\rho}^{-1}\bar{\theta}^{\beta-1}\right\},$$
and
$$\delta_0(\bar{\rho}, \bar{\theta})\left|\int(\D^2u)_L\cdot(\D^3\rho)_L\dif x\right|\leq\frac{1}{4}\left(R\bar{\rho}^{-2}\bar{\theta}\|(\D^2\rho)_L\|_{L^2}^2+\|(\D^2 u)_L\|_{L^2}^2\right),$$
and, after decreasing the same $\delta_0(\bar{\rho},\bar{\theta})$ if necessary,
$$\delta_0(\bar{\rho},\bar{\theta})\left|\int(\D u)_H:(\D^2\rho)_H\dif x\right|\leq\frac14\left(R\bar{\rho}^{-2}\bar{\theta}\|(\D^2\rho)_H\|_{L^2}^2+\|(\D^2u)_H\|_{L^2}^2\right),$$
then taking $4\bar{\rho}\varepsilon_0(\bar{\rho},\bar{\theta})=\delta_0(\bar{\rho}, \bar{\theta})R\bar{\theta}$, we obtain
\begin{align}\label{12.2}
&\frac{\dif}{\dif t}\left(R\bar{\rho}^{-2}\bar{\theta}\|(\D^2\rho)_L\|_{L^2}^2+\|(\D^2 u)_L\|_{L^2}^2+\bar{\theta}^{-1}\|(\D^2\theta)_L\|_{L^2}^2+\delta_0(\bar{\rho}, \bar{\theta})\int(\D^2u)_L\cdot(\D^3\rho)_L\dif x\right)\nonumber\\
&+2M_0(\bar{\rho}, \bar{\theta})\left(\|(\D^3 u)_L\|_{L^2}^2+\|(\D^3\theta)_L\|_{L^2}^2+\|(\D^3\rho)_L\|_{L^2}^2\right)\nonumber\\
\leq~&C(\bar{\rho}, \bar{\theta})\|(\D\rho, \D u, \D\theta)\|_{H^1}\left(\|(\D^2\rho)_H\|_{L^2}^2+\|(\D^3\rho)_L\|_{L^2}^2+\|\D^3u\|_{L^2}^2+\|\D^3\theta\|_{L^2}^2\right),
\end{align}
for some positive constant $M_0$ depending on $\bar{\rho}$ and $\bar{\theta}$. On the other hand, adding $\delta_0(\bar{\rho}, \bar{\theta})\times$ \eqref{11.1} to \eqref{10.1}, we obtain
\begin{align}\label{12.3}
&\frac{\dif}{\dif t}\left(R\bar{\rho}^{-2}\bar{\theta}\|(\D^2\rho)_H\|_{L^2}^2+\|(\D^2 u)_H\|_{L^2}^2+\bar{\theta}^{-1}\|(\D^2\theta)_H\|_{L^2}^2+\delta_0(\bar{\rho}, \bar{\theta})\int(\D u)_H\cdot(\D^2\rho)_H\dif x\right)\nonumber\\
&+2M_0(\bar{\rho}, \bar{\theta})\left(\|(\D^3 u)_H\|_{L^2}^2+\|(\D^3\theta)_H\|_{L^2}^2+\|(\D^2\rho)_H\|_{L^2}^2\right)\nonumber\\
\leq~&C(\bar{\rho}, \bar{\theta})\|(\D\rho, \D u, \D\theta)\|_{H^1}\left(\|(\D^2\rho)_H\|_{L^2}^2+\|(\D^3\rho)_L\|_{L^2}^2+\|\D^3u\|_{L^2}^2+\|\D^3\theta\|_{L^2}^2\right).
\end{align}

Adding \eqref{12.2} to \eqref{12.3} and noticing \eqref{5.0}, one finds that there exist energy functional $\calE_2(t)$ equivalent to $\|(\D^2\rho, \D^2u, \D^2\theta)\|_{L^2}^2$ and a large time $T_3(\bar{\rho},\bar{\theta})>0$, such that for any $t\geq T_3(\bar{\rho},\bar{\theta})$, it holds
\begin{equation}\label{12.4}
\frac{\dif}{\dif t}\calE_2(t)+M_0(\bar{\rho},\bar{\theta})\left(\|(\D^2\rho)_H\|_{L^2}^2+\|(\D^3\rho)_L\|_{L^2}^2+\|\D^3u\|_{L^2}^2+\|\D^3\theta\|_{L^2}^2\right)\leq 0.
\end{equation}

Recalling the definition of $\phi(\xi)$ in \eqref{cut-off} and \eqref{HL}, similar to \eqref{5.2}, we have
\begin{align}\label{12.41}
&\|(\D^3\rho)_L\|_{L^2}^2+\|\D^3u\|_{L^2}^2+\|\D^3\theta\|_{L^2}^2\nonumber\\
=~&\int_{\tilde{S}(t)}|\xi|^6\left(\phi(\xi)^2|\widehat{\rho-\bar{\rho}}|^2+|\widehat{u}|^2+|\widehat{\theta-\bar{\theta}}|^2\right)\dif\xi+\int_{\R^3\setminus\tilde{S}(t)}|\xi|^6\left(\phi(\xi)^2|\widehat{\rho-\bar{\rho}}|^2+|\widehat{u}|^2+|\widehat{\theta-\bar{\theta}}|^2\right)\dif\xi\nonumber\\
\geq~&\frac{c_*(\bar{\rho},\bar{\theta})}{1+t}\int_{\R^3\setminus\tilde{S}(t)}|\xi|^4\left(\phi(\xi)^2|\widehat{\rho-\bar{\rho}}|^2+|\widehat{u}|^2+|\widehat{\theta-\bar{\theta}}|^2\right)\dif\xi\nonumber\\
=~&\frac{c_*(\bar{\rho},\bar{\theta})}{1+t}\int_{\R^3}|\xi|^4\left(\phi(\xi)^2|\widehat{\rho-\bar{\rho}}|^2+|\widehat{u}|^2+|\widehat{\theta-\bar{\theta}}|^2\right)\dif\xi\nonumber\\
&-\frac{c_*(\bar{\rho},\bar{\theta})}{1+t}\int_{\tilde{S}(t)}|\xi|^4\left(\phi(\xi)^2|\widehat{\rho-\bar{\rho}}|^2+|\widehat{u}|^2+|\widehat{\theta-\bar{\theta}}|^2\right)\dif\xi \nonumber\\
\geq~&\frac{c_*(\bar{\rho},\bar{\theta})}{1+t}\left(\|\D^2 u\|_{L^2}^2+\|\D^2\theta\|_{L^2}^2+\|(\D^2\rho)_L\|_{L^2}^2\right)-\frac{c_*^3(\bar{\rho},\bar{\theta})}{(1+t)^3}\int_{\tilde{S}(t)}\left(\phi(\xi)^2|\widehat{\rho-\bar{\rho}}|^2+|\widehat{u}|^2+|\widehat{\theta-\bar{\theta}}|^2\right)\dif\xi\nonumber\\
\geq~&\frac{c_*(\bar{\rho},\bar{\theta})}{1+t}\left(\|\D^2 u\|_{L^2}^2+\|\D^2\theta\|_{L^2}^2+\|(\D^2\rho)_L\|_{L^2}^2\right)-\frac{c_*^3(\bar{\rho},\bar{\theta})}{(1+t)^3}\left(\|u\|_{L^2}^2+\|\theta-\bar{\theta}\|_{L^2}^2+\|\rho-\bar{\rho}\|_{L^2}^2\right).
\end{align}

Substituting \eqref{12.41} into \eqref{12.4}, we obtain
\begin{align}\label{12.42}
&\frac{\dif}{\dif t}\calE_2(t)+M_0(\bar{\rho},\bar{\theta})\|(\D^2\rho)_H\|_{L^2}^2+\frac{M_0(\bar{\rho},\bar{\theta})c_*(\bar{\rho},\bar{\theta})}{1+t}\left(\|\D^2 u\|_{L^2}^2+\|\D^2\theta\|_{L^2}^2+\|(\D^2\rho)_L\|_{L^2}^2\right)\nonumber\\
\leq~&\frac{M_0(\bar{\rho},\bar{\theta})c_*^3(\bar{\rho},\bar{\theta})}{(1+t)^3}\left(\|u\|_{L^2}^2+\|\theta-\bar{\theta}\|_{L^2}^2+\|\rho-\bar{\rho}\|_{L^2}^2\right).
\end{align}
Fix the existing constant $C(\bar{\rho},\bar{\theta})\geq1$ so that
\begin{equation*}
\calE_2(t)\leq C(\bar{\rho},\bar{\theta})\left(\|(\D^2\rho)_H\|_{L^2}^2+\|(\D^2\rho)_L\|_{L^2}^2+\|\D^2u\|_{L^2}^2+\|\D^2\theta\|_{L^2}^2\right).
\end{equation*}
Denote $T_4(\bar{\rho},\bar{\theta})=c_*(\bar{\rho},\bar{\theta})-1$. Then for any $t\geq 1+\max\big\{ T_i(\bar{\rho},\bar{\theta})\big\}_{i=1}^4$, from \eqref{decay-0:U} we have
\begin{equation}\label{12.43}
\frac{\dif}{\dif t}\calE_2(t)+\frac{M_0(\bar{\rho},\bar{\theta})c_*(\bar{\rho},\bar{\theta})}{C(\bar{\rho},\bar{\theta})(1+t)}\calE_2(t)\leq C(\bar{\rho},\bar{\theta})(1+t)^{-\frac{3}{2}\left(\frac{2}{p_0}-1\right)-3}.
\end{equation}
Taking $c_*(\bar{\rho},\bar{\theta})=\displaystyle{\frac{4C(\bar{\rho},\bar{\theta})}{M_0(\bar{\rho},\bar{\theta})}}$, we get
\begin{equation}\label{12.44}
\frac{\dif}{\dif t}\calE_2(t)+\frac{4\calE_2(t)}{1+t}\leq C(\bar{\rho},\bar{\theta})(1+t)^{-\frac{3}{2}\left(\frac{2}{p_0}-1\right)-3},
\end{equation}
based on which we can derive the desired decay estimate \eqref{decay-2}.
\end{proof}



\bigskip

\noindent {\bf Acknowledgments}\\
The research of this work was supported in part by the National Natural Science Foundation of China under grants 12371221, 12301277, 11831011, and 12161141004, by Science Foundation of Zhejiang Sci-Tech University (ZSTU) under grant 25062122-Y, and by Natural Science Foundation of Zhejiang Province under grant ZCLQN26A0102. This work was also partially supported by the Fundamental Research Funds for the Central Universities and Shanghai Frontiers Science Center of Modern Analysis.

\bigskip 
 
\noindent{\bf Data Availability Statements}\\ 
Data sharing not applicable to this article as no datasets were generated or analyzed during the current study.

\bigskip

\noindent{\bf Conflict of interests}\\
The authors declare that they have no competing interests.

\bigskip

\noindent{\bf Authors' contributions}\\
The authors have made the same contribution. All authors read and approved the final manuscript.

\bigskip

\bibliographystyle{plain}
\bibliography{references}

@article {yang,
    AUTHOR = {Yang, Xiuhui},
     TITLE = {Uniform well-posedness and low {M}ach number limit to the
              compressible nematic liquid crystal flows in a bounded domain},
   JOURNAL = {Nonlinear Anal.},
  FJOURNAL = {Nonlinear Analysis. Theory, Methods \& Applications. An
              International Multidisciplinary Journal},
    VOLUME = {120},
      YEAR = {2015},
     PAGES = {118--126},
      ISSN = {0362-546X},
   MRCLASS = {76A15 (35B25 35B35 35D35 35Q35 82D30)},
  MRNUMBER = {3348050},
MRREVIEWER = {Paul F. Bracken},
       DOI = {10.1016/j.na.2015.03.010},
       URL = {https://doi.org/10.1016/j.na.2015.03.010},
}

@article {MR3225502,
    AUTHOR = {Liu, Hongxia and Yang, Tong and Zhao, Huijiang and Zou,
              Qingyang},
     TITLE = {One-dimensional compressible {N}avier-{S}tokes equations with
              temperature dependent transport coefficients and large data},
   JOURNAL = {SIAM J. Math. Anal.},
  FJOURNAL = {SIAM Journal on Mathematical Analysis},
    VOLUME = {46},
      YEAR = {2014},
    NUMBER = {3},
     PAGES = {2185--2228},
      ISSN = {0036-1410},
   MRCLASS = {35Q35 (35A01 35A02 35D35 76N10)},
  MRNUMBER = {3225502},
MRREVIEWER = {Juhi Jang},
       DOI = {10.1137/130920617},
       URL = {https://doi.org/10.1137/130920617},
}

@article {MR1360077,
    AUTHOR = {Hoff, David},
     TITLE = {Strong convergence to global solutions for multidimensional
              flows of compressible, viscous fluids with polytropic
              equations of state and discontinuous initial data},
   JOURNAL = {Arch. Rational Mech. Anal.},
  FJOURNAL = {Archive for Rational Mechanics and Analysis},
    VOLUME = {132},
      YEAR = {1995},
    NUMBER = {1},
     PAGES = {1--14},
      ISSN = {0003-9527},
   MRCLASS = {76N10 (35Q30)},
  MRNUMBER = {1360077},
MRREVIEWER = {Vladimir V. Shelukhin},
       DOI = {10.1007/BF00390346},
       URL = {https://doi.org/10.1007/BF00390346},
}

@article {MR1339675,
    AUTHOR = {Hoff, David},
     TITLE = {Global solutions of the {N}avier-{S}tokes equations for
              multidimensional compressible flow with discontinuous initial
              data},
   JOURNAL = {J. Differential Equations},
  FJOURNAL = {Journal of Differential Equations},
    VOLUME = {120},
      YEAR = {1995},
    NUMBER = {1},
     PAGES = {215--254},
      ISSN = {0022-0396},
   MRCLASS = {35Q30 (76N10)},
  MRNUMBER = {1339675},
MRREVIEWER = {Vladimir V. Shelukhin},
       DOI = {10.1006/jdeq.1995.1111},
       URL = {https://doi.org/10.1006/jdeq.1995.1111},
}

@article {MR3597161,
    AUTHOR = {Wen, Huanyao and Zhu, Changjiang},
     TITLE = {Global solutions to the three-dimensional full compressible
              {N}avier-{S}tokes equations with vacuum at infinity in some
              classes of large data},
   JOURNAL = {SIAM J. Math. Anal.},
  FJOURNAL = {SIAM Journal on Mathematical Analysis},
    VOLUME = {49},
      YEAR = {2017},
    NUMBER = {1},
     PAGES = {162--221},
      ISSN = {0036-1410},
   MRCLASS = {35Q30 (35A09 35B40 35D35 35K65 35Q35 76N10)},
  MRNUMBER = {3597161},
MRREVIEWER = {Tsukasa Iwabuchi},
       DOI = {10.1137/16M1055414},
       URL = {https://doi.org/10.1137/16M1055414},
}

@article {MR2297248,
    AUTHOR = {Bresch, Didier and Desjard\^ins, Beno\^it},
     TITLE = {On the existence of global weak solutions to the {N}avier-{S}tokes equations for viscous compressible and heat conducting fluids.},
   JOURNAL = {J. Math. Pures Appl.},
  FJOURNAL = {},
    VOLUME = {87},
      YEAR = {2007},
    NUMBER = {1},
     PAGES = {57--90},
      ISSN = {},
   MRCLASS = {},
  MRNUMBER = {},
MRREVIEWER = {},
       DOI = {},
       URL = {},
}

@article {MR0106646,
    AUTHOR = {Serrin, James},
     TITLE = {On the uniqueness of compressible fluid motions},
   JOURNAL = {Arch. Ration. Mech. Anal.},
  FJOURNAL = {Archive for Rational Mechanics and Analysis},
    VOLUME = {3},
      YEAR = {1959},
     PAGES = {271--288 (1959)},
      ISSN = {0003-9527},
   MRCLASS = {76.00},
  MRNUMBER = {0106646},
MRREVIEWER = {R. Gerber},
       DOI = {10.1007/BF00284180},
       URL = {https://doi.org/10.1007/BF00284180},
}

@article {MR0149094,
    AUTHOR = {Nash, John},
     TITLE = {Le probl\`eme de {C}auchy pour les \'{e}quations diff\'{e}rentielles
              d'un fluide g\'{e}n\'{e}ral},
   JOURNAL = {Bull. Soc. Math. France},
  FJOURNAL = {Bulletin de la Soci\'{e}t\'{e} Math\'{e}matique de France},
    VOLUME = {90},
      YEAR = {1962},
     PAGES = {487--497},
      ISSN = {0037-9484},
   MRCLASS = {35.79},
  MRNUMBER = {0149094},
MRREVIEWER = {M. Schechter},
       URL = {http://www.numdam.org/item?id=BSMF_1962__90__487_0},
}

@article {MR0283426,
    AUTHOR = {Itaya, Nobutoshi},
     TITLE = {On the {C}auchy problem for the system of fundamental
              equations describing the movement of compressible viscous
              fluid},
   JOURNAL = {Kodai Math. Sem. Rep.},
  FJOURNAL = {Kodai Mathematical Seminar Reports},
    VOLUME = {23},
      YEAR = {1971},
     PAGES = {60--120},
      ISSN = {0023-2599},
   MRCLASS = {35.79 (76.00)},
  MRNUMBER = {0283426},
MRREVIEWER = {T. Kato},
       URL = {http://projecteuclid.org/euclid.kmj/1138846265},
}

@article {MR3744381,
    AUTHOR = {Huang, Xiangdi and Li, Jing},
     TITLE = {Global classical and weak solutions to the three-dimensional
              full compressible {N}avier-{S}tokes system with vacuum and
              large oscillations},
   JOURNAL = {Arch. Ration. Mech. Anal.},
  FJOURNAL = {Archive for Rational Mechanics and Analysis},
    VOLUME = {227},
      YEAR = {2018},
    NUMBER = {3},
     PAGES = {995--1059},
      ISSN = {0003-9527},
   MRCLASS = {76N10 (35A09 35D30 35Q35)},
  MRNUMBER = {3744381},
MRREVIEWER = {Alessandro Morando},
       DOI = {10.1007/s00205-017-1188-y},
       URL = {https://doi.org/10.1007/s00205-017-1188-y},
}

@article {MR0564670,
    AUTHOR = {Matsumura, Akitaka and Nishida, Takaaki},
     TITLE = {The initial value problem for the equations of motion of
              viscous and heat-conductive gases},
   JOURNAL = {J. Math. Kyoto Univ.},
  FJOURNAL = {Journal of Mathematics of Kyoto University},
    VOLUME = {20},
      YEAR = {1980},
    NUMBER = {1},
     PAGES = {67--104},
      ISSN = {0023-608X},
   MRCLASS = {35Q20 (35K55 76N15)},
  MRNUMBER = {564670},
MRREVIEWER = {A. A. Arsen\cprime ev},
       DOI = {10.1215/kjm/1250522322},
       URL = {https://doi.org/10.1215/kjm/1250522322},
}

@article {MR1488513,
    AUTHOR = {Xin, Zhouping},
     TITLE = {Blowup of smooth solutions to the compressible
              {N}avier--{S}tokes equation with compact density},
   JOURNAL = {Comm. Pure Appl. Math.},
  FJOURNAL = {Communications on Pure and Applied Mathematics},
    VOLUME = {51},
      YEAR = {1998},
    NUMBER = {3},
     PAGES = {229--240},
      ISSN = {0010-3640},
   MRCLASS = {35Q30 (35B40 76N10)},
  MRNUMBER = {1488513},
MRREVIEWER = {Emmanuel Grenier},
       DOI = {10.1002/(SICI)1097-0312(199803)51:3<229::AID-CPA1>3.3.CO;2-K},
       URL =
              {https://doi.org/10.1002/(SICI)1097-0312(199803)51:3<229::AID-CPA1>3.3.CO;2-K},
}

@article {MR3604614,
    AUTHOR = {Yu, Haibo and Zhao, Junning},
     TITLE = {Global classical solutions to the 3{D} isentropic compressible
              {N}avier-{S}tokes equations in a bounded domain},
   JOURNAL = {Nonlinearity},
  FJOURNAL = {Nonlinearity},
    VOLUME = {30},
      YEAR = {2017},
    NUMBER = {1},
     PAGES = {361--381},
      ISSN = {0951-7715},
   MRCLASS = {35Q35 (35A09 76N10)},
  MRNUMBER = {3604614},
       DOI = {10.1088/1361-6544/aa4f4c},
       URL = {https://doi.org/10.1088/1361-6544/aa4f4c},
}

@article {MR4079010,
    AUTHOR = {Yu, Haibo and Zhang, Peixin},
     TITLE = {Global strong solutions to the 3{D} full compressible
              {N}avier-{S}tokes equations with density-temperature-dependent
              viscosities in bounded domains},
   JOURNAL = {J. Differential Equations},
  FJOURNAL = {Journal of Differential Equations},
    VOLUME = {268},
      YEAR = {2020},
    NUMBER = {12},
     PAGES = {7286--7310},
      ISSN = {0022-0396},
   MRCLASS = {76N06 (35B40 35B45 35D35 35Q35 76N10)},
  MRNUMBER = {4079010},
MRREVIEWER = {Ionel Ciuperca},
       DOI = {10.1016/j.jde.2019.11.065},
       URL = {https://doi.org/10.1016/j.jde.2019.11.065},
}

@article {MR713680,
    AUTHOR = {Matsumura, Akitaka and Nishida, Takaaki},
     TITLE = {Initial-boundary value problems for the equations of motion of
              compressible viscous and heat-conductive fluids},
   JOURNAL = {Comm. Math. Phys.},
  FJOURNAL = {Communications in Mathematical Physics},
    VOLUME = {89},
      YEAR = {1983},
    NUMBER = {4},
     PAGES = {445--464},
      ISSN = {0010-3616},
   MRCLASS = {35Q10 (35K55 76N05)},
  MRNUMBER = {713680},
MRREVIEWER = {J. F. Heyda},
       URL = {http://projecteuclid.org/euclid.cmp/1103922925},
}

@book {Ladyzhenskaya1968,
    AUTHOR = {Ladyzhenskaya, Ol'ga  and  Solonnikov, Vsevolod  and Ural'ceva, Nina},
     TITLE = {Linear and Quasilinear Equations of Parabolic Type}, 
 PUBLISHER = {American Mathematical Society, Providence},
      YEAR = {1968},
}

@article {MR3564590,
    AUTHOR = {Wang, Tao and Zhao, Huijiang},
     TITLE = {One-dimensional compressible heat-conducting gas with
              temperature-dependent viscosity},
   JOURNAL = {Math. Models Methods Appl. Sci.},
  FJOURNAL = {Mathematical Models and Methods in Applied Sciences},
    VOLUME = {26},
      YEAR = {2016},
    NUMBER = {12},
     PAGES = {2237--2275},
      ISSN = {0218-2025,1793-6314},
   MRCLASS = {35Q35 (35B40 76N10 76N15)},
  MRNUMBER = {3564590},
       DOI = {10.1142/S0218202516500524},
       URL = {https://doi.org/10.1142/S0218202516500524},
}

@article {MR4235250,
    AUTHOR = {Sun, Ying and Zhang, Jianwen and Zhao, Xiaokui},
     TITLE = {Nonlinearly exponential stability for the compressible
              {N}avier-{S}tokes equations with temperature-dependent
              transport coefficients},
   JOURNAL = {J. Differential Equations},
  FJOURNAL = {Journal of Differential Equations},
    VOLUME = {286},
      YEAR = {2021},
     PAGES = {676--709},
      ISSN = {0022-0396,1090-2732},
   MRCLASS = {35Q35 (76N06)},
  MRNUMBER = {4235250},
MRREVIEWER = {Zhilei\ Liang},
       DOI = {10.1016/j.jde.2021.03.044},
       URL = {https://doi.org/10.1016/j.jde.2021.03.044},
}

@book{10.1115/1.3607836,
    author = {Zel’dovich, Ya. B. and Raizer, Yu. P.},
    title = {Physics of Shock Waves and High-Temperature Hydrodynamic Phenomena,},
    SERIES = {Vol. \uppercase \expandafter {\romannumeral 2}},
    PUBLISHER = {Academic Press, New York}, 
    year = {1967},
}

@book {MR258399,
    AUTHOR = {Chapman, Sydney and Cowling, T. G.},
     TITLE = {The mathematical theory of non-uniform gases. {A}n account of
              the kinetic theory of viscosity, thermal conduction and
              diffusion in gases},
   EDITION = {Third},
 PUBLISHER = {Cambridge University Press, London},
      YEAR = {1970},
     PAGES = {xxiv+423},
   MRCLASS = {82.00},
  MRNUMBER = {258399},
}

@article {MR4496704,
    AUTHOR = {Cao, Yue and Li, Yachun},
     TITLE = {Local strong solutions to the full compressible
              {N}avier-{S}tokes system with temperature-dependent viscosity
              and heat conductivity},
   JOURNAL = {SIAM J. Math. Anal.},
  FJOURNAL = {SIAM Journal on Mathematical Analysis},
    VOLUME = {54},
      YEAR = {2022},
    NUMBER = {5},
     PAGES = {5588--5628},
      ISSN = {0036-1410,1095-7154},
   MRCLASS = {35A01 (35A09 35B45 35B65 35Q35)},
  MRNUMBER = {4496704},
       DOI = {10.1137/21M1419544},
       URL = {https://doi.org/10.1137/21M1419544},
}

@article {MR4469007,
    AUTHOR = {Chaudhuri, Nilasis and Feireisl, Eduard},
     TITLE = {Navier-{S}tokes-{F}ourier system with {D}irichlet boundary
              conditions},
   JOURNAL = {Appl. Anal.},
  FJOURNAL = {Applicable Analysis. An International Journal},
    VOLUME = {101},
      YEAR = {2022},
    NUMBER = {12},
     PAGES = {4076--4094},
      ISSN = {0003-6811,1563-504X},
   MRCLASS = {35Q30 (76N06)},
  MRNUMBER = {4469007},
       DOI = {10.1080/00036811.2021.1992396},
       URL = {https://doi.org/10.1080/00036811.2021.1992396},
}

@article {MR4846859,
    AUTHOR = {Chiodaroli, Elisabetta and Feireisl, Eduard},
     TITLE = {On the long-time behavior of solutions to the
              {N}avier-{S}tokes-{F}ourier system on unbounded domains},
   JOURNAL = {J. Lond. Math. Soc. (2)},
  FJOURNAL = {Journal of the London Mathematical Society. Second Series},
    VOLUME = {111},
      YEAR = {2025},
    NUMBER = {1},
     PAGES = {Paper No. e70067, 21},
      ISSN = {0024-6107,1469-7750},
   MRCLASS = {35Q35 (35B40 35D30 76N10 76N15)},
  MRNUMBER = {4846859},
       DOI = {10.1112/jlms.70067},
       URL = {https://doi.org/10.1112/jlms.70067},
}

@article {MR4810457,
    AUTHOR = {Feireisl, Eduard and Lu, Yong and Sun, Yongzhong},
     TITLE = {Unconditional stability of equilibria in thermally driven
              compressible fluids},
   JOURNAL = {Arch. Ration. Mech. Anal.},
  FJOURNAL = {Archive for Rational Mechanics and Analysis},
    VOLUME = {248},
      YEAR = {2024},
    NUMBER = {6},
     PAGES = {Paper No. 98, 38},
      ISSN = {0003-9527,1432-0673},
   MRCLASS = {76R10 (35B50 35D30 35Q30 76N06)},
  MRNUMBER = {4810457},
       DOI = {10.1007/s00205-024-02044-1},
       URL = {https://doi.org/10.1007/s00205-024-02044-1},
}

@article {MR4803458,
    AUTHOR = {Hou, Meichen and Liu, Lingjun and Wang, Shu and Xu, Lingda},
     TITLE = {Vanishing viscosity limit to the planar rarefaction wave with
              vacuum for 3-{D} full compressible {N}avier-{S}tokes equations
              with temperature-dependent transport coefficients},
   JOURNAL = {Math. Ann.},
  FJOURNAL = {Mathematische Annalen},
    VOLUME = {390},
      YEAR = {2024},
    NUMBER = {3},
     PAGES = {3513--3566},
      ISSN = {0025-5831,1432-1807},
   MRCLASS = {35Q30 (35Q31 35Q35 76N06 76N10)},
  MRNUMBER = {4803458},
       DOI = {10.1007/s00208-024-02840-w},
       URL = {https://doi.org/10.1007/s00208-024-02840-w},
}

@article {MR2917121,
    AUTHOR = {Feireisl, Eduard and Mucha, Piotr B. and Novotn\'{y}, Anton\'{\i}n and
              Pokorn\'{y}, Milan},
     TITLE = {Time-periodic solutions to the full
              {N}avier-{S}tokes-{F}ourier system},
   JOURNAL = {Arch. Ration. Mech. Anal.},
  FJOURNAL = {Archive for Rational Mechanics and Analysis},
    VOLUME = {204},
      YEAR = {2012},
    NUMBER = {3},
     PAGES = {745--786},
      ISSN = {0003-9527},
   MRCLASS = {35Q35 (35B10 35D30 76N10)},
  MRNUMBER = {2917121},
MRREVIEWER = {Amin Esfahani},
       DOI = {10.1007/s00205-012-0492-9},
       URL = {https://doi.org/10.1007/s00205-012-0492-9},
}

@article {MR2909912,
    AUTHOR = {Feireisl, Eduard and Novotn\'{y}, Anton\'{\i}n},
     TITLE = {Weak-strong uniqueness property for the full
              {N}avier-{S}tokes-{F}ourier system},
   JOURNAL = {Arch. Ration. Mech. Anal.},
  FJOURNAL = {Archive for Rational Mechanics and Analysis},
    VOLUME = {204},
      YEAR = {2012},
    NUMBER = {2},
     PAGES = {683--706},
      ISSN = {0003-9527},
   MRCLASS = {35Q35 (35A02 35D30 35D35 76N10)},
  MRNUMBER = {2909912},
MRREVIEWER = {Denis Serre},
       DOI = {10.1007/s00205-011-0490-3},
       URL = {https://doi.org/10.1007/s00205-011-0490-3},
}

@article {MR3624545,
    AUTHOR = {Wan, Ling and Wang, Tao},
     TITLE = {Symmetric flows for compressible heat-conducting fluids with
              temperature dependent viscosity coefficients},
   JOURNAL = {J. Differential Equations},
  FJOURNAL = {Journal of Differential Equations},
    VOLUME = {262},
      YEAR = {2017},
    NUMBER = {12},
     PAGES = {5939--5977},
      ISSN = {0022-0396},
   MRCLASS = {35Q35 (35B35 35B40 76N10)},
  MRNUMBER = {3624545},
       DOI = {10.1016/j.jde.2017.02.022},
       URL = {https://doi.org/10.1016/j.jde.2017.02.022},
}

@article {MR3703560,
    AUTHOR = {Huang, Bingkang and Liao, Yongkai},
     TITLE = {Global stability of combination of viscous contact wave with
              rarefaction wave for compressible {N}avier-{S}tokes equations
              with temperature-dependent viscosity},
   JOURNAL = {Math. Models Methods Appl. Sci.},
  FJOURNAL = {Mathematical Models and Methods in Applied Sciences},
    VOLUME = {27},
      YEAR = {2017},
    NUMBER = {12},
     PAGES = {2321--2379},
      ISSN = {0218-2025},
   MRCLASS = {35Q35 (35B35 35B40 35Q30)},
  MRNUMBER = {3703560},
       DOI = {10.1142/S0218202517500464},
       URL = {https://doi.org/10.1142/S0218202517500464},
}

@article {MR775190,
    AUTHOR = {Schonbek, Maria Elena},
     TITLE = {{$L^2$} decay for weak solutions of the {N}avier-{S}tokes
              equations},
   JOURNAL = {Arch. Rational Mech. Anal.},
  FJOURNAL = {Archive for Rational Mechanics and Analysis},
    VOLUME = {88},
      YEAR = {1985},
    NUMBER = {3},
     PAGES = {209--222},
      ISSN = {0003-9527},
   MRCLASS = {35Q10 (76D05)},
  MRNUMBER = {775190},
MRREVIEWER = {Yoshikazu\ Giga},
       DOI = {10.1007/BF00752111},
       URL = {https://doi.org/10.1007/BF00752111},
}

@article{2024arXiv240805138L,
  title={Global existence and asymptotic behavior of large strong solutions to the 3{D} full compressible {N}avier-{S}tokes equations with density-dependent viscosities},
  author={{Li}, Yachun and {Lu}, Peng and {Shang}, Zhaoyang and {Yu}, Shaojun},
  journal={arXiv:2408.05138},
  year={2024}
}

@article {MR4980411,
    AUTHOR = {Li, Jiaxu and Li, Jing and L\"u, Boqiang},
     TITLE = {Global {C}lassical {S}olutions to the {F}ull {C}ompressible
              {N}avier--{S}tokes {S}ystem in 3{D} {E}xterior {D}omains},
   JOURNAL = {SIAM J. Math. Anal.},
  FJOURNAL = {SIAM Journal on Mathematical Analysis},
    VOLUME = {57},
      YEAR = {2025},
    NUMBER = {6},
     PAGES = {5854--5905},
      ISSN = {0036-1410,1095-7154},
   MRCLASS = {35Q30 (35A01 76N10)},
  MRNUMBER = {4980411},
       DOI = {10.1137/24M1716549},
       URL = {https://doi.org/10.1137/24M1716549},
}

@article {MR4097330,
    AUTHOR = {Li, Jinkai},
     TITLE = {Global small solutions of heat conductive compressible
              {N}avier-{S}tokes equations with vacuum: smallness on scaling
              invariant quantity},
   JOURNAL = {Arch. Ration. Mech. Anal.},
  FJOURNAL = {Archive for Rational Mechanics and Analysis},
    VOLUME = {237},
      YEAR = {2020},
    NUMBER = {2},
     PAGES = {899--919},
      ISSN = {0003-9527,1432-0673},
   MRCLASS = {35Q35 (76N10)},
  MRNUMBER = {4097330},
MRREVIEWER = {Alberto\ Valli},
       DOI = {10.1007/s00205-020-01521-7},
       URL = {https://doi.org/10.1007/s00205-020-01521-7},
}

@article {MR4346514,
    AUTHOR = {Li, Hao and Shang, Zhaoyang},
     TITLE = {Global strong solution to the two-dimensional full
              compressible {N}avier-{S}tokes equations with large viscosity},
   JOURNAL = {J. Math. Fluid Mech.},
  FJOURNAL = {Journal of Mathematical Fluid Mechanics},
    VOLUME = {24},
      YEAR = {2022},
    NUMBER = {1},
     PAGES = {Paper No. 7, 29},
      ISSN = {1422-6928,1422-6952},
   MRCLASS = {35Q30 (76N06)},
  MRNUMBER = {4346514},
       DOI = {10.1007/s00021-021-00641-8},
       URL = {https://doi.org/10.1007/s00021-021-00641-8},
}

@article{MR1480244,
    author ={Hoff, David} ,
    title = {Discontinuous solutions of the {N}avier-{S}tokes equations for multidimensional flows of heat-conducting fluids},
    journal = {Arch. Rational Mech. Anal.},
    VOLUME = {139},
    year = {1997},
    NUMBER = {4}, 
    PAGES = {303–-354},
}

@article{MR3603273,
    author ={Duan, Ran and Guo, Ai and Zhu, Changjiang} ,
    title = {Global strong solution to compressible {N}avier–{S}tokes
equations with density dependent viscosity and temperature dependent heat conductivity},
    journal = {J. Differential Equations},
    VOLUME = {262},
    year = {2017},
    NUMBER = {8}, 
    PAGES = {4314--4335},
}

@article{MR0785713,
    author = {Ponce, Gustavo} ,
    title = {Global existence of small solution to a class of nonlinear evolution equations},
    journal = {Nonlinear Anal., Theory Methods Appl.},
    VOLUME = {9},
    year = {1985},
    NUMBER = {5}, 
    PAGES = {339--418},
}

@article{Gao2020The,
 author = {Gao, Jincheng and Wei, Zhengzhen and Yao, Zheng-an} ,
    title = {The optimal decay rate of strong solution for the compressible {N}avier-{S}tokes equations with large initial data},
    journal = {Physica D},
    VOLUME = {406},
    year = {2020},
    PAGES = {132506 },
}

@article {MR4389852,
    AUTHOR = {He, Lingbing and Huang, Jingchi and Wang, Chao},
     TITLE = {Large-time behavior for compressible
              {N}avier-{S}tokes-{F}ourier system in the whole space},
   JOURNAL = {J. Math. Fluid Mech.},
  FJOURNAL = {Journal of Mathematical Fluid Mechanics},
    VOLUME = {24},
      YEAR = {2022},
    NUMBER = {2},
     PAGES = {Paper No. 31, 26},
      ISSN = {1422-6928,1422-6952},
   MRCLASS = {76N06 (35Q30)},
  MRNUMBER = {4389852},
MRREVIEWER = {Yongzhong\ Sun},
       DOI = {10.1007/s00021-022-00673-8},
       URL = {https://doi.org/10.1007/s00021-022-00673-8},
}

@article{Gao2021De,
 author = {Gao, Jincheng and Wei, Zhengzhen and Yao, Zheng-an} ,
    title = { Decay of strong solution for the compressible {N}avier–{S}tokes
equations with large initial data},
    journal = {Nonlinear Anal.},
    VOLUME = {213},
    year = {2021},
    PAGES = {112494 },
}

@article{Zhang2020Con,
 author = {Zhang, Zhifei  and Zi, Ruizhao} ,
    title = { Convergence to equilibrium for the solution of the full compressible
{N}avier-{S}tokes equations},
    journal = {Ann. Inst. H. Poincar\'e Anal. Non Lin\'eaire},
    VOLUME = {37},
    year = {2020},
    PAGES = {457--488},
}

@article {MR2562709,
    AUTHOR = {Villani, C\'edric},
     TITLE = {Hypocoercivity},
   JOURNAL = {Mem. Amer. Math. Soc.},
  FJOURNAL = {Memoirs of the American Mathematical Society},
    VOLUME = {202},
      YEAR = {2009},
    NUMBER = {950},
      ISSN = {0065-9266,1947-6221},
      ISBN = {978-0-8218-4498-4},
   MRCLASS = {35Q84 (35H10 76N10 76P05 82C70)},
  MRNUMBER = {2562709},
MRREVIEWER = {Andr\'as\ Domokos},
       DOI = {10.1090/S0065-9266-09-00567-5},
       URL = {https://doi.org/10.1090/S0065-9266-09-00567-5},
}

@article {MR3063918,
    AUTHOR = {Xin, Zhouping and Yan, Wei},
     TITLE = {On blowup of classical solutions to the compressible
              {N}avier-{S}tokes equations},
   JOURNAL = {Comm. Math. Phys.},
  FJOURNAL = {Communications in Mathematical Physics},
    VOLUME = {321},
      YEAR = {2013},
    NUMBER = {2},
     PAGES = {529--541},
      ISSN = {0010-3616,1432-0916},
   MRCLASS = {35Q35 (35A09 35B44 76N10)},
  MRNUMBER = {3063918},
MRREVIEWER = {Amin\ Esfahani},
       DOI = {10.1007/s00220-012-1610-0},
       URL = {https://doi.org/10.1007/s00220-012-1610-0},
}

@article {MR4813232,
    AUTHOR = {Chen, Gui-Qiang G. and Huang, Yucong and Zhu, Shengguo},
     TITLE = {Global {S}pherically {S}ymmetric {S}olutions of the
              {M}ultidimensional {F}ull {C}ompressible {N}avier--{S}tokes
              {E}quations with {L}arge {D}ata},
   JOURNAL = {Arch. Ration. Mech. Anal.},
  FJOURNAL = {Archive for Rational Mechanics and Analysis},
    VOLUME = {248},
      YEAR = {2024},
    NUMBER = {6},
     PAGES = {Paper No. 101},
      ISSN = {0003-9527,1432-0673},
   MRCLASS = {35Q30 (35A01 35D30 76N10)},
  MRNUMBER = {4813232},
       DOI = {10.1007/s00205-024-02018-3},
       URL = {https://doi.org/10.1007/s00205-024-02018-3},
}

@article {MR1389908,
    AUTHOR = {Jiang, Song},
     TITLE = {Global spherically symmetric solutions to the equations of a
              viscous polytropic ideal gas in an exterior domain},
   JOURNAL = {Comm. Math. Phys.},
  FJOURNAL = {Communications in Mathematical Physics},
    VOLUME = {178},
      YEAR = {1996},
    NUMBER = {2},
     PAGES = {339--374},
      ISSN = {0010-3616,1432-0916},
   MRCLASS = {35Q35 (35B60 76N15)},
  MRNUMBER = {1389908},
MRREVIEWER = {R.\ M.\ Gundersen},
       URL = {http://projecteuclid.org/euclid.cmp/1104286655},
}

@article {MR2091508,
    AUTHOR = {Hoff, David and Jenssen, Helge Kristian},
     TITLE = {Symmetric nonbarotropic flows with large data and forces},
   JOURNAL = {Arch. Ration. Mech. Anal.},
  FJOURNAL = {Archive for Rational Mechanics and Analysis},
    VOLUME = {173},
      YEAR = {2004},
    NUMBER = {3},
     PAGES = {297--343},
      ISSN = {0003-9527,1432-0673},
   MRCLASS = {35Q30 (35D05 76N10)},
  MRNUMBER = {2091508},
MRREVIEWER = {Sergey\ Nikolaevich\ Alekseenko},
       DOI = {10.1007/s00205-004-0318-5},
       URL = {https://doi.org/10.1007/s00205-004-0318-5},
}

@article {MR555060,
    AUTHOR = {Matsumura, Akitaka and Nishida, Takaaki},
     TITLE = {The initial value problem for the equations of motion of
              compressible viscous and heat-conductive fluids},
   JOURNAL = {Proc. Japan Acad. Ser. A Math. Sci.},
  FJOURNAL = {Japan Academy. Proceedings. Series A. Mathematical Sciences},
    VOLUME = {55},
      YEAR = {1979},
    NUMBER = {9},
     PAGES = {337--342},
      ISSN = {0386-2194},
   MRCLASS = {35Q20 (76A99)},
  MRNUMBER = {555060},
MRREVIEWER = {Cornelius\ O.\ Horgan},
       URL = {http://projecteuclid.org/euclid.pja/1195517111},
}

@article {MR651877,
    AUTHOR = {Kazhikhov, A. V.},
     TITLE = {On the {C}auchy problem for the equations of a viscous gas},
   JOURNAL = {Sibirsk. Mat. Zh.},
  FJOURNAL = {Akademiya Nauk SSSR. Sibirskoe Otdelenie. Sibirski\u i\
              Matematicheski\u i\ Zhurnal},
    VOLUME = {23},
      YEAR = {1982},
    NUMBER = {1},
     PAGES = {60--64, 220},
      ISSN = {0037-4474},
   MRCLASS = {35Q10 (76N15)},
  MRNUMBER = {651877},
}

@article {MR637519,
    AUTHOR = {Kawashima, Shuichi and Nishida, Takaaki},
     TITLE = {Global solutions to the initial value problem for the
              equations of one-dimensional motion of viscous polytropic
              gases},
   JOURNAL = {J. Math. Kyoto Univ.},
  FJOURNAL = {Journal of Mathematics of Kyoto University},
    VOLUME = {21},
      YEAR = {1981},
    NUMBER = {4},
     PAGES = {825--837},
      ISSN = {0023-608X},
   MRCLASS = {76N15 (35Q20)},
  MRNUMBER = {637519},
       DOI = {10.1215/kjm/1250521915},
       URL = {https://doi.org/10.1215/kjm/1250521915},
}

@article {MR4039142,
    AUTHOR = {Li, Jinkai and Xin, Zhouping},
     TITLE = {Entropy bounded solutions to the one-dimensional compressible
              {N}avier-{S}tokes equations with zero heat conduction and far
              field vacuum},
   JOURNAL = {Adv. Math.},
  FJOURNAL = {Advances in Mathematics},
    VOLUME = {361},
      YEAR = {2020},
     PAGES = {106923, 50},
      ISSN = {0001-8708,1090-2082},
   MRCLASS = {35Q35 (76N10)},
  MRNUMBER = {4039142},
MRREVIEWER = {Luc\ Paquet},
       DOI = {10.1016/j.aim.2019.106923},
       URL = {https://doi.org/10.1016/j.aim.2019.106923},
}

@article {MR4740640,
    AUTHOR = {Li, Jinkai and Xin, Zhouping},
     TITLE = {Instantaneous unboundedness of the entropy and uniform
              positivity of the temperature for the compressible
              {N}avier-{S}tokes equations with fast decay density},
   JOURNAL = {SIAM J. Math. Anal.},
  FJOURNAL = {SIAM Journal on Mathematical Analysis},
    VOLUME = {56},
      YEAR = {2024},
    NUMBER = {3},
     PAGES = {3004--3041},
      ISSN = {0036-1410,1095-7154},
   MRCLASS = {35A01 (35B44 35B45 35B50 76N06)},
  MRNUMBER = {4740640},
       DOI = {10.1137/23M1594352},
       URL = {https://doi.org/10.1137/23M1594352},
}

@article {MR4491875,
    AUTHOR = {Li, Jinkai and Xin, Zhouping},
     TITLE = {Entropy-bounded solutions to the one-dimensional heat
              conductive compressible {N}avier-{S}tokes equations with far
              field vacuum},
   JOURNAL = {Comm. Pure Appl. Math.},
  FJOURNAL = {Communications on Pure and Applied Mathematics},
    VOLUME = {75},
      YEAR = {2022},
    NUMBER = {11},
     PAGES = {2393--2445},
      ISSN = {0010-3640,1097-0312},
   MRCLASS = {35Q30},
  MRNUMBER = {4491875},
MRREVIEWER = {Jean\ C.\ Cortissoz},
}

@article {MR3461630,
    AUTHOR = {Wang, Tao},
     TITLE = {One dimensional {$p$}-th power {N}ewtonian fluid with
              temperature-dependent thermal conductivity},
   JOURNAL = {Commun. Pure Appl. Anal.},
  FJOURNAL = {Communications on Pure and Applied Analysis},
    VOLUME = {15},
      YEAR = {2016},
    NUMBER = {2},
     PAGES = {477--494},
      ISSN = {1534-0392,1553-5258},
   MRCLASS = {76N10 (35A01 35A02 35B45 35Q35)},
  MRNUMBER = {3461630},
MRREVIEWER = {Ivan\ Stra\v skraba},
       DOI = {10.3934/cpaa.2016.15.477},
       URL = {https://doi.org/10.3934/cpaa.2016.15.477},
}

@article {MR791841,
    AUTHOR = {Kawohl, Bernhard},
     TITLE = {Global existence of large solutions to initial-boundary value
              problems for a viscous, heat-conducting, one-dimensional real
              gas},
   JOURNAL = {J. Differential Equations},
  FJOURNAL = {Journal of Differential Equations},
    VOLUME = {58},
      YEAR = {1985},
    NUMBER = {1},
     PAGES = {76--103},
      ISSN = {0022-0396,1090-2732},
   MRCLASS = {35Q20 (76N15)},
  MRNUMBER = {791841},
MRREVIEWER = {L.\ Hsiao},
       DOI = {10.1016/0022-0396(85)90023-3},
       URL = {https://doi.org/10.1016/0022-0396(85)90023-3},
}

@article {MR4493879,
    AUTHOR = {Luo, Zhengyan and Zhang, Yinghui},
     TITLE = {Optimal large-time behavior of solutions to the full
              compressible {N}avier-{S}tokes equations with large initial
              data},
   JOURNAL = {J. Evol. Equ.},
  FJOURNAL = {Journal of Evolution Equations},
    VOLUME = {22},
      YEAR = {2022},
    NUMBER = {4},
     PAGES = {Paper No. 83, 35},
      ISSN = {1424-3199,1424-3202},
   MRCLASS = {35Q30 (35B40 76N06)},
  MRNUMBER = {4493879},
       DOI = {10.1007/s00028-022-00841-3},
       URL = {https://doi.org/10.1007/s00028-022-00841-3},
}

@article {MR1675129,
    AUTHOR = {Kobayashi, Takayuki and Shibata, Yoshihiro},
     TITLE = {Decay estimates of solutions for the equations of motion of
              compressible viscous and heat-conductive gases in an exterior
              domain in {${\bf R}^3$}},
   JOURNAL = {Comm. Math. Phys.},
  FJOURNAL = {Communications in Mathematical Physics},
    VOLUME = {200},
      YEAR = {1999},
    NUMBER = {3},
     PAGES = {621--659},
      ISSN = {0010-3616,1432-0916},
   MRCLASS = {35Q30 (35B40 76N10)},
  MRNUMBER = {1675129},
MRREVIEWER = {Ivan\ Stra\v skraba},
       DOI = {10.1007/s002200050543},
       URL = {https://doi.org/10.1007/s002200050543},
}

@article {MR4858604,
    AUTHOR = {Shen, Linxuan and Xu, Hao and Zhang, Jianwen},
     TITLE = {Global well-posedness and large-time behavior for the 3{D}
              full compressible {N}avier-{S}tokes equations with
              density-dependent viscosity and vacuum},
   JOURNAL = {J. Differential Equations},
  FJOURNAL = {Journal of Differential Equations},
    VOLUME = {426},
      YEAR = {2025},
     PAGES = {466--494},
      ISSN = {0022-0396,1090-2732},
   MRCLASS = {35Q30 (35B40 35B45 76N06)},
  MRNUMBER = {4858604},
       DOI = {10.1016/j.jde.2025.01.072},
       URL = {https://doi.org/10.1016/j.jde.2025.01.072},
}

@article {MR3054636,
    AUTHOR = {Bella, Peter and Feireisl, Eduard and Pra\v{z}\'ak, Dalibor},
     TITLE = {Long time behavior and stabilization to equilibria of
              solutions to the {N}avier-{S}tokes-{F}ourier system driven by
              highly oscillating unbounded external forces},
   JOURNAL = {J. Dynam. Differential Equations},
  FJOURNAL = {Journal of Dynamics and Differential Equations},
    VOLUME = {25},
      YEAR = {2013},
    NUMBER = {2},
     PAGES = {257--268},
      ISSN = {1040-7294,1572-9222},
   MRCLASS = {35Q35 (35B40 35B45 35D30)},
  MRNUMBER = {3054636},
MRREVIEWER = {Maria\ Specovius-Neugebauer},
       DOI = {10.1007/s10884-013-9299-0},
       URL = {https://doi.org/10.1007/s10884-013-9299-0},
}

@article {MR2350243,
    AUTHOR = {Feireisl, Eduard and Petzeltov\'a, Hana},
     TITLE = {On the long-time behaviour of solutions to the
              {N}avier-{S}tokes-{F}ourier system with a time-dependent
              driving force},
   JOURNAL = {J. Dynam. Differential Equations},
  FJOURNAL = {Journal of Dynamics and Differential Equations},
    VOLUME = {19},
      YEAR = {2007},
    NUMBER = {3},
     PAGES = {685--707},
      ISSN = {1040-7294,1572-9222},
   MRCLASS = {35Q30 (35B40 76D05 76N10)},
  MRNUMBER = {2350243},
MRREVIEWER = {Piotr\ Biler},
       DOI = {10.1007/s10884-006-9015-4},
       URL = {https://doi.org/10.1007/s10884-006-9015-4},
}

@article {MR4921984,
    AUTHOR = {Dong, Wenchao and Guo, Zhenhua},
     TITLE = {Cauchy problem for the {N}avier--{S}tokes equations with
              temperature-dependent transport coefficients and large data},
   JOURNAL = {Math. Models Methods Appl. Sci.},
  FJOURNAL = {Mathematical Models and Methods in Applied Sciences},
    VOLUME = {35},
      YEAR = {2025},
    NUMBER = {9},
     PAGES = {1889--1932},
      ISSN = {0218-2025,1793-6314},
   MRCLASS = {35B35 (35B40 76N06 76N10)},
  MRNUMBER = {4921984},
       DOI = {10.1142/S0218202525500319},
       URL = {https://doi.org/10.1142/S0218202525500319},
}

@article {MR4924422,
    AUTHOR = {Kang, Moon-Jin and Vasseur, Alexis F. and Wang, Yi},
     TITLE = {Time-{A}symptotic {S}tability of {G}eneric {R}iemann
              {S}olutions for {C}ompressible {N}avier--{S}tokes--{F}ourier
              {E}quations},
   JOURNAL = {Arch. Ration. Mech. Anal.},
  FJOURNAL = {Archive for Rational Mechanics and Analysis},
    VOLUME = {249},
      YEAR = {2025},
    NUMBER = {4},
     PAGES = {Paper No. 42},
      ISSN = {0003-9527,1432-0673},
   MRCLASS = {99-06},
  MRNUMBER = {4924422},
       DOI = {10.1007/s00205-025-02116-w},
       URL = {https://doi.org/10.1007/s00205-025-02116-w},
}

@article {MR4613439,
    AUTHOR = {Luo, Zhengyan and Zhang, Yinghui},
     TITLE = {Optimal large time behavior of the full compressible
              {N}avier-{S}tokes system in {$\Bbb R^3$}},
   JOURNAL = {Bull. Braz. Math. Soc. (N.S.)},
  FJOURNAL = {Bulletin of the Brazilian Mathematical Society. New Series.
              Boletim da Sociedade Brasileira de Matem\'atica},
    VOLUME = {54},
      YEAR = {2023},
    NUMBER = {3},
     PAGES = {Paper No. 38, 29},
      ISSN = {1678-7544,1678-7714},
   MRCLASS = {35Q30},
  MRNUMBER = {4613439},
MRREVIEWER = {Michael\ J.\ Carley},
       DOI = {10.1007/s00574-023-00354-8},
       URL = {https://doi.org/10.1007/s00574-023-00354-8},
}

@article {MR4492674,
    AUTHOR = {Lai, Suhua and Xu, Hao and Zhang, Jianwen},
     TITLE = {Well-posedness and exponential decay for the {N}avier-{S}tokes
              equations of viscous compressible heat-conductive fluids with
              vacuum},
   JOURNAL = {Math. Models Methods Appl. Sci.},
  FJOURNAL = {Mathematical Models and Methods in Applied Sciences},
    VOLUME = {32},
      YEAR = {2022},
    NUMBER = {9},
     PAGES = {1725--1784},
      ISSN = {0218-2025,1793-6314},
   MRCLASS = {35Q30 (35B30 35B40 35B65 76N10)},
  MRNUMBER = {4492674},
       DOI = {10.1142/S0218202522500403},
       URL = {https://doi.org/10.1142/S0218202522500403},
}

@article {MR4294284,
    AUTHOR = {Feireisl, Eduard and Novotn\'y, Anton\'in},
     TITLE = {Navier-{S}tokes-{F}ourier system with general boundary
              conditions},
   JOURNAL = {Comm. Math. Phys.},
  FJOURNAL = {Communications in Mathematical Physics},
    VOLUME = {386},
      YEAR = {2021},
    NUMBER = {2},
     PAGES = {975--1010},
      ISSN = {0010-3616,1432-0916},
   MRCLASS = {35Q30},
  MRNUMBER = {4294284},
       DOI = {10.1007/s00220-021-04091-1},
       URL = {https://doi.org/10.1007/s00220-021-04091-1},
}

@article {MR4978813,
    AUTHOR = {Wen, Huanyao},
     TITLE = {Global wellposedness of compressible {N}avier-{S}tokes
              equations with vacuum and smallness on scaling invariant
              quantity in {$\Bbb{R}^3$}},
   JOURNAL = {Adv. Math.},
  FJOURNAL = {Advances in Mathematics},
    VOLUME = {482},
      YEAR = {2025},
    NUMBER = {part C},
     PAGES = {Paper No. 110628},
      ISSN = {0001-8708},
   MRCLASS = {76N06 (35K65 76N10)},
  MRNUMBER = {4978813},
       DOI = {10.1016/j.aim.2025.110628},
       URL = {https://doi.org/10.1016/j.aim.2025.110628},
}

@article {MR3360663,
    AUTHOR = {Jiu, Quansen and Wang, Yuexun and Xin, Zhouping},
     TITLE = {Remarks on blow-up of smooth solutions to the compressible
              fluid with constant and degenerate viscosities},
   JOURNAL = {J. Differential Equations},
  FJOURNAL = {Journal of Differential Equations},
    VOLUME = {259},
      YEAR = {2015},
    NUMBER = {7},
     PAGES = {2981--3003},
      ISSN = {0022-0396,1090-2732},
   MRCLASS = {35Q35 (35B44)},
  MRNUMBER = {3360663},
       DOI = {10.1016/j.jde.2015.04.007},
       URL = {https://doi.org/10.1016/j.jde.2015.04.007},
}
\end{document}